\documentclass[a4paper, 11pt]{article}
\usepackage{amssymb,amsmath,amsthm,stmaryrd,mathrsfs,mathabx}

\numberwithin{equation}{section}

\newtheorem{thm}{Theorem}[section]
\newtheorem{lemma}[thm]{Lemma}
\newtheorem{cor}[thm]{Corollary}
\newtheorem{rmk}[thm]{Remark}
\newtheorem{defn}[thm]{Definition}
\newtheorem{hypothesis}[thm]{Hypothesis}
\newcommand{\op}[1]{\operatorname{\text{\rm #1}}}

\begin{document}

\title{Constant frequency and the higher regularity of branch sets}
\author{Brian Krummel}
\maketitle

\begin{abstract} 
We consider a two-valued function $u$ that is either Dirichlet energy minimizing, $C^{1,\mu}$ harmonic, or in $C^{1,\mu}$ with an area-stationary graph such that Almgren's frequency (restricted to the singular set) is continuous at a singular point $Y_0$.  As a corollary of recent work of Wickramasekera and the author, if the frequency of $u$ at $Y_0$ equals $1/2+k$ for some integer $k \geq 0$, then the singular set of $u$ is a $C^{1,\tau}$ submanifold and we have estimates on the asymptotic behavior of $u$ at singular points.  Using a nontrivial modification of the argument of Wickramasekera and author, we show that the frequency of $u$ at $Y_0$ cannot equal an integer and therefore must equal $1/2+k$ for some integer $k \geq 0$.  We then use the asymptotic behavior of $u$ and partial Legendre-type transformations based on those of Kinderlehrer, Nirenberg, and Spruck to show that the singular set in this case is in fact real analytic. 
\end{abstract}

\tableofcontents

\section{Introduction} \label{sec:introduction_sec}

We consider the singular set of two-valued functions on a domain $\Omega \subseteq \mathbb{R}^n$.  A two-valued function is a map $u : \Omega \rightarrow \mathcal{A}_2(\mathbb{R}^m)$, where $\mathcal{A}_2(\mathbb{R}^m)$ denotes the space of unordered pairs $\{a_1,a_2\}$ of $a_1,a_2 \in \mathbb{R}^m$ not necessarily distinct.  Given any two-valued function $u : \Omega \rightarrow \mathcal{A}_2(\mathbb{R}^m)$, we let $\mathcal{B}_u$ denote the singular set consisting of all points $Y \in \Omega$ for which there is no $\delta > 0$ such that $u(X) = \{u_1(X),u_2(X)\}$ on $B_{\delta}(Y)$ for some $C^1$ single-valued functions $u_1,u_2 : B_{\delta}(Y) \rightarrow \mathbb{R}^m$.  (Here and throughout we let $B^k_{\rho}(Y)$ denote the open ball in $\mathbb{R}^k$ of center $Y$ and radius $\rho > 0$ and we let $B_{\rho}(Y) = B^n_{\rho}(Y)$ where $n$ is fixed.)  We shall assume the basic facts and notation for two-valued functions used in~\cite{KrumWic1} (in particular, see Sections 2 and 3 of~\cite{KrumWic1}).  We will consider three classes of two-valued functions: \\

\noindent \textbf{Dirichlet energy minimizing two-valued functions.}  $u \in W^{1,2}(\Omega;\mathcal{A}_2(\mathbb{R}^m))$ that are \textit{Dirichlet energy minimizing} in the sense that 
\begin{equation*}
	\int_{\Omega} |Du|^2 \leq \int_{\Omega} |Dv|^2
\end{equation*}
for all $v \in W^{1,2}(\Omega;\mathcal{A}_2(\mathbb{R}^m))$ with $u = v$ a.e. in $\Omega \setminus K$ for some compact set $K \subset \Omega$. \\

\noindent \textbf{$C^{1,\mu}$ harmonic two-valued functions.}  $u \in C^{1,\mu}(\Omega;\mathcal{A}_2(\mathbb{R}^m))$, where $\mu \in (0,1)$, that are locally harmonic in $\Omega \setminus \mathcal{B}_u$ in the sense that for every ball $B \subset \subset \Omega \setminus \mathcal{B}_u$, $u(X) = \{u_1(X),u_2(X)\}$ on $B$ for harmonic single-valued functions $u_1,u_2 : B \rightarrow \mathbb{R}^m$. \\

\noindent \textbf{$C^{1,\mu}$ two-valued functions with area-stationary graphs.}  $u \in C^{1,\mu}(\Omega;\mathcal{A}_2(\mathbb{R}^m))$, where $\mu \in (0,1)$, whose graphs are area-stationary varifolds.  In other words, if we write $u(X) = \{u_1(X),u_2(X)\}$ for $X \in \Omega$, then the graph $\mathcal{M}$ of $u$ is the set of all points $(X,u_1(X))$ and $(X,u_2(X))$ for $X \in \Omega$ and can be regarded as an integral varifold with multiplicity function $\theta : \mathcal{M} \rightarrow \mathbb{Z}_+$ given by $\theta(X,u_1(X)) = \theta(X,u_2(X)) = 1$ whenever $u_1(X) \neq u_2(X)$ and $\theta(X,u_1(X)) = 2$ whenever $u_1(X) = u_2(X)$.  $\mathcal{M}$ is \textit{area-stationary} if 
\begin{equation*}
	\int_{\mathcal{M}} \op{div}_{\mathcal{M}} \zeta \, \theta \, d\mathcal{H}^n = 0
\end{equation*}
for all $\zeta \in C^1_c(\Omega \times \mathbb{R}^m;\mathbb{R}^{n+m})$. \\

Such multivalued functions arise in the study of minimal submanifolds at branch point singularities, i.e. singular points where at least one tangent cone is a plane of integer multiplicity $\geq 2$.  In particular, Almgren in his work posthumously published as~\cite{Alm83} used Dirichlet energy minimizing multivalued functions as approximations of area minimizing $n$-currents in order to show that the singular set of such currents has Hausdorff dimension at most $n-2$.  Later work by De Lellis and Spadaro of~\cite{DeLSpa} provides a more accessible proof of Almgren's result.  Wickramasekera in~\cite{Wic08} and~\cite{Wic} used $C^{1,\mu}$ two-valued harmonic functions as approximations of orientable, stable, minimal hypersurfaces close to a multiplicity two hyperplane in order to show that such hypersurfaces are locally the graphs of $C^{1,\mu'}$ two-valued functions over the plane for some $\mu' < \mu$.  Then $C^{1,\mu}$ two-valued harmonic functions were used as approximations of area-stationary $C^{1,\mu'}$ two-valued graphs in any codimension firstly by Simon and Wickramasekera in~\cite{SimWic11} to show that the singular set of such two-valued graphs has Hausdorff dimension at most $n-2$ and later by Wickramasekera and the author in~\cite{KrumWic2} to show that in any ball the singular set of such two-valued graphs is the finite union of locally $(n-2)$-rectifiable sets. 

Almgren~\cite{Alm83} showed in the case of Dirichlet energy minimizing multivalued functions that $\mathcal{B}_u$ has Hausdorff dimension at most $n-2$.  To accomplish this, Almgren introduced the frequency function 
\begin{equation} \label{defn_freqfn}
	N_{u,Y}(\rho) = \frac{\rho^{2-n} \int_{B_{\rho}(Y)} |Du|^2}{\rho^{1-n} \int_{\partial B_{\rho}(Y)} |u|^2} 
\end{equation}
for a Dirichlet energy minimizing two-valued function $u \in W^{1,2}(\Omega;\mathcal{A}_2(\mathbb{R}^m))$, $Y \in \Omega$, and $\rho \in (0,\op{dist}(Y,\partial \Omega))$.  Almgren showed that $N_{u,Y}(\rho)$ is monotone nondecreasing as a function of $\rho$ and additionally introduced the frequency $\mathcal{N}_u(Y) = \lim_{\rho \downarrow 0} N_{u,Y}(\rho)$ of $u$ at $Y$.  We regard $\mathcal{N}_u(Y)$ as quantifying the rate at which $u$ decays at $Y$ (see~\cite[Lemma 3.2(a)(b)]{KrumWic1}).  Simon and Wickramasekera later extended Almgren's arguments in~\cite{SimWic11} to show that in the cases of $C^{1,\mu}$ harmonic two-valued functions and $C^{1,\mu}$ two-valued functions with area-stationary graphs, $\mathcal{B}_u$ has Hausdorff dimension at most $n-2$.  Modifying arguments of Simon in~\cite{Sim93}, which were originally applied to multiplicity one classes of minimal submanifolds, and using Almgren's frequency function, Wickramasekera and the author in~\cite{KrumWic1} and~\cite{KrumWic2} showed that for all three cases of two-valued functions above, in any ball $\mathcal{B}_u$ is the finite union of locally $(n-2)$-rectifiable sets.  We will show that:

\begin{thm}
For all three classes of two-valued functions above, there is a relatively open dense subset of $\mathcal{B}_u$ on which $\mathcal{B}_u$ is a locally real analytic $(n-2)$-dimensional submanifold. 
\end{thm}

Examples show that for the three classes of multivalued functions $u$ that we consider, $\mathcal{B}_u$ is not a smoothly embedded $(n-2)$-dimensional submanifold in general.  Consider $u \in C^{1,1/2}(\mathbb{R}^4;\mathcal{A}_2(\mathbb{R}^2))$ given by $u(z,w) = \{ \pm (z^2 - w^4)^{3/2} \}$ for $z,w \in \mathbb{C}$, where here we identify $\mathbb{C} \cong \mathbb{R}^2$.  Since $u$ is holomorphic, $u$ is Dirichlet energy minimizing and the graph of $u$ is area-minimizing.  In this particular example, $\mathcal{B}_u$ is the holomorphic variety $\{z = -w^2\} \cup \{z = w^2\}$ consisting of locally real analytic $2$-dimensional submanifolds away from the origin along which $\mathcal{N}_u$ is constant $3/2$ and the origin at which $\mathcal{N}_u(0,0) = 3$.  Since one can regard the singular set of two-valued harmonic functions as analogous to the singular part of the nodal set of harmonic single-valued functions, the latter which is known to be a real analytic variety, one might conjecture that for the three classes of two-valued functions that we consider, $\mathcal{B}_u$ is a real analytic variety.  Motivated by examples like the one above, we will show that $\mathcal{B}_u$ consists of real analytic $(n-2)$-dimensional submanifolds along which $\mathcal{N}_u |_{\mathcal{B}_u}$ is constant and a relatively closed set $\mathcal{J}_u$ of points at which $\mathcal{N}_u |_{\mathcal{B}_u}$ is not lower semicontinuous.  Moreover, one can further conjecture that $\mathcal{J}_u$ is small in the sense that $\mathcal{H}^{n-2}(\mathcal{J}_u) = 0$ or $\dim \mathcal{J}_u \leq n-3$, though we will not prove that here.

Suppose that either (a) $u \in W^{1,2}(\Omega;\mathcal{A}_2(\mathbb{R}^m))$ is a Dirichlet energy minimizing two-valued function or (b) $u \in C^{1,\mu}(\Omega;\mathcal{A}_2(\mathbb{R}^m))$ is locally harmonic on $\Omega \setminus \mathcal{B}_u$.  (The case where $u \in C^{1,\mu}(\Omega;\mathcal{A}_2(\mathbb{R}^m))$ with an area-stationary graph has some important differences and will be discussed below.)  In these cases we may assume that $u$ is nonzero and symmetric, i.e. $u(X) = \{-u_1(X),+u_1(X)\}$ for some $u_1(X) \in \mathbb{R}^m$ for every $X \in \Omega$, so that $\mathcal{B}_u \subseteq \{ X : u(X) = \{0,0\} \}$ in case (a) and $\mathcal{B}_u \subseteq \{ X : u(X) = \{0,0\}, \, Du(X) = \{0,0\} \}$ in case (b).  Here we prove the following: 

\begin{thm} \label{maincor}
Suppose that either (a) $u \in W^{1,2}(\Omega;\mathcal{A}_2(\mathbb{R}^m))$ is a nonzero, symmetric, Dirichlet energy minimizing two-valued function or (b) $u \in C^{1,\mu}(\Omega;\mathcal{A}_2(\mathbb{R}^m))$, where $\mu \in (0,1)$, is a nonzero, symmetric two-valued function that is locally harmonic in $\Omega \setminus \mathcal{B}_u$.  Then 
\begin{equation*}
	\mathcal{B}_u = \bigcup_{k=0}^{\infty} \Gamma_{u,1/2+k} \cup \mathcal{J}_u
\end{equation*}
where: 
\begin{enumerate}
\item[(i)] For each $k = 0,1,2,\ldots$, $\Gamma_{u,1/2+k}$ is defined to be the set of points $Y \in \mathcal{B}_u$ for which there exists $\sigma > 0$ such that $\mathcal{N}_u(X) = 1/2+k$ for all $X \in \mathcal{B}_u \cap B_{\sigma}(Y)$.  In case (b), $\Gamma_{u,1/2} = \emptyset$.  $\Gamma_{u,1/2+k}$ is a locally real analytic $(n-2)$-dimensional submanifold for all $k$. 

\item[(ii)] $\mathcal{J}_u$ is the set of all points $Y \in \mathcal{B}_u$ for which there exists an integer $k \geq 0$ (depending on $Y$) such that $Y$ is a limit point of $\Gamma_{u,1/2+k}$ but $\mathcal{N}_u(Y) > 1/2+k$.  
\end{enumerate}
\end{thm}

Note that Theorem \ref{maincor} does not say anything about the size or local structure of $\mathcal{J}_u$; rather the theorem says that $\mathcal{B}_u \setminus \mathcal{J}_u$ is a relatively open dense subset of $\mathcal{B}_u$.  From standard stratification theory, we know that $\mathcal{J}_u$ is the union of a set of Hausdorff dimension $n-3$ and the subset of points in $\mathcal{J}_u$ where at least one blow-up is ``cylindrical'', i.e. translation invariant along an $(n-2)$-dimensional linear subspace, the latter set potentially having Hausdorff dimension $n-2$.  One can consider this analogous to the situation with the singular set of stationary integral varifolds, where by Allard regularity the set of regular points is an open dense subset of the support of the varifold and the singular set in principle could have Hausdorff dimension $n$ due to branch point singularities near which at least one tangent cone is an integer multiplicity hyperplane.  Very little is known about branch point singularities in general.  There are some results in special cases are obtained in~\cite{Alm83}~\cite{DeLSpa} for area-minimizing currents and in~\cite{Wic08}~\cite{Wic}~\cite{SimWic11}~\cite{KrumWic2} for stable minimal hypersurfaces near branch points with a multiplicity two tangent plane.  In the study stable minimal hypersurfaces near branch points with a tangent plane of multiplicity $\geq 3$, one must similarly consider the set of singular points where density restricted to the singular set is not lower semicontinuous.  Consequently, we regard establishing the size of $\mathcal{J}_u$ as a challenging problem for future study and do not consider the problem further here.

The proof of Theorem \ref{maincor} is broken up into several theorems as follows.  An easy consequence of the work of Wickramasekera and the author in~\cite{KrumWic1}, whose proof we discuss in Section~\ref{sec:prelims_sec}, is:

\begin{thm} \label{constfreqthm2}
Suppose that either (a) $u \in W^{1,2}(\Omega;\mathcal{A}_2(\mathbb{R}^m))$ is a symmetric Dirichlet energy minimizing two-valued function or (b) $u \in C^{1,\mu}(\Omega;\mathcal{A}_2(\mathbb{R}^m))$, where $\mu \in (0,1]$, is a two-valued function that is locally harmonic on $\Omega \setminus \mathcal{B}_u$.  Suppose that $Y_0 \in \mathcal{B}_u$ such that $\mathcal{N}_u(Y_0) = 1/2+k$ for some integer $k \geq 0$ ($k \geq 1$ in case (b)) and 
\begin{equation} \label{cft2_eqn1}
	\mathcal{N}_u(Y) \geq \mathcal{N}_u(Y_0) \text{ for all } Y \in \mathcal{B}_u.
\end{equation}
Then for some $\tau = \tau(n,m) \in (0,1)$ and $\rho_0 = \rho_0(n,m,u) \in (0,\op{dist}(Y_0,\partial \Omega)/2)$, 
\begin{enumerate}
\item[(i)] $\mathcal{N}_u(Y) = 1/2+k$ for all $Y \in \mathcal{B}_u \cap B_{\rho_0}(Y_0)$, 
\item[(ii)] for every $Y \in \mathcal{B}_u \cap B_{\rho_0}(Y_0)$ there is a unique homogeneous degree $1/2+k$ two-valued function $\varphi_Y : \mathbb{R}^n \rightarrow \mathcal{A}_2(\mathbb{R}^m)$ such that for some rotation $q_Y$ of $\mathbb{R}^n$ and $c_Y = (c^1_Y,c^2_Y,\ldots,c^m_Y) \in \mathbb{C}^m$ (with $c_Y \cdot c_Y = \sum_{\kappa=1}^m (c^{\kappa}_Y)^2 = 0$ in case (a)), 
\begin{equation*}
	\varphi_Y(q_Y^{-1} X) = \{ \pm \op{Re}(c_Y (x_1+ix_2)^{1/2+k}) \}
\end{equation*}
and 
\begin{equation*}
	\rho^{-n} \int_{B_{\rho}(0)} \mathcal{G}(u(Y+X),\varphi_Y(X))^2 dX \leq C \rho^{1+2k+2\tau} 
\end{equation*}
for all $\rho \in (0,\rho_0)$ and some constant $C = C(n,m,u) \in (0,\infty)$, where $\mathcal{G}$ is the metric on the space $\mathcal{A}_2(\mathbb{R}^m)$ of unordered pairs of $\mathbb{R}^m$ defined by 
\begin{equation} \label{metricGdefn}
	\mathcal{G}(\{a_1,a_2\},\{b_1,b_2\}) = \min \left\{ \sqrt{|a_1-b_1|^2 + |a_2-b_2|^2}, \sqrt{|a_1-b_2|^2 + |a_2-b_1|^2} \right\}
\end{equation}
for any $a_1,a_2,b_1,b_2 \in \mathbb{R}^m$,  
\item[(iii)] for every $Y,Z \in \mathcal{B}_u \cap B_{\rho_0}(Y_0)$, 
\begin{equation*}
	\int_{B_1(0)} \mathcal{G}(\varphi_Y(X),\varphi_Z(X))^2 dX \leq C |Y-Z|^{2\tau} 
\end{equation*}
for some constant $C = C(n,m,u) \in (0,\infty)$, and 
\item[(iv)] $\mathcal{B}_u$ is a $C^{1,\tau}$ $(n-2)$-dimensional submanifold in $B_{\rho_0}(Y)$.
\end{enumerate}
\end{thm}

The argument of~\cite{KrumWic1} involves an application of a blow-up method original due to Simon~\cite{Sim93}.  To prove Theorem \ref{constfreqthm2} using this blow-up method, one shows that at least one blow-up $\varphi$ of $u$ at $Y_0$ is ``cylindrical'' in the sense that $\mathcal{B}_{\varphi}$ is an $(n-2)$-dimensional subspace along which $\varphi$ is translation invariant and there are ``no small gaps'' in that there is a sufficiently high concentration of points $Z \in \mathcal{B}_u$ (with $\mathcal{N}_u(Z) \geq 1/2+k$) near $\mathcal{B}_{\varphi}$. 

Theorem \ref{constfreqthm2} naturally leads to the question of what happens if (\ref{cft2_eqn1}) holds true but $\mathcal{N}_u(Y_0)$ takes a value other than $1/2+k$ for some integer $k \geq 0$.  Using the theory of frequency functions, in particular Lemma \ref{cbu_dense_lemma} below and the upper semicontinuity of $\mathcal{N}_u$, one concludes that (\ref{cft2_eqn1}) holds true only when $\mathcal{N}_u(Y_0) = k/2$ for some integer $k \geq 1$ ($k \geq 3$ in case (b)).  For the case where $\mathcal{N}_u(Y_0)$ equals an integer, we assume without loss of generality that $\Omega = B_1(0)$ and $Y_0 = 0$ and show: 

\begin{thm} \label{constfreqthm1}
There is no symmetric two-valued function $u$ with either (a) $u \in W^{1,2}(B_1(0);\mathcal{A}_2(\mathbb{R}^m))$ Dirichlet energy minimizing or (b) $u \in C^{1,\mu}(B_1(0);\mathcal{A}_2(\mathbb{R}^m))$, where $\mu \in (0,1]$, harmonic on $B_1(0) \setminus \mathcal{B}_u$ such that $0 \in \mathcal{B}_u$, $\mathcal{N}_u(0)$ is an integer, and 
\begin{equation*}
	\mathcal{N}_u(Z) \geq \mathcal{N}_u(0) \text{ for all } Z \in \mathcal{B}_u.
\end{equation*}
\end{thm}

To prove Theorem \ref{constfreqthm1}, we suppose that $u$ did exist.  We extend the blow-up method of Wickramasekera and the author in~\cite{KrumWic1} to establish the asymptotic behavior of $u$ at its singular points and thereby show that $u$ decomposes into two harmonic single-valued functions near the origin, contradicting $0 \in \mathcal{B}_u$.  However, the blow-ups of $u$ at singular points $Y$ with $\mathcal{N}_u(Y) = \mathcal{N}_u(0)$ equal $\{ \pm h(x) \}$ for some homogeneous degree $\mathcal{N}_u(0)$, harmonic, single-valued polynomial $h : \mathbb{R}^n \rightarrow \mathbb{R}^m$.  Consequently it is possible that the blow-ups to $u$ at such singular points are not cylindrical and/or that the no gap condition fails to hold true.  Inspired by ideas of Wickramasekera in~\cite{Wic14} and Hughes in~\cite{Hughes}, we modify the blow-up method as follows.  Let $\alpha = \mathcal{N}_u(0)$ and $\Phi_{\alpha}$ denote the set of $\varphi = \{ \pm \varphi_1 \}$ where $\varphi_1$ is a homogeneous degree $\alpha$, harmonic, single-valued polynomial $\varphi_1$.  Given $\varphi \in \Phi_{\alpha}$, let the spine of $\varphi$ be the maximal subspace along which $\varphi$ is translation invariant.  Assuming the hypothesis that $u$ is closer to a particular $\varphi \in \Phi_{\alpha}$ than it is to any element of $\Phi_{\alpha}$ whose spine contains the spine of $\varphi$, we establish a De Giorgi-type excess decay lemma for the $L^2$ distance of $u$ to elements of $\Phi_{\alpha}$ using the blow-up method.  We weaken the above hypothesis to $u$ is close to $\varphi \in \Phi_{\alpha}$ using a multiple scales argument similarly used by Wickramasekera in~\cite{Wic14}.  Small gaps are permitted since if $\mathcal{N}_u(X) < \mathcal{N}_u(0)$ then there are no singular points of $u$ near $X$ and thus we can apply standard elliptic estimates for single-valued functions near $X$.

Next, we prove that: 

\begin{thm} \label{analyticity_thm}
Suppose that either (a) $u \in W^{1,2}(\Omega;\mathcal{A}_2(\mathbb{R}^m))$ is a symmetric Dirichlet energy minimizing two-valued function or (b) $u \in C^{1,\mu}(\Omega;\mathcal{A}_2(\mathbb{R}^m))$, where $\mu \in (0,1]$, is a two-valued function that is locally harmonic on $\Omega \setminus \mathcal{B}_u$.  Suppose that $Y_0 \in \mathcal{B}_u$ such that $\mathcal{N}_u(Y_0) = 1/2+k$ for some integer $k \geq 0$ ($k \geq 1$ in case (b)) and (\ref{cft2_eqn1}) holds true.  Then for some $\sigma \in (0,\op{dist}(Y_0,\partial \Omega))$, $\mathcal{B}_u$ is a real analytic $(n-2)$-dimensional submanifold in $B_{\sigma}(Y_0)$.
\end{thm}

To prove Theorem \ref{analyticity_thm}, by taking derivatives of $u$ we can reduce the proof of Theorem \ref{analyticity_thm} to the case of a two-valued function $u$ that is locally harmonic on $\Omega \setminus \mathcal{B}_u$ and is asymptotic to a homogeneous degree $1/2$, harmonic two-valued function at each point on $\mathcal{B}_u$.  The proof of Theorem \ref{analyticity_thm} then follows by using partial Legendre-type transformations based on those of Kinderlehrer, Nirenberg, and Spruck in~\cite{KNS}.  We use the partial Legendre-type transformation $\tilde x_1 = u^1(X)$, $\tilde x_2 = u^2(X)$, and $\tilde x_j = x_j$ if $j = 3,4,\ldots,n$.  One can show that there exists an inverse transformation, which necessarily has the form $x_1 = \phi^1(\tilde X)$, $x_2 = \phi^2(\tilde X)$, and $x_j = \tilde x_j$ if $j = 3,4,\ldots,n$ for some functions $\phi^1$ and $\phi^2$.  Under the partial Legendre-type transformation, $\mathcal{B}_u$ maps into $\{0\} \times \mathbb{R}^{n-2}$ and $\Delta u = 0$ in $B_1(0) \setminus \mathcal{B}_u$ transforms into a singular elliptic differential system of $\phi^1$ and $\phi^2$.  The analyticity of $\mathcal{B}_u$ is then established by an analysis of the regularity of $\phi^1$ and $\phi^2$ and then using the fact that $\mathcal{B}_u$ is the graph of $(\phi^1,\phi^2)$ over $\{0\} \times \mathbb{R}^{n-2}$.  We establish a preliminary regularity result for $\phi^1$ and $\phi^2$ using a difference quotient argument and a Schauder estimate, whose application uses the asymptotic behavior of $u$ at branch points established in Theorem \ref{constfreqthm2}.  The necessary regularity of $\phi^1$ and $\phi^2$ for proving $\mathcal{B}_u$ is locally real analytic follows from iteratively applying the Schauder estimate to obtain precise estimates on derivatives of $\phi^1$ and $\phi^2$ and bounding terms in the Schauder estimates using a method using majorants based on arguments of Friedman in~\cite{Friedman} and the author in~\cite{KrumThesis}. 

Finally, Theorem \ref{maincor} follows from combining Theorems \ref{constfreqthm2}, \ref{constfreqthm1}, and \ref{analyticity_thm} and using some facts about frequency, in particular the upper semicontinuity of frequency and Lemma \ref{cbu_dense_lemma} below.

Theorems \ref{maincor},  \ref{constfreqthm2}, \ref{constfreqthm1}, and \ref{analyticity_thm} can all be extended to the case of a two-valued function $u \in C^{1,\mu}(\Omega;\mathcal{A}_2(\mathbb{R}^m))$, where $\mu \in (0,1)$, such that $\mathcal{M} = \op{graph} u$ is area-stationary.  At each branch point $p$ of $\mathcal{M}$, write $\mathcal{M}$ as the graph of a $C^{1,\mu}$ two-valued function $\tilde u_p$ over the tangent plane to $\mathcal{M}$ at $p$.  Write $\tilde u_p(X) = \{\tilde u_{p,1}(X),\tilde u_{p,2}(X)\}$ at each $X$ and let $\tilde u_{p,a}(X) = (\tilde u_{p,1}(X)+\tilde u_{p,2}(X))/2$ and $\tilde u_{p,s}(X) = \{ \pm (\tilde u_{p,1}(X)-\tilde u_{p,2}(X))/2 \}$ so that $\tilde u_p = \tilde u_{p,a} + \tilde u_{p,s}$.  We define the frequency $\mathcal{N}_{\mathcal{M}}(p)$ of $\mathcal{M}$ at $p$ by $\mathcal{N}_{\mathcal{M}}(p) = \lim_{\rho \downarrow 0} \mathcal{N}_{\tilde u_{p,s},0}(\rho)$, where $N_{\tilde u_{p,s},0}(\rho)$ is given by (\ref{defn_freqfn}) with $\tilde u_{p,s}$ and $0$ in place of $u$ and $Y$ respectively.  The existence of the limit $\mathcal{N}_{\mathcal{M}}(p)$ is justified by~\cite{SimWic11} and~\cite{KrumWic2}.  The analogue to Theorem \ref{maincor} for area-stationary graphs, which follows by combining the analogues of Theorems \ref{constfreqthm2}, \ref{constfreqthm1}, and \ref{analyticity_thm}, is:

\begin{thm} \label{maincor_mss}
Suppose that $u \in C^{1,\mu}(\Omega;\mathcal{A}_2(\mathbb{R}^m))$, where $\mu \in (0,1)$, is a two-valued function whose graph $\mathcal{M}$ is area-stationary.  Then the set of points of $\mathcal{M}$ near which $\mathcal{M}$ is not locally the union of smoothly embedded submanifolds equals 
\begin{equation*}
	\op{sing} \mathcal{M} = \bigcup_{k=1}^{\infty} \Gamma_{\mathcal{M},1/2+k} \cup \mathcal{J}_{\mathcal{M}}
\end{equation*}
where: 
\begin{enumerate}
\item[(i)] For each $k = 1,2,3,\ldots$, $\Gamma_{\mathcal{M},1/2+k}$ is defined to be the set of points $p_0$ of $\mathcal{M}$ for which there exists $\sigma > 0$ such that $\mathcal{N}_{\mathcal{M}}(p) = 1/2+k$ for all branch points $p$ in $B^{n+m}_{\sigma}(p_0)$.  $\Gamma_{\mathcal{M},1/2+k}$ is a locally real analytic $(n-2)$-dimensional submanifold for all $k$. 

\item[(ii)] $\mathcal{J}_{\mathcal{M}}$ is the set of all branch points $p$ of $\mathcal{M}$ for which there exists an integer $k \geq 1$ (depending on $p$) such that $p$ is a limit point of $\Gamma_{\mathcal{M},1/2+k}$ but $\mathcal{N}_{\mathcal{M}}(p) > 1/2+k$.  
\end{enumerate}
\end{thm}

In particular, as a direct consequence of~\cite[Theorem 1.5]{Wic08}, we have: 

\begin{cor} \label{maincor_stablevar}
There exist $\varepsilon, \sigma \in (0,1)$ depending only on $n$ such that the following holds true.  Let $V$ be a stationary integral varifold of $B^{n+1}_2(0)$ in the varifold closure of orientable immersed stable minimal hypersurfaces $\mathcal{M}$ with $0 \in \overline{\mathcal{M}}$, $\mathcal{H}^{n-2}(\op{sing} \mathcal{M}) < \infty$, and $\frac{\mathcal{H}^n(\mathcal{M})}{\omega_n 2^n} \leq 2+\varepsilon$.  Let $\op{reg} V$ denote the set of points of $\op{spt} \|V\|$ near which $\op{spt} \|V\|$ is a union of embedded hypersurfaces (possibly intersecting) and $\op{sing} V = \op{spt} \|V\| \setminus \op{reg} V$.  Then
\begin{equation*}
	\op{sing} V \cap B^{n+1}_{\sigma}(0) = \bigcup_{k=0}^{\infty} \Gamma_{V,1/2+k} \cup \mathcal{J}_V \cup \mathcal{S}_V 
\end{equation*}
where:
\begin{enumerate}
\item[(i)] For each $k = 1,2,3,\ldots$, $\Gamma_{V,1/2+k}$ is the set of branch points $p_0$ of $V$ at which the unique tangent cone to $V$ is a multiplicity two hyperplane and there exists $\delta > 0$ such that $\mathcal{N}_{\mathcal{M}}(p) = 1/2+k$ for all branch points $p$ of $V$ in $B^{n+m}_{\delta}(p_0)$.  $\Gamma_{\mathcal{M},1/2+k}$ is a locally real analytic $(n-2)$-dimensional submanifold for all $k$. 

\item[(ii)] $\mathcal{J}_{\mathcal{M}}$ is the set of all branch points $p$ of $\mathcal{M}$ at which the unique tangent cone to $V$ is a multiplicity two hyperplane and for which there exists an integer $k \geq 1$ (depending on $p$) such that $p$ is a limit point of $\Gamma_{\mathcal{M},1/2+k}$ but $\mathcal{N}_{\mathcal{M}}(p) > 1/2+k$.  

\item[(iii)] $\mathcal{S}_V$ is a relatively closed subset of $\op{spt} \|V\| \cap B^{n+1}_{\sigma}(0)$ such that $\mathcal{S}_V = \emptyset$ if $n \leq 6$, $\mathcal{S}_V$ is discrete if $n = 7$, and the Hausdorff dimension of $\mathcal{S}_V$ is at most $n-7$ if $n \geq 8$.
\end{enumerate}
\end{cor}

\begin{rmk}
In~\cite{Wic}, which is currently in preparation, the a priori regularity assumption (i.e. $\mathcal{H}^{n-2}(\op{sing} \mathcal{M}) < \infty$) of~\cite{Wic08} is removed.  Specifically, Wickramasekera shows that there exist $\varepsilon, \sigma \in (0,1)$ depending only on $n$ such that if $V$ is a stationary integral $n$-varifold of $B^{n+1}_2(0)$ such that $0 \in \op{spt} \|V\|$, $\frac{\|V\|(B^{n+1}_2(0))}{\omega_n 2^n} \leq 2+\varepsilon$, $\op{reg} V$ is stable in the sense that 
\begin{equation*}
	\int_{\op{reg} V} |\nabla_{\op{reg} V} \zeta|^2 \leq \int_{\op{reg} V} |A_{\op{reg} V}|^2 \zeta^2 
\end{equation*}
for all $\zeta \in C^1_c(\op{reg} V)$ where $\nabla_{\op{reg} V}$ denotes the gradient of $\op{reg} V$ and $A_{\op{reg} V}$ denotes the second fundamental form of $\op{reg} V$, and there is no singular point of $V$ at which the tangent cone to $V$ is the sum of three half-hyperplanes meeting at equal angles along a common boundary, then the conclusion of~\cite[Theorem 1.5]{Wic08} holds true.  Hence for such a $V$ the conclusion of Corollary \ref{maincor_stablevar} also holds true. 
\end{rmk}

As an easy consequence of~\cite{KrumWic2} we obtain the analogue of Theorem \ref{constfreqthm2} of: 

\begin{thm} \label{constfreqthm2_mss}
Suppose that $u \in C^{1,\mu}(\Omega;\mathcal{A}_2(\mathbb{R}^m))$, where $\mu \in (0,1)$, is a two-valued function whose graph $\mathcal{M}$ is area-stationary.  Suppose that $p_0 \in \op{sing} \mathcal{M}$ such that $\mathcal{N}_{\mathcal{M}}(p_0) = 1/2+k$ for some integer $k \geq 1$ and 
\begin{equation} \label{cft2_eqn1_mss}
	\mathcal{N}_{\mathcal{M}}(p) \geq \mathcal{N}_{\mathcal{M}}(p_0) \text{ for all } p \in \op{sing} \mathcal{M}.
\end{equation}
Then for some $\tau = \tau(n,m) \in (0,1)$ and $\rho_0 = \rho_0(n,m,\mathcal{M}) \in (0,\op{dist}(Y,\partial \Omega)/2)$, 
\begin{enumerate}
\item[(i)] $\mathcal{N}_{\mathcal{M}}(p) = 1/2+k$ for all $p \in \op{sing} \mathcal{M} \cap B^{n+m}_{\rho_0}(p_0)$, 
\item[(ii)] for every $p \in \op{sing} \mathcal{M} \cap B_{\rho_0}(p_0)$ there is a unique homogeneous degree $1/2+k$ two-valued function $\varphi_p : T_p \mathcal{M} \rightarrow \mathcal{A}_2(T_p \mathcal{M}^{\perp})$ and there exists an orthogonal map $q_p : T_p \mathcal{M} \rightarrow \mathbb{R}^n$ and $c_p \in T_p \mathcal{M}^{\perp} \otimes \mathbb{C}$ such that 
\begin{equation*}
	\varphi_p(q_p^{-1} X) = \{ \pm \op{Re}(c_p (x_1+ix_2)^{1/2+k}) \}
\end{equation*}
and 
\begin{equation*}
	\rho^{-n} \int_{B_{\rho}(0)} \mathcal{G}\left( \frac{\tilde u_{p,s}(X)}{\|u_s\|_{L^2(B_{2\rho_0}(0))}}, \varphi_Z(X) \right)^2 dX 
	\leq C \rho^{1+2k+2\tau} 
\end{equation*}
for all $\rho \in (0,\rho_0)$ and some constant $C = C(n,m,\mathcal{M}) \in (0,\infty)$, and for every $p,p' \in \op{sing} \mathcal{M} \cap B_{\rho_0}(p_0)$, 
\begin{equation*}
	|q_p - q_{p'}| \leq C |p - p'|^{\tau} \text{ and } |c_p - c_{p'}| \leq C |p - p'|^{\tau} 
\end{equation*}
for some constant $C = C(n,m,\mathcal{M}) \in (0,\infty)$, and 
\item[(iii)] $\op{sing} \mathcal{M}$ is a $C^{1,\tau}$ $(n-2)$-dimensional submanifold in $B_{\rho_0}(0) \times \mathbb{R}^m$.
\end{enumerate}
\end{thm}

By modifying the proof of Theorem \ref{constfreqthm1} using ideas from~\cite{KrumWic2}, we obtain: 

\begin{thm} \label{constfreqthm1_mss}
There is no two-valued function $u \in C^{1,\mu}(\Omega;\mathcal{A}_2(\mathbb{R}^m))$, where $\mu \in (0,1)$, whose graph $\mathcal{M}$ is area-stationary such that $0 \in \mathcal{B}_u$, $u(0) = 0$, $Du(0) = 0$, $\mathcal{N}_{\mathcal{M}}(0)$ is an integer and 
\begin{equation*}
	\mathcal{N}_u(p) \geq \mathcal{N}_u(0) \text{ for all } p \in \op{sing} \mathcal{M}.
\end{equation*}
\end{thm}

By extending the arguments leading to Theorem \ref{analyticity_thm}, we will show that: 

\begin{thm} \label{analyticity_thm_mss}
Suppose that $u \in C^{1,\mu}(\Omega;\mathcal{A}_2(\mathbb{R}^m))$, where $\mu \in (0,1)$, is a two-valued function whose graph $\mathcal{M}$ is area-stationary.  Suppose that $p_0 \in \op{sing} \mathcal{M}$ such that $\mathcal{N}_{\mathcal{M}}(p_0) = 1/2+k$ for some integer $k \geq 1$ and (\ref{cft2_eqn1_mss}) holds true.  Then for some $\sigma \in (0,\op{dist}(p_0,\partial \Omega \times \mathbb{R}^m)/2)$, $\op{sing} \mathcal{M}$ is a real analytic $(n-2)$-dimensional submanifold in $B^{n+m}_{\sigma}(p_0)$.
\end{thm}

Note that since the derivatives of $u$ do not solve the minimal surface system, to prove Theorem \ref{analyticity_thm_mss} we consider a differential system satisfied by $u$ and itself derivatives.  In fact, we prove Theorems \ref{analyticity_thm} and \ref{analyticity_thm_mss} jointly by considering an abstract differential system and we gain some additional information about the structure of $u$ near singular points from this approach.

\section{Frequency functions and blow-ups} \label{sec:prelims_sec}

\noindent \textbf{Dirichlet energy minimizing two-valued functions.}  Let $\Omega$ be an open subset of $\mathbb{R}^n$.  We consider Dirichlet energy minimizing two-valued function $u \in W^{1,2}(\Omega;\mathcal{A}_2(\mathbb{R}^m))$.  We assume that $u$ is symmetric in that at each $X \in \Omega$, $u(X) = \{ -u_1(X),+u_1(X) \}$ for some $u_1(X) \in \mathbb{R}^m$ and that $u$ is nonzero. 

Recall that given a nonzero, symmetric, Dirichlet energy minimizing two-valued function $u \in W^{1,2}(\Omega;\mathcal{A}_2(\mathbb{R}^m))$ and $Y \in \Omega$, we define the \textit{frequency function} $N_{u,Y} : (0,\op{dist}(Y,\partial \Omega)) \rightarrow (0,\infty)$ by (\ref{defn_freqfn}).  $N_{u,Y}$ is well-defined and monotone nondecreasing~\cite[Theorem 2.6(10)]{Alm83}.  Moreover, $N_{u,Y} \equiv \alpha$ constant if and only if $u(Y+X)$ is a homogeneous degree $\alpha$ function of $X$.  By the monotonicity of $N_{u,Y}$, we can define the \textit{frequency} $\mathcal{N}_u(Y)$ of $u$ at $Y \in \Omega$ by $\mathcal{N}_u(Y) = \lim_{\rho \downarrow 0} N_{u,Y}(\rho)$.  Also by the monotonicity of frequency functions, frequency has the following upper-semicontinuity property:

\begin{lemma}[Upper semicontinuity of frequency] \label{usc_freq_dm}
Suppose that $u_j,u \in W^{1,2}(\Omega;\mathcal{A}_2(\mathbb{R}^m))$ are nonzero, symmetric, Dirichlet energy minimizing two-valued functions such that $u_j \rightarrow u$ in $L^2$ on compact subsets of $\Omega$ and suppose that $Y_j,Y \in \Omega$ such that $Y_j \rightarrow Y$.  Then 
\begin{equation*}
	\mathcal{N}_u(Y) \geq \limsup_{j \rightarrow \infty} \mathcal{N}_{u_j}(Y_j). 
\end{equation*}
\end{lemma}

For each Dirichlet energy minimizing two-valued function $u \in W^{1,2}(\Omega;\mathcal{A}_2(\mathbb{R}^m))$, let $\Sigma_u$ be the set of $Y \in \Omega$ such that $u(Y) = \{0,0\}$ but $u \not\equiv \{0,0\}$ near $Y$.  Since $u$ is continuous by~\cite[Theorem 2.13]{Alm83}, $\mathcal{B}_u \subseteq \Sigma_u$.  Clearly $\mathcal{N}_u(Y)$ is a positive integer for all $Y \in \Sigma_u \setminus \mathcal{B}_u$, whereas $\mathcal{N}_u(Y)$ can take both integer and non-integer values for $Y \in \mathcal{B}_u$.  By~\cite[Theorem 2.13]{Alm83}, there exists $\mu = \mu(n,m) \in (0,1)$ such that $u \in C^{0,\mu}(\Omega;\mathcal{A}_2(\mathbb{R}^m))$ and consequently (see~\cite[Lemma 3.1]{KrumWic1}) $\mathcal{N}_u(Y) \geq \mu$ for all $Y \in \Sigma_u$.

Given a nonzero, symmetric, Dirichlet energy minimizing two-valued function $u \in W^{1,2}(\Omega;\mathcal{A}_2(\mathbb{R}^m))$ and $Y \in \Sigma_u$, we define a \textit{blow-up} $\varphi \in W^{1,2}(\mathbb{R}^n;\mathcal{A}_2(\mathbb{R}^m))$ of $u$ at $Y$ to be any limit of the form 
\begin{equation} \label{defn_blowup}
	\varphi = \lim_{j \rightarrow \infty} u_{Y,\rho_j} \quad \text{letting} \quad  
	u_{Y,\rho}(X) = \frac{u(Y + \rho X)}{\rho^{-n/2} \|u\|_{L^2(B_{\rho}(Y))}} \text{ for } \rho > 0, 
\end{equation}
where $\rho_j \rightarrow 0^+$ and the limit is taken in $W^{1,2}$ on compact subsets of $\mathbb{R}^n$.  By the compactness properties of Dirichlet energy minimizing multivalued functions and monotonicity of frequency functions (see~\cite[Lemma 3.1(b)]{KrumWic1}), there exists at least one blow-up $\varphi$ of $u$ at each $Y$; however, the blow-up $\varphi$ might not be unique independent of the particular sequence $\rho_j$.  Every blow-up $\varphi$ of $u$ at $Y$ is a nonzero, homogeneous degree $\mathcal{N}_u(Y)$, locally Dirichlet energy minimizing two-valued function. 

For every nonzero, homogeneous, two-valued function $\varphi \in W^{1,2}(\mathbb{R}^n;\mathcal{A}_2(\mathbb{R}^m))$ that is Dirichlet energy minimizing, we define the \textit{spine} $S(\varphi)$ of $\varphi$ by 
\begin{equation} \label{defn_spine}
	S(\varphi) = \{ Z \in \mathbb{R}^n : \mathcal{N}_{\varphi}(Z) = \mathcal{N}_{\varphi}(0) \}. 
\end{equation}
Note that $S(\varphi)$ is always a linear subspace and $\varphi(Z+X) = \varphi(X)$ for all $X \in \mathbb{R}^n$ and $Z \in S(\varphi)$. 

Let $u \in W^{1,2}(\Omega;\mathcal{A}_2(\mathbb{R}^m))$ be a nonzero, symmetric, Dirichlet energy minimizing two-valued function.  For $j = 1,2,\ldots,n-2$, let $\Sigma^{(j)}_u$ denote the set of all $Y \in \Sigma_u$ such that $\dim S(\varphi) \leq j$ for every blow-up $\varphi$ of $u$ at $Y$ and let $\mathcal{B}^{(j)}_u = \mathcal{B}_u \cap \Sigma^{(j)}_u$.  Then 
\begin{equation*}
	\Sigma_u = \Sigma^{(n-2)}_u \supseteq \Sigma^{(n-3)}_u \supseteq \cdots \supseteq \Sigma^{(1)}_u \supseteq \Sigma^{(0)}_u 
\end{equation*}
and the Hausdorff dimension of $\Sigma^{(j)}_u$ is at most $j$ (see~\cite[Lemma 3.2]{KrumWic1}).  $\Sigma_u \setminus \Sigma^{(n-3)}_u$ consists of the set of all $Y \in \Sigma_u$ such that at least one blow-up $\varphi$ of $u$ at $Y$ has $\dim S(\varphi) = n-2$.  After rotating so that $S(\varphi) = \{0\} \times \mathbb{R}^{n-2}$, $\varphi$ has the form $\varphi(X) = \{ \pm \op{Re}(c(x_1+ix_2)^{k/2}) \}$ for some $c = (c^1,c^2,\ldots,c^m) \in \mathbb{C}^m$ with $\sum_{\kappa=1}^m (c^{\kappa})^2 = 1$ by~\cite{MicWhi94} and integer $k \geq 1$.  Consequently $\mathcal{N}_u(Y) = k/2$. 

\begin{lemma} \label{cbu_dense_lemma}
Let $u \in W^{1,2}(\Omega;\mathcal{A}_2(\mathbb{R}^m))$ be a nonzero, symmetric, Dirichlet energy minimizing two-valued function.  Then $\mathcal{B}_u \setminus \mathcal{B}_u^{(n-3)}$ is dense in $\mathcal{B}_u$. 
\end{lemma}
\begin{proof}
Given an open ball $B \subseteq \Omega$, by the appendix of~\cite{SimWic11} if $\mathcal{H}^{n-2}(\mathcal{B}_u \cap B) = 0$ then $B \setminus \mathcal{B}_u$ is simply-connected and thus $u$ decomposes into two or more harmonic single-valued functions on $B$.  Hence either $\mathcal{B}_u \cap B = \emptyset$ or $\mathcal{H}^{n-2}(\mathcal{B}_u \cap B) > 0$.  $\mathcal{B}_u^{(n-3)}$ has Hausdorff dimension at most $n-3$, so $\mathcal{B}_u \setminus \mathcal{B}_u^{(n-3)}$ is dense in $\mathcal{B}_u$ as claimed. 
\end{proof}

As a consequence of Lemma \ref{cbu_dense_lemma}, $\mathcal{N}_u(Y) \geq 1/2$ for every nonzero, symmetric, Dirichlet energy minimizing two-valued function $u \in W^{1,2}(\Omega;\mathcal{A}_2(\mathbb{R}^m))$ and $Y \in \Sigma_u$.  Note that Lemma \ref{cbu_dense_lemma} is specific to the case of two-valued functions since in the case of $q$-valued function for $q \geq 3$, $\mathcal{B}_u$ is not necessarily contained in the set of points $Y$ where $u(Y) = \{0,0,\ldots,0\}$. 

\begin{proof}[Proof of Theorem \ref{constfreqthm2} for case (a)]
First we claim that $\dim S(\varphi) = n-2$ for every blow-up $\varphi$ of $u$ at $Y_0$.  Write $\varphi = \lim_{j \rightarrow \infty} u_{Y_0,\rho_j}$ with convergence in locally in $W^{1,2}$ for $\rho_j \downarrow 0$ and let $Z \in \mathcal{B}_{\varphi}$.  Clearly $\mathcal{N}_{\varphi}(Z) \leq \mathcal{N}_u(Y_0)$.  For every $\delta > 0$, $B_{\delta}(Z) \cap \mathcal{B}_{u_{Y_0,\rho_j}} \neq \emptyset$ for $j$ sufficiently large since then otherwise $u_{Y_0,\rho_j}$ decomposes into two harmonic single-valued functions on $B_{\delta}(Z)$ and then by the Schauder estimates $\varphi$ decomposes into two harmonic single-valued functions on $B_{\delta/2}(Z)$, contradicting $Z \in \mathcal{B}_{\varphi}$.  Therefore there exists $Z_j \in \mathcal{B}_{u_{0,\rho_j}}$ such that $Z_j \rightarrow Z$.  By (\ref{cft2_eqn1}) and the upper semicontinuity of frequency, 
\begin{equation*}
	\mathcal{N}_{\varphi}(Z) 
	\geq \limsup_{j \rightarrow \infty} \mathcal{N}_{u_{Y_0,\rho_j}}(Z_j) 
	= \lim_{j \rightarrow \infty} \mathcal{N}_u(Y_0+\rho_j Z_j) 
	\geq \mathcal{N}_u(Y_0). 
\end{equation*}
Therefore $\mathcal{N}_{\varphi}(Z) = \mathcal{N}_u(Y_0)$ for all $Z \in \mathcal{B}_{\varphi}$, i.e. $\mathcal{B}_{\varphi} = S(\varphi)$.  Since $\mathcal{N}_{\varphi}(0) = 1/2+k$, $0 \in \mathcal{B}_{\varphi}$.  If $\dim S(\varphi) \leq n-3$, then $\mathbb{R}^n \setminus S(\varphi)$ is simply-connected and thus $u$ decomposes into two harmonic single-valued functions on $\mathbb{R}^n$, contradicting $\mathcal{B}_{\varphi} \neq \emptyset$.  Therefore $\dim S(\varphi) = n-2$. 

Let $\varphi^{(0)}$ be any blow-up of $u$ at $Y_0$.  Translate so that $Y_0 = 0$ and rotate so that $\mathcal{B}_{\varphi^{(0)}} = \{0\} \times \mathbb{R}^{n-2}$.  We claim that for every $\delta > 0$ there is a $\varepsilon_0 = \varepsilon_0(n,m,\varphi^{(0)},\delta) > 0$ such that if $v \in W^{1,2}(B_1(0);\mathcal{A}_2(\mathbb{R}^m))$ is a Dirichlet energy minimizing two-valued function such that 
\begin{equation*}
	\int_{B_1(0)} \mathcal{G}(v(X),\varphi^{(0)}(X))^2 dX < \varepsilon_0, 
\end{equation*}
then for every $y_0 \in B^{n-2}_{1/2}(0)$ there is a $Z \in \mathcal{B}_v \cap B_{\delta}(0,y_0)$.  Suppose $v_j \in W^{1,2}(B_1(0);\mathcal{A}_2(\mathbb{R}^m))$ are Dirichlet energy minimizing two-valued functions converging to $\varphi^{(0)}$ in $L^2(B_1(0);\mathcal{A}_2(\mathbb{R}^m))$ and $\mathcal{B}_{v_j} \cap B_{\delta}(0,y_j) = \emptyset$ for some $y_j \in B^{n-2}_{1/2}(0)$.  Let $y_j \rightarrow y$.  Then $v_j$ decomposes into two harmonic single-valued functions on $B_{\delta/2}(0,y)$ for $j$ sufficiently large.  By the Schauder estimates $\varphi^{(0)}$ decomposes into two harmonic single-valued functions on $B_{\delta/4}(0,y)$, contradicting $\mathcal{B}_{\varphi^{(0)}} = \{0\} \times \mathbb{R}^{n-2}$.  

By the proof of Proposition 8.1 of~\cite{KrumWic1}, noting that by the discussion in the previous paragraph and (\ref{cft2_eqn1}) we can rule out option (i) of Lemma 4.5 of~\cite{KrumWic1}, $\mathcal{B}_u \cap B_{\sigma}(0) \subseteq \Gamma \cap B_{\sigma}(0)$ for the graph $\Gamma$ of some Lipschitz function over $\{0\} \times \mathbb{R}^{n-2}$ and some $\sigma \in (0,\op{dist}(0,\partial \Omega))$.  Suppose there exists $z \in B^{n-2}_{\sigma/2}(0)$ and $\eta \in (0,\sigma/2-|z|)$ such that $\mathcal{B}_u \cap B^2_{\eta}(z) \times \mathbb{R}^{n-2} = \emptyset$.  Then $u(X) = \{\pm u_1(X)\}$ on $B_{\sigma/2}(0) \setminus (\Gamma \setminus (B^2_{\eta}(z) \times \mathbb{R}^{n-2}))$ for some harmonic single-valued function $u_1$ since the set is simply-connected.  We claim that $u_1$ extends to a continuous function on all of $B_{\sigma/2}(0)$.  If $z' \in B^{n-2}_{\sigma/2}(0)$ and $\eta' \in (0,\sigma/2-|z'|)$ such that $\mathcal{B}_u \cap B^2_{\eta'}(z') \times \mathbb{R}^{n-2} = \emptyset$, then $u_1$ extends to a unique harmonic single-valued function such that $u(X) = \{\pm u_1(X)\}$ on the set $B_{\sigma/2}(0) \setminus (\Gamma \setminus ((B^2_{\eta}(z) \cup B^2_{\eta'}(z')) \times \mathbb{R}^{n-2}))$ since the set is simply-connected and thus $u_1$ extends smoothly across $\Gamma \cap (B^2_{\eta'}(z') \times \mathbb{R}^{n-2})$.  Clearly $u_1$ extends continuously across $\mathcal{B}_u$ with $u_1 = 0$ on $\mathcal{B}_u$.  Therefore $u_1$ extends to a continuous function on $B_{\sigma/2}(0)$.  By Weyl's lemma, $u_1$ is harmonic on $B_{\sigma/2}(0)$.  Clearly $u(X) = \{ \pm u_1(X) \}$ on $B_{\sigma}(0)$, contradicting $0 \in \mathcal{B}_u$.  Therefore $\mathcal{B}_u \cap B_{\sigma/2}(0) = \Gamma \cap B_{\sigma/2}(0)$.  The rest of the conclusion of Theorem \ref{constfreqthm2}, in particular that $\mathcal{B}_u$ is a $C^{1,\tau}$ submanifold near $0$ and $u$ is exponentially asymptotic to unique blow-ups along $\mathcal{B}_u$ near $0$, now follows from the proof of Proposition 8.1 of~\cite{KrumWic1}.  
\end{proof}

Notice that the above proof of Theorem \ref{constfreqthm2} would not extend if we replaced the hypothesis of Theorem \ref{constfreqthm2} that $\mathcal{N}_u(Y_0) = 1/2+k$ with $\mathcal{N}_u(Y_0)$ being an integer since then neither the origin nor any other point on $S(\varphi)$ would necessarily be a singular point of $\varphi$.  Indeed, from the argument above we could conclude that if $\varphi$ is a blow-up of $u$ at $Y_0$ then $\mathcal{B}_{\varphi} \subseteq S(\varphi)$ due to $\mathcal{N}_u(Z) \geq \mathcal{N}_u(Y_0)$ for all $Z \in \mathcal{B}_u$.  Hence, using the fact that either $\dim S(\varphi) = n-2$ or $\dim S(\varphi) \leq n-3$ implying $\mathcal{B}_{\varphi} = \emptyset$, $\varphi(X) = \{\pm \varphi_1(X)\}$ for some nonzero, homogeneous degree $\mathcal{N}_u(Y_0)$, harmonic single-valued function $\varphi_1 : \mathbb{R}^n \rightarrow \mathbb{R}^m$.  Consequently we are not guaranteed the conditions that the blow-ups $\varphi$ at $Y_0$ are ``cylindrical'' (i.e. $S(\varphi) = n-2$) and there are ``no small gaps'' (i.e. $y_0 \in B^{n-2}_{1/2}(0)$ there exists $Z \in \mathcal{B}_u \cap B_{\delta}(0,y_0)$ with $\mathcal{N}_u(Z) \geq \mathcal{N}_u(0)$) which are essential to the blow-up method of~\cite{KrumWic1}.

Assuming Theorem \ref{constfreqthm2} and Theorem \ref{analyticity_thm}, which we will prove below, the proof of Theorem \ref{maincor} is as follows: 

\begin{proof}[Proof of Theorem \ref{maincor} for case (a)]
Let $Y \in \mathcal{B}_u$.  Suppose that $\mathcal{N}_u$ is continuous at $Y$, i.e. for every $\varepsilon > 0$ there exists $\delta \in (0,\op{dist}(Y,\partial \Omega))$ such that for every $X \in \mathcal{B}_u \cap B_{\delta}(Y)$, $\mathcal{N}_u(Y) - \varepsilon < \mathcal{N}_u(X) < \mathcal{N}_u(Y) + \varepsilon$.  By Lemma \ref{cbu_dense_lemma}, $\mathcal{N}_u(Y) = k/2$ for some integer $k \geq 1$ and $\mathcal{N}_u(X) \geq \mathcal{N}_u(Y)$ for all $X \in \mathcal{B}_u$ sufficiently close to $Y$.  By Theorem \ref{constfreqthm1}, $\mathcal{N}_u(Y)$ is not an integer.  By Theorem \ref{constfreqthm2}, $Y \in \Gamma_{u,\mathcal{N}_u(Y)}$.  By Theorem \ref{analyticity_thm}, $\Gamma_{u,\mathcal{N}_u(Y)}$ is real analytic near $Y$. 

If instead $\mathcal{N}_u$ is not continuous at $Y$, by the upper semicontinuity of $\mathcal{N}_u$, there exists a sequence $Z_j \in \mathcal{B}_u$ converging to $Y$ such that $\lim_{j \rightarrow \infty} \mathcal{N}_u(Z_j) < \mathcal{N}_u(Y)$.  By Lemma \ref{cbu_dense_lemma}, there exists integers $k_j \geq 0$, $Y_j \in (\mathcal{B}_u \setminus \mathcal{B}_u^{(n-3)}) \cap B_{1/j}(Z_j)$, and $\delta_j > 0$ such that $\mathcal{N}_u(Y_j) = k_j/2$, $\mathcal{N}_u(Y_j) \leq \mathcal{N}_u(Z_j)$, and $\mathcal{N}_u(X) \geq \mathcal{N}_u(Y_j)$ for all $X \in \mathcal{B}_u \cap B_{\delta_j}(Y_j)$.  After passing to a subsequence, we may assume that $k_j = k$ for some integer $k \geq 0$ independent of $j$.  Clearly $\mathcal{N}_u(Y) > k/2$.  By Theorems \ref{constfreqthm2} and \ref{constfreqthm1}, $k$ is an odd integer and $Y_j \in \Gamma_{u,k/2}$. 
\end{proof}

\noindent \textbf{$C^{1,\mu}$ harmonic two-valued functions.}  This case is similar to the case of Dirichlet energy minimizing two-valued functions.  We consider $u \in C^{1,\mu}(\Omega;\mathcal{A}_2(\mathbb{R}^m))$, where $\mu \in (0,1)$, such that $u$ is locally harmonic on $\Omega \setminus \mathcal{B}_u$.  We may assume that such $u$ are symmetric and nonzero.  Given such a $u$ and $Y \in \Omega$, we again define the frequency function $N_{u,Y} : (0,\op{dist}(Y,\partial \Omega)) \rightarrow (0,\infty)$ by (\ref{defn_freqfn}) and the frequency $\mathcal{N}_u(Y)$ of $u$ at $Y \in \Omega$ by $\mathcal{N}_u(Y) = \lim_{\rho \downarrow 0} N_{u,Y}(\rho)$.  By~\cite[Lemma 2.2]{SimWic11}, $N_{u,Y}$ is well-defined and monotone nondecreasing and thus the limit $\mathcal{N}_u(Y)$ exists.  $N_{u,Y} \equiv \alpha$ constant if and only if $u(Y+X)$ is a homogeneous degree $\alpha$ function of $X$.  By the monotonicity of frequency functions, the upper-semicontinuity property Lemma \ref{usc_freq_dm} above continues to hold true with $C^{1,\mu}$ harmonic two-valued functions in place of Dirichlet energy minimizing two-valued functions.  

For nonzero, symmetric two-valued functions $u \in C^{1,\mu}(\Omega;\mathcal{A}_2(\mathbb{R}^m))$ that are locally harmonic on $\Omega \setminus \mathcal{B}_u$, we replace $\Sigma_u$ above with $\mathcal{K}_u$, the set of $Y \in \Omega$ such that $u(Y) = \{0,0\}$ and $Du(Y) = \{0,0\}$ but $u \not\equiv \{0,0\}$ near $Y$.  Clearly $\mathcal{N}_u(Y)$ is a integer $\geq 2$ for all $Y \in \mathcal{K}_u \setminus \mathcal{B}_u$, whereas $\mathcal{N}_u(Y)$ can take both integer and non-integer values for $Y \in \mathcal{B}_u$.  By~\cite[Lemma 4.1]{SimWic11}, if $u \in C^{1,\mu}(\Omega;\mathcal{A}_2(\mathbb{R}^m))$ is nonzero, symmetric, and locally harmonic on $\Omega \setminus \mathcal{B}_u$, then $u \in C^{1,1/2}(\Omega;\mathcal{A}_2(\mathbb{R}^m))$ and $\mathcal{N}_u(Y) \geq 3/2$ for all $Y \in \mathcal{K}_u$.

Given a nonzero, symmetric, two-valued function $u \in C^{1,\mu}(\Omega;\mathcal{A}_2(\mathbb{R}^m))$ that is locally harmonic on $\Omega \setminus \mathcal{B}_u$ and $Y \in \mathcal{K}_u$, we define a blow-up $\varphi \in C^{1,1/2}(\mathbb{R}^n;\mathcal{A}_2(\mathbb{R}^m))$ of $u$ at $Y$ to be any limit given by (\ref{defn_blowup}) for some $\rho_j \rightarrow 0^+$ where the limit is taken in $C^1$ on compact subsets of $\mathbb{R}^n$.  By the Schauder estimates for $C^{1,\mu}$ harmonic two-valued function~\cite[Lemma 3.2]{SimWic11} and monotonicity of frequency functions (see~\cite[Lemma 3.1(b)]{KrumWic1}), there exists at least one blow-up $\varphi$ of $u$ at each $Y$; however, the blow-up $\varphi$ might not be unique independent of the particular sequence $\rho_j$.  Every blow-up $\varphi$ of $u$ at $Y$ is nonzero, homogeneous degree $\mathcal{N}_u(Y)$, and locally harmonic on $\mathbb{R}^n \setminus \mathcal{B}_{\varphi}$. 

For every nonzero, homogeneous, two-valued function $\varphi \in C^{1,\mu}(\mathbb{R}^n;\mathcal{A}_2(\mathbb{R}^m))$, where $\mu \in (0,1)$, we define the spine $S(\varphi)$ of $\varphi$ by (\ref{defn_spine}) above.  $S(\varphi)$ is always a linear subspace and $\varphi(Z+X) = \varphi(X)$ for all $X \in \mathbb{R}^n$ and $Z \in S(\varphi)$. 

Let $\mu \in (0,1)$ and $u \in C^{1,\mu}(\Omega;\mathcal{A}_2(\mathbb{R}^m))$ be nonzero, symmetric, and locally harmonic on $\Omega \setminus \mathcal{B}_u$.  For $j = 1,2,\ldots,n-2$, let $\mathcal{K}^{(j)}_u$ denote the set of all $Y \in \mathcal{K}_u$ such that $\dim S(\varphi) \leq j$ for every blow-up $\varphi$ of $u$ at $Y$ and let $\mathcal{B}^{(j)}_u = \mathcal{B}_u \cap \mathcal{K}^{(j)}_u$.  Then 
\begin{equation*}
	\mathcal{K}_u = \mathcal{K}^{(n-2)}_u \supseteq \mathcal{K}^{(n-3)}_u \supseteq \cdots \supseteq \mathcal{K}^{(1)}_u \supseteq \mathcal{K}^{(0)}_u 
\end{equation*}
and the Hausdorff dimension of $\mathcal{K}^{(j)}_u$ is at most $j$.  $\mathcal{K}_u \setminus \mathcal{K}^{(n-3)}_u$ consists of the set of all $Y \in \mathcal{K}_u$ such that at least one blow-up $\varphi$ of $u$ at $Y$ has $\dim S(\varphi) = n-2$.  After rotating so that $S(\varphi) = \{0\} \times \mathbb{R}^{n-2}$, $\varphi$ has the form $\varphi(X) = \{ \pm \op{Re}(c(x_1+ix_2)^{k/2}) \}$ for some $c \in \mathbb{C}^m$ and integer $k \geq 3$ and consequently $\mathcal{N}_u(Y) = k/2$.  Lemma \ref{cbu_dense_lemma} above continues to hold true with $C^{1,\mu}$ harmonic two-valued functions in place of Dirichlet energy minimizing two-valued functions.  

The proofs of Theorem \ref{constfreqthm2} and Theorem \ref{maincor} all continue to hold in the case of $C^{1,\mu}$ harmonic two-valued functions with obvious changes. \\

\noindent \textbf{$C^{1,\mu}$ two-valued functions with area-stationary graphs.}  The case of $C^{1,\mu}$ two-valued functions with area-stationary graphs is significantly different due to the fact that these area-stationary submanifolds are only approximately the graphs of harmonic functions when expressed as graphs over the tangent planes at branch points.  Let $\mu \in (0,1)$ and $u \in C^{1,\mu}(B_2(0);\mathcal{A}_2(\mathbb{R}^m))$ be a two-valued function whose $\mathcal{M} = \op{graph} u$ is area-stationary as a varifold.  Note that by~\cite{SimWic11}, $u \in C^{1,1/2}(B_1(0);\mathcal{A}_2(\mathbb{R}^m))$.  Assume that $u_s(X) = \{ \pm (u_1(X)-u_2(X))/2 \}$ is not identically $\{0,0\}$, where $u(X) = \{u_1(X),u_2(X)\}$ at $X \in B_1(0)$, so that $\mathcal{M}$ is not the graph of smooth single-valued function with multiplicity two.  Assume that for $\varepsilon > 0$ to be determined, 
\begin{equation} \label{close2plane}
	\|u\|_{C^{1,1/2}(B_1(0))} = \sup_{B_1(0)} |u| + \sup_{B_1(0)} |Du| + [Du]_{\mu,B_1(0)} < \varepsilon. 
\end{equation}

Recall that $\op{sing} \mathcal{M}$ denotes the set of all points $p$ of $\mathcal{M}$ for which there is no $\delta > 0$ such that $\mathcal{M}$ is the union of smooth embedded submanifolds in $B^{n+m}_{\delta}(p)$.  In place of $\Sigma_u$ and $\mathcal{K}_u$, we will consider the set $\mathcal{K}_{\mathcal{M}}$ of all points $p$ of $\mathcal{M}$ at which $\mathcal{M}$ has a multiplicity two tangent plane.  Observe that $\op{sing} \mathcal{M} \subseteq \mathcal{K}_{\mathcal{M}}$. 

At each $p \in \mathcal{K}_{\mathcal{M}}$, we can write $\mathcal{M}$ as the graph of a two-valued function $\tilde u_p \in C^{1,1/2}(T_p \mathcal{M} \cap B^{n+m}_{\sqrt{1-4\varepsilon^2} (2-|p|)}(0); T_p \mathcal{M}^{\perp})$ over the tangent plane $T_p \mathcal{M}$ to $\mathcal{M}$ at $p$~\cite{KrumWic2}.  Writing $\tilde u_p(X) = \{\tilde u_{p,1}(X),\tilde u_{p,2}(X)\}$ at each $X$, let $\tilde u_{p,a}(X) = (\tilde u_{p,1}(X)+\tilde u_{p,2}(X))/2$ and $\tilde u_{p,s}(X) = \{ \pm (\tilde u_{p,1}(X)-\tilde u_{p,2}(X))/2 \}$ so that $\tilde u_p = \tilde u_{p,a} + \tilde u_{p,s}$.  By~\cite[Theorems 7.1 and 7.4]{SimWic11}, assuming $\varepsilon$ is sufficiently small depending on $n$ and $m$, 
\begin{gather}
	|\tilde u_{p,s}(X)| + \op{dist}(X,\mathcal{B}_{\tilde u_p}) |D\tilde u_{p,s}(X)| + \op{dist}(X,\mathcal{B}_{\tilde u_p})^2 |D^2 \tilde u_{p,s}(X)| 
		\leq C(n,m) \varepsilon \op{dist}(X,\mathcal{B}_{\tilde u_p})^{3/2}, \nonumber \\
	|D^2 \tilde u_{p,a}(X)| \leq C(n,m) \varepsilon. \label{SimWicThm7}
\end{gather}

Let $B$ be an open ball in $T_p \mathcal{M} \cap B^{n+m}_{\sqrt{1-4\varepsilon^2} (2-|p|)}(0) \setminus \mathcal{B}_{\tilde u_p}$.  Observe that $\tilde u_p(X) = \{\tilde u_{p,1}(X),\tilde u_{p,2}(X)\}$ on $B$ for single-valued solutions $\tilde u_{p,1},\tilde u_{p,2}$ to the minimal surface system 
\begin{equation} \label{mss1}
	D_i \left( \sqrt{G(D\tilde u_l)} G^{ij}(D\tilde u_l) D_j \tilde u_l^{\kappa} \right) = 0
\end{equation}
on $B$ for $\kappa = 1,2,\ldots,m$ and $l = 1,2$, where $G(P)$ and $(G^{ij}(P))$ denote the determinant and inverse matrix respectively of the $n \times n$ matrix $(G_{ij}(P))$ where $G_{ij}(P) = \delta_{ij} + P_i^{\kappa} P_j^{\kappa}$ for $P \in \mathbb{R}^{mn}$.  Subtracting (\ref{mss1}) for $l = 1,2$ yields 
\begin{equation} \label{mss2}
	D_i \left( A^{ij}_{\kappa \lambda}(D\tilde u_{p,a},D\tilde u_{p,s}) D_j \tilde u_{p,s}^{\lambda} \right) = 0 \text{ on } B
\end{equation}
for $\kappa = 1,2,\ldots,m$, where $\tilde u_{p,a}(X) = (\tilde u_1(X)+\tilde u_2(X))/2$, $\tilde u_{p,s}(X) = (\tilde u_2(X)-\tilde u_1(X))/2$, and 
\begin{equation} \label{mss2_notation} 
	A^{ij}_{\kappa \lambda}(P,Q) = \frac{1}{2} \int_{-1}^1 (D_{P^{\kappa}_i P^{\lambda}_j} (\sqrt{G}))(P+tQ) dt. 
\end{equation}
Note that $A^{ij}_{\kappa \lambda}(0,0) = \delta_{ij} \delta_{\kappa \lambda}$ and $A^{ij}_{\kappa \lambda}(P,Q) = A^{ij}_{\kappa \lambda}(-P,Q) = A^{ij}_{\kappa \lambda}(P,-Q)$ for all $P,Q \in \mathbb{R}^{mn}$.  We can rewrite (\ref{mss2}) as 
\begin{equation} \label{mss3}
	\Delta \tilde u_s^{\kappa} + f_{\kappa} = 0 \text{ on } B \text{ where } 
	f_{\kappa} = D_i \left( (A^{ij}_{\kappa \lambda}(D\tilde u_{p,a},D\tilde u_{p,s}) - \delta_{ij} \delta_{\kappa \lambda}) D_j \tilde u_{p,s}^{\lambda} \right) 
\end{equation}
for $\kappa = 1,2,\ldots,m$.  Note that by adding (\ref{mss1}) for $l = 1,2$ we get that 
\begin{equation} \label{mss4}
	D_i \left( A^{ij}_{\kappa \lambda}(D\tilde u_{p,s},D\tilde u_{p,a}) D_j \tilde u_{p,a}^{\lambda} \right) = 0 \text{ on } B
\end{equation}
for $\kappa = 1,2,\ldots,m$.  

By~\cite{KrumWic2}, provided $\varepsilon$ is sufficiently small, $e^{C\varepsilon \rho^{1/4}} N_{\tilde u_{p,s},0}(\rho)$ is monotone nondecreasing as a function of $\rho$ for some constant $C = C(n,m) \in (0,\infty)$, where $\mathcal{N}_{\tilde u_{p,s},0}(\rho)$ is given by (\ref{defn_freqfn}) with $\tilde u_{p,s}$ and $0$ in place of $u$ and $Y$ respectively.  Thus we may define the frequency $\mathcal{N}_{\mathcal{M}}(p)$ of $\mathcal{M}$ at $p$ by $\mathcal{N}_{\mathcal{M}}(p) = \lim_{\rho \downarrow 0} \mathcal{N}_{\tilde u_{p,s},0}(\rho)$.  $\mathcal{N}_{\mathcal{M}}(p)$ is an integer $\geq 2$ for each $p \in \mathcal{K}_{\mathcal{M}} \setminus \op{sing} \mathcal{M}$, but $\mathcal{N}_{\mathcal{M}}(p)$ can take both integer and non-integer values for $p \in \op{sing} \mathcal{M}$.  By~\cite{SimWic11}, $\mathcal{N}_{\mathcal{M}}(p) \geq 3/2$ for all $p \in \mathcal{K}_{\mathcal{M}}$.  Frequency has the following upper-semicontinuity properties:

\begin{lemma}[Upper semicontinuity of frequency] \label{usc_freq_mss}
(i) Suppose that $u \in C^{1,\mu}(B_1(0);\mathcal{A}_2(\mathbb{R}^m))$ is a two-valued functions such that $u_s \not\equiv \{0,0\}$, (\ref{close2plane}) holds true for some $\varepsilon = \varepsilon(n,m,\mu) > 0$ sufficiently small, and $\mathcal{M} = \op{graph} u$ is area stationary.  Let $p_j,p \in \mathcal{K}_{\mathcal{M}}$ such that $p_j \rightarrow p$.  Then 
\begin{equation*}
	\mathcal{N}_{\mathcal{M}}(p) \geq \limsup_{j \rightarrow \infty} \mathcal{N}_{\mathcal{M}}(p_j). 
\end{equation*}

\noindent (ii) Suppose that $u_j \in C^{1,\mu}(B_1(0);\mathcal{A}_2(\mathbb{R}^m))$ is a two-valued functions for which $\mathcal{M}_j = \op{graph} u_j$ is area stationary.  Let $u_{j,s}(X) = \{ \pm (u_{j,1}(X)-u_{j,2}(X))/2 \}$ where $u_j(X) = \{u_{j,1}(X),u_{j,2}(X)\}$ and assume that $u_{j,s} \not\equiv \{0,0\}$.  Suppose that $\|u_j\|_{C^{1,\mu}(B_1(0))} \rightarrow 0$ and $u_{j,s}/\|u_{j,s}\|_{L^2(B_1(0))}$ converge to a nonzero two-valued function $\varphi \in C^{1,\mu}(B_1(0);\mathcal{A}_2(\mathbb{R}^m))$ that is locally harmonic on $B_1(0) \setminus \mathcal{B}_{\varphi}$.  Let $Y_j \in B_1(0)$ with $(Y_j,u_j(Y_j)) \in \mathcal{K}_{\mathcal{M}_j}$ (slightly abusing notation so that $u_j(Y_j)$ denotes both the unordered pair $\{u_{j,1}(Y_j),u_{j,2}(Y_j)\}$ and common value $u_{j,1}(Y_j) = u_{j,2}(Y_j)$) and $Y \in B_1(0) \cap \mathcal{K}_{\varphi}$ such that $Y_j \rightarrow Y$.  Then 
\begin{equation*}
	\mathcal{N}_{\varphi}(Y) \geq \limsup_{j \rightarrow \infty} \mathcal{N}_{\mathcal{M}_j}(Y_j,u_j(Y_j)). 
\end{equation*}
\end{lemma}
\begin{proof}
To prove (i), recall that $\mathcal{M}$ is the graph of $\tilde u_p$ over the tangent plane to $\mathcal{M}$ at $p$ and let $Y_j$ be the projection of $p_j$ onto the tangent plane of $\mathcal{M}$ at $p$.  Using the monotonicity of $N_{\tilde u_{p_j,s},0}$, we compute that 
\begin{align*}
	\mathcal{N}_{\mathcal{M}}(p) 
	&= \lim_{\rho \downarrow 0} N_{\tilde u_{p,s},0}(\rho) 
	= \lim_{\rho \downarrow 0} \lim_{j \rightarrow \infty} N_{\tilde u_{p,s},Y_j}(\rho) 
	\\&\geq \lim_{\rho \downarrow 0} \lim_{j \rightarrow \infty} \left( 1 - C(n,m) \sup_{B_{2|Y_j|}(0)} |D\tilde u_p| \right) N_{\tilde u_{p_j,s},0}(\rho) 
	\geq \limsup_{j \rightarrow \infty} \mathcal{N}_{\mathcal{M}}(p_j). 
\end{align*}
To prove (ii), write $\mathcal{M}$ as the graph of $\tilde u_j$ over the tangent plane to $\mathcal{M}_j$ at $p_j$ and let $\tilde u_{j,s}(X) = \{ \pm (u_{j,1}(X)-u_{j,2}(X))/2 \}$ where $\tilde u_j(X) = \{u_{j,1}(X),u_{j,2}(X)\}$.  Using the monotonicity of $N_{\tilde u_{j,s},0}$, we compute that 
\begin{align*}
	\mathcal{N}_{\varphi}(Y) 
	&= \lim_{\rho \downarrow 0} N_{\varphi,Y}(\rho) 
	= \lim_{\rho \downarrow 0} \lim_{j \rightarrow \infty} N_{u_{j,s}/\|u_{j,s}\|_{L^2(B_1(0))},Y_j}(\rho) 
	= \lim_{\rho \downarrow 0} \lim_{j \rightarrow \infty} N_{u_{j,s},Y_j}(\rho) 
	\\&\geq \lim_{\rho \downarrow 0} \lim_{j \rightarrow \infty} \left( 1 - C(n,m) \sup_{B_1(0)} |Du_j| \right) N_{\tilde u_{j,s},0}(\rho) 
	\geq \limsup_{j \rightarrow \infty} \mathcal{N}_{\mathcal{M}}(p_j). 
\end{align*}
\end{proof}

Given $u \in C^{1,\mu}(B_1(0);\mathcal{A}_2(\mathbb{R}^m))$ with $u_s \not\equiv \{0,0\}$ whose $\mathcal{M} = \op{graph} u$ is area-stationary as a varifold and $p \in \mathcal{K}_{\mathcal{M}}$, we define a \textit{blow-up} $\varphi \in C^{1,\mu}(T_p \mathcal{M};\mathcal{A}_2(T_p \mathcal{M}^{\perp}))$ of $\mathcal{M}$ at $p$ to be any limit of the form 
\begin{equation*}
	\varphi = \lim_{j \rightarrow \infty} \frac{\tilde u_p(\rho_j X)}{\rho_j^{-n/2} \|\tilde u_p\|_{L^2(B_{\rho}(Y))}}, 
\end{equation*}
where $T_p \mathcal{M}$ denotes the tangent plane to $\mathcal{M}$ at $p$, $\rho_j \rightarrow 0^+$, and the limit is taken in $W^{1,2}$ on compact subsets of $\mathbb{R}^n$.  By~\cite{KrumWic2}, for each $\mathcal{M}$ and $p$ at least one blow-up $\varphi$ exists; however, the blow-up $\varphi$ might not be unique independent of the particular sequence $\rho_j$.  Every blow-up $\varphi$ of $\mathcal{M}$ at $p$ is a nonzero, homogeneous degree $\mathcal{N}_{\mathcal{M}}(p)$ two-valued function that is locally harmonic on $T_p \mathcal{M} \setminus \mathcal{B}_{\varphi}$.  For $j = 1,2,\ldots,n-2$, let $\mathcal{K}^{(j)}_{\mathcal{M}}$ denote the set of all $p \in \mathcal{K}_{\mathcal{M}}$ such that $\dim S(\varphi) \leq j$ for every blow-up $\varphi$ of $\mathcal{M}$ at $p$.  Then 
\begin{equation*}
	\mathcal{K}_{\mathcal{M}} = \mathcal{K}^{(n-2)}_{\mathcal{M}} \supseteq \mathcal{K}^{(n-3)}_{\mathcal{M}} \supseteq \cdots 
	\supseteq \mathcal{K}^{(1)}_{\mathcal{M}} \supseteq \mathcal{K}^{(0)}_{\mathcal{M}} 
\end{equation*}
and the Hausdorff dimension of $\mathcal{K}^{(j)}_{\mathcal{M}}$ is at most $j$ by~\cite{KrumWic2}.  $\mathcal{K}_{\mathcal{M}} \setminus \mathcal{K}^{(n-3)}_{\mathcal{M}}$ consists of the set of all $p \in \mathcal{K}_{\mathcal{M}}$ such that at least one blow-up $\varphi$ of $\mathcal{M}$ at $p$ has $\dim S(\varphi) = n-2$ and thus after rotating so that $\mathcal{M}$ is tangent to $\mathbb{R}^n \times \{0\}$ at $p$ and $S(\varphi) = \{0\} \times \mathbb{R}^{n-2} \subset \mathbb{R}^n$, $\varphi$ has the form $\varphi(X) = \{ \pm \op{Re}(c(x_1+ix_2)^{k/2}) \}$ for some $c \in \mathbb{C}^m$ and integer $k \geq 3$ and consequently $\mathcal{N}_u(Y) = k/2$.  Arguing much like we did to prove Lemma \ref{cbu_dense_lemma}, we obtain that $\op{sing} \mathcal{M} \setminus \mathcal{K}_{\mathcal{M}}^{(n-3)}$ is dense in $\op{sing} \mathcal{M}$. 

The proof of Theorem \ref{constfreqthm2_mss} is similar to the proof of Theorem \ref{constfreqthm2} and is included for completion. 

\begin{proof}[Proof of Theorem \ref{constfreqthm2_mss}]
First we claim that $\dim S(\varphi) = n-2$ for every blow-up $\varphi$ of $\mathcal{M}$ at $p_0$.  Since $\mathcal{N}_{\varphi}(0) = \mathcal{N}_{\mathcal{M}}(p_0)$ is not an integer, $0 \in \mathcal{B}_{\varphi}$.  In particular $\mathcal{B}_{\varphi} \neq \emptyset$, so $\mathcal{B}_{\varphi}$ has Hausdorff dimension $n-2$.  Write $\varphi = \lim_{j \rightarrow \infty} \tilde u_{p_0,s}(\rho_j X)/\rho_j^{-n/2} \|\tilde u_{p_0,s}\|_{L^2(B_{\rho_j}(0))}$ with convergence in locally in $C^1$ for $\rho_j \downarrow 0$ and let $Z \in \mathcal{B}_{\varphi}$.  Clearly $\mathcal{N}_{\varphi}(Z) \leq \mathcal{N}_u(0)$.  Then for every $\delta > 0$, $B_{\delta}(Z) \cap \mathcal{B}_{\tilde u_{p_0,s}(\rho_j X)} \neq \emptyset$ for $j$ sufficiently large since then otherwise $\tilde u_{p_0,s}(\rho_j X)$ decomposes into two harmonic single-valued functions on $B_{\delta}(Z)$ and then by the Schauder estimates applied to $\tilde u_{p_0,s}(\rho_j X)/\rho_j^{-n/2} \|\tilde u_{p_0,s}\|_{L^2(B_{\rho_j}(0))}$, which solves a homogeneous elliptic differential equation in divergence form due to $\tilde u_{p_0}$ solving the minimal surface system (see Section 5 of~\cite{SimWic11}), $\varphi$ decomposes into two harmonic single-valued functions on $B_{\delta/2}(Z)$, contradicting $Z \in \mathcal{B}_{\varphi}$.  Therefore there exists $Z_j \in \mathcal{B}_{\tilde u_{p_0,s}(\rho_j X)}$ such that $Z_j \rightarrow Z$.  By (\ref{cft2_eqn1}) and Lemma \ref{usc_freq_mss}(ii), $\mathcal{N}_{\varphi}(Z) \geq \mathcal{N}_{\mathcal{M}}(p_0)$.  Therefore $\mathcal{N}_{\varphi}(Z) = \mathcal{N}_u(p_0)$ for all $Z \in \mathcal{B}_{\varphi}$.  Consequently $S(\varphi) = \mathcal{B}_{\varphi}$ and $\dim S(\varphi) = n-2$. 

Let $\varphi^{(0)}$ be any blow-up of $\mathcal{M}$ at $p_0$.  Rotate so that $\mathcal{M}$ is tangent to $\mathbb{R}^n \times \{0\}$ at $p_0$ and $\mathcal{B}_{\varphi^{(0)}} = \{0\} \times \mathbb{R}^{n-2} \subset \mathbb{R}^n$.  Using the Schauder estimates, we can show that for every $\delta > 0$ there is a $\varepsilon_0 = \varepsilon_0(n,m,\varphi^{(0)},\delta) > 0$ such that if $v \in C^{1,1/2}(B_1(0);\mathcal{A}_2(\mathbb{R}^m))$ with $v_s \not\equiv \{0,0\}$ and $\op{graph} v$ area-stationary such that 
\begin{equation*}
	\|v\|_{C^{1,1/4}(B_1(0))} \leq \varepsilon_0, \quad 
	\int_{B_1(0)} \mathcal{G}\left( \frac{v(X)}{\|v\|_{L^2(B_1(0))}}, \varphi^{(0)}(X) \right)^2 dX < \varepsilon_0, 
\end{equation*}
then for every $y_0 \in B^{n-2}_{1/2}(0)$ there is a $Z \in \mathcal{B}_v \cap B_{\delta}(0,y_0)$.  

The conclusion of Theorem \ref{constfreqthm2_mss} now follows from the arguments of~\cite{KrumWic2} like in the proof of Theorem \ref{constfreqthm2} for case (a) above, noting that by the discussion in the previous paragraph and (\ref{cft2_eqn1_mss}) we can rule out small gaps.  
\end{proof}

The proof of Theorem \ref{maincor_mss} is nearly identical to the proof of Theorem \ref{maincor} with obvious changes, in particular replacing Lemma \ref{usc_freq_dm} and Theorems \ref{constfreqthm2}, \ref{constfreqthm1}, and \ref{analyticity_thm} with Lemma \ref{usc_freq_mss}(i) and Theorems \ref{constfreqthm2_mss}, \ref{constfreqthm1_mss}, and \ref{analyticity_thm_mss} respectively.  Corollary \ref{maincor_stablevar} is a direct consequence of~\cite{Wic08}~\cite{Wic} and Theorem \ref{maincor}. 

Now Sections 3-6 will be devoted to the proofs of Theorems \ref{constfreqthm1} and \ref{constfreqthm1_mss} and Sections 7-10 will be devoted to the proofs of Theorems \ref{analyticity_thm} and \ref{analyticity_thm_mss}.

\section{Apriori estimates - Part I: Statement of estimates and blow-ups} \label{sec:L2estimates_sec}

In the next two sections, we will state and prove the main a priori estimates needed to prove Theorem \ref{constfreqthm1}.  These estimates are extensions of the estimates of Wickramasekera and the author in Section 6 of~\cite{KrumWic1}, which itself is based on the earlier work of Simon in~\cite{Sim93}.  Recall that unlike in~\cite{Sim93} and~\cite{KrumWic1} here we consider a situation where there may be small gaps and the homogeneous degree $\alpha$, harmonic two-valued functions can have spines of different dimensions.  The small gaps are handled by the fact that in any open ball where frequency $< \alpha$ by assumption there are no singular points and thus we may apply the Schauder estimates.  The main challenge with the spines having different dimensions is that in our setting that a given homogeneous degree $\alpha$, harmonic two-valued function may be close to another homogeneous degree $\alpha$, harmonic two-valued function with a  spine of larger dimension.  To handle this issue, we introduce and use the concept of an optimal sequence with increasing spines and, motivated by the work of Wickramasekera of~\cite{Wic04} and~\cite{Wic14}, we assume Hypothesis \ref{cft1_hyp}(ii) and prove an estimate Corollary \ref{branchdist_cor} for the distance of singular points to spines of homogeneous degree $\alpha$, harmonic two-valued functions.

\begin{defn}
Fix an integer $\alpha \geq 1$.  

Let $\mathcal{F}_{\alpha}$ denote the set of all Dirichlet energy minimizing two-valued functions $u \in W^{1,2}(B_1(0);\mathcal{A}_2(\mathbb{R}^m))$ such that $\mathcal{N}_u(Y) \geq \alpha$ for all $Y \in \mathcal{B}_u$.

Let $\Phi_{\alpha}$ denote the set of all nonzero two-valued functions $\varphi : \mathbb{R}^n \rightarrow \mathcal{A}_2(\mathbb{R}^m)$ such that $\varphi(X) = \{ \pm \varphi_1(X) \}$ for all $X \in \mathbb{R}^n$ and some nonzero, homogeneous degree $\alpha$, harmonic single-valued polynomial $\varphi_1$. 
\end{defn}

\begin{hypothesis} \label{cft1_hyp}
Let $\varepsilon, \beta > 0$.  Let $\varphi^{(0)} \in \Phi_{\alpha}$ be Dirichlet energy minimizing and write $\varphi^{(0)}(X) = \{ \pm \varphi^{(0)}_1(X) \}$ for all $X \in \mathbb{R}^n$ and some harmonic single-valued function $\varphi^{(0)}_1 : \mathbb{R}^n \rightarrow \mathbb{R}^m$.  Let $\varphi \in \Phi_{\alpha}$, and $u \in \mathcal{F}_{\alpha}$ such that 
\begin{enumerate}
\item[(i)] $\int_{B_1(0)} \mathcal{G}(\varphi,\varphi^{(0)})^2 < \varepsilon$ and $\int_{B_1(0)} \mathcal{G}(u,\varphi^{(0)})^2 < \varepsilon$ and 
\item[(ii)] either $\dim S(\varphi) = \dim S(\varphi^{(0)})$ or 
\begin{equation*}
	\int_{B_1(0)} \mathcal{G}(u,\varphi)^2 
	\leq \beta \inf_{\varphi' \in \Phi_{\alpha}, \, S(\varphi) \subset S(\varphi')} \int_{B_1(0)} \mathcal{G}(u,\varphi')^2, 
\end{equation*}
where for two sets $A$ and $B$ we let $A \subset B$ denote that $A$ is a proper subset of $B$. 
\end{enumerate}
\end{hypothesis}

\begin{defn}
Given $\varphi \in \Phi_{\alpha}$, there exists a finite sequence $\{\varphi_j\}_{j = 0,1,2,\ldots,N} \subset \Phi_{\alpha}$, called \textit{optimal sequence with increasing spines}, such that $\varphi_0 = \varphi$, $\varphi_j$ satisfies $S(\varphi_{j-1}) \subset S(\varphi_j)$ and 
\begin{equation*}
	\int_{B_1(0)} \mathcal{G}(\varphi,\varphi_j)^2 
	\leq \frac{3}{2} \inf_{\varphi' \in \Phi_{\alpha}, \, S(\varphi_{j-1}) \subset S(\varphi')} \int_{B_1(0)} \mathcal{G}(\varphi,\varphi')^2 
\end{equation*}
for $j = 1,2,\ldots,N$, and $\dim S(\varphi_N) = \dim S(\varphi^{(0)})$. 
\end{defn}

The statement of the a priori estimates are as follows.  For each $l \in \{1,2,\ldots,n\}$, $\zeta \in \mathbb{R}^{n-l}$, $\kappa \in [0,1]$, and $\rho > 0$, let
\begin{equation*}
	A^l_{\rho,\kappa}(\zeta) = \{ (x,y) \in \mathbb{R}^l \times \mathbb{R}^{n-l} : (|x|-\rho)^2 + |y-\zeta|^2 < \kappa^2 \rho^2/4 \} 
\end{equation*}
if $l < n$ and $A^n_{\rho,\kappa} = \{ x \in \mathbb{R}^n : (1-\kappa/2) \rho < |x| < (1+\kappa/2) \rho \}$.

\begin{lemma} \label{graphrep_lemma} 
Suppose $\varphi^{(0)} \in \Phi_{\alpha}$ is a Dirichlet energy minimizing two-valued function with $\dim S(\varphi^{(0)}) = n-l_0$ for some integer $l_0 \in \{2,3,\ldots,n\}$.  Let $l \in \{l_0,l_0+1,\ldots,n\}$.  For every $\gamma,\nu \in (0,1)$ there exists $\varepsilon_0, \beta_0 > 0$ depending on $n$, $l$, $m$, $\varphi^{(0)}$, $\gamma$, and $\nu$ such that the following holds true.  Let $\varphi \in \Phi_{\alpha}$ and after an orthogonal change of coordinates suppose that $S(\varphi) = \{0\} \times \mathbb{R}^{n-l}$.  (Note that we need not assume that $S(\varphi) \subseteq S(\varphi^{(0)})$.)  Let $u \in W^{1,2}(A^l_{1,\gamma}(0);\mathcal{A}_2(\mathbb{R}^m))$ be a nonzero, symmetric, Dirichlet energy minimizing two-valued function.  Suppose that 
\begin{equation*}
	\int_{A^l_{1,\gamma}(0)} \mathcal{G}(\varphi,\varphi^{(0)})^2 < \varepsilon_0, \quad 
	\int_{A^l_{1,\gamma}(0)} \mathcal{G}(u,\varphi^{(0)})^2 < \varepsilon_0, 
\end{equation*} 
and either $l = l_0$ or 
\begin{equation*}
	\int_{A^l_{1,\gamma}(0)} \mathcal{G}(u,\varphi)^2 
	\leq \beta_0 \inf_{\varphi' \in \Phi_{\alpha}, \, S(\varphi) \subset S(\varphi')} \int_{A^l_{1,\gamma}(0)} \mathcal{G}(u,\varphi')^2. 
\end{equation*}
Then there is a harmonic single-valued function $v : A^l_{1,\gamma/2}(0) \rightarrow \mathbb{R}^m$ such that 
\begin{equation*}
	u(X) = \{ \varphi_1(X) + v(X), -\varphi_1(X) - v(X) \}
\end{equation*}
for all $X \in A^l_{1,\gamma/2}(0)$, where $\varphi(X) = \{ \pm \varphi_1(X) \}$ for all $X \in \mathbb{R}^n$ and some harmonic single-valued function $\varphi_1 : \mathbb{R}^n \rightarrow \mathbb{R}^m$ close to $\varphi^{(0)}_1$ in $L^2(B_1(0);\mathbb{R}^m)$, and
\begin{gather*}
	\sup_{A^l_{1,\gamma/2}(0)} (|v| + |Dv|) \leq \nu, \\
	\int_{A^l_{1,\gamma/2}(0)} (|v|^2 + |Dv|^2) \leq C \int_{A^l_{1,\gamma}(0)} \mathcal{G}(u,\varphi)^2, 
\end{gather*}
for some $C = C(n,l,m,\varphi^{(0)},\gamma) \in (0,\infty)$. 
\end{lemma}

As an obvious consequence of Lemma \ref{graphrep_lemma} we obtain: 

\begin{cor} \label{graphrep_cor}
Suppose $\varphi^{(0)} \in \Phi_{\alpha}$ is a Dirichlet energy minimizing two-valued function with $\dim S(\varphi^{(0)}) = n-l_0$ for some integer $l_0 \in \{2,3,\ldots,n\}$.  Let $l \in \{l_0,l_0+1,\ldots,n\}$.  For every $\gamma,\nu,\tau \in (0,1)$ with $\tau < 1-\gamma$ there exists $\varepsilon_0, \beta_0 > 0$ depending on $n$, $l$, $m$, $\varphi^{(0)}$, $\gamma$, $\nu$, and $\tau$ such that if $\varphi \in \Phi_{\alpha}$ and $u \in \mathcal{F}_{\alpha}$ such that $\dim S(\varphi) = n-l$ and Hypothesis \ref{cft1_hyp} holds true with $\varepsilon = \varepsilon_0$ and $\beta = \beta_0$, then there is a harmonic single-valued function $v : B_{\gamma}(0) \cap \{ r > \tau \} \rightarrow \mathbb{R}^m$ such that 
\begin{equation*}
	u(X) = \{ \varphi_1(X) + v(X), -\varphi_1(X) - v(X) \}
\end{equation*}
for all $X \in B_{\gamma}(0) \cap \{ r > \tau \}$ and
\begin{gather*}
	\sup_{B_{\gamma}(0) \cap \{ r > \tau \}} (r^{-\alpha} |v| + r^{1-\alpha} |Dv|) \leq \nu, \nonumber \\
	\int_{B_{\gamma}(0) \cap \{ r > \tau \}} (|v|^2 + r^2 |Dv|^2) \leq C \int_{B_1(0)} \mathcal{G}(u,\varphi)^2, 
\end{gather*}
where $C = C(n,l,m,\varphi^{(0)},\gamma) \in (0,\infty)$ and $r(X) = \op{dist}(X,S(\varphi))$. 
\end{cor}

\begin{lemma} \label{keyest_lemma} 
Suppose $\varphi^{(0)} \in \Phi_{\alpha}$ is a Dirichlet energy minimizing two-valued function with $\dim S(\varphi^{(0)}) = n-l_0$ for some integer $l_0 \in \{2,3,\ldots,n\}$.  Let $l \in \{l_0,l_0+1,\ldots,n\}$.  For every $\gamma \in (0,1)$ there exists $\varepsilon_0, \beta_0 > 0$ depending on $n$, $l$, $m$, $\varphi^{(0)}$, and $\gamma$ such that if $\varphi \in \Phi_{\alpha}$ and $u \in \mathcal{F}_{\alpha}$ such that $\dim S(\varphi) = n-l$, Hypothesis \ref{cft1_hyp} holds true with $\varepsilon = \varepsilon_0$ and $\beta = \beta_0$, and $\mathcal{N}_u(0) \geq \alpha$, then 
\begin{equation} \label{keyest_eqn1}
	\int_{B_{\gamma}(0)} |\nabla_{S(\varphi)} u|^2 + \int_{B_{\gamma}(0)} R^{2-n} \left( \frac{\partial (u/R^{\alpha})}{\partial R} \right)^2 
	\leq C \int_{B_1(0)} \mathcal{G}(u,\varphi^{(0)})^2 
\end{equation} 
where $\nabla_{S(\varphi)}$ denotes the tangential derivative along $S(\varphi)$, $R(X) = |X|$, and $C = C(n,l,m,\varphi^{(0)},\gamma) \in (0,\infty)$.
\end{lemma}

\begin{cor} \label{radialdecay_cor} 
Suppose $\varphi^{(0)} \in \Phi_{\alpha}$ is a Dirichlet energy minimizing two-valued function with $\dim S(\varphi^{(0)}) = n-l_0$ for some integer $l_0 \in \{2,3,\ldots,n\}$.  Let $l \in \{l_0,l_0+1,\ldots,n\}$.  For every $\gamma,\sigma \in (0,1)$ there exists $\varepsilon_0, \beta_0 > 0$ depending on $n$, $l$, $m$, $\varphi^{(0)}$, $\gamma$, and $\sigma$ such that if $\varphi \in \Phi_{\alpha}$ and $u \in \mathcal{F}_{\alpha}$ such that $\dim S(\varphi) = n-l$, Hypothesis \ref{cft1_hyp} holds true with $\varepsilon = \varepsilon_0$ and $\beta = \beta_0$, and $\mathcal{N}_u(0) \geq \alpha$, then 
\begin{equation*} 
	\int_{B_{\gamma}(0)} \frac{\mathcal{G}(u,\varphi)^2}{|X|^{n+2\alpha-\sigma}} \leq C \int_{B_1(0)} \mathcal{G}(u,\varphi)^2 
\end{equation*} 
for some $C = C(n,l,m,\varphi^{(0)},\gamma) \in (0,\infty)$.
\end{cor}

\begin{cor} \label{branchdist_cor} 
Suppose $\varphi^{(0)} \in \Phi_{\alpha}$ is a Dirichlet energy minimizing two-valued function with $\dim S(\varphi^{(0)}) = \{0\} \times \mathbb{R}^{n-l_0}$ for some integer $l_0 \in \{2,3,\ldots,n\}$.  Let $l \in \{l_0,l_0+1,\ldots,n\}$.  For every $\gamma \in (0,1)$ there exists $\varepsilon_0, \beta_0 > 0$ depending on $n$, $l$, $m$, $\varphi^{(0)}$, and $\gamma$ such that the following holds true.  Let $\varphi \in \Phi_{\alpha}$ and $u \in \mathcal{F}_{\alpha}$ such that Hypothesis \ref{cft1_hyp} holds true with $\varepsilon = \varepsilon_0$ and $\beta = \beta_0$.  Let $\{\varphi_j\}_{j = 0,1,2,\ldots,N} \subset \Phi_{\alpha}$ be an optimal sequence with increasing spines such that $\varphi_0 = \varphi$ and $S(\varphi_N) = \{0\} \times \mathbb{R}^{n-l_0}$.  Let $Z \in B_{1/2}(0)$ with $\mathcal{N}_u(Z) \geq \alpha$.  Then 
\begin{gather*}
	\op{dist}(Z,S(\varphi_j))^2 \leq \frac{C \int_{B_1(0)} \mathcal{G}(u,\varphi)^2}{\int_{B_1(0)} \mathcal{G}(\varphi,\varphi_{j+1})^2} 
		\text{ for } j = 0,1,2,\ldots,N-1, \nonumber \\
	\op{dist}(Z,S(\varphi^{(0)}))^2 \leq C \int_{B_1(0)} \mathcal{G}(u,\varphi)^2, 
\end{gather*} 
for some $C = C(n,l,m,\varphi^{(0)},\gamma) \in (0,\infty)$. 
\end{cor}

\begin{cor} \label{nonconest_cor} 
Suppose $\varphi^{(0)} \in \Phi_{\alpha}$ is a Dirichlet energy minimizing two-valued function with $\dim S(\varphi^{(0)}) = \{0\} \times \mathbb{R}^{n-l_0}$ for some integer $l_0 \in \{2,3,\ldots,n\}$.  Let $l \in \{l_0,l_0+1,\ldots,n\}$.  For every $\gamma,\tau,\sigma \in (0,1)$ there exists $\varepsilon_0, \beta_0 > 0$ depending on $n$, $l$, $m$, $\varphi^{(0)}$, $\gamma$, $\tau$, and $\sigma$ such that the following holds true.  Let $\varphi \in \Phi_{\alpha}$ and $u \in \mathcal{F}_{\alpha}$ such that Hypothesis \ref{cft1_hyp} holds true with $\varepsilon = \varepsilon_0$ and $\beta = \beta_0$ and $\dim S(\varphi) = n-l$.  Let $\varphi_1$ and $v$ be as in Corollary \ref{graphrep_cor}.  Let $Z \in B_{1/2}(0)$ with $\mathcal{N}_u(Z) \geq \alpha$.  Then 
\begin{align*}
	&\text{(i)} \hspace{3mm} \int_{B_{\gamma}(0)} \frac{\mathcal{G}(u,\varphi)^2}{|X-Z|^{n-\sigma}} 
		\leq C \int_{B_1(0)} \mathcal{G}(u,\varphi)^2, \nonumber \\
	&\text{(ii)} \hspace{3mm} \int_{B_{\gamma}(0) \cap \{|x| > \tau\}} \frac{|v(X) - D\varphi_1(X) \cdot Z|^2}{|X-Z|^{2\alpha+n-\sigma}} 
		\leq C \int_{B_1(0)} \mathcal{G}(u,\varphi)^2, 
\end{align*} 
for some $C = C(n,l,m,\varphi^{(0)},\gamma,\sigma) \in (0,\infty)$. 
\end{cor}

We will prove Lemma \ref{graphrep_lemma}, Lemma \ref{keyest_lemma}, Corollary \ref{radialdecay_cor}, Corollary \ref{branchdist_cor}, and Corollary \ref{nonconest_cor} in Section~\ref{sec:proof_est_sec} via an induction argument on $l$.  Before doing so, we consider an important consequence of these results, the existence of blow-ups relative to elements of $\Phi_{\alpha}$.  First we need the following lemma: 

\begin{lemma} \label{twovalL2_lemma}
Let $\gamma \in (0,1)$, $d \geq 0$ be an integer, and $\psi : B_1(0) \rightarrow \mathbb{R}^m$ be a nonzero single-valued polynomial of degree at most $d$.  There exists $\varepsilon = \varepsilon(n,m,\gamma,\psi) > 0$ such that if $u,v \in C^d(B_1(0);\mathbb{R}^m)$ are single-valued functions such that  
\begin{equation} \label{twovalL2_eqn1}
	\|u - \psi\|_{C^d(B_1(0))} + \|v - \psi\|_{C^d(B_1(0))} < \varepsilon,
\end{equation}
then 
\begin{equation} \label{twovalL2_eqn2}
	\int_{B_{\gamma}(0)} |u - v|^2 \leq C \int_{B_1(0)} \mathcal{G}(\{\pm u\}, \{\pm v\})^2,  
\end{equation}
for some constant $C = C(n,m,\gamma,\psi) \in (0,\infty)$.  
\end{lemma}
\begin{proof}
We shall proceed by induction on the degree $d$ of the function $\psi$.  The case $d = 0$ is obvious.  Assume the induction hypothesis that for some integer $d \geq 1$, Lemma \ref{twovalL2_lemma} holds true in the case that $\psi$ is a polynomial of degree at most $d-1$.  We want to show that Lemma \ref{twovalL2_lemma} holds true in the case that $\psi$ is a polynomial of degree $d$.  By locally approximating $\psi$ by homogeneous polynomials, it suffices to prove Lemma \ref{twovalL2_lemma} in the special case that $\psi$ is a homogeneous degree $d$ polynomial. 

Given a homogeneous degree $d$ polynomial $\psi$, let 
\begin{equation*}
	L(\psi) = \{ X : D^k \psi(X) = 0 \text{ for all } k < d \} ,
\end{equation*}
noting that $L(\psi)$ is always a linear subspace and $\psi(Y+X) = \psi(X)$ for all $X \in \mathbb{R}^n$ and $Y \in L(\psi)$.  We claim that it suffices to prove Lemma \ref{twovalL2_lemma} in the special case that $\psi$ is a homogeneous degree $d$ polynomial with $L(\psi) = \{0\}$.  Having shown that, if $\psi : B_1(0) \rightarrow \mathbb{R}^m$ is any nonzero homogeneous degree $d$ polynomial and $u,v \in C^d(B_1(0);\mathbb{R}^m)$ satisfy (\ref{twovalL2_eqn1}), we could rotate so that $L(\psi) = \{0\} \times \mathbb{R}^{n-l}$ for some integer $l$ and conclude that 
\begin{equation} \label{twovalL2_eqn3}
	\int_{B^l_{\gamma \sqrt{1 - |y|^2}}(0)} |u(x,y) - v(x,y)|^2 dx \leq C \int_{B^l_{\sqrt{1 - |y|^2}}(0)} \mathcal{G}(\{\pm u\}, \{\pm v\})^2 
\end{equation}
for all $y \in B^{n-l}_{\gamma}(0)$ and some constant $C = C(n,m,\gamma,\psi) \in (0,\infty)$, provided $\varepsilon$ is sufficiently small.  Then using the fact that $B^l_{\sqrt{\gamma^2 - |y|^2}}(0) \subseteq B^l_{\gamma \sqrt{1 - |y|^2}}(0)$ and integrating both sides of (\ref{twovalL2_eqn3}) over $y \in B^{n-l}_{\gamma}(0)$, we obtain (\ref{twovalL2_eqn2}). 

Suppose that $\psi : B_1(0) \rightarrow \mathbb{R}^m$ is a nonzero, homogeneous degree $d$, harmonic single-valued function such that $L(\psi) = \{0\}$.  Suppose that for every integer $j \geq 1$, there exists single-valued functions $u_j,v_j \in C^d(B_1(0);\mathbb{R}^m)$ such that $u_j \rightarrow \psi$ and $v_j \rightarrow \psi$ in $C^d(B_1(0);\mathbb{R}^m)$ but 
\begin{equation} \label{twovalL2_eqn4}
	\int_{B_1(0)} \mathcal{G}(\{\pm u_j\}, \{\pm v_j\})^2 < \frac{1}{j} \int_{B_{\gamma}(0)} |u_j - v_j|^2.  
\end{equation}
Without loss of generality assume that 
\begin{align} \label{twovalL2_eqn5}
	&\|u_j - \psi\|_{C^d(B_1(0))} + \|v_j - \psi\|_{C^d(B_1(0))} \nonumber \\ 
	&\leq 2 \inf_{Z \in \mathbb{R}^n} \left( \|u_j(X) - \psi(X-Z)\|_{C^d(B_1(0))} + \|v_j(X) - \psi(X-Z)\|_{C^d(B_1(0))} \right) . 
\end{align}
By the induction hypothesis and the homogeneity of $\psi$, there exists $\varepsilon_0 = \varepsilon_0(n,m,\gamma,\psi) > 0$ such that if $u,v \in C^d(B_1(0);\mathbb{R}^m)$ such that 
\begin{equation*}
	\sum_{k=0}^d \rho^k \sup_{B_{\rho}(0) \setminus B_{\gamma^3 \rho}(0)} |D^k u - D^k \psi| 
	+ \sum_{k=0}^d \rho^k \sup_{B_{\rho}(0) \setminus B_{\gamma^3 \rho}(0)} |D^k v - D^k \psi| < \varepsilon_0 \rho^d,
\end{equation*}
then 
\begin{equation} \label{twovalL2_eqn6}
	\int_{B_{\gamma \rho}(0) \setminus B_{\gamma^2 \rho}(0)} |u - v|^2 \leq C \int_{B_1(0) \setminus B_{\gamma^3 \rho}(0)} \mathcal{G}(\{\pm u\}, \{\pm v\})^2 
\end{equation}
for some constant $C = C(n,m,\gamma,\psi) \in (0,\infty)$.   Thus if $D^k \tilde u_j(0) = D^k \tilde v_j(0) = 0$ for $k < d$, then 
\begin{align*}
	|D^k u_j(X) - D^k \psi(X)| \leq C(n,d,k) \sup_{B_1(0)} |D^d u_j - D^d \psi| |X|^{d-k}, \\
	|D^k v_j(X) - D^k \psi(X)| \leq C(n,d,k) \sup_{B_1(0)} |D^d v_j - D^d \psi| |X|^{d-k}, 
\end{align*}
for $k = 0,1,2,\ldots,d$ and $X \in B_1(0)$, so by (\ref{twovalL2_eqn6}) and the obvious covering argument, (\ref{twovalL2_eqn2}) holds true with $u = u_j$ and $v = v_j$, contradicting (\ref{twovalL2_eqn4}).  Thus we may rescale letting $\tilde u_j(X) = \lambda_j^{-d} u_j(\lambda_j X)$ and $\tilde v_j(X) = \lambda_j^{-d} v_j(\lambda_j X)$ where $\lambda_j > 0$ is chosen such that 
\begin{equation} \label{twovalL2_eqn7}
	\sum_{k=1}^{d-1} (|D^k \tilde u_j(0)|^2 + |D^k \tilde v_j(0)|^2) = 1. 
\end{equation}
Since $u_j \rightarrow \psi$ and $v_j \rightarrow \psi$ in $C^d(B_1(0);\mathbb{R}^m)$, $\lambda_j \rightarrow 0$ and 
\begin{align} \label{twovalL2_eqn8}
	&\lim_{j \rightarrow \infty} \sup_{B_{1/\lambda_j}(0)} |D^d \tilde u_j - D^d \psi| 
		= \lim_{j \rightarrow \infty} \sup_{B_1(0)} |D^d u_j - D^d \psi| = 0, \nonumber \\ 
	&\lim_{j \rightarrow \infty} \sup_{B_{1/\lambda_j}(0)} |D^d \tilde v_j - D^d \psi| 
		= \lim_{j \rightarrow \infty} \sup_{B_1(0)} |D^d v_j - D^d \psi| = 0. 
\end{align}
By (\ref{twovalL2_eqn4}), 
\begin{equation} \label{twovalL2_eqn9}
	\int_{B_{1/\lambda_j}(0)} \mathcal{G}(\{\pm \tilde u_j\}, \{\pm \tilde v_j\})^2 
	< \frac{1}{j} \int_{B_{\gamma/\lambda_j}(0)} |\tilde u_j - \tilde v_j|^2.  
\end{equation}
By (\ref{twovalL2_eqn7}) and (\ref{twovalL2_eqn8}), 
\begin{align*}
	|D^k \tilde u_j(X) - D^k \psi(X)| \leq C(n,d,k) (|X|^{d-k-1} + \sup_{B_1(0)} |D^d u_j - D^d \psi| |X|^{d-k}), \\
	|D^k \tilde v_j(X) - D^k \psi(X)| \leq C(n,d,k) (|X|^{d-k-1} + \sup_{B_1(0)} |D^d v_j - D^d \psi| |X|^{d-k}), 
\end{align*}
for $k = 0,1,2,\ldots,d$ and $X \in B_{1/\lambda_j}(0) \setminus B_1(0)$, so by (\ref{twovalL2_eqn6}) and the obvious covering argument there exists $R > 0$ depending on $\psi$ such that 
\begin{equation} \label{twovalL2_eqn10}
	\int_{B_{\gamma/\lambda_j}(0) \setminus B_R(0)} |\tilde u_j - \tilde v_j|^2 
	\leq C \int_{B_{1/\lambda_j}(0) \setminus B_{\gamma R}(0)} \mathcal{G}(\{\pm \tilde u_j\}, \{\pm \tilde v_j\})^2 
\end{equation}
for large $j$ and some constant $C = C(n,m,\gamma,\psi) \in (0,\infty)$.  By (\ref{twovalL2_eqn7}) and (\ref{twovalL2_eqn8}), $\tilde u_j - \psi \rightarrow f$ and $\tilde v_j - \psi \rightarrow g$ in $C^d$ on compact subsets of $\mathbb{R}^n$ for some nonzero polynomials $f$ and $g$ of degree at most $d-1$.  By (\ref{twovalL2_eqn9}) and (\ref{twovalL2_eqn10}), 
\begin{equation*}
	\int_{B_{\gamma/\lambda_j}(0) \setminus B_R(0)} |\tilde u_j - \tilde v_j|^2 \leq \frac{2C}{j} \int_{B_R(0)} |\tilde u_j - \tilde v_j|^2
\end{equation*}
for large $j$ and so by letting $j \rightarrow \infty$, $f = g$ on $\mathbb{R}^n \setminus B_R(0)$.  Since $f$ and $g$ are polynomials, $f = g$ on $\mathbb{R}^n$.  Observe that we cannot have $D^k f(Z) + D^k \psi(Z) = 0$ for all $k = 1,2,\ldots,d-1$ for some $Z \in B_R(0)$ since otherwise $f(X) + \psi(X) = \psi(X-Z)$ for all $X \in \mathbb{R}^n$ and thus 
\begin{align*}
	&\|u_j(X) - \psi(X- \lambda_j Z)\|_{C^d(B_1(0))} + \|v_j(X) - \psi(X - \lambda_j Z)\|_{C^d(B_1(0))}
	\\&< \frac{1}{2} \|u_j(X) - \psi(X-Z)\|_{C^d(B_1(0))} + \|v_j(X) - \psi(X-Z)\|_{C^d(B_1(0))}
\end{align*}
for $j$ sufficiently large, contradicting (\ref{twovalL2_eqn5}).  By locally approximating $f+\psi$ by Taylor polynomials of degree less than $d$ near points in $B_R(0)$ and applying the induction hypothesis,  
\begin{equation} \label{twovalL2_eqn11}
	\int_{B_R(0)} |\tilde u_j - \tilde v_j|^2 \leq C \int_{B_{2R}(0)} \mathcal{G}(\{\pm \tilde u_j\}, \{\pm \tilde v_j\})^2 
\end{equation}
for large $j$ and some constant $C = C(n,m,f,\psi) \in (0,\infty)$.  But (\ref{twovalL2_eqn10}) and (\ref{twovalL2_eqn11}) contradict (\ref{twovalL2_eqn4}) for large $j$. 
\end{proof}

\noindent \textbf{Blow-ups.} Suppose $\varepsilon_j, \beta_j \rightarrow 0^+$ and $\varphi^{(0)}, \varphi_j \in \Phi_{\alpha}$ and $u_j \in \mathcal{F}_{\alpha}$ such that Hypothesis \ref{cft1_hyp} holds true with $\varepsilon = \varepsilon_j$, $\beta = \beta_j$, $\varphi = \varphi_j$, and $u = u_j$.  After passing to a subsequence, we may assume that $l = n-\dim S(\varphi_j)$ is independent of $j$.  Assume that Corollary \ref{graphrep_cor} and Corollary \ref{nonconest_cor}(i) all hold true for this particular value of $l$.  Let $\tau_j \rightarrow 0^+$ slowly enough that Corollary \ref{graphrep_cor} holds true with $(1+\gamma)/2$ in place of $\gamma$, $\tau = \tau_j$, $\sigma = 1/2$, $\varphi = \varphi_j$, and $u = u_j$.  

By the sequential compactness of the space of closed subsets of a compact space equipped with the Hausdorff metric, after passing to a subsequence $\mathcal{B}_{u_j}$ converges to some closed subset $\mathcal{D} \subseteq \{0\} \times B^{n-l}_1(0)$ in Hausdorff distance on compact subsets of $B_1(0)$.  By Corollary \ref{graphrep_cor} and the Schauder estimates, we get $v_j \in C^2(U_j;\mathbb{R}^m)$, where $U_j = \{X \in B_{(1+\gamma)/2}(0) : \op{dist}(X,\mathcal{D}) > \tau_j\}$, such that 
\begin{equation*}
	u_j(X) = \{ \pm (\varphi{j,1}(X) + v_j(X)) \} 
\end{equation*}
for all $X \in U_j$, where $\varphi_j(X) = \{ \pm \varphi_{j,1}(X) \}$ for all $X \in \mathbb{R}^n$ and a unique harmonic single-valued function $\varphi_{j,1} : \mathbb{R}^n \rightarrow \mathbb{R}^m$ that is close to $\varphi^{(0)}_1$ in $L^2(B_1(0);\mathbb{R}^m)$.  Define 
\begin{equation*}
	w_j = v_j/E_j \quad \text{for} \quad E_j = \left( \int_{B_1(0)} \mathcal{G}(u_j,\varphi_j)^2 \right)^{1/2}. 
\end{equation*}
By elliptic estimates and Lemma \ref{twovalL2_lemma}, 
\begin{equation*}
	\|w_j\|_{C^3(K)} \leq C, 
\end{equation*}
for large $j$ and some constant $C = C(n,m,K,\varphi^{(0)}) \in (0,\infty)$ for every compact subset $K$ of $B_1(0) \setminus \mathcal{D}$, so after passing to a subsequence $w_j$ converge to some $w$ in $C^2(K;\mathbb{R}^m)$ for every compact subset $K$ of $B_1(0) \setminus \mathcal{D}$.  We call any such $w$ a \textit{blow-up} of $u_j$ relative to $\varphi_j$ over $B_{\gamma}(0)$.  We claim that $w_j \rightarrow w$ in $L^2(B_{\gamma}(0);\mathbb{R}^m)$. 

Given $\delta \in (0,(1-\gamma)/2)$, for every $(0,y_0) \in B_{\gamma}(0) \cap \mathcal{D}$ and for $j$ sufficiently large there exists $Z_j \in \mathcal{B}_{u_j} \cap B_{\delta}(0,y_0)$ and thus by Corollary \ref{nonconest_cor}(i) then  
\begin{equation} \label{bu_eqn1}
	\int_{B_{(1+\gamma)/2}(0)} \frac{\mathcal{G}(u_j,\varphi^{(0)})^2}{|X-Z_j|^{n-1/2}} \leq C \int_{B_1(0)} \mathcal{G}(u_j,\varphi^{(0)})^2 
\end{equation} 
for some $C = C(n,l,m,\varphi^{(0)},\gamma) \in (0,\infty)$.  By (\ref{bu_eqn1}),  
\begin{equation} \label{bu_eqn2}
	\int_{B_{2\delta}(0,y_0)} \mathcal{G}(u_j,\varphi_j)^2 \leq C \delta^{n-1/2} \int_{B_1(0)} \mathcal{G}(u_j,\varphi_j)^2. 
\end{equation} 
By covering $\{ X \in B_{\gamma}(0) : \op{dist}(X,\mathcal{D}) \leq \delta \}$ by $N$ balls $B_{2\delta}(0,z_k)$ for $z_k \in B_{\rho/2}(0) \cap \mathcal{D}$ and $N \leq C(n,l,\gamma) \delta^{l-n}$ and summing (\ref{bu_eqn2}) with $y_0 = z_k$ over $k = 1,2,\ldots,N$, 
\begin{equation*}
	\int_{\{ X \in B_{\gamma}(0) : \op{dist}(X,\mathcal{D}) \leq \delta \}} \mathcal{G}(u_j,\varphi_j)^2 
	\leq C \delta^{l-1/2} \int_{B_1(0)} \mathcal{G}(u_j,\varphi_j)^2 
\end{equation*} 
for $j$ sufficiently large and for some $C = C(n,l,m,\varphi^{(0)},\gamma) \in (0,\infty)$.  Consequently $w_j \rightarrow w$ in $L^2(B_{\gamma}(0);\mathbb{R}^m)$ with 
\begin{equation*}
	\lim_{j \rightarrow \infty} E_j^{-2} \int_{B_{\gamma}(0)} \mathcal{G}(u_j,\varphi_j)^2 = \int_{B_{\gamma}(0)} |w|^2.  
\end{equation*}

\section{Apriori estimates - Part II: Proof of estimates} \label{sec:proof_est_sec}

Fix a Dirichlet energy minimizing two-valued function $\varphi^{(0)} \in \Phi_{\alpha}$ and let $l_0 = n - \dim S(\varphi^{(0)})$.  The proof of the results Lemma \ref{graphrep_lemma}, Lemma \ref{keyest_lemma}, Corollary \ref{radialdecay_cor}, Corollary \ref{branchdist_cor}, and Corollary \ref{nonconest_cor} will proceed by induction, assuming for some $\lambda \in \{l_0,l_0+1,\ldots,n\}$ that:

\begin{hypothesis} \label{graphrep_hyp}
	Either $\lambda = l_0$ or $\lambda > l_0$ and Lemma \ref{graphrep_lemma}, Lemma \ref{keyest_lemma}, Corollary \ref{radialdecay_cor}, Corollary \ref{branchdist_cor}, and Corollary \ref{nonconest_cor} all hold true whenever $l_0 \leq l < \lambda$. 
\end{hypothesis}

\begin{proof}[Proof of Lemma \ref{graphrep_lemma} for $l = \lambda$]
We claim that whenever $\varepsilon_0$ and $\beta_0$ are sufficiently small and the hypotheses of Lemma \ref{graphrep_lemma} hold true with $l = \lambda$, $\mathcal{B}_u \cap A^{\lambda}_{1,3\gamma/4}(0) = \emptyset$.  Then by the Schauder estimates and the structure of $\varphi^{(0)}$ when $\lambda = 2$ and $A^{\lambda}_{1,5\gamma/8}(0)$ being simply connected when $\lambda \geq 3$, there exists a harmonic single-valued function $v : A^{\lambda}_{1,5\gamma/8}(0) \rightarrow \mathbb{R}^m$ such that $u(X) = \{ \pm (\varphi_1(X) + v(X)) \}$ for all $X \in A^{\lambda}_{1,5\gamma/8}(0)$ and the estimates on $v$ in $A^{\lambda}_{1,\gamma/2}(0)$ follow from the Schauder estimates and Lemma \ref{twovalL2_lemma}.  Since the claim is obviously true when $\lambda = l_0$, we shall assume that $\lambda > l_0$.  By rotating $\mathbb{R}^n$ slightly, it suffices to prove the claim under the assumption that $S(\varphi^{(0)}) = \{0\} \times \mathbb{R}^{n-l_0}$ and $S(\varphi) = \{0\} \times \mathbb{R}^{n-\lambda}$.

Suppose there exist $\varepsilon_j, \beta_j \downarrow 0$, $\varphi_j$, and $u_j$ such that the hypotheses of Lemma \ref{graphrep_lemma} hold true with $\varepsilon_j$, $\beta_j$, $\varphi^{(0)}$, $\varphi_j$, and $u_j$ in place of $\varepsilon_0$, $\beta_0$, $\varphi^{(0)}$, $\varphi$, and $u$, $S(\varphi^{(0)}) = \{0\} \times \mathbb{R}^{n-l_0}$, and $S(\varphi_j) = \{0\} \times \mathbb{R}^{n-\lambda}$ but there exists $Z_j \in \mathcal{B}_{u_j} \cap A^{\lambda}_{1,3\gamma/4}(0)$.  

Let $\{\varphi_{j,k}\}_{k = 0,1,2,\ldots,N_j} \subset \Phi_{\alpha}$ be an optimal sequence with increasing spines over $A^{\lambda}_{1,1}(0)$ instead of $B_1(0)$ with $\varphi_{j,0} = \varphi_j$.  After passing to a subsequence and making an orthogonal change of coordinates fixing $\{0\} \times \mathbb{R}^{n-\lambda}$ and $\{0\} \times \mathbb{R}^{n-l_0}$, assume that $N_j$ and $\dim S(\varphi_{j,k})$ are independent of $j$ and that $S(\varphi_{j,k})$ converge to $\{0\} \times \mathbb{R}^{n-p_k}$ for some integers $\lambda = p_0 > p_1 > p_2 > \cdots > p_N = l_0$.  Then after small rotations of $u_j$, $\varphi_j$, and $\varphi_{j,k}$, we may assume that $S(\varphi_{j,k}) = \{0\} \times \mathbb{R}^{n-p_k}$.  After passing to a subsequence, we can find $K \in \{1,2,\ldots,N\}$ such that 
\begin{equation*}
	\liminf_{j \rightarrow \infty} \frac{\int_{A^{\lambda}_{1,1}(0)} \mathcal{G}(\varphi_j,\varphi_{j,k})^2}{
		\int_{A^{\lambda}_{1,1}(0)} \mathcal{G}(\varphi_j,\varphi_{j,k+1})^2} > 0 \text{ for } k = 1,2,\ldots,N-1
\end{equation*}
and either $K = N$ or 
\begin{equation} \label{graphrep_eqn1}
	\lim_{j \rightarrow \infty} \frac{\int_{A^{\lambda}_{1,1}(0)} \mathcal{G}(\varphi_j,\varphi_{j,K})^2}{
		\int_{A^{\lambda}_{1,1}(0)} \mathcal{G}(\varphi_j,\varphi_{j,K+1})^2} = 0.
\end{equation}
Note that by the definition of an optimal sequence with increasing spines, for large $j$, 
\begin{align} \label{graphrep_eqn2}
	\int_{A^{\lambda}_{1,1}(0)} \mathcal{G}(\varphi_{j,k},\varphi_{j,k+1})^2 
	&\leq 2 \int_{A^{\lambda}_{1,1}(0)} \mathcal{G}(\varphi_{j,K},\varphi_{j,k})^2 + 2 \int_{A^{\lambda}_{1,1}(0)} \mathcal{G}(\varphi_{j,K},\varphi_{j,k+1})^2 
		\nonumber \\
	&\leq 6 \int_{A^{\lambda}_{1,1}(0)} \mathcal{G}(\varphi_{j,K},\varphi_{j,k+1})^2. 
\end{align}

By Corollary \ref{branchdist_cor}, we can find, in order, rotations $\Gamma_{j,N}, \Gamma_{j,N-1}, \ldots, \Gamma_{j,K}$ of $\mathbb{R}^n$ such that $\Gamma_{j,k}$ is the identity map on $\{0\} \times \mathbb{R}^{n-\lambda}$ for all $k$, $\Gamma_{j,k}$ is the identity map on $\mathbb{R}^{p_{k+1}} \times \{0\}$ for $k < N$, and 
\begin{align} \label{graphrep_eqn3}
	\sup_{\mathbb{R}^n} |\Gamma_{j,k}-I| &\leq C \frac{\int_{A^{\lambda}_{1,1}(0)} \mathcal{G}(u_j,\varphi_{j,K})^2}{
		\int_{A^{\lambda}_{1,1}(0)} \mathcal{G}(\varphi_{j,K},\varphi_{j,k+1})^2} \text{ for } k < N, \nonumber \\
	\sup_{\mathbb{R}^n} |\Gamma_{j,N}-I| &\leq C \int_{A^{\lambda}_{1,1}(0)} \mathcal{G}(u_j,\varphi_{j,K})^2, 
\end{align}
for some $C = C(n,\lambda,m,\varphi^{(0)}) \in (0,\infty)$ and 
\begin{equation} \label{graphrep_eqn4}
	(\Gamma_{j,k} \circ \Gamma_{j,k+1} \circ \cdots \circ \Gamma_{j,N})(Z_j) \in \{0\} \times \mathbb{R}^{n-p_k} 
\end{equation}
for all $k$.  Let $\Gamma_j = \Gamma_{j,K} \circ \Gamma_{j,K+1} \circ \cdots \circ \Gamma_{j,N}$.  By Hypothesis \ref{graphrep_hyp}, (\ref{graphrep_eqn1}), and (\ref{graphrep_eqn3}), there exists $\eta_j \rightarrow 0^+$ such that 
\begin{equation*}
	\int_{A^{\lambda}_{1,1}(0)} \mathcal{G}(u_j \circ \Gamma_j^{-1},\varphi^{(0)})^2 
	\leq 2\int_{A^{\lambda}_{1,1}(0)} \mathcal{G}(u_j,\varphi^{(0)})^2 + C \sup_{\mathbb{R}^n} |\Gamma_j - I|^2 
	\leq 2\varepsilon_j + \eta_j, 
\end{equation*}
which converges to zero as $j \rightarrow \infty$, and 
\begin{align} \label{graphrep_eqn5}
	\int_{A^{\lambda}_{1,1}(0)} \mathcal{G}(u_j \circ \Gamma_j^{-1},\varphi_{j,K})^2 
	&\leq C\int_{A^{\lambda}_{1,1}(0)} \mathcal{G}(u_j,\varphi_{j,K})^2 \nonumber \\&\hspace{3mm}
		+ C \sum_{k=K}^{N-1} \int_{A^{\lambda}_{1,1}(0)} |D\varphi_{j,K,1} - D\varphi_{j,k+1,1}|^2 
			\sup_{\{0\} \times \mathbb{R}^{p_k-p_{k+1}} \times \{0\}} |\Gamma_j - I|^2 \nonumber \\&\hspace{3mm}
		+ C \int_{A^{\lambda}_{1,1}(0)} |D\varphi_{j,K,1}|^2 \sup_{\mathbb{R}^{l_0} \times \{0\}} |\Gamma_j - I|^2 \nonumber \\
	&\leq C\int_{A^{\lambda}_{1,1}(0)} \mathcal{G}(u_j,\varphi_{j,K})^2 
\end{align}
and so 
\begin{align*}
	\int_{A^{\lambda}_{1,1}(0)} \mathcal{G}(u_j \circ \Gamma_j^{-1},\varphi_{j,K})^2 
	\leq \eta_j \int_{A^{\lambda}_{1,1}(0)} \mathcal{G}(\varphi_{j,K},\varphi_{j,K+1})^2,  
\end{align*}
where $\varphi_{j,k}(X) = \{ \pm \varphi_{j,k,1}(X) \}$ for all $X \in \mathbb{R}^n$ and some harmonic single-valued function $\varphi_{j,k,1}$ that is close to $\varphi^{(0)}_1$ in $L^2(B_1(0);\mathbb{R}^m)$ and $C = C(n,\lambda,m,\varphi^{(0)}) \in (0,\infty)$.  Hence Hypothesis \ref{graphrep_hyp} holds true with $\varepsilon = 2\varepsilon_j + \eta_j$, $\beta = \eta_j$, $u = u_j \circ \Gamma_j^{-1}$, and $\varphi = \varphi_{j,K}$.  

Now, after passing to a subsequence, $u_j \circ \Gamma_j^{-1}$ blows up to some $w$ relative to $\varphi_{j,K}$ over $A^{\lambda}_{1,7\gamma/8}(0)$ (in place of the ball $B_1(0)$).  By Hypothesis \ref{cft1_hyp}(ii), $w$ is also the blow-up of $\varphi_j \circ \Gamma_j^{-1}$ relative to $\varphi_{j,K}$, so in particular $w$ is a nonzero, homogeneous degree $\alpha$, harmonic polynomial.  Write $\varphi_{j,k}(X) = \{\pm \varphi_{j,k,1}(X)\}$ for $X \in \mathbb{R}^n$ for a harmonic single-valued function $\varphi_{j,k,1}$ that is close to $\varphi^{(0)}_1$ in $L^2(B_1(0);\mathbb{R}^m)$ and let $\widetilde{\varphi}_j(X) = \{ \pm (\varphi_{j,K,1}(X) + E_j w(X)) \}$ for $X \in \mathbb{R}^n$, where $E_j = \left( \int_{A^{\lambda}_{1,\gamma}(0)} \mathcal{G}(u_j,\varphi_{j,K})^2 \right)^{1/2}$.  Obviously $\{0\} \times \mathbb{R}^{n-\lambda} \subseteq S(\widetilde{\varphi}_j) \subseteq \{0\} \times \mathbb{R}^{n-p_K}$ for large $j$.

Write $Z_j = (0,\zeta_j)$ for $\zeta_j \in \mathbb{R}^{n-p_K}$.  After passing to a subsequence, $\zeta_j \rightarrow \zeta$ in $\mathbb{R}^{n-p_K}$ and $\op{dist}((0,\zeta),\{0\} \times \mathbb{R}^{n-\lambda}) \geq 1-3\gamma/4$.  By Corollary \ref{nonconest_cor}(ii) applied on a ball whose center lies on $S(\varphi_{j,K})$, 
\begin{equation} \label{graphrep_eqn6}
	\int_{B_{\gamma/16}(Z_j) \cap \{|x| > \tau_j\}} \frac{|v_j(X)|^2}{|X-Z_j|^{2\alpha+n-1/2}} \leq C \int_{B_1(0)} \mathcal{G}(u_j,\varphi_{j,K})^2 
\end{equation}
for some $C = C(n,\lambda,m,\varphi^{(0)},\gamma) \in (0,\infty)$, where $\tau_j \rightarrow 0^+$ slowly enough that Corollaries \ref{graphrep_cor} and \ref{nonconest_cor} hold with $(1+\gamma)/2$ in place of $\gamma$, $\tau = \tau_j$, $\sigma = 1/2$, $\varphi = \varphi_{j,K}$, and $u = u_j \circ \Gamma_j^{-1}$ and $v = v_j$ in Corollary \ref{graphrep_cor}.  By dividing (\ref{graphrep_eqn6}) by $E_j^2$ and using (\ref{graphrep_eqn1}) and (\ref{graphrep_eqn4}), 
\begin{equation*}
	\int_{B_{\gamma/32}(0,\zeta)} \frac{|w(X)|^2}{|X-Z|^{2\alpha+n-1/2}} \leq C 
\end{equation*} 
for some $C = C(n,\lambda,m,\varphi^{(0)}) \in (0,\infty)$, which implies that $w$ vanishes at order $\alpha$ at $(0,\zeta)$.  Hence $S(\varphi_j) = \{0\} \times \mathbb{R}^{n-\lambda} \subset S(\widetilde{\varphi}_j)$.  

Since $w$ is the blow-up of both $\varphi_j \circ \Gamma_j^{-1}$ and $\widetilde{\varphi}_j$ relative to $\varphi_{j,K}$ over $A^{\lambda}_{1,7\gamma/8}(0)$, 
\begin{equation*}
	\lim_{j \rightarrow \infty} \frac{\int_{A^{\lambda}_{1,1}(0)} \mathcal{G}(\varphi_j \circ \Gamma_j^{-1},\widetilde{\varphi}_j)^2}{
		\int_{A^{\lambda}_{1,1}(0)} \mathcal{G}(u_j \circ \Gamma_j^{-1},\varphi_{j,K})^2} = 0
\end{equation*}
and thus by Hypothesis \ref{graphrep_hyp}(ii) and (\ref{graphrep_eqn5}), 
\begin{equation*}
	\lim_{j \rightarrow \infty} \frac{\int_{A^{\lambda}_{1,1}(0)} \mathcal{G}(u_j,\widetilde{\varphi}_j \circ \Gamma_j)^2}{
		\int_{A^{\lambda}_{1,1}(0)} \mathcal{G}(u_j,\varphi_{j,K})^2} = 0, 
\end{equation*}
contradicting the definitions of $\{\varphi_{j,k}\}_{k=0,1,\ldots,N}$ and of $K$ as the smallest value satisfying (\ref{graphrep_eqn1}). 
\end{proof}

\begin{lemma} \label{translate_lemma}
Suppose $\varphi^{(0)} \in \Phi_{\alpha}$ is a Dirichlet energy minimizing two-valued function with $S(\varphi^{(0)}) = \{0\} \times \mathbb{R}^{n-l_0}$ for some integer $l_0 \in \{2,3,\ldots,n\}$.  Let $\lambda$ and $l$ be integers such that Hypothesis \ref{graphrep_hyp} holds true and $l_0 \leq \lambda \leq l \leq n$.  Given $\varepsilon_1,\beta_1 > 0$, and $\rho \in (0,1/2]$, there exists $\varepsilon_2, \beta_2 > 0$ depending on $n$, $\lambda$, $m$, $\varphi^{(0)}$, $\varepsilon_1$, $\beta_1$, and $\rho$ such that if $\varphi \in \Phi_{\alpha}$, $u \in \mathcal{F}_{\alpha}$, $Z \in B_{1/2}(0)$ such that Hypothesis \ref{cft1_hyp} holds true with $\varepsilon = \varepsilon_2$ and $\beta = \beta_2$, $S(\varphi) = \{0\} \times \mathbb{R}^{n-l}$, and $\mathcal{N}_u(Z) \geq \alpha$, then Hypothesis \ref{cft1_hyp} holds true with $\widetilde{u}(X) = \rho^{-\alpha} u(Z + \rho X)$ in place of $u$ and $\varepsilon = \varepsilon_1$ and $\beta = \beta_1$. 
\end{lemma}
\begin{proof}
By Hypothesis \ref{graphrep_hyp}(i), given $\delta > 0$ if $\varepsilon_2$ and $\beta_2$ are sufficiently small then 
\begin{equation*}
	\op{dist}(Z,S(\varphi^{(0)})) \leq \delta, 
\end{equation*}
so by Hypothesis \ref{graphrep_hyp}(i), 
\begin{align*}
	\int_{B_1(0)} \mathcal{G}(\widetilde{u},\varphi^{(0)})^2 
	&= \rho^{-n-2\alpha} \int_{B_{\rho}(Z)} \mathcal{G}(u(X),\varphi^{(0)}(X-Z))^2 dX \\
	&\leq 2\rho^{-n-2\alpha} \int_{B_1(0)} \mathcal{G}(u,\varphi^{(0)})^2 + C \op{dist}(Z,S(\varphi^{(0)}))^2 \\
	&\leq 2\rho^{-n-2\alpha} \varepsilon_2 + C \delta^2 
\end{align*}
for $C = C(n,\lambda,m,\varphi^{(0)},\rho) \in (0,\infty)$.  Therefore Hypothesis \ref{graphrep_hyp}(i) holds true with $\varepsilon = 2\rho^{-n-2\alpha} \varepsilon_2 + C\delta$ and $u = \widetilde{u}$. 

Suppose that $l > \lambda$.  Let $\{\varphi_k\}_{k=0,1,2,\ldots,N}$ be an optimal sequence with increasing spines and $\varphi_0 = \varphi$.  Write $\varphi_k(X) = \{\pm \varphi_{k,1}(X)\}$ for all $X \in \mathbb{R}^n$ and some harmonic single-valued function $\varphi_{k,1}$ that is close to $\varphi^{(0)}_1$ in $L^2(B_1(0);\mathbb{R}^m)$.  After a small rotation of $\mathbb{R}^n$ fixing $\{0\} \times \mathbb{R}^{n-l}$, assume that $S(\varphi_N) = S(\varphi^{(0)})$.  Assume that $S(\varphi_k) = \{0\} \times \mathbb{R}^{n-p_k}$ for $k = 0,1,2,\ldots,N$ for some integers $\lambda = p_0 > p_1 > p_2 > \cdots > p_N = l$.  Write $Z = (\xi_0,\xi_1,\xi_2,\ldots,\xi_N,\zeta) \in \mathbb{R}^{l_0} \times \mathbb{R}^{p_{N-1}-p_N} \times \mathbb{R}^{p_{N-2}-p_{N-1}} \times \cdots \times \mathbb{R}^{p_0-p_1} \times \mathbb{R}^{n-l}$.  By Hypothesis \ref{graphrep_hyp} and Corollary \ref{branchdist_cor}, 
\begin{align} \label{translate_eqn1}
	\int_{B_1(0)} \mathcal{G}(\widetilde{u},\varphi)^2 
	&= \rho^{-n-2\alpha} \int_{B_{\rho}(Z)} \mathcal{G}(u(X),\varphi(X-Z))^2 dX \nonumber \\
	&\leq C \int_{B_1(0)} \mathcal{G}(u,\varphi)^2 + C \int_{B_1(0)} \int_0^1 |D\varphi_{0,1}(X+tZ)|^2 |\xi_0|^2 dt dX \nonumber \\&\hspace{4mm}
		+ C \sum_{k=1}^N \int_{B_1(0)} \int_0^1 |(D\varphi_{0,1} - D\varphi_{k,1})(X+tZ)|^2 |\xi_{N-k+1}|^2 \nonumber \\
	&\leq C \int_{B_1(0)} \mathcal{G}(u,\varphi)^2 + C \int_{B_1(0)} |\varphi|^2 |\xi_0|^2 
		+ C \sum_{k=1}^N \int_{B_1(0)} \mathcal{G}(\varphi,\varphi_k)^2 |\xi_{N-k+1}|^2 \nonumber \\
	&\leq C \int_{B_1(0)} \mathcal{G}(u,\varphi)^2 
\end{align}
for $C = C(n,\lambda,m,\varphi^{(0)},\rho) \in (0,\infty)$ provided $\varepsilon_2$ and $\beta_2$ are sufficiently small and thus 
\begin{equation*}
	\int_{B_1(0)} \mathcal{G}(\widetilde{u},\varphi)^2 
	\leq C\beta_2 \inf_{\varphi' \in \Phi_{\alpha}, \, S(\varphi) \subset S(\varphi')} \int_{B_1(0)} \mathcal{G}(\varphi,\varphi')^2. 
\end{equation*}
Therefore Hypothesis \ref{graphrep_hyp}(ii) holds true with $\beta = 4C\beta_2$ and $u = \widetilde{u}$ provided $C\beta_2 < 1/4$.  

Suppose that $l = \lambda > l_0$.  Let $\{\varphi_k\}_{k=0,1,2,\ldots,N}$ be an optimal sequence with increasing spines and $\varphi_0 = \varphi$.  We want to show that given $\eta > 0$, if $\varepsilon_2$ and $\beta_2$ are sufficiently small, then 
\begin{equation} \label{translate_eqn2}
	\op{dist}(Z,S(\varphi^{(0)})) \leq \eta, 
\end{equation}
and 
\begin{align} \label{translate_eqn3}
	\op{dist}(Z,S(\varphi_k))^2 &\leq \frac{\eta \inf_{\varphi' \in \Phi_{\alpha}, \, S(\varphi) \subset S(\varphi')} 
		\int_{B_1(0)} \mathcal{G}(\varphi,\varphi')^2}{\int_{B_1(0)} \mathcal{G}(\varphi,\varphi_{k+1})^2} \text{ for } k = 1,2,\ldots,N-1, \nonumber \\
	\op{dist}(Z,S(\varphi^{(0)}))^2 &\leq \eta \inf_{\varphi' \in \Phi_{\alpha}, \, S(\varphi) \subset S(\varphi')} 
		\int_{B_1(0)} \mathcal{G}(\varphi,\varphi')^2. 
\end{align} 
(\ref{translate_eqn2}) is clearly true by Corollary \ref{graphrep_cor}.  Assuming that (\ref{translate_eqn3}) is also true, by the computation of (\ref{translate_eqn1}) with (\ref{translate_eqn2}) and (\ref{translate_eqn3}) in place of Corollary \ref{branchdist_cor}, 
\begin{equation*}
	\int_{B_1(0)} \mathcal{G}(\widetilde{u},\varphi)^2 
	\leq C (\beta_2 + \eta) \inf_{\varphi' \in \Phi_{\alpha}, \, S(\varphi) \subset S(\varphi')} \int_{B_1(0)} \mathcal{G}(\varphi,\varphi')^2 
\end{equation*}
for $C = C(n,\lambda,m,\varphi^{(0)},\rho) \in (0,\infty)$ provided $\varepsilon_2$ and $\beta_2$ are sufficiently small.  By choosing $\eta$ such that $C\eta < \beta_1/2$ and then choosing $\varepsilon_2$ and $\beta_2$ so that (\ref{translate_eqn2}) and (\ref{translate_eqn3}) hold true and $C\beta_2 < \beta_1/2$, Hypothesis \ref{graphrep_hyp}(ii) holds true with $\beta = \beta_1$ and $u = \widetilde{u}$.  

Suppose there exists $\eta > 0$, $\varepsilon_j, \beta_j \downarrow 0$, $\varphi_j \in \Phi_{\alpha}$, $u_j \in \mathcal{F}_{\alpha}$, $Z_j \in B_{1/2}(0)$, and $\{\varphi_{j,k}\}_{k=0,1,2,\ldots,N}$ an optimal sequence with increasing spines with $\varphi_{j,0} = \varphi_j$ such that Hypothesis \ref{graphrep_hyp} holds true with $\varepsilon = \varepsilon_j$, $\beta = \beta_j$, $\varphi = \varphi_j$, and $u = u_j$, $\dim S(\varphi_j) = n-\lambda$, $\mathcal{N}_{u_j}(Z_j) \geq \alpha$, and (\ref{translate_eqn3}) fails to hold true with $Z = Z_j$, $\varphi = \varphi_j$, and $\varphi_k = \varphi_{j,k}$.  After passing to a subsequence of $j$, we can find an increasing sequence $\{k_i\}_{i=1,2,\ldots,I} \subset \in \{1,2,\ldots,N\}$ such that 
\begin{gather}
	\liminf_{j \rightarrow \infty} \frac{\int_{B_1(0)} \mathcal{G}(\varphi_j,\varphi_{j,k})^2}{\int_{B_1(0)} \mathcal{G}(\varphi_j,\varphi_{j,k+1})^2} > 0 
		\text{ if } k \not\in \{k_i\}, \nonumber \\
	\lim_{j \rightarrow \infty} \frac{\int_{B_1(0)} \mathcal{G}(\varphi_j,\varphi_{j,k_i})^2}{\int_{B_1(0)} \mathcal{G}(\varphi_j,\varphi_{j,{k_{i+1}}})^2} = 0 
		\text{ for } i = 1,2,\ldots,I-1, \label{translate_eqn4}
\end{gather}
and $k_I = N$.  As in the proof of Lemma \ref{graphrep_lemma}, we may assume that $S(\varphi_{j,k_i}) = \{0\} \times \mathbb{R}^{n-p_i}$ for some integers $\lambda = p_0 > p_1 > p_2 > \cdots > p_I = l_0$.  Write $Z_j = (\xi_{j,0},\xi_{j,1},\xi_{j,2},\ldots,\xi_{j,I},\zeta_j) \in \mathbb{R}^{l_0} \times \mathbb{R}^{p_{I-1}-p_I} \times \mathbb{R}^{p_{I-2}-p_{I-1}} \times \cdots \times \mathbb{R}^{p_0-p_1} \times \mathbb{R}^{n-\lambda}$.  Let $w$ be a blow-up of $u_j$ relative to $\varphi_{j,k_1}$ over $B_{3/4}(0)$.  By Hypothesis \ref{graphrep_hyp}(ii), $w$ is also a blow-up of $\varphi_j$ relative to $\varphi_{j,k_1}$ and thus $w$ is homogeneous degree $\alpha$ and $w(Y+X) = w(X)$ for all $X \in B_{3/4}(0)$ if $Y \in \{0\} \times \mathbb{R}^{n-\lambda}$.  Write $\varphi_{j,k}(X) = \{\pm \varphi_{j,k,1}(X)\}$ for all $X \in \mathbb{R}^n$ and some harmonic single-valued function $\varphi_{j,k,1}$ that is close to $\varphi^{(0)}_1$ in $L^2(B_1(0);\mathbb{R}^m)$.  After passing to a subsequence, the limit 
\begin{equation*}
	\widehat{\varphi}_i = \lim_{j \rightarrow \infty} \frac{\varphi_{j,k_1,1} - \varphi_{j,k_{i+1},1}}{
		\left( \int_{B_1(0)} \mathcal{G}(\varphi_{j,k_1},\varphi_{j,k_{i+1}})^2 \right)^{1/2}}
\end{equation*}
exists in $L^2(B_1(0);\mathbb{R}^m)$ for $k = 1,2,\ldots,I-1$.  Let $\widehat{\varphi}_I = \varphi^{(0)}$.  By (\ref{translate_eqn4}), 
\begin{equation*}
	\widehat{\varphi}_i = \lim_{j \rightarrow \infty} \frac{\varphi_{j,k_i,1} - \varphi_{j,k_{i+1},1}}{
		\left( \int_{B_1(0)} \mathcal{G}(\varphi_{j,k_1},\varphi_{j,k_{i+1}})^2 \right)^{1/2}}
\end{equation*}
and consequently $\widehat{\varphi}_i$ is translation invariant along $\{0\} \times \mathbb{R}^{n-p_i}$.  After passing to a subsequence, the limits 
\begin{gather*}
	\kappa_i = \lim_{j \rightarrow \infty} 
		\left( \frac{\int_{B_1(0)} \mathcal{G}(\varphi_{j,k_1},\varphi_{j,k_{i+1}})^2}{\int_{B_1(0)} \mathcal{G}(u_j,\varphi_{j,k_1})^2} \right)^{1/2} 
		\cdot \xi_{j,I-i} \text{ for } k = 1,2,\ldots,I-1, \nonumber \\
	\kappa_I = \lim_{j \rightarrow \infty} \frac{\xi_{j,0}}{\left( \int_{B_1(0)} \mathcal{G}(u_j,\varphi_{j,k_1})^2 \right)^{1/2}}, 
\end{gather*}
exist and $\zeta_j$ converges to some $\zeta$ in $\{0\} \times \mathbb{R}^{n-\lambda}$.  Since (\ref{translate_eqn3}) fails to hold true with $Z = Z_j$, $\varphi = \varphi_j$, and $\varphi_k = \varphi_{j,k}$, one of $\kappa_i$ is nonzero.  By Corollary \ref{nonconest_cor}(ii), 
\begin{align*} 
	&\int_{B_{3/4}(0) \cap \{|x| > \tau_j\}} \frac{\left| v_j - \sum_{i=1}^{I-1} (D\varphi_{j,k_1,1} - D\varphi_{j,k_{i+1},1}) \cdot (0,\xi_{j,I-i},0) 
		- D\varphi_{j,k_1,1} \cdot (0,\xi_{j,0},0) \right|^2}{|X-Z_j|^{2\alpha+n-1/2}} \\&\leq C \int_{B_1(0)} \mathcal{G}(u_j,\varphi_{j,K})^2 
\end{align*}
for some $C = C(n,\lambda,m,\varphi^{(0)}) \in (0,\infty)$, where $\tau_j \rightarrow 0^+$ slowly enough that Corollaries \ref{graphrep_cor} and \ref{nonconest_cor} hold with $\gamma = 3/4$, $\tau = \tau_j$, $\sigma = 1/2$, $\varphi = \varphi_{j,K}$, and $u = u_j \circ \Gamma_j^{-1}$ and $v = v_j$ in Corollary \ref{graphrep_cor} and where $(0,\xi_{j,I-i},0) \in \mathbb{R}^{p_{i+1}} \times \mathbb{R}^{p_i - p_{i+1}} \times \mathbb{R}^{n-p_i}$.  Dividing by $\int_{B_1(0)} \mathcal{G}(u_j,\varphi_{j,k_1})^2$ and letting $j \rightarrow \infty$, 
\begin{equation} \label{translate_eqn5}
	\int_{B_{3/4}(0)} \frac{\left| w - \sum_{i=1}^I D\widehat{\varphi}_i \cdot (0,\kappa_i,0) \right|^2}{|X-(0,\zeta)|^{2\alpha+n-1/2}} \leq C 
\end{equation} 
for some $C = C(n,\lambda,m,\varphi^{(0)}) \in (0,\infty)$, where $(0,\kappa_i,0) \in \mathbb{R}^{p_{i+1}} \times \mathbb{R}^{p_i - p_{i+1}} \times \mathbb{R}^{n-p_i}$.  By integration by parts and the fact that $\widehat{\varphi}_i$ is translation invariant along $\{0\} \times \mathbb{R}^{n-p_i}$, 
\begin{equation} \label{translate_eqn6}
	\int_{B_{1/8}(0,\zeta)} D_{(0,\kappa_i,0)} \widehat{\varphi}_i \cdot D_{(0,\kappa_j,0)} \widehat{\varphi}_j 
	= \int_{B_{1/8}(0,\zeta)} \widehat{\varphi}_i \cdot D_{(0,\kappa_i,0)} D_{(0,\kappa_j,0)} \widehat{\varphi}_j 
	= 0 
\end{equation}
if $1 \leq i < j \leq I$, where we rewrite the partial derivatives $D\widehat{\varphi}_i \cdot (0,\kappa_i,0)$ as $D_{(0,\kappa_i,0)} \widehat{\varphi}_i$ and let $\cdot$ denote the inner product on $\mathbb{R}^m$.  Since $w$ is homogeneous degree $\alpha$, (\ref{translate_eqn6}) implies that 
\begin{equation} \label{translate_eqn7}
	\int_{B_{1/8}(0,\zeta)} \frac{D_{(0,\kappa_i,0)} \widehat{\varphi}_i \cdot D_{(0,\kappa_j,0)} \widehat{\varphi}_j}{|X-(0,\zeta)|^{2\alpha+n-1/2}} = 0 
\end{equation}
if $1 \leq i < j \leq I$.  By (\ref{translate_eqn7}) and $w$ being homogeneous degree $\alpha$ at $Z$, 
\begin{equation} \label{translate_eqn8}
	\sum_{i=1}^I \int_{B_{1/8}(0,\zeta)} \frac{|D\widehat{\varphi}_i \cdot (0,\kappa_i,0)|^2}{|X-(0,\zeta)|^{2\alpha+n-1/2}} \leq C. 
\end{equation} 
But by (\ref{translate_eqn4}), $\widehat{\varphi}_i$ is not translation invariant in the $(0,\kappa_i,0)$ direction whenever $\kappa_i \neq 0$ and thus since $\widehat{\varphi}_i$ is homogeneous degree $\alpha$ at $(0,\zeta)$, (\ref{translate_eqn8}) is impossible.  
\end{proof}

\begin{proof}[Proof of Lemma \ref{keyest_lemma} for $l = \lambda$]
Let $\psi : [0,\infty) \rightarrow \mathbb{R}$ be a smooth function with $\psi(t) = 1$ for $t \in [0,\gamma]$, $\psi(t) = 0$ for $t \geq (1+\gamma)/2$, and $0 \leq \psi'(t) \leq 3/(1 - \gamma)$ for $t \in [0,\infty)$.  Arguing as in~\cite{KrumWic1} using the fact that $\mathcal{N}_u(0) \geq \alpha$, we obtain (6.6) of~\cite{KrumWic1}, i.e. 
\begin{equation} \label{keyest_eqn2}
	\int_{B_{\gamma}(0)} R^{2-n} \left( \frac{\partial (u/R^{\alpha})}{\partial R} \right)^2 
	\leq C \int (r^2 |Du|^2 \psi(R)^2 + 2 \alpha R^{-1} |u|^2 \psi(R) \psi'(R)), 
\end{equation}
where $r(X) = \op{dist}(X,S(\varphi))$, $R(X) = |X|$, and $C = C(n,m,\varphi^{(0)},\gamma) \in (0,\infty)$. 

After an orthogonal change of coordinates suppose that $S(\varphi) = \{0\} \times \mathbb{R}^{n-\lambda}$.  Equip $\mathbb{R}^n$ with cylindrical coordinates $(r \omega,y)$ for $r > 0$, $\omega \in S^{\lambda-1}$, and $y \in \mathbb{R}^{n-\lambda}$ and let $\nabla_{S^{\lambda-1}}$ denote the gradient with respect to $\omega$.  Recall (3.3) and (3.4) of~\cite{KrumWic1}, i.e. 
\begin{gather}
	\int_{B_1(0)} \left( \frac{1}{2} |Du|^2 \delta_{ij} - D_i u^{\kappa} D_j u^{\kappa} \right) D_i \zeta^j = 0, \label{keyest_eqn4} \\
	\int_{B_1(0)} (|Du|^2 \zeta + u^{\kappa} D_i u^{\kappa} D_i \zeta) = 0, \label{keyest_eqn5}
\end{gather}
for $\zeta^1,\zeta^2,\ldots,\zeta^n,\zeta \in C^{\infty}_c(B_1(0))$.  Let $(\zeta^1,\zeta^2,\ldots,\zeta^n) = \psi(R)^2 (x_1,x_2,\ldots,x_{\lambda},0,\ldots,0)$ in (\ref{keyest_eqn4}) to obtain 
\begin{equation} \label{keyest_eqn6}
	\int \left( \frac{\lambda}{2} |Du|^2 - |D_x u|^2 \right) \psi(R)^2 
	= -2 \int \left( \frac{1}{2} r^2 |Du|^2 - r^2 |D_r u|^2 - r D_r u (y \cdot D_y u) \right) R^{-1} \psi(R) \psi'(R), 
\end{equation}
where $r(X) = \op{dist}(X,\{0\} \times \mathbb{R}^{n-\lambda})$.  Let $\zeta = \psi(R)^2$ in (\ref{keyest_eqn5}) to obtain 
\begin{equation} \label{keyest_eqn7}
	\int |Du|^2 \psi(R)^2 = -2 \int \left( r u D_r u + u (y \cdot D_y u) \right) R^{-1} \psi(R) \psi'(R). 
\end{equation}
By multiplying (\ref{keyest_eqn7}) by $\alpha$ and adding it to (\ref{keyest_eqn6}), and adding also $2\alpha (\alpha+\lambda/2-1) \int R^{-1} |u|^2 \psi(R) \psi'(R)$ to both sides, we obtain 
\begin{align*} 
	&\int |D_y u|^2 \psi(R)^2 + (\alpha+\lambda/2-1) \int \left( |Du|^2 \psi(R)^2 + 2\alpha R^{-1} |u|^2 \psi(R) \psi'(R) \right) 
	\\&= -2 \int \left( \frac{1}{2} r^2 |Du|^2 - \alpha (\alpha+\lambda/2-1) R^{-1} |u|^2 - r D_r u (r D_r u - \alpha u) \right) R^{-1} \psi(R) \psi'(R) 
	\\&+ 2 \int (y \cdot D_y u)(r D_r u - \alpha u) R^{-1} \psi(R) \psi'(R). 
\end{align*}
By Cauchy's inequality,
\begin{align} \label{keyest_eqn8}
	&\frac{1}{2} \int |D_y u|^2 \psi(R)^2 + (\alpha+\lambda/2-1) \int \left( |Du|^2 \psi(R)^2 + 2\alpha R^{-1} |u|^2 \psi(R) \psi'(R) \right) \nonumber \\
	&\leq -2 \int \left( \frac{1}{2} r^2 |Du|^2 - \alpha (\alpha+\lambda/2-1) |u|^2 - r D_r u (r D_r u - \alpha u) \right) R^{-1} \psi(R) \psi'(R) \nonumber \\
	&+ 2 \int |r D_r u - \alpha u|^2 \psi'(R)^2. 
\end{align}

Using the Besicovitch covering theorem, cover $B^{1+n-\lambda}_{\gamma}(0) \cap (0,\infty) \times \mathbb{R}^{n-\lambda}$ by a collection of balls $\{ B_{(1-\gamma) \rho/2}(\rho,\zeta) \}_{(\rho,\zeta) \in \mathcal{I}}$, where $\mathcal{I}$ is an index set consisting of points of $B^{1+n-\lambda}_{\gamma}(0) \cap (0,\infty) \times \mathbb{R}^{n-\lambda}$, such that $\{ B_{(1-\gamma) \rho/2}(\rho,\zeta) \}_{(\rho,\zeta) \in \mathcal{I}}$ is the union of $N \leq C(n,\lambda)$ subcollections of mutually disjoint balls.  Observe that then $B_{\gamma}(0) \setminus \{0\} \times \mathbb{R}^{n-\lambda}$ is covered by the collection of annuli $\{ A^{\lambda}_{\rho,(1-\gamma)/2}(\zeta) \}_{(\rho,\zeta) \in \mathcal{I}}$ and this collection is the union of $N$ subcollections of mutually disjoint annuli.  Let $\{\chi_{(\rho,\zeta)}\}_{(\rho,\zeta) \in \mathcal{I}}$ be a partition of unity subordinate to $\{ A^{\lambda}_{\rho,(1-\gamma)/2}(\zeta) \}_{(\rho,\zeta) \in \mathcal{I}}$ such that $\chi_{(\rho,\zeta)}(x,y)$ depends only on $|x|$ and $y$.  We want to show that 
\begin{align} \label{keyest_eqn9}
	&-2 \int \left( \frac{1}{2} r^2 |Du|^2 - \alpha (\alpha+\lambda/2-1) |u|^2 - r D_r u (r D_r u - \alpha u) \right) R^{-1} \psi(R) \psi'(R) \chi_{(\rho,\zeta)}
	\nonumber \\&+ 2 \int |r D_r u - \alpha u|^2 \psi'(R)^2 \chi_{(\rho,\zeta)} \leq C \int_{A^{\lambda}_{\rho,1-\gamma}(\zeta)} \mathcal{G}(u,\varphi)^2 
\end{align}
for some $C = C(n,\lambda,m,\varphi^{(0)},\gamma) \in (0,\infty)$.  Since given $(\rho_0,\zeta_0) \in \mathcal{I}$, $A^{\lambda}_{\rho_0,(1-\gamma)/2}(\zeta_0) \cap A^{\lambda}_{\rho,(1-\gamma)/2}(\zeta) \neq \emptyset$ for at most $N' \leq C(n,\lambda)$ points $(\rho,\zeta) \in \mathcal{I}$, by summing (\ref{keyest_eqn9}) over $(\rho,\zeta) \in \mathcal{I}$ and combining with (\ref{keyest_eqn2}) and (\ref{keyest_eqn8}), (\ref{keyest_eqn1}) holds true.  To prove (\ref{keyest_eqn9}), fix $(\rho,\zeta) \in \mathcal{I}$ and let $\chi = \chi_{(\rho,\zeta)}$.  Let $\widetilde{\varepsilon}, \widetilde{\beta} \in (0,1)$ depend only on $n$, $\lambda$, $m$, $\varphi^{(0)}$, and $\gamma$ and will be determined below.  We shall consider three cases: 
\begin{enumerate}
	\item[I.] $\rho^{-n-2\alpha} \int_{A^{\lambda}_{\rho,1-\gamma}(\zeta)} \mathcal{G}(u,\varphi)^2 \geq (2\widetilde{\beta}/3)^n \widetilde{\varepsilon}$, 
	\item[II.] $\rho^{-n-2\alpha} \int_{A^{\lambda}_{\rho,1-\gamma}(\zeta)} \mathcal{G}(u,\varphi)^2 < (2\widetilde{\beta}/3)^n \widetilde{\varepsilon}$ and $A^{\lambda}_{\rho,1-\gamma}(\zeta) \cap \mathcal{B}_u = \emptyset$, and
	\item[III.] $\rho^{-n-2\alpha} \int_{A^{\lambda}_{\rho,1-\gamma}(\zeta)} \mathcal{G}(u,\varphi)^2 < (2\widetilde{\beta}/3)^n \widetilde{\varepsilon}$ and $A^{\lambda}_{\rho,1-\gamma}(\zeta) \cap \mathcal{B}_u \neq \emptyset$. 
\end{enumerate}

In Case I, by the $W^{1,2}$ estimates for harmonic two-valued functions (see Section 2 of~\cite{KrumWic1}), 
\begin{align*}
	\int_{A^{\lambda}_{\rho,(1-\gamma)/2}(\zeta)} (|u|^2 + r^2 |Du|^2) 
	&\leq C \int_{A^{\lambda}_{\rho,1-\gamma}(\zeta)} |u|^2 
	\leq C \int_{A^{\lambda}_{\rho,1-\gamma}(\zeta)} |\varphi|^2 + C \int_{A^{\lambda}_{\rho,1-\gamma}(\zeta)} \mathcal{G}(u,\varphi)^2 \nonumber \\
	&\leq C \rho^{n+2\alpha} + C \int_{A^{\lambda}_{\rho,1-\gamma}(\zeta)} \mathcal{G}(u,\varphi)^2
	\leq C \int_{A^{\lambda}_{\rho,1-\gamma}(\zeta)} \mathcal{G}(u,\varphi)^2
\end{align*}
for $C = C(n,\lambda,m,\varphi^{(0)},\gamma) \in (0,\infty)$ and thus (\ref{keyest_eqn9}) follows. 

In Case II, by the Schauder estimates and Lemma \ref{twovalL2_lemma}, provided $\widetilde{\varepsilon}$ is sufficiently small, there exists $v : A^{\lambda}_{\rho,(1-\gamma)/2}(\zeta) \rightarrow \mathbb{R}^m$ such that 
\begin{gather*}
	u(X) = \{ \varphi_1(X) + v(X), -\varphi_1(X) - v(X) \} \text{ for all } A^{\lambda}_{\rho,(1-\gamma)/2}(\zeta), \\
	\int_{A^{\lambda}_{\rho,(1-\gamma)/2}(\zeta)} (|v|^2 + r^2 |Dv|^2) \leq C \int_{A^{\lambda}_{\rho,1-\gamma}(\zeta)} \mathcal{G}(u,\varphi)^2
\end{gather*}
for $C = C(n,\lambda,m,\varphi^{(0)},\gamma) \in (0,\infty)$.  Hence 
\begin{align} \label{keyest_eqn10}
	&-2 \int \left( \frac{1}{2} r^2 |Du|^2 - \alpha (\alpha+\lambda/2-1) |u|^2 - r D_r u (r D_r u - \alpha u) \right) R^{-1} \psi(R) \psi'(R) \chi
	\nonumber \\&+ 2 \int |r D_r u - \alpha u|^2 \psi'(R)^2 \chi 
	\leq \frac{1}{2} \int \left( |\nabla_{S^{\lambda-1}} \varphi|^2 - \alpha (\alpha+\lambda-2) |\varphi|^2 \right) R^{-1} \psi(R) \psi'(R) \chi \nonumber \\
	&+ \sqrt{2} \int \left( \nabla_{S^{n-\lambda-1}} v \nabla_{S^{\lambda-1}} \varphi - \alpha (\alpha+\lambda-2) v \varphi \right) R^{-1} \psi(R) \psi'(R) \chi 
	+ C \int_{A^{\lambda}_{\rho,1-\gamma}(\zeta)} \mathcal{G}(u,\varphi)^2 
\end{align}
for some $C = C(n,\lambda,m,\varphi^{(0)},\gamma) \in (0,\infty)$.  Since $\varphi$ is harmonic and homogeneous degree $\alpha$ and $S(\varphi) = \{0\} \times \mathbb{R}^{n-\lambda}$, $\Delta_{S^{\lambda-1}} \varphi + \alpha (\alpha + \lambda - 2) \varphi = 0$ and thus by integration by parts 
\begin{equation*}
	\int |\nabla_{S^{\lambda-1}} \varphi|^2  R^{-1} \psi(R) \psi'(R) \chi 
	= \alpha (\alpha + \lambda - 2) \int |\varphi|^2  R^{-1} \psi(R) \psi'(R) \chi 
\end{equation*}
and  
\begin{equation*}
	\int \nabla_{S^{\lambda-1}} \varphi \cdot \nabla_{S^{\lambda-1}} v  R^{-1} \psi(R) \psi'(R) \chi 
	= \alpha (\alpha + \lambda - 2) \int \varphi v  R^{-1} \psi(R) \psi'(R) \chi. 
\end{equation*}
Hence by (\ref{keyest_eqn10}), (\ref{keyest_eqn9}) holds true. 

Finally we will consider Case III.  By Lemma \ref{graphrep_lemma}, provided $\widetilde{\varepsilon}$ and $\widetilde{\beta}$ is sufficiently small, there exists $J \geq 1$ and a sequence $\varphi_0,\varphi_1,\varphi_2,\ldots,\varphi_J$ such that $\varphi_0 = \varphi$, $S(\varphi_{j-1}) \subset S(\varphi_j)$ and 
\begin{equation} \label{keyest_eqn11}
	\int_{A^{\lambda}_{\rho,1-\gamma}(\zeta)} \mathcal{G}(u,\varphi_{j-1})^2 > \frac{2}{3} \widetilde{\beta} \int_{A^{\lambda}_{\rho,1-\gamma}(\zeta)} \mathcal{G}(u,\varphi_j)^2
\end{equation}
for all $j = 1,2,\ldots,J$, and letting $\widetilde{\varphi} = \varphi_J$ either $\dim S(\widetilde{\varphi}) = l_0$ or 
\begin{equation} \label{keyest_eqn12}
	\int_{A^{\lambda}_{\rho,1-\gamma}(\zeta)} \mathcal{G}(u,\widetilde{\varphi})^2 
	\leq \widetilde{\beta} \in_{\psi' \in \Phi_{\alpha}, \, S(\widetilde{\varphi}) \subset S(\varphi')} 
	\int_{A^{\lambda}_{\rho,1-\gamma}(\zeta)} \mathcal{G}(u,\varphi')^2. 
\end{equation}
Observe that by (\ref{keyest_eqn11}), 
\begin{equation} \label{keyest_eqn13}
	\int_{A^{\lambda}_{\rho,1-\gamma}(\zeta)} \mathcal{G}(u,\widetilde{\varphi})^2 
	< (2\widetilde{\beta}/3)^{-n} \int_{A^{\lambda}_{\rho,1-\gamma}(\zeta)} \mathcal{G}(u,\varphi_j)^2 < \widetilde{\varepsilon}. 
\end{equation}
After an orthogonal change of coordinates, assume that $\dim S(\widetilde{\varphi}) = \{0\} \times \mathbb{R}^{n-l}$.  Let $(s\omega,z)$ denote cylindrical coordinates on $\mathbb{R}^n$ where $s \geq 0$, $\omega \in S^{l-1}$, and $z \in \mathbb{R}^{n-l}$.  Note that by (\ref{keyest_eqn11}), (\ref{keyest_eqn13}), Corollary \ref{branchdist_cor}, provided $\widetilde{\varepsilon}$ and $\widetilde{\beta}$ are sufficiently small,  
\begin{equation} \label{keyest_eqn14}
	\op{dist}(Z,S(\widetilde{\varphi})) \leq c E \text{ for all } Z \in \mathcal{B}_u \cap A^{\lambda}_{\rho,1-\gamma}(\zeta)
\end{equation} 
for some constant $c = c(n,\lambda,m,\varphi^{(0)},\gamma) \in (0,\infty)$, where $E = \left( \rho^{-n-2\alpha} \int_{A^{\lambda}_{\rho,1-\gamma}(\zeta)} \mathcal{G}(u,\widetilde{\varphi})^2 \right)^{1/2}$.  Since $\widetilde{\varphi}$ is harmonic and homogeneous degree $\alpha$ and $\{0\} \times \mathbb{R}^{n-\lambda} \subset S(\widetilde{\varphi})$, $\Delta_{S^{\lambda-1}} \widetilde{\varphi} + \alpha (\alpha + \lambda - 2) \widetilde{\varphi} = 0$ and thus by integration by parts 
\begin{equation*}
	\int |\nabla_{S^{\lambda-1}} \widetilde{\varphi}|^2  R^{-1} \psi(R) \psi'(R) \chi 
	= \alpha (\alpha + \lambda - 2) \int |\widetilde{\varphi}|^2  R^{-1} \psi(R) \psi'(R) \chi 
\end{equation*}
and since (\ref{keyest_eqn12}) holds true, in order to prove (\ref{keyest_eqn9}) it suffices to show that 
\begin{align} \label{keyest_eqn15}
	&-2 \int \left( \frac{1}{2} r^2 (|Du|^2 - |D\widetilde{\varphi}|^2) - \alpha (\alpha+\lambda/2-1) (|u|^2 - |\widetilde{\varphi}|^2) 
		- r D_r u (r D_r u - \alpha u) \right) R^{-1} \psi(R) \psi'(R) \chi \nonumber \\
	&+ 2 \int |r D_r u - \alpha u|^2 \psi'(R)^2 \chi \leq C \int_{A^{\lambda}_{\rho,1-\gamma}(\zeta)} \mathcal{G}(u,\varphi)^2 
\end{align}

For $X \in A^{\lambda}_{\rho,1-\gamma}(\zeta)$, let $d(X) = \max\{\op{dist}(X,\mathcal{B}_u),(1-\gamma) \rho/32\}$.  Observe that we can cover $A^{\lambda}_{\rho,(1-\gamma)/2}(\zeta)$ by $A^{\lambda}_{\rho,(1-\gamma)/2}(\zeta) \cap \{ s > (1-\gamma)\rho/64 \}$, $B_{\max\{4cE,d(0,\xi)/2\}}(0,\xi)$ where $(0,\xi) \in A^{\lambda}_{\rho,(1-\gamma)/2}(\zeta) \cap \{0\} \times \mathbb{R}^{n-l}$, and $A^l_{\sigma,1/2}(\xi)$ where $(0,\xi) \in A^{\lambda}_{\rho,(1-\gamma)/2}(\zeta) \cap \{0\} \times \mathbb{R}^{n-l}$ and $\max\{4cE,d(0,\xi)/2\} \leq \sigma \leq (1-\gamma) \rho/32$.  Using the Besicovitch covering theorem, cover $A^{\lambda}_{\rho,(1-\gamma)/2}(\zeta)$ by a collection of such balls and annuli that equals the union of $N \leq C(n,\lambda,l)$ subcollections of mutually disjoint sets.  Let $\{\eta_S\}_{S \in \mathcal{C}}$ be a partition of unity subordinate to $\mathcal{C}$ that $\eta_S(X)$ depends only on $s$ and $z$.  Note that since two sets in $\mathcal{C}$ intersect only if they are roughly the same size, by defining $\eta'_S : \mathbb{R}^n \rightarrow [0,1]$ by $\eta'_S(X) = \op{dist}(X, \mathbb{R}^n \setminus S)$, approximating $\eta''_S$ by an appropriate smooth function $\eta''_S$ via convolution, and letting 
\begin{equation*}
	\eta_S = \frac{\eta''_S}{\sum_{T \in \mathcal{C}} \eta''_S}, 
\end{equation*}
we can assume that $|D\eta_{0,S}| \leq C \rho^{-1}$ when $S = A^{\lambda}_{\rho,(1-\gamma)/2}(\zeta) \cap \{ s > (1-\gamma)\rho/64 \}$ and $|D\eta_{0,S}| \leq C \sigma^{-1}$ whenever $S = B_{\sigma}(\xi) \in \mathcal{C}$ or $S = A^l_{\sigma,1/2}(\xi) \in \mathcal{C}$.

Suppose that $B_{d(0,\xi)/2}(0,\xi) \in \mathcal{C}$ and $d(0,\xi) \geq 8cE$.  If $d(0,\xi)^{-n-2\alpha} \int_{B_{d(0,\xi)}(0,\xi)} \mathcal{G}(u,\widetilde{\varphi})^2 < \widetilde{\varepsilon}$, then by the Schauder estimates there exists a harmonic function $v : B_{d(0,\xi)/2}(0,\xi) \rightarrow \mathbb{R}^m$ such that 
\begin{gather}
	u(X) = \{ \widetilde{\varphi}_1(X) + v(X), -\widetilde{\varphi}_1(X) - v(X) \} \text{ for all } X \in B_{d(0,\xi)/2}(0,\xi), \nonumber \\
	\int_{B_{d(0,\xi)/2}(0,\xi)} (|v|^2 + s^2 |Dv|^2) \leq C \int_{B_{d(0,\xi)}(0,\xi)} \mathcal{G}(u,\widetilde{\varphi})^2 \label{keyest_eqn16}
\end{gather}
for some $C = C(n,\lambda,m,\varphi^{(0)},\gamma) \in (0,\infty)$, where $\widetilde{\varphi}(X) = \{ \pm \widetilde{\varphi}_1(X) \}$ for all $X \in \mathbb{R}^n$ and some harmonic single-valued function $\widetilde{\varphi}_1 : \mathbb{R}^n \rightarrow \mathbb{R}^m$ that is close to $\varphi^{(0)}_1$ in $L^2(B_1(0);\mathbb{R}^m)$.  By Lemma \ref{translate_lemma} and Corollary \ref{radialdecay_cor}, 
\begin{equation} \label{keyest_eqn17}
	\int_{B_{d(0,\xi)}(0,\xi)} \mathcal{G}(u,\widetilde{\varphi})^2 
	\leq C \left(\frac{d(0,\xi)}{\rho}\right)^{n+2\alpha-1/2} \int_{A^{\lambda}_{\rho,1-\gamma}(\zeta)} \mathcal{G}(u,\widetilde{\varphi})^2 
\end{equation}
for some $C = C(n,\lambda,m,\varphi^{(0)},\gamma) \in (0,\infty)$.  By (\ref{keyest_eqn16}), (\ref{keyest_eqn17}), and $\Delta_{S^{\lambda-1}} \widetilde{\varphi} + \alpha (\alpha + \lambda - 2) \widetilde{\varphi} = 0$, 
\begin{align} \label{keyest_eqn18}
	&-2 \int \left( \frac{1}{2} r^2 (|Du|^2 - |D\widetilde{\varphi}|^2) - \alpha (\alpha+\lambda/2-1) (|u|^2 - |\widetilde{\varphi}|^2) 
		- r D_r u (r D_r u - \alpha u) \right) R^{-1} \psi(R) \psi'(R) \chi \nonumber \\
	&+ 2 \int |r D_r u - \alpha u|^2 \psi'(R)^2 \chi \eta \nonumber \\
	&\leq -2 \int \left( \nabla_{S^{\lambda-1}} v \cdot \nabla_{S^{\lambda-1}} \widetilde{\varphi} - \alpha (\alpha+\lambda-2) v \widetilde{\varphi} \right) 
		R^{-1} \psi(R) \psi'(R) \chi \eta 
	+ C \left(\frac{d(0,\xi)}{\rho}\right)^{n+2\alpha-1/2} \int_{A^{\lambda}_{\rho,1-\gamma}(\zeta)} \mathcal{G}(u,\widetilde{\varphi})^2 \nonumber \\
	&\leq 2 \int v \nabla_{S^{\lambda-1}} \widetilde{\varphi} \cdot \nabla_{S^{\lambda-1}} \eta R^{-1} \psi(R) \psi'(R) \chi 
	+ C \rho^{-n-2\alpha+5/2} \left( \int_{B_{d(0,\xi)/2}(0,\xi)} s^{2\alpha-5/2} \right) 
		\left( \int_{A^{\lambda}_{\rho,1-\gamma}(\zeta)} \mathcal{G}(u,\varphi)^2 \right), 
\end{align}
where $\eta = \eta_{B_{d(0,\xi)/2}(0,\xi)}$ and $C = C(n,\lambda,m,\varphi^{(0)},\gamma) \in (0,\infty)$.  Similarly, when $A^l_{\sigma,1/2}(\xi) \in \mathcal{C}$ with $\sigma^{-n-2\alpha} \int_{A^l_{\sigma,1}(\xi)} \mathcal{G}(u,\widetilde{\varphi})^2 < \widetilde{\varepsilon}$, 
\begin{align} \label{keyest_eqn19}
	&-2 \int \left( \frac{1}{2} r^2 (|Du|^2 - |D\widetilde{\varphi}|^2) - \alpha (\alpha+\lambda/2-1) (|u|^2 - |\widetilde{\varphi}|^2) 
		- r D_r u (r D_r u - \alpha u) \right) R^{-1} \psi(R) \psi'(R) \chi \nonumber \\
	&+ 2 \int |r D_r u - \alpha u|^2 \psi'(R)^2 \chi \eta \nonumber \\
	&\leq 2 \int v \nabla_{S^{\lambda-1}} \widetilde{\varphi} \cdot \nabla_{S^{\lambda-1}} \eta R^{-1} \psi(R) \psi'(R) \chi 
	+ C \rho^{-n-2\alpha+5/2} \left( \int_{A^l_{\sigma,1/2}(\xi)} s^{2\alpha-5/2} \right) 
		\left( \int_{A^{\lambda}_{\rho,1-\gamma}(\zeta)} \mathcal{G}(u,\varphi)^2 \right), 
\end{align}
where $\eta = \eta_{A^l_{\sigma,1/2}(\xi)}$ and $C = C(n,\lambda,m,\varphi^{(0)},\gamma) \in (0,\infty)$.  Also, 
\begin{align} \label{keyest_eqn20}
	&-2 \int \left( \frac{1}{2} r^2 (|Du|^2 - |D\widetilde{\varphi}|^2) - \alpha (\alpha+\lambda/2-1) (|u|^2 - |\widetilde{\varphi}|^2) 
		- r D_r u (r D_r u - \alpha u) \right) R^{-1} \psi(R) \psi'(R) \chi \nonumber \\
	&+ 2 \int |r D_r u - \alpha u|^2 \psi'(R)^2 \chi \eta 
	\leq 2 \int v \nabla_{S^{\lambda-1}} \widetilde{\varphi} \cdot \nabla_{S^{\lambda-1}} \eta R^{-1} \psi(R) \psi'(R) \chi 
	+ C \int_{A^{\lambda}_{\rho,1-\gamma}(\zeta)} \mathcal{G}(u,\varphi)^2, 
\end{align}
where $\eta = \eta_{A^{\lambda}_{\rho,(1-\gamma)/2}(\zeta) \cap \{ s > (1-\gamma)\rho/64 \}}$ and $C = C(n,\lambda,m,\varphi^{(0)},\gamma) \in (0,\infty)$.

Suppose instead that $B_{d(0,\xi)/2}(0,\xi) \in \mathcal{C}$ with $d(0,\xi) \geq 8cE$ and $d(0,\xi)^{-n-2\alpha} \int_{B_{d(0,\xi)}(0,\xi)} \mathcal{G}(u,\widetilde{\varphi})^2 \geq \widetilde{\varepsilon}$.  Then by the $W^{1,2}$ estimates for harmonic two-valued functions, 
\begin{align*}
	\int_{B_{d(0,\xi)/2}(0,\xi)} (|u|^2 + |\widetilde{\varphi}|^2 + s^2 |Du|^2 + s^2 |D\widetilde{\varphi}|^2) 
	&\leq \int_{B_{d(0,\xi)}(0,\xi)} (|u|^2 + |\widetilde{\varphi}|^2) 
	\\&\leq C d(0,\xi)^{n+2\alpha} + C \int_{B_{d(0,\xi)}(0,\xi)} \mathcal{G}(u,\widetilde{\varphi})^2  
	\\&\leq C \int_{B_{d(0,\xi)}(0,\xi)} \mathcal{G}(u,\widetilde{\varphi})^2 
\end{align*}
for $C = C(n,\lambda,m,\varphi^{(0)},\gamma) \in (0,\infty)$, so by (\ref{keyest_eqn17}), 
\begin{align} \label{keyest_eqn21}
	&-2 \int \left( \frac{1}{2} r^2 (|Du|^2 - |D\widetilde{\varphi}|^2) - \alpha (\alpha+\lambda/2-1) (|u|^2 - |\widetilde{\varphi}|^2) 
		- r D_r u (r D_r u - \alpha u) \right) R^{-1} \psi(R) \psi'(R) \chi \nonumber \\
	&+ 2 \int |r D_r u - \alpha u|^2 \psi'(R)^2 \chi \eta 
	\leq C \left(\frac{d(0,\xi)}{\rho}\right)^{n+2\alpha-2-1/2} \int_{A^{\lambda}_{\rho,1-\gamma}(\zeta)} \mathcal{G}(u,\varphi)^2 \nonumber \\ 
	&\leq C \rho^{-n-2\alpha+5/2} \left( \int_{B_{d(0,\xi)/2}(0,\xi)} s^{2\alpha-5/2} \right) 
		\left( \int_{A^{\lambda}_{\rho,1-\gamma}(\zeta)} \mathcal{G}(u,\varphi)^2 \right), 
\end{align}
where $\eta = \eta_{B_{d(0,\xi)/2}(0,\xi)}$ and $C = C(n,\lambda,m,\varphi^{(0)},\gamma) \in (0,\infty)$.  Similarly, Similarly, when $A^l_{\sigma,1/2}(\xi) \in \mathcal{C}$ with $\sigma^{-n-2\alpha} \int_{A^l_{\sigma,1}(\xi)} \mathcal{G}(u,\widetilde{\varphi})^2 \geq \widetilde{\varepsilon}$, (\ref{keyest_eqn21}) holds true with $\eta = \eta_{A^l_{\sigma,1/2}(\xi)}$ and $\int_{A^l_{\sigma,1/2}(\xi)} s^{2\alpha-1/2}$ in place of $\int_{B_{d(0,\xi)/2}(0,\xi)} s^{2\alpha-1/2}$.

Suppose that $B_{d(0,\xi)/2}(0,\xi) \in \mathcal{C}$ and $d(0,\xi) < 8cE$.  Then by the $W^{1,2}$ estimates for harmonic two-valued functions, 
\begin{align*}
	\int_{B_{4cE}(0,\xi)} (|u|^2 + |\widetilde{\varphi}|^2 + s^2 |Du|^2 + s^2 |D\widetilde{\varphi}|^2) 
	&\leq \int_{B_{8cE}(0,\xi)} (|u|^2 + |\widetilde{\varphi}|^2) 
	\\&\leq C E^{n+2\alpha} + C \int_{B_{8cE}(0,\xi)} \mathcal{G}(u,\widetilde{\varphi})^2 
\end{align*}
for $C = C(n,\lambda,m,\varphi^{(0)},\gamma) \in (0,\infty)$, so by a computation similar to (\ref{keyest_eqn17}) and by the definition of $E$, (\ref{keyest_eqn19}), 
\begin{align} \label{keyest_eqn22}
	&-2 \int \left( \frac{1}{2} r^2 (|Du|^2 - |D\widetilde{\varphi}|^2) - \alpha (\alpha+\lambda/2-1) (|u|^2 - |\widetilde{\varphi}|^2) 
		- r D_r u (r D_r u - \alpha u) \right) R^{-1} \psi(R) \psi'(R) \chi \nonumber \\
	&+ 2 \int |r D_r u - \alpha u|^2 \psi'(R)^2 \chi \eta 
	\leq C \left(\frac{E}{\rho}\right)^{n+2\alpha-5/2} \int_{A^{\lambda}_{\rho,1-\gamma}(\zeta)} \mathcal{G}(u,\varphi)^2 \nonumber \\ 
	&\leq C \rho^{-n-2\alpha+5/2} \left( \int_{B_{d(0,\xi)/2}(0,\xi)} s^{2\alpha-5/2} \right) 
		\left( \int_{A^{\lambda}_{\rho,1-\gamma}(\zeta)} \mathcal{G}(u,\varphi)^2 \right), 
\end{align}
where $\eta = \eta_{B_{4cE}(0,\xi)}$ and $C = C(n,\lambda,m,\varphi^{(0)},\gamma) \in (0,\infty)$. 

By (\ref{keyest_eqn18}) - (\ref{keyest_eqn22}), 
\begin{align} \label{keyest_eqn23}
	&-2 \int \left( \frac{1}{2} r^2 (|Du|^2 - |D\widetilde{\varphi}|^2) - \alpha (\alpha+\lambda/2-1) (|u|^2 - |\widetilde{\varphi}|^2) 
		- r D_r u (r D_r u - \alpha u) \right) R^{-1} \psi(R) \psi'(R) \chi \nonumber \\
	&+ 2 \int |r D_r u - \alpha u|^2 \psi'(R)^2 \chi \eta 
	\leq \sum_{S \in \mathcal{C}_*} 2 \int v \nabla_{S^{\lambda-1}} \widetilde{\varphi} \cdot \nabla_{S^{\lambda-1}} \eta_S R^{-1} \psi(R) \psi'(R) \chi 
	\nonumber \\&+ C \rho^{-n-2\alpha-5/2} \left( \int_{A^{\lambda}_{\rho,1-\gamma}(\zeta)} s^{2\alpha-5/2} \right) 
		\left( \int_{A^{\lambda}_{\rho,1-\gamma}(\zeta)} \mathcal{G}(u,\varphi)^2 \right) , 
\end{align}
for some $C = C(n,\lambda,m,\varphi^{(0)},\gamma) \in (0,\infty)$, where $\mathcal{C}_*$ consists of $A^{\lambda}_{\rho,(1-\gamma)/2}(\zeta) \cap \{ s > (1-\gamma)\rho/64 \}$, $B_{d(0,\xi)/2}(0,\xi) \in \mathcal{C}$ with $d(0,\xi) \geq 8cE$ and $d(0,\xi)^{-n-2\alpha} \int_{B_{d(0,\xi)}(0,\xi)} \mathcal{G}(u,\widetilde{\varphi})^2 < \widetilde{\varepsilon}$, and $A^l_{\sigma,1/2}(\xi) \in \mathcal{C}$ with $\sigma^{-n-2\alpha} \int_{A^l_{\sigma,1}(\xi)} \mathcal{G}(u,\widetilde{\varphi})^2 < \widetilde{\varepsilon}$.   By the definition of $\eta_S$, the $W^{1,2}$ estimates for harmonic two-valued functions, and (\ref{keyest_eqn17}) and similar computations, 
\begin{align} \label{keyest_eqn24}
	&\sum_{S \in \mathcal{C}_*} 2 \int v \nabla_{S^{\lambda-1}} \widetilde{\varphi} \cdot \nabla_{S^{\lambda-1}} \eta_S R^{-1} \psi(R) \psi'(R) \chi \nonumber \\
	&= -\sum_{S \in \mathcal{C} \setminus \mathcal{C}_*} 
		2 \int v \nabla_{S^{\lambda-1}} \widetilde{\varphi} \cdot \nabla_{S^{\lambda-1}} \eta_S R^{-1} \psi(R) \psi'(R) \chi \nonumber \\
	&\leq C \sum_{S \in \mathcal{C} \setminus \mathcal{C}_*} \int_S \frac{r^2}{s} (|u| + |\widetilde{\varphi}|) |D\widetilde{\varphi}| \nonumber \\
	&\leq C \sum_{S \in \mathcal{C} \setminus \mathcal{C}_*} \rho^{-n-2\alpha+5/2} \left( \int_S s^{2\alpha-5/2} \right) 
		\left( \int_{A^{\lambda}_{\rho,1-\gamma}(\zeta)} \mathcal{G}(u,\varphi)^2 \right) \nonumber \\
	&\leq C \rho^{-n-2\alpha-5/2} \left( \int_{A^{\lambda}_{\rho,1-\gamma}(\zeta)} s^{2\alpha-5/2} \right) 
		\left( \int_{A^{\lambda}_{\rho,1-\gamma}(\zeta)} \mathcal{G}(u,\varphi)^2 \right) 
\end{align}
for $C = C(n,\lambda,m,\varphi^{(0)},\gamma) \in (0,\infty)$.  Moreover, 
\begin{equation} \label{keyest_eqn25}
	\int_{A^{\lambda}_{\rho,1-\gamma}(\zeta)} s^{2\alpha-5/2} \leq C(n,\lambda,l,\gamma) \rho^{n+2\alpha-5/2}. 
\end{equation}
By combining (\ref{keyest_eqn23}), (\ref{keyest_eqn24}), and (\ref{keyest_eqn25}), (\ref{keyest_eqn9}).
\end{proof}

\begin{proof}[Proof of Corollary \ref{radialdecay_cor} for $l = \lambda$]
Follows from Lemma \ref{keyest_lemma} using the same argument as the proof of Corollary 6.4 of~\cite{KrumWic1}. 
\end{proof}

\begin{proof}[Proof of Corollary \ref{branchdist_cor} for $l = \lambda$]
First we claim that given $\rho \in (0,1/2]$ there exists $\overline{\delta} = \overline{\delta}(\varphi^{(0)}) > 0$, $\overline{\varepsilon} = \overline{\varepsilon}(\varphi^{(0)},\rho) > 0$, and $\overline{\beta} = \overline{\beta}(\varphi^{(0)},\rho) > 0$ such that if $\varphi \in \Phi_{\alpha}$, $u \in \mathcal{F}_{\alpha}$, $Z \in B_{1/2}(0)$, $\{\varphi_j\}_{j = 0,1,2,\ldots,N} \subset \Phi_{\alpha}$ is an optimal sequence with increasing spines and $\varphi_0 = \varphi$, and $a \in \mathbb{R}^{\lambda}$ with $a \neq 0$ such that Hypothesis \ref{graphrep_hyp} holds true with $\varepsilon = \overline{\varepsilon}$ and $\beta = \overline{\beta}$, $\mathcal{N}_u(Z) \geq \alpha$, and $S(\varphi_j) = \{0\} \times \mathbb{R}^{n-p_j}$ for some integers $\lambda = p_0 > p_1 > p_2 > \cdots > p_N = l_0$ (recall that $S(\varphi^{(0)}) = \{0\} \times \mathbb{R}^{n-l_0}$), then 
\begin{equation} \label{branchdist_eqn1}
	\mathcal{L}^n \left\{ X \in B_{\rho}(Z) : \overline{\delta} r(X)^{\alpha-1} 
	\left( |a_0|^2 + \sum_{j=0}^{N-1} \int_{B_1(0)} \mathcal{G}(\varphi,\varphi_{j+1})^2 |a_{N-j}|^2 \right)^{1/2} 
	\leq |D\varphi(X) \cdot (a,0)| \right\} \geq \overline{\delta} \rho^n, 
\end{equation}
where $r(X) = \op{dist}(X,\{0\} \times \mathbb{R}^{n-\lambda}$ and $a = (a_0,a_1,a_2,\ldots,a_N) \in \mathbb{R}^{l_0} \times \mathbb{R}^{p_{N-1}-p_N} \times \mathbb{R}^{p_{N-2}-p_{N-1}} \times \cdots \times \mathbb{R}^{p_0-p_1}$. 

Suppose that for every $j = 1,2,3,\ldots$ there existed $\rho_j \in (0,1/2]$, $\varepsilon_{j,k}, \beta_{j,k} \downarrow 0$ as $k \rightarrow \infty$, $\varphi_{j,k} \in \Phi_{\alpha}$, $u_{j,k} \in \mathcal{F}_{\alpha}$, $Z_{j,k} \in B_{1/2}(0)$,  $\{\varphi_{j,k,s}\}_{s = 0,1,2,\ldots,N} \subset \Phi_{\alpha}$ is an optimal sequence with increasing spines and $\varphi_{j,k,0} = \varphi_{j,k}$, and $a_{j,k} \in S^{\lambda-1}$ such that Hypothesis \ref{graphrep_hyp} holds true with $\varepsilon = \varepsilon_{j,k}$ and $\beta = \beta_{j,k}$, $\mathcal{N}_u(Z_{j,k}) \geq \alpha$, and $S(\varphi_{j,k,s}) = \{0\} \times \mathbb{R}^{n-p_s}$ for some integers $\lambda = p_0 > p_1 > p_2 > \cdots > p_N = l_0$ but  
\begin{align} \label{branchdist_eqn2}
	&\mathcal{L}^n \left\{ X \in B_{\rho_j}(Z_{j,k}) : (1/j) r(X)^{\alpha-1} 
	\left( |a_{j,k,0}|^2 + \sum_{s=0}^{N-1} \int_{B_1(0)} \mathcal{G}(\varphi_j,\varphi_{j,k,s+1})^2 |a_{j,k,N-s}|^2 \right)^{1/2} \right. \nonumber \\
	&\leq \left. |D\varphi(X) \cdot (a_{j,k},0)| \right\} < (1/j) \rho_j^n, 
\end{align}
where $a_{j,k} = (a_{j,k,0},a_{j,k,1},a_{j,k,2},\ldots,a_{j,k,N}) \in \mathbb{R}^{l_0} \times \mathbb{R}^{p_{N-1}-p_N} \times \mathbb{R}^{p_{N-2}-p_{N-1}} \times \cdots \times \mathbb{R}^{p_0-p_1}$.  Let 
\begin{align*}
	\widetilde{a}_{j,k,s} &= \frac{\left( \int_{B_1(0)} \mathcal{G}(\varphi_{j,k},\varphi_{j,k,s+1})^2 \right)^{1/2}}{
		\left( |a_{j,k,0}|^2 + \sum_{t=0}^{N-1} \int_{B_1(0)} \mathcal{G}(\varphi_{j,k},\varphi_{j,k,t+1})^2 |a_{j,k,N-t}|^2 \right)^{1/2}} a_{j,k,N-s} 
		\text{ for } s = 0,1,\ldots,N-1, \\
	\widetilde{a}_{j,k,N} &= \frac{a_{j,k,0}}{
		\left( |a_{j,k,0}|^2 + \sum_{t=0}^{N-1} \int_{B_1(0)} \mathcal{G}(\varphi_{j,k},\varphi_{j,k,t+1})^2 |a_{j,k,N-t}|^2 \right)^{1/2}}, 
\end{align*}
and 
\begin{align*}
	\widetilde{\varphi}_{j,k,s} &= \frac{\varphi_{j,k,0,1} - \varphi_{j,k,s+1,1}}{
		\left( \int_{B_1(0)} \mathcal{G}(\varphi_{j,k},\varphi_{j,k,s+1})^2 \right)^{1/2}} \text{ for } s = 0,1,\ldots,N-1, \\
	\widetilde{\varphi}_{j,k,s} &= \varphi^{(0)},
\end{align*}
where $\varphi_{j,k,s}(X) = \{\pm \varphi_{j,k,s,1}(X)\}$ for all $X \in \mathbb{R}^n$ and some harmonic single-valued function $\varphi_{j,k,s,1}$ close to $\varphi^{(0)}_1$ in $L^2(B_1(0);\mathbb{R}^m)$.  After passing to a subsequence, let $Z_{j,k}$ converge to $Z_j \in \{0\} \times \mathbb{R}^{n-\lambda}$ by Corollary \ref{graphrep_cor}, $\widetilde{a}_{j,k,s} \rightarrow \widetilde{a}_{j,s}$, and $\widetilde{\varphi}_{j,k,s} \rightarrow \widetilde{\varphi}_{j,s}$ in $L^2(B_1(0);\mathbb{R}^m)$ as $k \rightarrow \infty$.  Note that $\sum_{s=0}^N |\widetilde{a}_{j,s}|^2 = 1$ for all $j$.  As $k \rightarrow \infty$, (\ref{branchdist_eqn2}) yields 
\begin{equation*}
	\mathcal{L}^n \left\{ X \in B_{\rho_j}(Z_j) : (1/j) r(X)^{\alpha-1} \leq \left| \sum_{s=0}^N D\widetilde{\varphi}_{j,s} \cdot (0,\widetilde{a}_{j,s},0) 
		\right| \right\} < (1/j) \rho_j^n, 
\end{equation*}
where $(0,\widetilde{a}_{j,s},0) \in \mathbb{R}^{p_{s+1}} \times \mathbb{R}^{p_s - p_{s+1}} \times \mathbb{R}^{n-p_s}$.  By translating and scaling, we can take $Z = 0$ and $\rho = 1$ so that 
\begin{equation} \label{branchdist_eqn3}
	\mathcal{L}^n \left\{ X \in B_1(0) : (1/j) r(X)^{\alpha-1} \leq \left| \sum_{s=0}^N D\widetilde{\varphi}_{j,s} \cdot (0,\widetilde{a}_{j,s},0) 
		\right| \right\} < 1/j. 
\end{equation}
After passing to a subsequence, $\widetilde{a}_{j,s} \rightarrow \widetilde{a}_s$ and $\widetilde{\varphi}_{j,s} \rightarrow \widetilde{\varphi}_s$ in $L^2(B_1(0);\mathbb{R}^m)$ as $j \rightarrow \infty$.  $\sum_{s=0}^N |\widetilde{a}_s|^2 = 1$.  By (\ref{branchdist_eqn3}), 
\begin{equation} \label{branchdist_eqn4}
	\sum_{s=0}^N D\widetilde{\varphi}_s \cdot (0,\widetilde{a}_s,0) \equiv 0 \text{ on } B_1(0). 
\end{equation}
For some increasing sequence of integer $k(j)$, $\widetilde{a}_{j,k(j),s} \rightarrow \widetilde{a}_s$, and $\widetilde{\varphi}_{j,k(j),s} \rightarrow \widetilde{\varphi}_s$ in $L^2(B_1(0);\mathbb{R}^m)$ as $j \rightarrow \infty$.  After passing to a subsequence of $j$, we can find an increasing sequence $\{s_i\}_{i=1,2,\ldots,I} \subset \in \{1,2,\ldots,N\}$ such that 
\begin{gather}
	\lim_{j \rightarrow \infty} \frac{\int_{B_1(0)} \mathcal{G}(\varphi_{j,k(j),0},\varphi_{j,k(j),s})^2}{
		\int_{B_1(0)} \mathcal{G}(\varphi_{j,k(j),0},\varphi_{j,k(j),s+1})^2} > 0 \text{ if } s \not\in \{s_i\}, \nonumber \\
	\lim_{j \rightarrow \infty} \frac{\int_{B_1(0)} \mathcal{G}(\varphi_{j,k(j),0},\varphi_{j,k(j),s_i})^2}{
		\int_{B_1(0)} \mathcal{G}(\varphi_{j,k(j),0},\varphi_{j,k(j),s_{i+1}})^2} = 0 \text{ for } i = 1,2,\ldots,I-1, \label{branchdist_eqn5}
\end{gather}
and $s_I = N$.  Let 
\begin{align*}
	\widehat{\varphi}_i &= \lim_{j \rightarrow \infty} \frac{\varphi_{j,k(j),s_i,1} - \varphi_{j,k(j),s_{i+1},1}}{
		\left( \int_{B_1(0)} \mathcal{G}(\varphi_{j,k(j)},\varphi_{j,k(j),s_{i+1}})^2 \right)^{1/2}} \text{ for } i = 1,\ldots,I-1, \\
	\widehat{\varphi}_I &= \varphi^{(0)}. 
\end{align*}
By the definition of $\widetilde{\varphi}_s$ and (\ref{branchdist_eqn5}), 
\begin{equation*}
	D\widetilde{\varphi}_s \cdot (0,\widetilde{a}_s,0) = c_s D\widehat{\varphi}_i \cdot (0,\widetilde{a}_s,0) \text{ if } s_i \leq s < s_{i+1} 
\end{equation*}
for some constant $c_s \in \mathbb{R}$, so (\ref{branchdist_eqn4}) implies that 
\begin{equation} \label{branchdist_eqn6}
	\sum_{i=1}^I D\widehat{\varphi}_i \cdot \widehat{a}_i \equiv 0 \text{ on } B_1(0) \text{ where } 
	\widehat{a}_i = \sum_{s_i \leq s < s_{i+1}} c_s (0,\widetilde{a}_s,0). 
\end{equation}
By integration by parts and the fact that $\widehat{\varphi}_i$ is translation invariant along $\{0\} \times \mathbb{R}^{n-p_i}$, 
\begin{equation*}
	\int_{B_1(0)} D_{\widehat{a}_i} \widehat{\varphi}_i \cdot D_{\widehat{a}_j} \widehat{\varphi}_j 
	= \int_{B_1(0)} \widehat{\varphi}_i \cdot D_{\widehat{a}_i} D_{\widehat{a}_j} \widehat{\varphi}_j 
	= 0 
\end{equation*}
if $1 \leq i < j \leq I$, where we rewrite the partial derivatives $D\widehat{\varphi}_i \cdot \widehat{a}_i$ as $D_{\widehat{a}_i} \widehat{\varphi}_i$ and let $\cdot$ denote the inner product on $\mathbb{R}^m$, so by (\ref{branchdist_eqn6}), 
\begin{equation*}
	\sum_{i=1}^I \int_{B_1(0)} |D\widehat{\varphi}_i \cdot \widehat{a}_i|^2 = 0; 
\end{equation*}
in other words, $\widehat{\varphi}_i$ is translation invariant in the $\widehat{a}_i$ direction whenever $\widehat{a}_i \neq 0$.  But $\widehat{a}_i \neq 0$ for some $i$ since $\sum_{i=1}^I |\widehat{a}_i|^2 = 1$ and $\widehat{\varphi}_i$ cannot be translation invariant along the $\widehat{a}_i$ direction by (\ref{branchdist_eqn5}), so we reach a contradiction. 

Now let $\varphi$, $\{\varphi_j\}_{j=0,1,2,\ldots,N}$, $u$, and $Z$ be as in the statement of Corollary \ref{branchdist_cor}.  Write $\varphi_j(X) = \{\pm \varphi_{j,1}(X)\}$ for all $X \in \mathbb{R}^n$ and some harmonic single-valued $\varphi_{j,1}$ close to $\varphi^{(0)}_1$ in $L^2(B_1(0);\mathbb{R}^m)$.  Write $Z = (\xi,\zeta) \in \mathbb{R}^{\lambda} \times \mathbb{R}^{n-\lambda}$ and, recalling that $S(\varphi_j) = \{0\} \times \mathbb{R}^{n-p_j}$, write $\xi = (\xi_0,\xi_1,\xi_2,\ldots,\xi_N) \in \mathbb{R}^{l_0} \times \mathbb{R}^{p_{N-1}-p_N} \times \mathbb{R}^{p_{N-2}-p_{N-1}} \times \cdots \times \mathbb{R}^{p_0-p_1}$.  By Taylor's theorem applied to $\varphi_{0,1}$ and the homogeneity of $\varphi$, if $|\xi| \leq r(X)/2$ then
\begin{align} \label{branchdist_eqn7}
	&\varphi_{0,1}(X-Z) = \varphi_{0,1}(X) - D_{(\xi,0)} \varphi_{0,1}(X) + \int_0^1 (1-t) D_{(\xi,0)} D_{(\xi,0)} \varphi_{0,1}(X - tZ) dt \nonumber \\
	&= \varphi_{0,1}(X) - D_{(\xi,0)} \varphi_{0,1}(X) 
		+ \sum_{j=0}^{N-1} \int_0^1 (1-t) D_{(0,\xi_{N-j},0)} D_{(\xi,0)} (\varphi_{0,1} - \varphi_{j+1,1})(X - tZ)) dt \nonumber \\&
		+ \int_0^1 (1-t) D_{(\xi_0,0)} D_{(\xi,0)} \varphi_{0,1}(X - tZ) dt, 
\end{align}
so 
\begin{align} \label{branchdist_eqn8}
	&G(u(X),\varphi(X-Z)) \geq |D\varphi(X) \cdot \xi| - G(u(X),\varphi(X)) \nonumber \\&
		- \sum_{j=0}^{N-1} \int_0^1 |D^2 (\varphi_{0,1} - \varphi_{j,1})(X - tZ))| dt \cdot |\xi| |\xi_{N-j}| 
		- \int_0^1 |D^2 \varphi_{0,1}(X - tZ)| dt \cdot |\xi| |\xi_0|. 
\end{align}
Let $\rho > 0$ to be determined.  Let $a = \xi$ in (\ref{branchdist_eqn1}) and use (\ref{branchdist_eqn8}) to obtain that for some set $S \subset B_{\rho}(Z)$ with $\mathcal{L}^n(S) \geq \overline{\delta} \rho^n$ 
\begin{align} \label{branchdist_eqn9}
	&\left( |\xi_0|^2 + \sum_{j=0}^{N-1} \int_{B_1(0)} \mathcal{G}(\varphi,\varphi_{j+1})^2 |\xi_{N-j}|^2 \right) \int_S r(X)^{2\alpha-2} dX 
	\leq \overline{\delta}^{-2} \int_{B_{\rho}(Z)} |D\varphi \cdot (\xi,0)|^2 \nonumber \\
	&\leq (N+3) \overline{\delta}^{-2} \int_{B_{\rho}(Z)} \mathcal{G}(u(X),\varphi(X-Z))^2 dX 
		+ (N+3) \overline{\delta}^{-2} \int_{B_1(0)} \mathcal{G}(u,\varphi)^2 \nonumber \\&\hspace{4mm} 
		+ (N+3) \overline{\delta}^{-2} \sum_{j=0}^{N-1} \int_{B_{\rho}(Z) \cap \{r(X) \geq 2|\xi|\}} \int_0^1 
			|D^2 (\varphi_{0,1} - \varphi_{j,1})(X - tZ))|^2 dt dX \cdot |\xi|^2 |\xi_{N-j}|^2 \nonumber \\&\hspace{4mm} 
		+ (N+3) \overline{\delta}^{-2} \int_{B_{\rho}(Z) \cap \{r(X) \geq 2|\xi|\}} \int_0^1 |D^2 \varphi_{0,1}(X - tZ)|^2 dt dX \cdot |\xi|^2 |\xi_0|^2 
			\nonumber \\&\hspace{4mm} 
		+ \overline{\delta}^{-2} \int_{B_{\rho}(Z) \cap \{r(X) \leq 2|\xi|\}} |D\varphi \cdot (\xi,0)|^2. 
\end{align}
For some $\kappa = \kappa(n,\lambda)$, $\mathcal{L}^n(B^{\lambda}_{\kappa \overline{\delta}^{1/2} \rho}(0) \times B^{n-\lambda}_{\rho}(0)) < \overline{\delta} \rho^n/2$ and thus $\mathcal{L}^n(\{ X \in S : r(X) \geq \kappa \overline{\delta}^{1/2} \rho \} \geq \overline{\delta} \rho^n/2$.  To bound the integral on the left-hand side of (\ref{branchdist_eqn9}), we have 
\begin{equation} \label{branchdist_eqn10}
	\int_S r(X)^{2\alpha-2} dX \geq \frac{1}{2} \kappa^{2\alpha-2} \overline{\delta}^{\alpha} \rho^{n+2\alpha-4} 
\end{equation}
for $C = C(n,\lambda,\varphi^{(0)}) \in (0,\infty)$.  We need to bound the terms on the right-hand side of (\ref{branchdist_eqn9}).  For the first term on the right-hand side of (\ref{branchdist_eqn9}), by Lemma \ref{translate_lemma} and Corollary \ref{radialdecay_cor}
\begin{align} \label{branchdist_eqn11}
	&\rho^{-n-2\alpha+1/2} \int_{B_{\rho}(Z)} \mathcal{G}(u(X),\varphi(X-Z))^2 dX \nonumber \\
	&\leq C \int_{B_1(0)} \mathcal{G}(u(X),\varphi(X-Z))^2 dX \nonumber \\
	&\leq C \int_{B_1(0)} \mathcal{G}(u(X),\varphi(X-Z))^2 dX + C \int_{B_1(0)} \int_0^1 |D\varphi(X-tZ) \cdot (\xi,0)|^2 dt dX \nonumber \\
	&\leq C \int_{B_1(0)} \mathcal{G}(u(X),\varphi(X-Z))^2 dX \nonumber \\&\hspace{4mm}
		+ C \sum_{j=0}^{N-1} \int_{B_1(0)} \int_0^1 |(D\varphi_{0,1} - D\varphi_{j+1,1})(X-tZ) \cdot (0,\xi_{N-j},0)|^2 dt dX \nonumber \\&\hspace{4mm}
		+ C \int_{B_1(0)} \int_0^1 |D\varphi_{0,1}(X-tZ) \cdot (0,\xi_0,0)|^2 dt dX \nonumber \\
	&\leq C \int_{B_1(0)} \mathcal{G}(u(X),\varphi(X-Z))^2 dX 
		+ C \sum_{j=0}^{N-1} \int_{B_1(0)} \mathcal{G}(\varphi,\varphi_{j+1})^2 |\xi_{N-j}|^2 + C |\xi_0|^2 
\end{align}
for $C = C(n,\lambda,m,\varphi^{(0)}) \in (0,\infty)$.  For the third and fourth terms on the right-hand side of (\ref{branchdist_eqn9}), the terms are nonzero only if $\rho > |\xi|$ and thus by direct computation 
\begin{align} \label{branchdist_eqn12}
	&\int_{B_{\rho}(Z) \cap \{|x| \geq 2|\xi|\}} \int_0^1 |D^2 (\varphi_{0,1} - \varphi_{j+1,1})(X - tZ))|^2 dt dX \nonumber \\
	&\leq \int_{B_{2\rho}(0,\zeta)} |D^2 \varphi_{0,1} - D^2 \varphi_{j+1,1}|^2
	\leq C \rho^{n+2\alpha-4} \int_{B_1(0)} \mathcal{G}(\varphi,\varphi_{j+1})^2
\end{align}
for some constant $C = C(n,\alpha) \in (0,\infty)$ and similarly 
\begin{equation} \label{branchdist_eqn13}
	\int_{B_{\rho}(Z) \cap \{|x| \geq 2|\xi|\}} \int_0^1 |D^2 \varphi_{0,1}(X - tZ)|^2 dt dX \leq C \rho^{n+2\alpha-4} 
\end{equation}
for some constant $C = C(n,\alpha,\varphi^{(0)}) \in (0,\infty)$.  Moreover, by Corollary \ref{graphrep_cor}, for $\tau > 0$ to be determined, provided $\varepsilon_0$ and $\beta_0$ are sufficiently small depending on $\tau$, $|\xi| < \tau$.  For the last term on the right-hand side of (\ref{branchdist_eqn9}) we have 
\begin{align} \label{branchdist_eqn14}
	&\int_{B_{\rho}(Z) \cap \{|x| \leq 2|\xi|\}} |D\varphi \cdot (\xi,0)|^2 \nonumber \\
	&\leq (N+1) \sum_{j=0}^{N-1} \int_{B_{\rho}(Z) \cap \{r(X) \leq 2|\xi|\}} |D\varphi_{0,1} - D\varphi_{j+1,1}|^2 |\xi_{N-j}|^2 
		+ (N+1) \int_{B_{\rho}(Z) \cap \{r(X) \leq 2|\xi|\}} |D\varphi|^2 |\xi_0|^2 \nonumber \\
	&\leq C \rho^{n-\lambda} \sum_{j=0}^{N-1} \int_{B^{\lambda}_{2|\xi|}(0)} |D\varphi_{0,1} - D\varphi_{j+1,1}|^2 |\xi_{N-j}|^2 
		+ C \rho^{n-\lambda} \int_{B^{\lambda}_{2|\xi|}(0)} |D\varphi|^2 |\xi_0|^2 \nonumber \\
	&\leq C \rho^{n-\lambda} |\xi|^{\lambda+2\alpha-2} \sum_{j=0}^{N-1} \int_{B_1(0)} \mathcal{G}(\varphi,\varphi_{j+1})^2 |\xi_{N-j}|^2 
		+ C \rho^{n-\lambda} |\xi|^{\lambda+2\alpha-2} |\xi_0|^2
\end{align}
for $C = C(n,\lambda) \in (0,\infty)$ and again $|\xi| < \tau$.  Combining (\ref{branchdist_eqn9}) - (\ref{branchdist_eqn14}), and $|\xi| < \tau$, we obtain 
\begin{equation} \label{branchdist_eqn15}
	\left( |\xi_0|^2 + \sum_{j=0}^{N-1} \int_{B_1(0)} \mathcal{G}(\varphi,\varphi_{j+1})^2 |\xi_{N-j}|^2 \right) 
		\left( 1 - C \rho^{7/2} + C\tau^2 + \frac{C \tau^{\lambda+2\alpha-2}}{\rho^{\lambda+2\alpha-4}} \right) \rho^{n+2\alpha-4} 
	\leq C \int_{B_1(0)} \mathcal{G}(u,\varphi)^2 
\end{equation}
for $C = C(n,\lambda,m,\varphi^{(0)}) \in (0,\infty)$.  Choose $\rho$ such that $C \rho^{7/2} < 1/6$, then choose $\tau$ such that $C \tau^2 < 1/6$ and $\tau \leq \rho$, and finally choose $\varepsilon_0$ and $\beta_0$ small enough that we have (\ref{branchdist_eqn1}) and $|\xi| < \tau$ so that (\ref{branchdist_eqn15}) yields 
\begin{equation*}
	|\xi_0|^2 + \sum_{j=0}^{N-1} \int_{B_1(0)} \mathcal{G}(\varphi,\varphi_{j+1})^2 |\xi_{N-j}|^2 \leq C \int_{B_1(0)} \mathcal{G}(u,\varphi)^2 
\end{equation*}
for $C = C(n,\lambda,m,\varphi^{(0)}) \in (0,\infty)$.
\end{proof}

\begin{proof}[Proof of Corollary \ref{nonconest_cor} for $l = \lambda$]
By Corollary \ref{branchdist_cor}, the computation of (\ref{translate_eqn1}) holds true when $l = \lambda$ and $\rho = 1/2$ to give us 
\begin{equation} \label{nonconest_eqn1}
	\int_{B_{1/2}(Z)} \mathcal{G}(u(X),\varphi(X-Z))^2 dX \leq C \int_{B_1(0)} \mathcal{G}(u,\varphi)^2 
\end{equation}
for $C = C(n,\lambda,m,\varphi^{(0)}) \in (0,\infty)$.  By Corollary \ref{radialdecay_cor} with $2^{\alpha} u(Z+X/2)$ in place of $u$, which holds true by Lemma \ref{translate_lemma}, and by (\ref{nonconest_eqn1}),  
\begin{equation} \label{nonconest_eqn2}
	\int_{B_{1/4}(Z)} \frac{\mathcal{G}(u(X),\varphi(X-Z))^2}{|X-Z|^{n+2\alpha-\sigma}} dX 
	\leq C \int_{B_{1/2}(Z)} \mathcal{G}(u,\varphi)^2 
	\leq C \int_{B_1(0)} \mathcal{G}(u,\varphi)^2 
\end{equation}
for $C = C(n,\lambda,m,\varphi^{(0)}) \in (0,\infty)$. 

Let $\{\varphi_j\}_{j = 0,1,2,\ldots,N} \subset \Phi_{\alpha}$ be an optimal sequence with increasing spines such that $\varphi_0 = \varphi$.  After a small rotation of $\mathbb{R}^n$, take $S(\varphi_N) = \{0\} \times \mathbb{R}^{n-l_0}$.  Write $\varphi_j(X) = \{\pm \varphi_{j,1}(X)\}$ for all $X \in \mathbb{R}^n$ and some harmonic single-valued function $\varphi_{j,1}$ close to $\varphi^{(0)}_1$ in $L^2(B_1(0);\mathbb{R}^m)$.  Let $Z = (\xi,\zeta) \in \mathbb{R}^{\lambda} \times \mathbb{R}^{n-\lambda}$ and $\xi = (\xi_0,\xi_1,\xi_2,\ldots,\xi_N) \in \mathbb{R}^{l_0} \times \mathbb{R}^{p_N-p_{N-1}} \times \mathbb{R}^{p_{N-1}-p_{N-2}} \times \cdots \times \mathbb{R}^{p_0-p_1}$.  Observe that by (\ref{nonconest_eqn2}), Corollary \ref{branchdist_cor}, and elliptic estimates applied to $\varphi$ and $\varphi_j$, 
\begin{align*}
	&\int_{B_{\gamma}(0)} \frac{\mathcal{G}(u,\varphi)^2}{|X-Z|^{n-\sigma}} \nonumber \\
	&\leq (N+2) \int_{B_{1/4}(Z)} \frac{\mathcal{G}(u,\varphi)^2}{|X-Z|^{n-\sigma}} 
		+ (N+2) \sum_{j=0}^{N-1} \int_{B_{1/4}(Z)} \int_0^1 \frac{|D(\varphi_{0,1} - \varphi_{j,1})(X - tZ)|^2 |\xi_{N-j}|^2}{|X-Z|^{n-\sigma}} dt dX 
			\nonumber \\& \hspace{4mm} 
		+ (N+2) \int_{B_{1/4}(Z)} \int_0^1 \frac{|D\varphi(X - tZ)|^2 |\xi_0|^2}{|X-Z|^{n-\sigma}} dt dX
		+ 4^{n-\sigma} \int_{B_1(0)} \mathcal{G}(u,\varphi)^2 \\
	&\leq (N+2) \int_{B_{1/4}(Z)} \frac{\mathcal{G}(u,\varphi)^2}{|X-Z|^{n-\sigma}} 
		+ (N+2) \sum_{j=0}^{N-1} \int_{B_1(0)} \mathcal{G}(\varphi,\varphi_j)^2 |\xi_{N-j}|^2 \int_{B_{1/4}(Z)} \frac{dX}{|X-Z|^{n-\sigma}} 
			\nonumber \\& \hspace{4mm} 
		+ (N+2) \int_{B_1(0)} |\varphi|^2 |\xi_0|^2 \int_{B_{1/4}(Z)} \frac{dX}{|X-Z|^{n-\sigma}} 
		+ 4^{n-\sigma} \int_{B_1(0)} \mathcal{G}(u,\varphi)^2 \\
	&\leq C \int_{B_1(0)} \mathcal{G}(u,\varphi)^2 
\end{align*}
for some constant $C = C(n,\lambda,m,\varphi^{(0)},\sigma) \in (0,\infty)$. 

Let $v$ be as in Corollary \ref{graphrep_cor}.  By (\ref{branchdist_eqn7}), 
\begin{align} \label{nonconest_eqn3}
	&\int_{B_{\gamma}(0) \cap \{r(X) \geq \tau\}} \frac{|v - D\varphi_{0,1} \cdot (\xi,0)|^2}{|X-Z|^{n+2\alpha-\sigma}} \nonumber \\
	&\leq (N+2) \int_{B_{1/4}(Z) \cap \{r(X) \geq \tau\}} \frac{|u_1(X) - \varphi_{0,1}(X-Z)|^2}{|X-Z|^{n+2\alpha-\sigma}} \nonumber \\
		&+ (N+2) \sum_{j=0}^{N-1} \int_{B_{1/4}(Z) \cap \{r(X) \geq \tau\}} \int_0^1 \frac{|D^2 (\varphi_{0,1} - \varphi_{j,1})(X - tZ)|^2 |\xi|^2 |\xi_{N-j}|^2}{
			|X-Z|^{n+2\alpha-\sigma}} dt dX \nonumber \\
		&+ (N+2) \int_{B_{1/4}(Z) \cap \{r(X) \geq \tau\}} \int_0^1 \frac{|D^2 \varphi(X - tZ)|^2 |\xi|^2 |\xi_0|^2}{|X-Z|^{n+2\alpha-\sigma}} dt dX 
			\nonumber \\
		&+ (N+2) 4^{n+2\alpha-\sigma} \int_{(B_{\gamma}(0) \setminus B_{1/4}(Z)) \cap \{r(X) \geq \tau\}} |v|^2 \nonumber \\ 
		&+ (N+2) 4^{n+2\alpha-\sigma} \sum_{j=0}^{N-1} \int_{(B_{\gamma}(0) \setminus B_{1/4}(Z)) \cap \{r(X) \geq \tau\}} |D\varphi_{0,1} - D\varphi_{j,1}|^2 
			|\xi|^2 |\xi_{N-j}|^2 \nonumber \\
		&+ (N+2) 4^{n+2\alpha-\sigma} \int_{(B_{\gamma}(0) \setminus B_{1/4}(Z)) \cap \{r(X) \geq \tau\}} |D\varphi|^2 |\xi|^2 |\xi_0|^2. 
\end{align}
By locally approximating $\varphi |_{S^{l-1}}$ by polynomials of degree at most $\alpha$ using Taylor's theorem and then using Lemma \ref{twovalL2_lemma}, 
\begin{equation*}
	\int_{S^{{\lambda}-1}} |u_1(\xi+r\omega,\zeta+y) - \varphi_{0,1}(r\omega,y)|^2 d\omega  
	\leq C \int_{S^{{\lambda}-1}} \mathcal{G}(u(\xi+r\omega,\zeta+y),\varphi(r\omega,y))^2 d\omega  
\end{equation*}
for all $r \in [\tau,1]$ and $y \in \mathbb{R}^{n-\lambda}$ with $r^2 + |y|^2 \leq 1/16$ for some constant $C = C(\lambda,m,\varphi^{(0)}) \in (0,\infty)$, so by (\ref{nonconest_eqn2}), 
\begin{equation} \label{nonconest_eqn4}
	\int_{B_{1/4}(Z)} \frac{|u_1(X) - \varphi_{0,1}(X-Z)|^2}{|X-Z|^{n+2\alpha-\sigma}} 
	\leq C \int_{B_{1/4}(Z)} \frac{\mathcal{G}(u(X),\varphi(X-Z))^2}{|X-Z|^{n+2\alpha-\sigma}} 
	\leq C \int_{B_1(0)} \mathcal{G}(u,\varphi)^2  
\end{equation}
for $C = C(n,\lambda,m,\varphi^{(0)},\sigma) \in (0,\infty)$.  By Lemma \ref{twovalL2_lemma}, 
\begin{equation} \label{nonconest_eqn5}
	\int_{(B_{\gamma}(0) \setminus B_{1/4}(Z)) \cap \{r(X) \geq \tau\}} |v|^2 \leq C \int_{B_1(0)} \mathcal{G}(u,\varphi)^2 
\end{equation}
for $C = C(n,\lambda,m,\varphi^{(0)},\gamma) \in (0,\infty)$.  By elliptic estimates applied to $\varphi$ and $\varphi_j$ and Corollary \ref{branchdist_cor}, 
\begin{align} \label{nonconest_eqn6}
	&\int_{B_{1/4}(Z) \cap \{r(X) \geq \tau\}} \int_0^1 \frac{|D^2 (\varphi_{0,1} - \varphi_{j,1})(X - tZ)|^2 |\xi|^2 |\xi_{N-j}|^2}{|X-Z|^{n+2\alpha-\sigma}} 
		dt dX \nonumber \\
	&\leq C |\xi|^2 |\xi_{N-j}|^2 \int_{B_1(0)} \mathcal{G}(\varphi,\varphi_j)^2 \int_{B_{1/4}(Z) \cap \{r(X) \geq \tau\}} \frac{dX}{|X-Z|^{n+2\alpha-\sigma}} 
		\nonumber \\
	&\leq \frac{C |\xi|^2}{\tau^{2\alpha}} \int_{B_1(0)} \mathcal{G}(u,\varphi)^2
\end{align}
for $C = C(n,\lambda,m,\varphi^{(0)},\sigma) \in (0,\infty)$ and similarly 
\begin{equation} \label{nonconest_eqn7}
	\int_{B_{1/4}(Z) \cap \{r(X) \geq \tau\}} \int_0^1 \frac{|D^2 \varphi(X - tZ)|^2 |\xi|^2 |\xi_0|^2}{|X-Z|^{n+2\alpha-\sigma}} dt dX 
	\leq \frac{C |\xi|^2}{\tau^{2\alpha}} \int_{B_1(0)} \mathcal{G}(u,\varphi)^2 
\end{equation}
for $C = C(n,\lambda,m,\varphi^{(0)},\sigma) \in (0,\infty)$.  Similarly, 
\begin{align} \label{nonconest_eqn8}
	\int_{(B_{\gamma}(0) \setminus B_{1/4}(Z)) \cap \{r(X) \geq \tau\}} |D\varphi_{0,1} - D\varphi_{j,1}|^2 |\xi|^2 |\xi_{N-j}|^2 
		&\leq C |\xi|^2 \int_{B_1(0)} \mathcal{G}(u,\varphi)^2, \nonumber \\
	\int_{(B_{\gamma}(0) \setminus B_{1/4}(Z)) \cap \{r(X) \geq \tau\}} |D\varphi|^2 |\xi|^2 |\xi_0|^2 
		&\leq C |\xi|^2 \int_{B_1(0)} \mathcal{G}(u,\varphi)^2
\end{align}
for $C = C(n,\lambda,m,\varphi^{(0)}) \in (0,\infty)$.  By (\ref{nonconest_eqn3}) - (\ref{nonconest_eqn8}) and choosing $\varepsilon_0$ and $\beta_0$ sufficiently small that $|\xi| < \tau^{\alpha}$, 
\begin{equation*}
	\int_{B_{\gamma}(0) \cap \{r(X) \geq \tau\}} \frac{|v - D\varphi_{0,1} \cdot (\xi,0)|^2}{|X-Z|^{n+2\alpha-\sigma}} 
	\leq C \int_{B_1(0)} \mathcal{G}(u,\varphi)^2
\end{equation*}
for $C = C(n,\lambda,m,\varphi^{(0)},\gamma,\sigma) \in (0,\infty)$.  
\end{proof}

\section{Improvement of excess and proof of Theorem \ref{constfreqthm1}} \label{sec:excessdecay_sec}

\begin{lemma} \label{excessdecay1_lemma}
Let $\vartheta \in (0,1/4)$ and let $\varphi^{(0)} \in \Phi_{\alpha}$ be a Dirichlet energy minimizing two-valued function.  Then there are $\overline{\varepsilon} = \overline{\varepsilon}(n,m,\varphi^{(0)},\vartheta) \in (0,1)$ and $\overline{\beta} = \overline{\beta}(n,m,\varphi^{(0)},\vartheta) \in (0,1)$ such that if $\varphi \in \Phi_{\alpha}$ and $u \in \mathcal{F}_{\alpha}$ such that Hypothesis \ref{cft1_hyp} holds true with $\varepsilon = \overline{\varepsilon}$ and $\beta = \overline{\beta}$ and $\mathcal{N}_u(0) \geq \alpha$, then there exists $\widetilde{\varphi} \in \Phi_{\alpha}$ with 
\begin{equation} \label{excessdecay1_eqn1}
	\int_{B_1(0)} \mathcal{G}(\widetilde{\varphi},\varphi)^2 \leq 3 \int_{B_1(0)} \mathcal{G}(u,\varphi)^2
\end{equation}
and 
\begin{equation} \label{excessdecay1_eqn2}
	\vartheta^{-n-2\alpha} \int_{B_{\vartheta}(0)} \mathcal{G}(u,\widetilde{\varphi})^2 \leq \overline{C} \vartheta^2 \int_{B_1(0)} \mathcal{G}(u,\varphi)^2 
\end{equation}
for some $\overline{C} = \overline{C}(n,m,\varphi^{(0)}) \in [1,\infty)$.  
\end{lemma}
\begin{proof}
We argue as in the proof of Lemma 4.1 of~\cite{KrumWic1}, which itself is based on the proof of Lemma 1 of~\cite{Sim93}.  Let $\vartheta \in (0,1/4)$, $\varepsilon_j,\beta_j \downarrow 0$, $\varphi_j \in \Phi_{\alpha}$, and $u_j \in \mathcal{F}_{\alpha}$ such that Hypothesis \ref{cft1_hyp} holds true with $\varepsilon = \varepsilon_j$ and $\beta = \beta_j$, and $N_{u_j}(0) \geq \alpha$.  We want to show that for infinitely many $j$ there exists $\widetilde{\varphi}_j \in \Phi_{\alpha}$ such that (\ref{excessdecay1_eqn3}) and (\ref{excessdecay1_eqn4}) hold true with $u = u_j$, $\varphi = \varphi_j$, and $\widetilde{\varphi} = \widetilde{\varphi}_j$ for some constant $\overline{C} \in (0,\infty)$ depending only on $n$, $m$, and $\varphi^{(0)}$ and independent of $j$. 

Let $\{\varphi_{j,k}\}_{k=0,1,2,\ldots,N}$ be a sequence with increasing spines such that $\varphi_{j,0} = \varphi_j$.  Write $\varphi_{j,k}(X) = \{\pm \varphi_{j,k,1}(X)\}$ for all $X \in \mathbb{R}^n$ and some harmonic single-valued function $\varphi_{j,k,1}$ close to $\varphi^{(0)}_1$ in $L^2(B_1(0);\mathbb{R}^m)$.  After passing to a subsequence and rotating $u_j$, $\varphi_j$, and $\varphi_{j,k}$ slightly, we may assume that $S(\varphi^{(0)}) = \{0\} \times \mathbb{R}^{n-l_0}$ and $S(\varphi_{j,k}) = \{0\} \times \mathbb{R}^{n-p_k}$ for integers $2 \leq l_0 = p_N < p_{N-1} < \cdots < p_1 < p_0 = l \leq n$.  Let $\tau_j \downarrow 0$ slowly enough that Corollary \ref{graphrep_cor} and Corollary \ref{nonconest_cor} hold with $\gamma = 3/4$, $\tau = \tau_j$, $\sigma = 1/2$, $u(X) = \rho^{-\alpha} u_j(\rho X)$, and $\varphi(X) = \rho^{-\alpha} \varphi_j(\rho X)$ for $\rho \in [\vartheta,1]$.  By Corollary \ref{graphrep_cor}, we get $v_j \in C^2(\op{graph} \varphi |_{U_j};\mathbb{R}^m)$, where $U_j = \{(x,y) \in B_{3/4}(0) : |x| > \tau_j\}$, such that
\begin{equation*}
	u_j(X) = \{ \varphi_{j,1}(X) + v_j(X,\varphi_{j,1}(X)), -\varphi_{j,1}(X) + v_j(X,-\varphi_{j,1}(X)) \} .
\end{equation*}
By the sequential compactness of the space of closed subsets of a compact space equipped with the Hausdorff metric, after passing to a subsequence $\mathcal{B}_{u_j}$ converges to some closed subset $\mathcal{D} \subseteq \{0\} \times B^{n-l}_1(0)$ in Hausdorff distance.  Recall that after passing to a subsequence, $v_j/E_j$, where $E_j = \left( \int_{B_1(0)} \mathcal{G}(u_j,\varphi_j)^2 \right)^{1/2}$, converges to a blow-up $w$ in $C^2$ on compact subsets of the interior of $B_1(0) \setminus \mathcal{D}$ and in $L^2(B_{1/2}(0);\mathbb{R}^m)$.  Since $w$ is a blow-up of $u_j$ relative to $\varphi_j$, 
\begin{equation} \label{excessdecay1_eqn3}
	\lim_{j \rightarrow \infty} E_j^{-2} \int_{B_{\rho}(0)} \mathcal{G}(u_j,\varphi_j)^2 = \sqrt{2} \int_{B_{\rho}(0)} |w|^2  
\end{equation} 
whenever $\rho \in [\vartheta,1/4]$. 

Let $(0,\zeta) \in B_{1/2}(0) \cap \mathcal{D}$.  Let $Z_j \in \mathcal{B}_{u_j}$ converge to $(0,\zeta)$.  After passing to a subsequence of $j$, we can find an increasing sequence $\{k_i\}_{i=1,2,\ldots,I} \subset \in \{1,2,\ldots,N\}$ such that 
\begin{gather}
	\lim_{j \rightarrow \infty} \frac{\int_{B_1(0)} \mathcal{G}(\varphi_{j,0},\varphi_{j,k})^2}{
		\int_{B_1(0)} \mathcal{G}(\varphi_{j,0},\varphi_{j,k+1})^2} > 0 \text{ if } k \not\in \{k_i\}, \nonumber \\
	\lim_{j \rightarrow \infty} \frac{\int_{B_1(0)} \mathcal{G}(\varphi_{j,0},\varphi_{j,k_i})^2}{
		\int_{B_1(0)} \mathcal{G}(\varphi_{j,0},\varphi_{j,k_{i+1}})^2} = 0 \text{ for } i = 1,2,\ldots,I-1, \label{excessdecay1_eqn4}
\end{gather}
and $k_I = N$.  Let $k_0 = l$.  Write $Z_j = (\xi_{j,0},\xi_{j,1},\xi_{j,2},\ldots,\xi_{j,I},\zeta) \in \mathbb{R}^{l_0} \times \mathbb{R}^{p_{k_I}-p_{k_I-1}} \times \mathbb{R}^{p_{k_I-1}-p_{k_I-2}} \times \cdots \times \mathbb{R}^{p_{k_0}-p_{k_1}} \times \mathbb{R}^{n-l}$.  After passing to a subsequence, the limit 
\begin{equation*}
	\widehat{\varphi}_i = \frac{\varphi_{j,k_1,1} - \varphi_{j,k_{i+1},1}}{\left( \int_{B_1(0)} \mathcal{G}(\varphi_{j,k_1},\varphi_{j,k_{i+1}})^2 \right)^{1/2}}
\end{equation*}
exists in $L^2(B_1(0);\mathbb{R}^m)$ for $k = 1,2,\ldots,I-1$.  Let $\widehat{\varphi}_I = \varphi^{(0)}$.  By (\ref{translate_eqn4}), 
\begin{equation*}
	\widehat{\varphi}_i = \frac{\varphi_{j,k_i,1} - \varphi_{j,k_{i+1},1}}{\left( \int_{B_1(0)} \mathcal{G}(\varphi_{j,k_1},\varphi_{j,k_{i+1}})^2 \right)^{1/2}}
\end{equation*}
and consequently $\widehat{\varphi}_i$ is translation invariant along $\{0\} \times \mathbb{R}^{n-p_{k_i}}$.  By Corollary \ref{branchdist_cor}, after passing to a subsequence the limits 
\begin{gather*}
	\kappa_i = \lim_{j \rightarrow \infty} 
		\left( \frac{\int_{B_1(0)} \mathcal{G}(\varphi_{j,k_1},\varphi_{j,k_{i+1}})^2}{\int_{B_1(0)} \mathcal{G}(u_j,\varphi_{j,k_1})^2} \right)^{1/2} 
		\cdot \xi_{I-i} \text{ for } k = 1,2,\ldots,I-1, \nonumber \\
	\kappa_I = \lim_{j \rightarrow \infty} \frac{\xi_{j,0}}{\left( \int_{B_1(0)} \mathcal{G}(u_j,\varphi_{j,K})^2 \right)^{1/2}}, 
\end{gather*}
exist.  By Corollary \ref{nonconest_cor}(ii), 
\begin{equation} \label{excessdecay1_eqn5}
	\int_{B_{3/4}(0)} \frac{\left|v_j - \sum_{i=0}^{I-1} (D\varphi_{j,0,1} - D\varphi_{j,k_{i+1},1}) \cdot \xi_{I-i} + D\varphi_{j,N,1} \cdot \xi_0 \right|^2}{
		|X-Z_j|^{n+2\alpha-1/2}} \leq C \int_{B_1(0)} \mathcal{G}(u_j,\varphi_j)^2. 
\end{equation}
for some constant $C = C(n,m,\varphi^{(0)}) \in (0,\infty)$.  Dividing (\ref{excessdecay1_eqn5}) by $\mathcal{G}(u_j,\varphi_j)^2$ and letting $j \rightarrow \infty$, 
\begin{equation} \label{excessdecay1_eqn6}
	\int_{B_{3/4}(0)} \frac{\left| w - \sum_{i=0}^I D\widehat{\varphi}_i \cdot \kappa_{I-i} \right|}{|X-Z_j|^{n+2\alpha-1/2}} 
	\leq C \int_{B_1(0)} \mathcal{G}(u_j,\varphi_j)^2. 
\end{equation}
and thus 
\begin{equation} \label{excessdecay1_eqn7}
	\rho^{-n} \int_{B_{\rho}(0,\zeta)} \left|w - \sum_{i=0}^I D\widehat{\varphi}_i \cdot (0,\kappa_{I-i},0) \right|^2 
	\leq C \rho^{2\alpha-1/2}
\end{equation} 
for all $\rho \in (0,3/4-|\zeta|)$.  Note that by integration by parts and the fact that $\widehat{\varphi}_i$ is translation invariant along $\{0\} \times \mathbb{R}^{n-p_{k_i}}$, 
\begin{equation*}
	\int_{B_1(0)} D_{(0,\kappa_i,0)} \widehat{\varphi}_i \cdot D_{(0,\kappa_j,0)} \widehat{\varphi}_j 
	= \int_{B_1(0)} \widehat{\varphi}_i \cdot D_{(0,\kappa_i,0)} D_{(0,\kappa_j,0)} \widehat{\varphi}_j 
	= 0 
\end{equation*}
if $1 \leq i < j \leq I$, which since $w$ is homogeneous degree $\alpha$ implies that 
\begin{equation} \label{excessdecay1_eqn8}
	\int_{S^{n-1}} D_{(0,\kappa_i,0)} \widehat{\varphi}_i \cdot D_{(0,\kappa_j,0)} \widehat{\varphi}_j = 0 
\end{equation}
if $1 \leq i < j \leq I$.  By (\ref{excessdecay1_eqn7}) and (\ref{excessdecay1_eqn8}), $\kappa_i$ are independent of the subsequence of $j$ used to obtain the limit $\kappa_i$ and thus (\ref{excessdecay1_eqn7}) holds true with $\kappa_i = \kappa_i(\zeta)$ for some functions $\kappa_i : B_{1/2}(0) \cap \mathcal{D} \rightarrow \mathbb{R}^2$.  Since $0 \in \mathcal{B}_{u_j}$ for all $j$, $0 \in \mathcal{D}$ and $\kappa_i(0) = 0$.  

By (\ref{excessdecay1_eqn7}), if $\alpha \geq 2$ then 
\begin{equation*}
	\rho^{-n} \int_{B_{\rho}(0,\zeta)} |w|^2 \leq C \rho^{2\alpha-2}, 
\end{equation*} 
so by the Schauder estimates $w$ is $C^{\alpha-2,1}$ on the interior of $B_{1/2}(0)$ with $D^k w = \{0,0\}$ on $\{0\} \times B^{n-l}_{1/2}(0)$ for $k \leq \alpha-2$.  If instead $\alpha = 1$, $D\widehat{\varphi}_i$ are constant and thus so $w$ is continuous on the interior of $B_{1/2}(0)$.  Since, regardless of the value of $\alpha$, $w$ is a continuous on the interior of $B_{1/2}(0)$ and harmonic on the interior of $B_{1/2}(0) \setminus \{0\} \times \mathbb{R}^{n-l}$, $w$ is a smooth harmonic single-valued function on the interior of $B_{1/2}(0)$.  Recall that $D^k w(0,0) = \{0,0\}$ for $\kappa \leq \alpha-2$ and observe that by (\ref{excessdecay1_eqn7}) with $\zeta = 0$ and $\kappa_i = 0$, $D^{\alpha-1} v(0,0) = \{0,0\}$.  Hence by standard elliptic theory, 
\begin{equation} \label{excessdecay1_eqn9}
	\int_{B_{\vartheta}(0)} |w - \psi(X)|^2 \leq C \vartheta^2 \int_{B_1(0)} |w - \psi(X)|^2
\end{equation}
for some constant $C = C(n,m) \in (0,\infty)$ and some homogeneous degree $\alpha$, harmonic single-valued function $\psi : \mathbb{R}^n \rightarrow \mathbb{R}^m$ such that 
\begin{equation} \label{excessdecay1_eqn10}
	\int_{B_{\rho}(0)} |w|^2 = \int_{B_{\rho}(0)} |\psi|^2 + \int_{B_1(0)} |v(X,\varphi^{(0)}_1(X)) - \psi(X)|^2 
\end{equation}
for all $\rho \in (0,1]$. 

Let 
\begin{equation*}
	\widetilde{\varphi}_j(X) = \{ \pm (\varphi_{j,1}(X) + E_j \psi(X)) \}.
\end{equation*}
By (\ref{excessdecay1_eqn10}) and (\ref{excessdecay1_eqn10}), 
\begin{equation*}
	\int_{B_{1/4}(0)} |\psi|^2 \leq 1/\sqrt{2} 
\end{equation*}
and so (\ref{excessdecay1_eqn1}) holds true.  Repeating the argument above with $\widetilde{\varphi}_j$ in place of $\varphi_j$, we obtain 
\begin{equation} \label{excessdecay1_eqn11}
	\lim_{j \rightarrow \infty} E_j^{-2} \int_{B_{\rho}(0)} \mathcal{G}(u_j,\widetilde{\varphi}_j)^2 
	= \sqrt{2} \int_{B_{\rho}(0)} |w - \psi(X)|^2 
\end{equation} 
whenever $\rho \in [\vartheta,1/4]$.  By (\ref{excessdecay1_eqn3}), (\ref{excessdecay1_eqn9}), and (\ref{excessdecay1_eqn11}), we obtain (\ref{excessdecay1_eqn2}) with $u = u_j$, $\varphi = \varphi_j$, and $\widetilde{\varphi} = \widetilde{\varphi}_j$ for large $j$ as required. 
\end{proof}

In the following two lemmas, we remove Hypothesis \ref{graphrep_hyp}(ii) using an argument from Section 13 of~\cite{Wic14}. 

\begin{lemma} \label{excessdecay2_lemma}
Let $\varphi^{(0)} \in \Phi_{\alpha}$ be a Dirichlet energy minimizing two-valued function with $S(\varphi^{(0)}) = n-l_0$ for some $l_0 \in \{2,3,\ldots,n\}$ and let $l \in \{l_0,l_0+1,\ldots,n\}$.  For $j = 0,1,2,\ldots,l-l_0$, let $\vartheta_j \in (0,1/4)$ such that $8\vartheta_j < \vartheta_{j+1}$ for $j = 0,1,2,\ldots,l-l_0-1$.  There exists $\varepsilon^{(l)} = \varepsilon^{(l)}(n,m,\varphi^{(0)},\vartheta_0,\vartheta_1,\ldots,\vartheta_{l-l_0}) \in (0,1)$ such that if $\varphi \in \Phi_{\alpha}$ and $u \in \mathcal{F}_{\alpha}$ such that Hypothesis \ref{cft1_hyp}(i) holds true with $\varepsilon = \varepsilon^{(l)}$, $S(\varphi) = n-l$, and $\mathcal{N}_u(0) \geq \alpha$, then there exists $\widetilde{\varphi} \in \Phi_{\alpha}$ with 
\begin{equation*}
	\int_{B_1(0)} \mathcal{G}(\widetilde{\varphi},\varphi)^2 \leq 3 \int_{B_1(0)} \mathcal{G}(u,\varphi)^2
\end{equation*}
and for some $j \in \{0,1,2,\ldots,l-l_0\}$
\begin{equation*}
	\vartheta_j^{-n-2\alpha} \int_{B_{\vartheta_j}(0)} \mathcal{G}(u,\widetilde{\varphi})^2 
	\leq C^{(l)}_j \vartheta_j^2 \int_{B_1(0)} \mathcal{G}(u,\varphi)^2 
\end{equation*}
where $C^{(l)}_0 = \overline{C}_0(n,m,\varphi^{(0)}) \in [1,\infty)$ and $C^{(l)}_j = C^{(l)}_j(n,m,\varphi^{(0)},\vartheta_0,\vartheta_1,\ldots,\vartheta_{j-1}) \in [1,\infty)$ for $j = 1,2,\ldots,l-l_0$ are constants.  
\end{lemma}
\begin{proof}
We shall proceed by induction on $l$.  If $l = l_0$, then Lemma \ref{excessdecay2_lemma} with $\varepsilon^{(l_0)} = \overline{\varepsilon}(n,m,\varphi^{(0)},\vartheta_0)$ follows from Lemma \ref{excessdecay1_lemma} with $\vartheta = \vartheta_0$.  Suppose that Lemma \ref{excessdecay2_lemma} holds true whenever $l < l_0+p$ for some $p \in \{1,2,\ldots,n-l_0\}$.  We shall prove Lemma \ref{excessdecay2_lemma} for the case $l = l_0+p$ with 
\begin{align*}
	\varepsilon^{(l_0+p)} &= \min\{ \overline{\varepsilon}(n,m,\varphi^{(0)},\vartheta_0) \} \cup 
		\{ (2\widetilde{\beta}/3) \varepsilon^{(l)}(n,m,\varphi^{(0)},\vartheta_1,\vartheta_2,\ldots,\vartheta_{l-l_0+1}) : l_0 \leq l < l_0+p \}, \\
	C^{(l_0+p)}_0 &= \overline{C}(n,m,\varphi^{(0)}), \\
	C^{(l_0+p)}_1 &= \max\{ C^{(l)}_0(n,m,\varphi^{(0)}) : l_0 \leq l < l_0+p \}, \\
	C^{(l_0+p)}_j &= \max\{ (2\widetilde{\beta}/3)^{-1} C^{(l)}_j(n,m,\varphi^{(0)},\vartheta_1,\vartheta_2,\ldots,\vartheta_{j+1}) : l_0 \leq l < l_0+p \} 
		\text{ if } 2 \leq j \leq p, 
\end{align*} 
where $\widetilde{\beta} = \overline{\beta}(n,m,\varphi^{(0)},\vartheta_0)$.  Let $\varphi \in \Phi_{\alpha}$ and $u \in \mathcal{F}_{\alpha}$ such that Hypothesis \ref{cft1_hyp}(i) holds true with $\varepsilon = \varepsilon^{(l_0+p)}$, $S(\varphi) = n-l_0+p$, and $\mathcal{N}_u(0) \geq \alpha$.  If Hypothesis \ref{cft1_hyp}(ii) holds true with $\beta = \widetilde{\beta}$, then the conclusion of Lemma \ref{excessdecay2_lemma} with $j = 0$ follows from Lemma \ref{excessdecay1_lemma} with $\vartheta = \vartheta_0$.  If Hypothesis \ref{cft1_hyp}(ii) fails for $\beta = \widetilde{\beta}$, we can find $\varphi' \in \Phi_{\alpha}$ with $S(\varphi) \subset S(\varphi')$ and 
\begin{equation} \label{excessdecay2_eqn1}
	\int_{B_1(0)} \mathcal{G}(u,\varphi')^2 \leq \frac{2}{3\widetilde{\beta}} \int_{B_1(0)} \mathcal{G}(u,\varphi)^2. 
\end{equation}
Then the conclusion of Lemma \ref{excessdecay2_lemma} with $\varphi$ follows from Lemma \ref{excessdecay2_lemma} with $\varphi'$ in place of $\varphi$ and $\vartheta_{j+1}$ in place of $\vartheta_j$ and from (\ref{excessdecay2_eqn1}). 
\end{proof}

\begin{lemma} \label{excessdecay3_lemma}
Let $\varphi^{(0)} \in \Phi_{\alpha}$ be a Dirichlet energy minimizing two-valued function with $S(\varphi^{(0)}) = n-l_0$ for some $l_0 \in \{2,3,\ldots,n\}$.  For $j = 0,1,2,\ldots,n-l_0$, let $\vartheta_j \in (0,1/4)$ such that $8\vartheta_j < \vartheta_{j+1}$ for $j = 0,1,2,\ldots,n-l_0-1$.  There exists $\varepsilon = \varepsilon(n,m,\varphi^{(0)},\vartheta_0,\vartheta_1,\ldots,\vartheta_{n-l_0}) \in (0,1)$ such that if $\varphi \in \Phi_{\alpha}$ and $u \in \mathcal{F}_{\alpha}$ such that Hypothesis \ref{cft1_hyp}(i) holds true and $\mathcal{N}_u(0) \geq \alpha$, then there exists $\widetilde{\varphi} \in \Phi_{\alpha}$ with 
\begin{equation*}
	\int_{B_1(0)} \mathcal{G}(\widetilde{\varphi},\varphi)^2 \leq 3 \int_{B_1(0)} \mathcal{G}(u,\varphi)^2
\end{equation*}
and for some $j \in \{0,1,2,\ldots,n-l_0\}$
\begin{equation*}
	\vartheta_j^{-n-2\alpha} \int_{B_{\vartheta_j}(0)} \mathcal{G}(u,\widetilde{\varphi})^2 
	\leq C_j \vartheta_j^2 \int_{B_1(0)} \mathcal{G}(u,\varphi)^2 
\end{equation*}
where $C_0 =C_0(n,m,\varphi^{(0)}) \in [1,\infty)$ and $C_j = C_j(n,m,\varphi^{(0)},\vartheta_0,\vartheta_1,\ldots,\vartheta_{j-1}) \in [1,\infty)$ for $j = 1,2,\ldots,n-l_0$ are constants.  
\end{lemma}
\begin{proof}
Let 
\begin{align*}
	\varepsilon &= \min\{ \varepsilon^{(l)}(n,m,\varphi^{(0)},\vartheta_0,\vartheta_1,\ldots,\vartheta_{l-l_0}) : l_0 \leq l \leq n \}, \\
	C_0 &= \max\{ C^{(l)}_0(n,m,\varphi^{(0)}) : l_0 \leq l \leq n \} \text{ if } 1 \leq j \leq n-l_0, \\
	C_j &= \max\{ C^{(l)}_j(n,m,\varphi^{(0)},\vartheta_0,\vartheta_1,\ldots,\vartheta_{j-1}) : l_0 \leq l \leq n \} \text{ if } 1 \leq j \leq n-l_0. 
\end{align*}
If $\varphi \in \Phi_{\alpha}$ and $u \in \mathcal{F}_{\alpha}$ such that Hypothesis \ref{cft1_hyp}(i) holds true and $\mathcal{N}_u(0) \geq \alpha$, then Hypothesis \ref{cft1_hyp}(i) holds true with $\varepsilon = \varepsilon^{(l)}(n,m,\varphi^{(0)},\vartheta_0,\vartheta_1,\ldots,\vartheta_{l-l_0})$ and thus the conclusion of Lemma \ref{excessdecay3_lemma} follows from Lemma \ref{excessdecay2_lemma}. 
\end{proof}

\begin{proof}[Proof of Theorem \ref{constfreqthm1}]
Let $\varphi^{(0)}$ be any blow-up of $u$ at the origin and recall that $\varphi^{(0)} \in \Phi_{\alpha}$.  Let $l = n-\dim S(\varphi^{(0)})$ and without loss of generality assume $n-l$ the minimal dimension of the spine of a blow-up at point $Y \in B_{1/2}(0) \cap \mathcal{B}_u$ with $\mathcal{N}_u(Y) \geq \alpha$.  Choose $\vartheta_j \in (0,1/4)$ for $j = 0,1,2,\ldots,n-l_0$ such that $8\vartheta_j < \vartheta_{j+1}$ for $j = 0,1,2,\ldots,n-l_0-1$ and $C_j \vartheta_j \leq 1$, where $C_0 =C_0(n,m,\varphi^{(0)})$ and $C_j = C_j(n,m,\varphi^{(0)},\vartheta_0,\vartheta_1,\ldots,\vartheta_{j-1})$ as in Lemma \ref{excessdecay3_lemma}.  By scaling, assume that 
\begin{equation*}
	\int_{B_1(0)} \mathcal{G}(u,\varphi^{(0)})^2 < 2^{-n-2\alpha} \varepsilon 
\end{equation*}
for $\varepsilon \in (0,1)$ to be determined.  Observe that then 
\begin{equation} \label{cft1_eqn1}
	2^{n+2\alpha} \int_{B_{1/2}(Z)} \mathcal{G}(u(X),\varphi^{(0)}(X-Z))^2 dX < \varepsilon 
\end{equation}
for every $Z \in B_{1/2}(0) \cap \mathcal{B}_u$.  

We claim that by iteratively applying Lemma \ref{excessdecay3_lemma}, for each $Z \in B_{1/2}(0) \cap \mathcal{B}_u$ we can find a sequence $\{\psi_k\}_{k=0,1,2,\ldots} \subset \Phi_{\alpha}$ and a sequence $1/2 = \sigma_0 > \sigma_1 > \sigma_2 > \cdots$ such that $\psi_0 = \varphi^{(0)}$ and whenever $k = 0,1,2,\ldots$, 
\begin{equation} \label{cft1_eqn2}
	\int_{B_1(0)} \mathcal{G}(\psi_k,\psi_{k+1})^2 \leq 3 \sigma_k^{-n-2\alpha} \int_{B_{\sigma_k}(Z)} \mathcal{G}(u(X),\psi_k(X-Z))^2 dX
\end{equation}
and for some $j_k \in \{0,1,2,\ldots,n-l_0\}$, $\sigma_{k+1} = \vartheta_{j_k} \sigma_k$ and 
\begin{equation} \label{cft1_eqn3}
	\sigma_{k+1}^{-n-2\alpha} \int_{B_{\sigma_{k+1}}(Z)} \mathcal{G}(u(X),\psi_{k+1}(X-Z))^2 dX 
	\leq \vartheta_{j_k} \sigma_k^{-n-2\alpha} \int_{B_{\sigma_k}(Z)} \mathcal{G}(u(X),\psi_k(X-Z))^2 dX. 
\end{equation}
Suppose that we found such $\psi_0,\psi_1,\ldots,\psi_k$ for some $k \geq 0$.  By (\ref{cft1_eqn1}) and (\ref{cft1_eqn3}), 
\begin{equation} \label{cft1_eqn4}
	\sigma_k^{-n-2\alpha} \int_{B_{\sigma_k}(Z)} \mathcal{G}(u(X),\psi_k(X-Z))^2 dX 
	\leq 2^{n+2\alpha+1} \sigma_k \int_{B_{1/2}(Z)} \mathcal{G}(u(X),\varphi^{(0)}(X-Z))^2 dX 
	< 2\sigma_k \varepsilon,  
\end{equation}
so by (\ref{cft1_eqn2}) and (\ref{cft1_eqn4}) and the fact that $\vartheta_j < 1/4$, 
\begin{align} \label{cft1_eqn5}
	\left( \int_{B_1(0)} \mathcal{G}(\psi_k,\varphi^{(0)})^2 \right)^{1/2}
	&\leq \sum_{i=0}^{k-1} \left( \int_{B_1(0)} \mathcal{G}(\psi_{i+1},\psi_i)^2 \right)^{1/2} \nonumber \\
	&\leq \sum_{i=0}^{k-1} \left( 3 \sigma_i^{-n-2\alpha} \int_{B_{\sigma_i}(Z)} \mathcal{G}(u(X),\psi_i(X-Z))^2 dX \right)^{1/2} \nonumber \\
	&\leq \sum_{i=0}^{k-1} (6 \sigma_i \varepsilon)^{1/2} 
	\leq \sum_{i=0}^{k-1} (3 \cdot 4^{-i} \varepsilon)^{1/2} 
	\leq 2 (3 \varepsilon)^{1/2}.  
\end{align}
Thus if $\varepsilon$ is sufficiently small then by Lemma \ref{excessdecay3_lemma} $\psi_k$ and $j_k$ exist. 

Computing like in (\ref{cft1_eqn5}) using the fact that $\vartheta_1 = \max_j \vartheta_j < 4$, 
\begin{align} \label{cft1_eqn6}
	\left( \int_{B_1(0)} \mathcal{G}(\psi_k,\psi_K)^2 \right)^{1/2}
	&\leq \sum_{i=k}^{K-1} \left( \int_{B_1(0)} \mathcal{G}(\psi_{i+1},\psi_i)^2 \right)^{1/2} \nonumber \\
	\leq \sum_{i=k}^{K-1} (3 \vartheta_1^{i-k} \sigma_k \varepsilon)^{1/2} 
	\leq 2 (3 \sigma_k \varepsilon)^{1/2} 
\end{align}
for $0 \leq k < K$.  By (\ref{cft1_eqn6}), $\{\psi_k\}_{k=0,1,2,\ldots}$ is a Cauchy sequence in $L^2(B_1(0);\mathcal{A}_2(\mathbb{R}^m))$ and thus converges to some $\varphi_Z \in \Phi_{\alpha}$ such that $\varphi_Z$ is the unique blow-up to $u$ at $Z$.  Clearly $\dim S(\varphi_Z) \leq \dim S(\varphi^{(0)}) = n-l$, so by the minimality of $l$, $\dim S(\varphi_Z) = n-l$.  By letting $K \rightarrow \infty$ in (\ref{cft1_eqn6}), 
\begin{equation} \label{cft1_eqn7}
	\int_{B_1(0)} \mathcal{G}(\psi_k,\varphi_Z)^2 \leq 12 \sigma_k \varepsilon .
\end{equation}
By (\ref{cft1_eqn4}) and (\ref{cft1_eqn7}), 
\begin{equation} \label{cft1_eqn8}
	\sigma_k^{-n-2\alpha} \int_{B_{\sigma_k}(Z)} \mathcal{G}(u(X),\varphi_Z(X-Z))^2 dX \leq 48 \sigma_k \varepsilon 
\end{equation}
for $k = 0,1,2,\ldots$.  Now for $\rho \in (0,1/2]$, by choosing an integer $k \geq 1$ such that $\sigma_k < \rho \leq \sigma_{k-1}$, (\ref{cft1_eqn8}) yields 
\begin{equation} \label{cft1_eqn9}
	\rho^{-n-2\alpha} \int_{B_{\rho}(Z)} \mathcal{G}(u(X),\varphi_Z(X-Z))^2 dX \leq C \rho \varepsilon 
\end{equation}
for some constant $C = C(n,m,\varphi^{(0)}) \in (0,\infty)$.

By (\ref{cft1_eqn9}), Corollary \ref{branchdist_cor}, and $\dim S(\varphi_X) = n-l$ for all $X \in B_{1/2}(0) \cap \mathcal{B}_u$, 
\begin{equation} \label{cft1_eqn11}
	\op{dist}(Z,Y + S(\varphi_Y)) \leq C\rho^{3/2} \text{ for all } Z \in \mathcal{B}_u \cap B_{\rho}(Y) 
\end{equation}
for all $Y \in B_{1/2}(0) \cap \mathcal{B}_u$ and $\rho \in (0,1/2]$ for some $C = C(n,m,\varphi^{(0)}) \in (0,\infty)$.  

Suppose $Y,Z \in B_{1/2}(0) \cap \mathcal{B}_u$.  Let $\rho = 2|Y-Z|$.  Repeating this argument with $\rho^{-\alpha} u(Y + \rho X)$ in place of $u$ and $\varphi_Y$ in place of $\varphi^{(0)}$, by the uniqueness of blow-ups, (\ref{cft1_eqn7}) with $k = 0$, and (\ref{cft1_eqn9}) the same argument yields 
\begin{equation} \label{cft1_eqn10}
	\int_{B_1(0)} \mathcal{G}(\varphi_Y,\varphi_Z)^2 \leq C \varepsilon |X-Y|^2
\end{equation}
for some $C = C(n,m,\varphi^{(0)}) \in (0,\infty)$ and thus 
\begin{equation} \label{cft1_eqn12}
	|\op{pr}_{S(\varphi_Y)} - \op{pr}_{S(\varphi_Z)}| \leq C \varepsilon |X-Y| 
\end{equation}
for all $Y,Z \in B_{1/2}(0) \cap \mathcal{B}_u$ for some $C = C(n,m,\varphi^{(0)}) \in (0,\infty)$, where $\op{pr}_S$ denotes the orthogonal projection onto a given subspace $S$.  After an orthogonal change of coordinates, suppose that $S(\varphi^{(0)}) = \{0\} \times \mathbb{R}^{n-l}$.  By (\ref{cft1_eqn11}) and (\ref{cft1_eqn12}), provided $\varepsilon$ is sufficiently small, $\mathcal{B}_u$ is contained in the graph $\Gamma$ of of a Lipschitz function over $S(\varphi^{(0)}) = \{0\} \times \mathbb{R}^{n-l}$ with small gradient.  If $l \geq 3$, then $\mathcal{B}_u$ has Hausdorff dimension at most $n-3$ and thus $B_{1/4}(0) \setminus \mathcal{B}_u$ is simply connected by the appendix of~\cite{SimWic11}.  Consequently $u$ decomposes into two harmonic single-valued functions on $B_{1/4}(0)$, contradicting $0 \in \mathcal{B}_u$.  If $l = 2$, then by arguing like in the last paragraph of the proof of Theorem \ref{constfreqthm2} for case (a), we must have that $B_{1/4}(0) \cap \mathcal{B}_u = B_{1/4}(0) \cap \Gamma$.  Moreover, by (\ref{cft1_eqn9}), standard elliptic estimates, and Lemma \ref{twovalL2_lemma}, $u$ is asymptotic to $\varphi_Z$ at each $Z \in \Gamma \cap B_{1/8}(0)$.  By the asymptotic behavior of $u$ and the structure of the blow-ups $\varphi_Z$, $u$ decomposes into two harmonic single-valued functions on $B_{1/8}(0)$, contradicting $0 \in \mathcal{B}_u$. 
\end{proof}

\section{Proof of Theorem \ref{constfreqthm1_mss}} \label{sec:cft1_mss_sec}

In this section we will extend the proof of Theorem \ref{constfreqthm1} to the context of area-stationary two-valued graphs in order to prove Theorem \ref{constfreqthm1_mss}.  Our approach is to modify the blow-up method in a way similar to~\cite{KrumWic2} in order to extend Lemma \ref{excessdecay1_lemma} as follows.

Fix an integer $\alpha \geq 1$.  In place of $\mathcal{F}_{\alpha}$ we use the set $\mathcal{F}^{\text{MSS}}_{\alpha}$ of all pairs $(u,\Lambda)$ of two-valued functions $u \in C^{1,1/2}(B_1(0);\mathcal{A}_2(\mathbb{R}^m))$ whose graph $\mathcal{M}$ is area-stationary such that $\mathcal{N}_{\mathcal{M}}(p) \geq \alpha$ for all $p \in \op{sing} \mathcal{M}$ and $\Lambda \in (0,\infty)$.  We replace Hypothesis \ref{cft1_hyp} with: 

\begin{hypothesis} \label{cft1_hyp_mss}
Let $\varepsilon, \beta > 0$.  Let $\varphi^{(0)}, \varphi \in \Phi_{\alpha}$ and $(u,\Lambda) \in \mathcal{F}^{\text{MSS}}_{\alpha}$ such that $\|u\|_{C^{0,1/2}(B_2(0))} \leq \varepsilon$ and 
\begin{enumerate}
\item[(i)] $\int_{B_1(0)} \mathcal{G}(\varphi,\varphi^{(0)})^2 < \varepsilon$ and $\int_{B_1(0)} \mathcal{G}(u_s/\Lambda,\varphi^{(0)})^2 < \varepsilon$ and 
\item[(ii)] either $\dim S(\varphi) = \dim S(\varphi^{(0)})$ or 
\begin{equation*}
	\int_{B_1(0)} \mathcal{G}(u_s/\Lambda,\varphi)^2 + \|Du\|_{C^{0,1/4}(B_1(0))}^4 
	\leq \beta \inf_{\varphi' \in \Phi_{\alpha}, \, S(\varphi) \subset S(\varphi')} 
	\left( \int_{B_1(0)} \mathcal{G}(u_s/\Lambda,\varphi')^2 + \|Du\|_{C^{0,1/4}(B_1(0))}^2 \right) , 
\end{equation*}
where $u_s(X) = \{ \pm (u_1(X)-u_2(X))/2 \}$. 
\end{enumerate}
\end{hypothesis}

Observe that Hypothesis \ref{cft1_hyp_mss}(ii) implies that 
\begin{align} \label{aprioriest_mss_eqn7} 
	\int_{B_1(0)} \mathcal{G}(u_s/\Lambda,\varphi)^2 &\leq \frac{\beta}{1-\beta} \inf_{\varphi' \in \Phi_{\alpha}, \, S(\varphi) \subset S(\varphi')} 
		\int_{B_1(0)} \mathcal{G}(u_s/\Lambda,\varphi')^2, \nonumber \\
	\|Du\|_{C^{0,1/4}(B_1(0))}^2 &\leq \frac{\beta}{1-\beta} \inf_{\varphi' \in \Phi_{\alpha}, \, S(\varphi) \subset S(\varphi')} 
		\int_{B_1(0)} \mathcal{G}(u_s/\Lambda,\varphi')^2.
\end{align}

\begin{lemma} \label{excessdecay1_lemma_mss}
Let $\vartheta \in (0,1/4)$ and let $\varphi^{(0)} \in \Phi_{\alpha}$.  Then there are $\overline{\varepsilon} = \overline{\varepsilon}(n,m,\varphi^{(0)},\vartheta) \in (0,1)$ and $\overline{\beta} = \overline{\beta}(n,m,\varphi^{(0)},\vartheta) \in (0,1)$ such that if $\varphi \in \Phi_{\alpha}$ and $(u,\Lambda) \in \mathcal{F}^{\text{MSS}}_{\alpha}$ such that Hypothesis \ref{cft1_hyp_mss} holds true with $\varepsilon = \overline{\varepsilon}$ and $\beta = \overline{\beta}$, $u(0) = 0$, $Du(0) = 0$, and $\mathcal{N}_{\mathcal{M}}(0) \geq \alpha$, then there exists $\widetilde{\varphi} \in \Phi_{\alpha}$ with 
\begin{equation*}
	\int_{B_1(0)} \mathcal{G}(\widetilde{\varphi},\varphi)^2 \leq 3 \left( \int_{B_1(0)} \mathcal{G}(u_s/\Lambda,\varphi)^2 + \|Du\|_{C^{0,1/4}(B_1(0))}^2 \right)
\end{equation*}
and 
\begin{equation*}
	\vartheta^{-n-2\alpha} \int_{B_{\vartheta}(0)} \mathcal{G}(u_s/\Lambda,\widetilde{\varphi})^2 + \sup_{B_{\vartheta}(0)} |Du| 
		+ \vartheta^{1/2} [Du]_{1/4,B_{\vartheta}(0)} 
	\leq \overline{C} \vartheta^{1/4} \left( \int_{B_1(0)} \mathcal{G}(u_s/\Lambda,\varphi)^2 + \|Du\|_{C^{0,1/4}(B_1(0))} \right)
\end{equation*}
for some $\overline{C} = \overline{C}(n,m,\varphi^{(0)}) \in [1,\infty)$.  
\end{lemma}

Upon proving Lemma \ref{excessdecay1_lemma_mss}, we can drop Hypothesis \ref{cft1_hyp_mss}(ii) using a multiple scales argument like we did in Section \ref{sec:excessdecay_sec}.  We then extend the proof of Theorem \ref{constfreqthm1} to show that 
\begin{equation*}
	\rho^{-n-2\alpha} \int_{B_{\rho}(0)} \mathcal{G} \left( \frac{\tilde u_{p,s}}{\|u_s\|_{L^2(B_1(0))}}, \varphi_p \right)^2 
	\leq C(n,m) \varepsilon \rho^{2\tau}
\end{equation*}
for every $p \in B_{1/2}(0) \times \mathbb{R}^m \cap \op{sing} \mathcal{M}$ for some $\tau = \tau(n,m) \in (0,1)$ and a unique homogeneous degree $\alpha$, harmonic two-valued function $\varphi_p : T_p \mathcal{M} \rightarrow T_p \mathcal{M}^{\perp}$, where recall that we write $\mathcal{M}$ as the graph of $\tilde u_p(X) = \{\tilde u_{p,1}(X),\tilde u_{p,2}(X)\}$ over the tangent plane to $\mathcal{M}$ at $p$ and let $\tilde u_{p,s}(X) = \{ \pm (\tilde u_{p,1}(X) - \tilde u_{p,2}(X))/2 \}$.  Then $\op{sing} \mathcal{M}$ is contained in a graph over an $(n-l)$-dimensional subspace near the origin, where $n-l$ is the dimension of the spines of each $\varphi_p$, and one can argue that $\op{sing} \mathcal{M}$ equals the graph and has $C^{1,\tau}$ regularity.  Using the fact that $\tilde u_p$ is asymptotic to $\varphi_p$ at each $p \in \op{sing} \mathcal{M}$ near the origin, we can then show that $\mathcal{M}$ is the union of two smoothly embedded $n$-dimensional submanifolds near the origin, contradicting $0 \in \op{sing} \mathcal{M}$.

To prove Lemma \ref{excessdecay1_lemma_mss}, we extend the statement of the a priori estimates of Section~\ref{sec:L2estimates_sec} as follows.  

\begin{lemma} \label{graphrep_lemma_mss} 
Suppose $\varphi^{(0)} \in \Phi_{\alpha}$ with $\dim S(\varphi^{(0)}) = n-l_0$ for some integer $l_0 \in \{2,3,\ldots,n\}$.  Let $l \in \{l_0,l_0+1,\ldots,n\}$.  For every $\gamma,\nu \in (0,1)$ there exists $\varepsilon_0, \beta_0 > 0$ depending on $n$, $l$, $m$, $\varphi^{(0)}$, $\gamma$, and $\nu$ such that the following holds true.  Let $\varphi \in \Phi_{\alpha}$ and after an orthogonal change of coordinates suppose that $S(\varphi) = \{0\} \times \mathbb{R}^{n-l}$.  (Note that we need not assume that $S(\varphi) \subseteq S(\varphi^{(0)})$.)  Let $u \in C^{1,1/2}(A^l_{1,\gamma}(0);\mathcal{A}_2(\mathbb{R}^m))$ be a two-valued function whose graph is area-stationary and $\Lambda \in (0,\infty)$.  Suppose that 
\begin{equation*}
	\|u\|_{C^{1,1/2}(A^l_{1,\gamma}(0))} \leq \varepsilon_0, \quad
	\int_{A^l_{1,\gamma}(0)} \mathcal{G}(\varphi,\varphi^{(0)})^2 < \varepsilon_0, \quad 
	\int_{A^l_{1,\gamma}(0)} \mathcal{G}(u_s/\Lambda,\varphi^{(0)})^2 < \varepsilon_0, 
\end{equation*} 
and either $l = l_0$ or 
\begin{equation*}
	\int_{A^l_{1,\gamma}(0)} \mathcal{G}(u_s/\Lambda,\varphi)^2 + \|u\|_{C^{1,1/4}(A^l_{1,\gamma}(0))}
	\leq \beta_0 \left( \inf_{\varphi' \in \Phi_{\alpha}, \, S(\varphi) \subset S(\varphi')} \int_{A^l_{1,\gamma}(0)} \mathcal{G}(u_s/\Lambda,\varphi')^2 
		+ \|u\|_{C^{1,1/4}(A^l_{1,\gamma}(0))} \right) . 
\end{equation*}
Then there is a harmonic single-valued function $v : A^l_{1,\gamma/2}(0) \rightarrow \mathbb{R}^m$ such that 
\begin{equation*}
	u_s(X)/\Lambda = \{ \varphi_1(X) + v(X), -\varphi_1(X) - v(X) \}
\end{equation*}
for all $X \in A^l_{1,\gamma/2}(0)$ and
\begin{gather*}
	\sup_{A^l_{1,\gamma/2}(0)} (|v| + |Dv|) \leq \nu, \\
	\int_{A^l_{1,\gamma/2}(0)} (|v|^2 + |Dv|^2) \leq C \left( \int_{A^l_{1,\gamma}(0)} \mathcal{G}(u_s/\Lambda,\varphi)^2
		+ \Lambda^{-2} \int_{A^l_{1,\gamma}(0)} |f|^2 \right) , 
\end{gather*}
for some $C = C(n,l,m,\varphi^{(0)},\gamma) \in (0,\infty)$, where $f_{\kappa} = D_i \left( (A^{ij}_{\kappa \lambda}(D\tilde u_{p,a},D\tilde u_{p,s}) - \delta_{ij} \delta_{\kappa \lambda}) D_j \tilde u_{p,s}^{\lambda} \right)$. 
\end{lemma}

\begin{lemma} \label{keyest_lemma_mss} 
Suppose $\varphi^{(0)} \in \Phi_{\alpha}$ with $\dim S(\varphi^{(0)}) = n-l_0$ for some integer $l_0 \in \{2,3,\ldots,n\}$.  Let $l \in \{l_0,l_0+1,\ldots,n\}$.  For every $\gamma \in (0,1)$ there exists $\varepsilon_0, \beta_0 > 0$ depending on $n$, $l$, $m$, $\varphi^{(0)}$, and $\gamma$ such that if $\varphi \in \Phi_{\alpha}$ and $(u,\Lambda) \in \mathcal{F}^{\text{MSS}}_{\alpha}$ such that $\dim S(\varphi) = n-l$, Hypothesis \ref{cft1_hyp_mss} holds true with $\varepsilon = \varepsilon_0$ and $\beta = \beta_0$, $u(0) = 0$, $Du(0) = 0$, and $\mathcal{N}_{\mathcal{M}}(0) \geq \alpha$, then 
\begin{equation*}
	\Lambda^{-2} \int_{B_{\gamma}(0)} |\nabla_{S(\varphi)} u|^2 
	+ \Lambda^{-2} \int_{B_{\gamma}(0)} R^{2-n} \left( \frac{\partial (u/R^{\alpha})}{\partial R} \right)^2 
	\leq C \left( \int_{B_1(0)} \mathcal{G}(u_s/\Lambda,\varphi)^2 + \|Du\|_{C^{0,1/4}(B_1(0))}^2 \right)
\end{equation*} 
where $\nabla_{S(\varphi)}$ denotes the tangential derivative along $S(\varphi)$, $R(X) = |X|$, and $C = C(n,l,m,\varphi^{(0)},\gamma) \in (0,\infty)$.
\end{lemma}

\begin{cor} \label{radialdecay_cor_mss} 
Suppose $\varphi^{(0)} \in \Phi_{\alpha}$ with $\dim S(\varphi^{(0)}) = n-l_0$ for some integer $l_0 \in \{2,3,\ldots,n\}$.  Let $l \in \{l_0,l_0+1,\ldots,n\}$.  For every $\gamma,\sigma \in (0,1)$ there exists $\varepsilon_0, \beta_0 > 0$ depending on $n$, $l$, $m$, $\varphi^{(0)}$, $\gamma$, and $\sigma$ such that if $\varphi \in \Phi_{\alpha}$ and $(u,\Lambda) \in \mathcal{F}^{\text{MSS}}_{\alpha}$ such that $\dim S(\varphi) = n-l$, Hypothesis \ref{cft1_hyp_mss} holds true with $\varepsilon = \varepsilon_0$ and $\beta = \beta_0$, and $\mathcal{N}_{\mathcal{M}}(0) \geq \alpha$, then 
\begin{equation} \label{aprioriest_harm_eqn4} 
	\int_{B_{\gamma}(0)} \frac{\mathcal{G}(u,\varphi)^2}{|X|^{n+2\alpha-\sigma}} \leq C \int_{B_1(0)} \mathcal{G}(u,\varphi)^2 
\end{equation} 
for some $C = C(n,l,m,\varphi^{(0)},\gamma) \in (0,\infty)$.
\end{cor}

\begin{cor} \label{nonconest_cor_mss} 
Suppose $\varphi^{(0)} \in \Phi_{\alpha}$ with $\dim S(\varphi^{(0)}) = \{0\} \times \mathbb{R}^{n-l_0}$ for some integer $l_0 \in \{2,3,\ldots,n\}$.  Let $l \in \{l_0,l_0+1,\ldots,n\}$.  For every $\gamma \in (0,1)$ there exists $\varepsilon_0, \beta_0 > 0$ depending on $n$, $l$, $m$, $\varphi^{(0)}$, and $\gamma$ such that the following holds true.  Let $\varphi \in \Phi_{\alpha}$ and $(u,\Lambda) \in \mathcal{F}^{\text{MSS}}_{\alpha}$ such that Hypothesis \ref{cft1_hyp_mss} holds true with $\varepsilon = \varepsilon_0$ and $\beta = \beta_0$.  Let $\{\varphi_j\}_{j = 0,1,2,\ldots,N} \subset \Phi_{\alpha}$ be an optimal sequence with increasing spines such that $\varphi_0 = \varphi$ and $S(\varphi_N) = \{0\} \times \mathbb{R}^{n-l_0}$.  Write $\varphi_j(X) = \{ \pm \varphi_{j,1}(X) \}$ for all $X \in \mathbb{R}^n$ and some harmonic single-valued function $\varphi_{j,1} : \mathbb{R}^n \rightarrow \mathbb{R}^m$ that is close to $\varphi^{(0)}_1$ in $L^2(B_1(0);\mathbb{R}^m)$.  Let $v : B_{\gamma}(0) \setminus \{|x| > \tau\} \rightarrow \mathbb{R}^m$ be a single-valued such that 
\begin{equation*}
	u_s(X)/\Lambda = \{ \varphi_{0,1}(X) + v(X), -\varphi_{0,1}(X) - v(X) \}
\end{equation*}
on $B_{\gamma}(0) \setminus \{|x| > \tau\}$ and $\sup |v|$ is small.  Let $Z \in B_{1/2}(0)$ with $\mathcal{N}_{\mathcal{M}}(Z,u(Z)) \geq \alpha$.  Then 
\begin{gather*}
	\op{dist}(Z,S(\varphi_j))^2 \leq C \frac{\int_{B_1(0)} \mathcal{G}(u_s/\Lambda,\varphi)^2 + \|Du\|_{C^{0,1/4}(B_1(0))}^2}{
		\int_{B_1(0)} \mathcal{G}(\varphi,\varphi_{j+1})^2} \text{ for } j = 0,1,2,\ldots,N-1, \\
	\op{dist}(Z,S(\varphi^{(0)}))^2 \leq C \left( \int_{B_1(0)} \mathcal{G}(u_s/\Lambda,\varphi)^2 + \|Du\|_{C^{0,1/4}(B_1(0))}^2 \right), 
\end{gather*} 
for some $C = C(n,l,m,\varphi^{(0)},\gamma) \in (0,\infty)$ and 
\begin{gather*}
	\int_{B_{\gamma}(0)} \frac{\mathcal{G}(u_s/\Lambda,\varphi)^2}{|X-Z|^{n-\sigma}} 
		\leq C \left( \int_{B_1(0)} \mathcal{G}(u_s/\Lambda,\varphi)^2 + \|Du\|_{C^{0,1/4}(B_1(0))}^2 \right), \\
	\int_{B_{\gamma}(0) \cap \{|x| > \tau\}} \frac{|v(X) - D\varphi{0,1}(X) \cdot Z|^2}{|X-Z|^{2\alpha+n-\sigma}} 
		\leq C \left( \int_{B_1(0)} \mathcal{G}(u_s/\Lambda,\varphi)^2 + \|Du\|_{C^{0,1/4}(B_1(0))}^2 \right),  
\end{gather*} 
for some $C = C(n,l,m,\varphi^{(0)},\gamma,\sigma) \in (0,\infty)$
\end{cor}

We modify the definition of blow-ups of Section~\ref{sec:L2estimates_sec} as follows.  Suppose $\varepsilon_j, \beta_j \rightarrow 0^+$ and $\varphi^{(0)}, \varphi_j \in \Phi_{\alpha}$ and $(u_j,\Lambda_j) \in \mathcal{F}^{\text{MSS}}_{\alpha}$ such that Hypothesis \ref{cft1_hyp_mss} holds true with $\varepsilon = \varepsilon_j$, $\beta = \beta_j$, $\varphi = \varphi_j$, and $(u,\Lambda) = (u_j,\Lambda_j)$.  Writing $u_j(X) = \{u_{j,1}(X),u_{j,2}(X)\}$, let $u_{j,a}(X) = (u_{j,1}(X)+u_{j,2}(X))/2$ and $u_{j,s}(X) = \{ \pm (u_{j,1}(X)-u_{j,2}(X))/2 \}$.  Write $\varphi_j(X) = \{ \pm \varphi_{j,1}(X) \}$ for a harmonic single-valued function $\varphi_{j,1} : \mathbb{R}^n \rightarrow \mathbb{R}^m$ that is close to $\varphi^{(0)}_1$ in $L^2(B_1(0);\mathbb{R}^m)$.  After passing to a subsequence, we may assume that $l = n-\dim S(\varphi_j)$ is independent of $j$.  By the sequential compactness of the space of closed subsets of a compact space equipped with the Hausdorff metric, after passing to a subsequence $\mathcal{B}_{u_j}$ converges to some closed subset $\mathcal{D} \subseteq \{0\} \times B^{n-l}_1(0)$ in Hausdorff distance on compact subsets of $B_1(0)$.  Write 
\begin{equation*}
	u_{j,s}(X)/\Lambda = \{ \pm (\varphi_{j,1}(X) + v_j(X)) \} 
\end{equation*}
for all $X \in U_j$, where $U_j = \{X \in B_{(1+\gamma)/2}(0) : \op{dist}(X,\mathcal{D}) > \tau_j\}$ for $\tau_j \rightarrow 0^+$ slowly and $v_j \in C^2(U_j;\mathbb{R}^m)$.  Define 
\begin{equation*}
	w_j = v_j/E_j \quad \text{for} \quad 
	E_j = \left( \int_{B_1(0)} \mathcal{G}(u_{s,j}/\Lambda_j,\varphi_j)^2 + \|Du_j\|_{C^{0,1/4}(B_1(0))}^2 \right)^{1/2}. 
\end{equation*}
Note that $w_j$ satisfies $\Delta w_j^{\kappa} = f_{j,\kappa}$ where $f_{j,\kappa} = D_i ((A^{il}_{\kappa \lambda}(Du_{j,a},Du_{j,s}) - \delta_{il} \delta_{\kappa \lambda}) D_l u_{j,s}^{\lambda})$.  By elliptic estimates for $w_j$ and Lemma \ref{twovalL2_lemma}, after passing to a subsequence $w_j$ converge to some $w$ in $C^2(K;\mathbb{R}^m)$ for every compact subset $K$ of $B_1(0) \setminus \mathcal{D}$.  Using Lemma \ref{nonconest_cor_mss}, one can show that $w_j \rightarrow w$ in $L^2(B_{\gamma}(0);\mathbb{R}^m)$.  Any such $w$ is called a \textit{blow-up} of $(u_j,\Lambda_j)$ relative to $\varphi_j$ over $B_{\gamma}(0)$. 

As before, the proofs of the a priori estimates Lemma \ref{graphrep_lemma_mss}, Lemma \ref{keyest_lemma_mss}, Corollary \ref{radialdecay_cor_mss}, and Corollary \ref{nonconest_cor_mss} proceed by induction on $l = n - \dim S(\varphi)$, assuming that for some $\lambda \in \{l_0,l_0+1,\ldots,n\}$ either $\lambda = l_0$ or the a priori estimates hold true whenever $l_0 \leq l < \lambda$ and then proving the a priori estimates for $l = \lambda$.  The proofs of Lemma \ref{graphrep_lemma}, Lemma \ref{translate_lemma}, Corollary \ref{branchdist_cor}, and Corollary \ref{nonconest_cor} in Section~\ref{sec:proof_est_sec} extend by using (\ref{aprioriest_mss_eqn7}) and making obvious modifications to account for the changes in the statements of the a priori estimates and definitions of blow-ups.  The $W^{1,2}$ estimate in Lemma \ref{graphrep_lemma} changes since $\Delta v = f$ by (\ref{mss3}), where $f$ is as in (\ref{mss3}).  The proof of Corollary \ref{radialdecay_cor_mss} is identical to a proof of~\cite{KrumWic2} and involves extending the proof of Corollary \ref{radialdecay_cor} in Section~\ref{sec:proof_est_sec}, accounting for the changes in the statements of the a priori estimates, to prove the corollary in the special case where $u(0) = 0$ and $Du(0) = 0$ and then proving the general case by rotating the tangent plane at the origin, picking up a term $|Du(0)|$ on the right-hand side which is absorbed into $\|Du\|_{C^{0,1/4}(B_1(0))}$.  The proof of Lemma \ref{keyest_lemma_mss} proceeds as follows: 

\begin{proof}[Proof of Lemma \ref{keyest_lemma_mss} for $l = \lambda$]  
In place of (6.6) of~\cite{KrumWic1} we observe that by~\cite{KrumWic2}, 
\begin{equation} \label{keyest_mss_eqn1}
	\int_{B_{\gamma}(0)} R^{2-n} \left( \frac{\partial (u_s/R^{\alpha})}{\partial R} \right)^2 
	\leq C \int (r^2 |Du_s|^2 \psi(R)^2 + 2 \alpha R^{-1} |u_s|^2 \psi(R) \psi'(R)) + C \int_{B_{(1+\gamma)/2}(0)} R^{4-n+2\alpha} |f|^2 
\end{equation}
for $C = C(n,m,\varphi^{(0)},\gamma) \in (0,\infty)$, where $f_{\kappa} = D_i ((A^{ij}_{\kappa \lambda}(D\tilde u_{p,a},D\tilde u_{p,s}) - \delta_{ij} \delta_{\kappa \lambda}) D_j \tilde u_{p,s}^{\lambda})$.  We extend (\ref{keyest_eqn8}) using (\ref{mss3}) like in~\cite{KrumWic2} to 
\begin{align} \label{keyest_mss_eqn2}
	&\frac{1}{2} \int |D_y u_s|^2 \psi(R)^2 + (\alpha+\lambda/2-1) \int \left( |Du_s|^2 \psi(R)^2 + 2\alpha R^{-1} |u_s|^2 \psi(R) \psi'(R) \right) \nonumber \\
	&\leq -2 \int \left( \frac{1}{2} r^2 |Du_s|^2 - \alpha (\alpha+\lambda/2-1) |u_s|^2 - r D_r u_s (r D_r u_s - \alpha u_s) \right) R^{-1} \psi(R) \psi'(R) 
	\nonumber \\&+ 2 \int |r D_r u_s - \alpha u_s|^2 \psi'(R)^2 + \frac{1}{4} \int |f|^2 \psi(R)^2. 
\end{align}
In Cases I and II, we extend (\ref{keyest_eqn9}) using (\ref{mss3}) in obtaining $W^{1,2}$ estimates and thereby obtain 
\begin{align} \label{keyest_mss_eqn3}
	&-2 \Lambda^{-2} \int \left( \frac{1}{2} r^2 |Du_s|^2 - \alpha (\alpha+\lambda/2-1) |u_s|^2 - r D_r u_s (r D_r u_s - \alpha u_s) \right) 
	R^{-1} \psi(R) \psi'(R) \chi_{(\rho,\zeta)} \nonumber \\&+ 2 \Lambda^{-2} \int |r D_r u_s - \alpha u_s|^2 \psi'(R)^2 \chi_{(\rho,\zeta)} 
	\leq C \int_{A^{\lambda}_{\rho,1-\gamma}(\zeta)} \mathcal{G}(u_s/\Lambda,\varphi)^2 + C \int_{A^{\lambda}_{\rho,3(1-\gamma)/4}(\zeta)} r^4 |f|^2
\end{align}
for $C = C(n,\lambda,m,\varphi^{(0)},\gamma) \in (0,\infty)$.  In Case III, we cover $A^{\lambda}_{\rho,(1-\gamma)/2}(\zeta)$ by a cover $\mathcal{C}$ balls and annuli and find a subordinate partition of unity $\{\eta_S\}_{S \in \mathcal{C}}$ as before.  If $B_{d(0,\xi)/2}(0,\xi) \in \mathcal{C}$ with $d(0,\xi) \geq 8cE$ and $d(0,\xi)^{-n-2\alpha} \int_{B_{d(0,\xi)}(0,\xi)} \mathcal{G}(u,\widetilde{\varphi})^2 < \widetilde{\varepsilon}$, then by the Schauder estimates there exists a harmonic function $v : B_{d(0,\xi)/2}(0,\xi) \rightarrow \mathbb{R}^m$ such that 
\begin{gather}
	u(X) = \{ \varphi_1(X) + v(X), -\varphi_1(X) - v(X) \} \text{ for all } X \in B_{d(0,\xi)/2}(0,\xi), \nonumber \\
	\int_{B_{d(0,\xi)/2}(0,\xi)} (|v|^2 + s^2 |Dv|^2) \leq C \int_{B_{d(0,\xi)}(0,\xi)} \mathcal{G}(u,\widetilde{\varphi})^2 
		+ \int_{B_{d(0,\xi)}(0,\xi)} s^4 |f|^2 \label{keyest_mss_eqn4}
\end{gather}
for some $C = C(n,\lambda,m,\varphi^{(0)},\gamma) \in (0,\infty)$.  By Lemma \ref{translate_lemma} and Corollary \ref{radialdecay_cor}, 
\begin{equation} \label{keyest_mss_eqn5}
	\int_{B_{d(0,\xi)}(0,\xi)} \mathcal{G}(u,\widetilde{\varphi})^2 
	\leq C \left(\frac{d(0,\xi)}{\rho}\right)^{n+2\alpha-1/2} \left( \int_{A^{\lambda}_{\rho,1-\gamma}(\zeta)} \mathcal{G}(u,\widetilde{\varphi})^2 
		+ \rho^n \|Du\|_{C^{0,1/4}(A^{\lambda}_{\rho,1-\gamma}(\zeta))} \right)
\end{equation}
for some $C = C(n,\lambda,m,\varphi^{(0)},\gamma) \in (0,\infty)$.  By (\ref{keyest_mss_eqn4}), (\ref{keyest_mss_eqn5}), and $\Delta_{S^{\lambda-1}} \widetilde{\varphi} + \alpha (\alpha + \lambda - 2) \widetilde{\varphi} = 0$, 
\begin{align} \label{keyest_mss_eqn6}
	&-2 \int \left( \frac{1}{2} r^2 (|Du|^2 - |D\widetilde{\varphi}|^2) - \alpha (\alpha+\lambda/2-1) (|u|^2 - |\widetilde{\varphi}|^2) 
		- r D_r u (r D_r u - \alpha u) \right) R^{-1} \psi(R) \psi'(R) \chi \nonumber \\
	&+ 2 \int |r D_r u - \alpha u|^2 \psi'(R)^2 \chi \eta 
	\leq 2 \int v \nabla_{S^{\lambda-1}} \widetilde{\varphi} \cdot \nabla_{S^{\lambda-1}} \eta R^{-1} \psi(R) \psi'(R) \chi 
	+ C \int_{B_{d(0,\xi)}(0,\xi)} s^2 r^2 |f|^2 \nonumber \\
	&+ C \rho^{-n-2\alpha+5/2} \left( \int_{B_{d(0,\xi)/2}(0,\xi)} s^{2\alpha-5/2} \right) 
		\left( \int_{A^{\lambda}_{\rho,1-\gamma}(\zeta)} \mathcal{G}(u,\varphi)^2 + \rho^n \|Du\|_{C^{0,1/4}(A^{\lambda}_{\rho,1-\gamma}(\zeta))} \right) .  
\end{align}
where $\eta = \eta_{B_{d(0,\xi)/2}(0,\xi)}$ and $C = C(n,\lambda,m,\varphi^{(0)},\gamma) \in (0,\infty)$.  By similarly extending (\ref{keyest_eqn19}) - (\ref{keyest_eqn22}), (\ref{keyest_eqn24}) and adding the resulting estimates and using (\ref{keyest_eqn25}), 
\begin{align} \label{keyest_mss_eqn7}
	&-2 \int \left( \frac{1}{2} r^2 (|Du|^2 - |D\widetilde{\varphi}|^2) - \alpha (\alpha+\lambda/2-1) (|u|^2 - |\widetilde{\varphi}|^2) 
		- r D_r u (r D_r u - \alpha u) \right) R^{-1} \psi(R) \psi'(R) \chi \nonumber \\
	&+ 2 \int |r D_r u - \alpha u|^2 \psi'(R)^2 \chi \eta 
	\leq C \left( \int_{A^{\lambda}_{\rho,1-\gamma}(\zeta)} \mathcal{G}(u,\varphi)^2 + \rho^n \|Du\|_{C^{0,1/4}(A^{\lambda}_{\rho,1-\gamma}(\zeta))} \right) 
	+ C \sum_{S \in \mathcal{C}} \int_{\widehat{S}} r^4 |f|^2,  
\end{align}
where $\widehat{S} = B_{2\sigma}(0,\xi)$ when $S = B_{\sigma}(0,\xi) \in \mathcal{C}$, $\widehat{S} = A^l_{\sigma,1}(\xi)$ when $S = A^l_{\sigma,1/2}(\xi) \in \mathcal{C}$, and $\widehat{S} = A^{\lambda}_{\rho,3(1-\gamma)/4}(\zeta)$ for $S = A^{\lambda}_{\rho,(1-\gamma)/2}(\zeta) \cap \{ d \geq (1-\gamma) \rho/64 \}$.  Since any ball or annuli in $\mathcal{C}$ that intersect have roughly the same size and $\mathcal{C}$ is the union of $N \leq C(n)$ subcollections of mutually disjoint sets, $\{ \widehat{S} : S \in \mathcal{C} \}$ equals the union of $N' \leq C(n)$ subcollections of mutually disjoint sets and thus 
\begin{equation} \label{keyest_mss_eqn8}
	\sum_{S \in \mathcal{C}} \int_{\widehat{S}} r^4 |f|^2 \leq C(n) \int_{A^{\lambda}_{\rho,1-\gamma}(\zeta)} r^4 |f|^2
\end{equation}
By combining (\ref{keyest_mss_eqn1}), (\ref{keyest_mss_eqn2}), (\ref{keyest_mss_eqn3}), (\ref{keyest_mss_eqn7}), and (\ref{keyest_mss_eqn8})
\begin{align*}
	&\Lambda^{-2} \int_{B_{\gamma}(0)} |\nabla_{S(\varphi)} u|^2 
	+ \Lambda^{-2} \int_{B_{\gamma}(0)} R^{2-n} \left( \frac{\partial (u/R^{\alpha})}{\partial R} \right)^2 
	\\&\leq C \left( \int_{B_1(0)} \mathcal{G}(u_s/\Lambda,\varphi)^2 + \|Du\|_{C^{0,1/4}(B_1(0))}^2 
	+ C \Lambda^{-2} \int_{B_{(3+\gamma)/4}(0)} (1 + R^{4-n+2\alpha}) |f|^2 \right)
\end{align*} 
for $C = C(n,l,m,\varphi^{(0)},\gamma) \in (0,\infty)$.  We complete the proof of the extension of Lemma \ref{keyest_lemma} we observe that by~\cite{KrumWic2},  
\begin{equation*}
	\int_{B_{(3+\gamma)/4}(0)} (1 + R^{4-n+2\alpha}) |f|^2 \leq C \|Du\|_{C^{0,1/4}(B_1(0))}^2 \int_{B_1(0)} |u_s|^2 
\end{equation*} 
for $C = C(n,m,\gamma) \in (0,\infty)$. 
\end{proof} 

Having proved the a priori estimates above, the proof of Lemma \ref{excessdecay1_lemma_mss}] proceeds almost exactly like before with obvious changes except we define the blow-up $w$ of $(u_j,\Lambda_j)$ relative to $\varphi_j$ to be the limit of $v_j/E_j$ where 
\begin{equation*}
	E_j = \left( \int_{B_1(0)} \mathcal{G}(u_{s,j}/\Lambda_j,\varphi_j)^2 + \|Du_j\|_{C^{0,1/4}(B_1(0))} \right)^{1/2}
\end{equation*}
so that $w$ is harmonic and thereby show that 
\begin{equation*}
	\vartheta^{-n-2\alpha} \int_{B_{\vartheta}(0)} \mathcal{G}(u_s/\Lambda,\widetilde{\varphi})^2 
	\leq \overline{C} \vartheta^2 \left( \int_{B_1(0)} \mathcal{G}(u_s/\Lambda,\varphi)^2 + \|Du\|_{C^{0,1/4}(B_1(0))} \right)
\end{equation*}
for some $\overline{C} = \overline{C}(n,m,\varphi^{(0)}) \in [1,\infty)$.  Since $u(0) = 0$ and $Du(0) = 0$, 
\begin{equation*}
	\sup_{B_{\vartheta}(0)} |Du| + \vartheta^{1/4} [Du]_{1/4,B_{\vartheta}(0)} \leq \vartheta^{1/4} \|Du\|_{C^{0,1/4}(B_1(0))}. 
\end{equation*}

\section{Partial Legendre transformation for harmonic 2-valued function} \label{sec:harm_plt_sec}

In the remainder of the paper, we will prove Theorems \ref{analyticity_thm} and \ref{analyticity_thm_mss}.  By Lemma \ref{constfreqthm2}, to prove Theorem \ref{analyticity_thm} it suffices to consider the following setup.  Let $\mathcal{N} = 1/2+k$ for some integer $k \geq 0$.  Let $\varepsilon > 0$ to be determined depending only on $n$, $m$, and $\mathcal{N}$.  Suppose that $u : B_1(0) \rightarrow \mathcal{A}_2(\mathbb{R}^m)$ is either $W^{1,2}$ Dirichlet energy minimizing or $C^{1,1/2}$ and locally harmonic on $B_1(0) \setminus \mathcal{B}_u$.  Observe that $u$ is locally harmonic on $B_1(0) \setminus \mathcal{B}_u$ in the sense that for every ball $B \subset \subset \Omega \setminus \mathcal{B}_u$, $u(X) = \{u_1(X),u_2(X)\}$ on $B$ for harmonic single-valued functions $u_1,u_2 : B \rightarrow \mathbb{R}^m$.  Suppose that $\mathcal{B}_u$ is contained in the graph $(x_1,x_2) = g(x_3,\ldots,x_n)$ of a $C^{1,\tau}$ function $g : B^{n-2}_1(0) \rightarrow \mathbb{R}^2$ such that 
\begin{equation} \label{branchsetgraph} 
	g(0) = 0, \quad Dg(0) = 0, \quad [Dg]_{\tau;B^{n-2}_1(0)} \leq \varepsilon, 
\end{equation}
Suppose that for every $Y \in B_{1/2}(0) \cap \mathcal{B}$, there exists a homogeneous degree $\mathcal{N}$ two-valued function $\varphi_Y : \mathbb{R}^n \rightarrow \mathcal{A}_2(\mathbb{R}^m)$ such that for some rotation $q_Y$ of $\mathbb{R}^n$ and $c_Y \in \mathbb{C}^m$, 
\begin{equation} \label{harm_asymptotics1}
	\varphi_Y(q_Y^{-1} X) = \{ \pm \op{Re}(c_Y (x_1+ix_2)^{\mathcal{N}}) \}
\end{equation}
and 
\begin{equation} \label{harm_asymptotics2}
	\rho^{-n} \int_{B_{\rho}(0)} \mathcal{G}(u(Y+X),\varphi_Y(X))^2 dX \leq \varepsilon \rho^{2\mathcal{N}+2\tau} 
\end{equation}
for all $\rho \in (0,1/2)$ for some constants $C \in (0,\infty)$ and $\tau = \tau(n,m) \in (0,1)$.  In the case that $u$ is Dirichlet energy minimizing and $\mathcal{N} = 1/2$, we may assume that $q_0$ equals the identity map on $\mathbb{R}^n$ and $c_0 = (1,i,0,0,\ldots,0)$.  In the case that $\mathcal{N} \geq 3/2$, we may assume that $q_0$ equals the identity map on $\mathbb{R}^n$, $\op{Re}(c_0) = (1,0,0,\ldots,0)$ and $|\op{Im}(c_0)| \leq 1$.  By the compactness of properties of homogeneous degree $\mathcal{N}$, harmonic two-valued functions with $(n-2)$-dimensional spine and the proof of Lemma \ref{constfreqthm2}, we may take $C$ in (\ref{harm_asymptotics2}) to depend only on $n$, $m$, and $\mathcal{N}$.  By (\ref{harm_asymptotics1}) and (\ref{harm_asymptotics2}), we may assume that 
\begin{equation} \label{varphicont} 
	|c_Y - c_Z| \leq C(n,m) \varepsilon |Y-Z|^{\tau}, \quad |q_Y - q_Z| \leq C(n,m) \varepsilon |Y-Z|^{\tau}
\end{equation}
whenever $Y,Z \in B_{1/2}(0) \cap \mathcal{B}$.  

Let $X \in B_{1/4}(0)$ and $Y \in \mathcal{B}_u \cap B_{1/2}(0)$ with $|Z-Y| \leq 2\op{dist}(Z,\mathcal{B}_u)$.  Write $u(X) = \{ \pm u_1(X) \}$ and $\varphi_Y(X-Y) = \{ \pm \varphi_{Y,1}(X-Y) \}$ where $u_1$ and $\varphi_{Y,1}$ are close to $\varphi_{0,1}$ in $L^2(B_1(0);\mathcal{A}_2(\mathbb{R}^m))$.  Then $u_1 - \varphi_{Y,1}$ is harmonic on $B_{|X-Y|/2}(X)$.  Thus provided $\varepsilon$ is sufficiently small, by (\ref{harm_asymptotics2}),~\cite[Lemma 5.1]{KrumWic1}, and the Schauder estimates, 
\begin{equation} \label{harm_asymptotics3}
	|D^{\alpha} u_1(X) - D^{\alpha} \varphi_{Y,1}(X-Y)| \leq C(n,m,|\alpha|) \varepsilon |X-Y|^{\mathcal{N}+\tau-|\alpha|} 
\end{equation}
for every multi-index $\alpha$. 

We define the following variant of the partial Legendre transformation of~\cite{KNS}.  In the case that $\mathcal{N} = 1/2$ and $u$ is Dirichlet energy minimizing, we let $\tilde X = (\tilde x_1,\tilde x_2,\tilde y_1,\tilde y_2,\ldots,\tilde y_{n-2})$ for 
\begin{equation*}
	\tilde x_1 = u^1(X), \quad \tilde x_2 = u^2(X), \quad \tilde y_{j-2} = x_j \text{ for } j = 3,4,\ldots,n, 
\end{equation*}
where $u^{\kappa}$ denotes the $\kappa$-th coordinate function of $u$.  In other words under the transformation each point $X \in B_1(0)$ maps to two points, $(+u^1_1(X),+u^2_1(X),x_3,\ldots,x_n)$ and $(-u^1_1(X),-u^2_1(X),x_3,\ldots,x_n)$.  In the case that $\mathcal{N} \geq 3/2$, we will consider $u$, $D_1^{\mathcal{N}-1/2} u$, and $D_2 D_1^{\mathcal{N}-3/2} u$ as locally harmonic two-valued functions and let $\tilde X = (\tilde x_1,\tilde x_2,\tilde y_1,\tilde y_2,\ldots,\tilde y_{n-2})$ for 
\begin{equation*}
	\tilde x_1 = D_1^{\mathcal{N}-1/2} u^1(X), \quad \tilde x_2 = D_2 D_1^{\mathcal{N}-3/2} u^1(X), \quad \tilde y_{j-2} = x_j \text{ for } j = 3,4,\ldots,n. 
\end{equation*}
where we let $u^{\kappa}$ denote the $\kappa$-th coordinate function of $u$.  In other words under the transformation each point $X \in B_1(0)$ maps to two points, $(+D_1 D_1^{\mathcal{N}-3/2} u^1_1(X),+D_2 D_1^{\mathcal{N}-3/2} u^1_1(X),x_3,\ldots,x_n)$ and $(-D_1 D_1^{\mathcal{N}-3/2} u^1_1(X),-D_2 D_1^{\mathcal{N}-3/2} u^1_1(X),x_3,\ldots,x_n)$.  Note that the transformation maps $\mathcal{B}_u$ into $\{0\} \times \mathbb{R}^{n-2}$. 

Let $\eta' = (\eta_3,\eta_4,\ldots,\eta_n) \in B^{n-2}_{1/4}(0)$ and $Y = (\eta_1,\eta_2,\eta')$ be the unique point in $\mathcal{B}_u \cap B^2_{1/4}(0) \times \{\eta'\}$.  Let $c = (c^1,c^2) = (c^1_Y,c^2_Y)$ if $\mathcal{N} = 1/2$ and $u$ is Dirichlet energy minimizing and $c = (c^1,c^2) = \mathcal{N} (\mathcal{N}-1) \cdots (5/2) (3/2) c^1_Y (1,i)$ if $\mathcal{N} \geq 3/2$ and let $q_Y = (q^i_j)_{i,j=1,2,\ldots,n}$.  Suppose $X \in B_{1/4}(Y)$ with $|X-Y| \leq 2 \op{dist}(X,\mathcal{B})$.  By the definition of $\tilde x_1, \tilde x_2$ and (\ref{harm_asymptotics3}), 
\begin{align*}
	\tilde x_k 
	&= \pm \left( \op{Re} \left( c^k \left( \sum_{j=1}^n (q^1_j + i q^2_j) (x_j - \eta_j) \right)^{1/2} \right) + E \right) \\
	&= \pm \left( \op{Re}(c^k) \op{Re} \left( \sum_{j=1}^n (q^1_j + i q^2_j) (x_j - \eta_j) \right)^{1/2} 
		- \op{Im}(c^k) \op{Im} \left( \sum_{j=1}^n (q^1_j + i q^2_j) (x_j - \eta_j) \right)^{1/2} + E \right)
\end{align*}
for $k = 1,2$ where $E$ denotes error terms satisfying 
\begin{equation} \label{harm_tildex_asym1}
	|D_X^{\alpha} E| \leq C(n,m,|\alpha|) \varepsilon |X-Y|^{1/2+\tau-|\alpha|}. 
\end{equation}
We want to express $x_1$ and $x_2$ in terms of $\tilde X$.  Solving for the real and imaginary parts of $\left( \sum_{j=1}^n (q^1_j + i q^2_j) (x_j - \eta_j) \right)^{1/2}$ yields 
\begin{equation} \label{harm_tildex_asym2}
	\pm \left( \left( \sum_{j=1}^n (q^1_j + i q^2_j) (x_j - \eta_j) \right)^{1/2} + E \right)
	= \frac{i \overline{c^2} \tilde x_1 - i \overline{c^1} \tilde x_2}{\op{Im}(\overline{c^1} c^2)}  
\end{equation}
where $E$ again satisfies (\ref{harm_tildex_asym1}).  Squaring both sides yields 
\begin{align*}
	\frac{(\overline{c^2} \tilde x_1 - \overline{c^1} \tilde x_2)^2}{(\op{Im}(\overline{c^1} c^2))^2} 
	&= -\sum_{j=1}^n (q^1_j + i q^2_j) (x_j - \eta_j) + E \nonumber \\
	&= -\sum_{j=1}^2 (q^1_j + i q^2_j) (x_j - \eta_j) - \sum_{j=3}^n (q^1_j + i q^2_j) (\tilde y_{j-2} - \eta_j) + E 
\end{align*}
where 
\begin{equation} \label{harm_tildex_asym3}
	|X-Y| |D_X^{\alpha} E| \leq C(n,m,|\alpha|) \varepsilon |X-Y|^{1+\tau-|\alpha|}. 
\end{equation}
By solving for $x_1$ and $x_2$, 
\begin{align} \label{harm_tildex_asym4}
	\left( \begin{array}{c} x_1 \\ x_2 \end{array} \right) 
	&= \left( \begin{array}{c} \eta_1 \\ \eta_2 \end{array} \right) 
		- \left( \begin{array}{cc} q^1_1 & q^1_2 \\ q^2_1 & q^2_2 \end{array} \right)^{-1} 
		\left( \frac{1}{(\op{Im}(\overline{c^1} c^2))^2} 
		\left( \begin{array}{c} \op{Re} (c^2 \tilde x_1 - c^1 \tilde x_2)^2 \\ 
		-\op{Im} (c^2 \tilde x_1 - c^1 \tilde x_2)^2 \end{array} \right) 
		+ \sum_{j=3}^n \left( \begin{array}{c} q^1_j \\ q^2_j \end{array} \right) (\tilde y_{j-2} - \eta_j) \right) + E, 
\end{align}
where $E$ satisfies (\ref{harm_tildex_asym3}).  By (\ref{branchsetgraph}), $|X-Y| \leq 2 \op{dist}(X,\mathcal{B}_u)$ implies that 
\begin{equation*}
	|X-Y| \leq 2 (1 + 3 \varepsilon) \left( |x_1-\eta_1|^2 + |x_2-\eta_2|^2 \right)^{1/2}, 
\end{equation*} 
so by (\ref{harm_tildex_asym4}), 
\begin{equation} \label{harm_tildex_asym5} 
	\frac{1}{2} (1 - C(n,m) \varepsilon) |X-Y| \leq |\tilde x|^2 \leq (1 + C(n,m) \varepsilon) |X-Y|. 
\end{equation}
Hence $E$ in (\ref{harm_tildex_asym4}) satisfies 
\begin{equation} \label{harm_tildex_asym6}
	|D_X^{\alpha} E| \leq C(n,m,|\alpha|) \varepsilon (|\tilde x_1|^2 + |\tilde x_2|^2)^{1+\tau-|\alpha|} 
\end{equation}
whenever $|X-Y| \leq 2 \op{dist}(X,\mathcal{B}_u)$. 

We claim that the transformation $X \mapsto \tilde X$ is invertible on a neighborhood of the origin.  Observe that to show this it suffices to check that $X \mapsto \tilde X$ is invertible on the slice $B^2_{1/8}(0) \times \{\eta''\}$ for each $\eta' \in B^{n-2}_{1/8}(0)$.  Let $Y = (\eta_1,\eta_2,\eta')$ be the unique branch point of $u$ in $B^2_{1/8}(0) \times \{\eta\}$.  By (\ref{harm_tildex_asym4}) and (\ref{harm_tildex_asym4}), $F : B^2_{3/16}(0) \rightarrow \mathbb{C} \equiv \mathbb{R}^2$ defined by 
\begin{equation*}
	F(x_1,x_2) = \left( \begin{array}{c} \eta_1 \\ \eta_2 \end{array} \right) 
		- \frac{1}{(\op{Im}(\overline{c^1} c^2))^2} \left( \begin{array}{cc} q^1_1 & q^1_2 \\ q^2_1 & q^2_2 \end{array} \right)^{-1} 
		\left( \begin{array}{c} \op{Re} (\overline{c^2} u^1(x_1,x_2,\eta') - \overline{c^1} \tilde u^2(x_1,x_2,\eta'))^2 \\ 
		\op{Im} (\overline{c^2} u^1(x_1,x_2,\eta) - \overline{c^1} u^2(x_1,x_2,\eta'))^2 \end{array} \right) 
\end{equation*}
extends across $(y_1,y_2)$ to a $C^1$ function such that $F(\eta_1,\eta_2) = (\eta_1,\eta_2)$ and $|DF - I| \leq C(n,m) \varepsilon$ on $B^2_{3/16}(0)$, where $I$ denotes the $2 \times 2$ identity matrix.  Obviously $F |_{B^2_{1/8}(0)}$ is invertible.  Consequently $X \mapsto \tilde X$ is invertible on an neighborhood of the origin.  The inverse tranformation is given by 
\begin{equation} \label{harm_transvar}
	x_1 = \phi^1(\tilde X), \quad x_2 = \phi^2(\tilde X), \quad y_{j-2} = \tilde x_j \text{ for } j = 3,4,\ldots,n, 
\end{equation}
for $\tilde X \in B_{1/8}(0)$ for some functions $\phi^1, \phi^2: B_{1/8}(0) \rightarrow \mathbb{R}$.  Note that assuming $\varepsilon$ is sufficiently small, 
\begin{equation} \label{harm_branchset}
	\mathcal{B} \cap B_{1/16}(0) = \{ (\phi^1(0,y),\phi^2(0,y),y) : y \in B^{n-2}_{1/16}(0) \} \cap B_{1/16}(0). 
\end{equation}
Therefore our goal is to study the regularity of $\phi^1$ and $\phi^2$.

By the definitions of $\phi^1$ and $\phi^2$, 
\begin{gather} 
	\frac{\partial}{\partial \tilde x_1} = \phi^1_1 \frac{\partial}{\partial x_1} + \phi^2_1 \frac{\partial}{\partial x_2}, \nonumber \\
	\frac{\partial}{\partial \tilde x_2} = \phi^1_2 \frac{\partial}{\partial x_1} + \phi^2_2 \frac{\partial}{\partial x_2}, \nonumber \\
	\frac{\partial}{\partial \tilde y_{j-2}} = \phi^1_j \frac{\partial}{\partial x_1} + \phi^2_j \frac{\partial}{\partial x_2} + \frac{\partial}{\partial x_j} 
		\text{ for } j = 3,4,\ldots,n, \label{transder1} 
\end{gather}
where $\phi^{\kappa}_j = \partial \phi^{\kappa}/\partial \tilde x_j$ for $\kappa,j = 1,2$ and $\phi^{\kappa}_j = \partial \phi^{\kappa}/\partial \tilde y_{j-2}$ for $\kappa = 1,2$ and $j = 3,4,\ldots,n$.  Thus 
\begin{gather} 
	\frac{\partial}{\partial x_1} = \frac{\phi^2_2}{\phi^1_1 \phi^2_2 - \phi^1_2 \phi^2_1} \frac{\partial}{\partial \tilde x_1} 
		- \frac{\phi^2_1}{\phi^1_1 \phi^2_2 - \phi^1_2 \phi^2_1} \frac{\partial}{\partial \tilde x_2} \nonumber \\
	\frac{\partial}{\partial x_2} = \frac{-\phi^1_2}{\phi^1_1 \phi^2_2 - \phi^1_2 \phi^2_1} \frac{\partial}{\partial \tilde x_1} 
		+ \frac{\phi^1_1}{\phi^1_1 \phi^2_2 - \phi^1_2 \phi^2_1} \frac{\partial}{\partial \tilde x_2} \nonumber \\
	\frac{\partial}{\partial x_j} = \frac{-\phi^2_2 \phi^1_j + \phi^1_2 \phi^2_j}{\phi^1_1 \phi^2_2 - \phi^1_2 \phi^2_1} \frac{\partial}{\partial \tilde x_1} 
		+ \frac{\phi^2_1 \phi^1_j - \phi^1_1 \phi^2_j}{\phi^1_1 \phi^2_2 - \phi^1_2 \phi^2_1} \frac{\partial}{\partial \tilde x_2} 
		+ \frac{\partial}{\partial \tilde y_{j-2}} \text{ for } j = 3,4,\ldots,n. \label{transder2}
\end{gather}

We will close this section by examining the asymptotic behavior of $\phi^1$ and $\phi^2$ as $\tilde X \rightarrow (0,0,\eta')$ for $\eta' = (\eta_3,\eta_4,\ldots,\eta_n) \in B^{n-2}_{1/16}(0)$.  Let $Y = (\eta_1,\eta_2,\eta')$ be the unique point in $\mathcal{B}_u \cap B^2_{1/8}(0) \times \{\eta'\}$.  Let $c = (c^1,c^2) = (c^1_Y,c^2_Y)$ if $\mathcal{N} = 1/2$ and $u$ is Dirichlet energy minimizing and $c = (c^1,c^2) = \mathcal{N} (\mathcal{N}-1) \cdots (5/2) (3/2) c^1_Y (1,i)$ if $\mathcal{N} \geq 3/2$ and let $q_Y = (q^i_j)_{i,j=1,2,\ldots,n}$.  For now suppose $X \in B_{1/16}(Y)$ with $|X-Y| \leq 2 \op{dist}(X,\mathcal{B}_u)$.  By the computation leading to (\ref{harm_tildex_asym4}) and the fact that $x_k = \phi^k(\tilde X)$ for $k = 1,2$, 
\begin{equation} \label{phi_asym1}
	\left( \begin{array}{c} \phi^1 \\ \phi^2 \end{array} \right) 
	= \left( \begin{array}{c} \eta_1 \\ \eta_2 \end{array} \right) 
		- \left( \begin{array}{cc} q^1_1 & q^1_2 \\ q^2_1 & q^2_2 \end{array} \right)^{-1} 
		\left( \frac{1}{(\op{Im}(\overline{c^1} c^2))^2} 
		\left( \begin{array}{c} \op{Re} (c^2 \tilde x_1 - c^1 \tilde x_2)^2 \\ 
		-\op{Im} (c^2 \tilde x_1 - c^1 \tilde x_2)^2 \end{array} \right) 
		+ \sum_{j=3}^n \left( \begin{array}{c} q^1_j \\ q^2_j \end{array} \right) (\tilde y_{j-2} - \eta_j) \right) + E, 
\end{equation}
where $c = (c^1,c^2)$ and $q = (q^i_j)$ are as above and $|D_X^{\alpha} E| \leq C(n,m) \varepsilon |X-Y|^{1+\tau-|\alpha|}$.  Let 
\begin{equation*}
	\zeta = \frac{i \overline{c^2} \tilde x_1 - i \overline{c^1} \tilde x_2}{\op{Im}(\overline{c^1} c^2)}.  
\end{equation*} 
In the case that $\mathcal{N} = 1/2$ and $u$ is Dirichlet energy minimizing, by the definition of $\phi$, 
\begin{equation*}
	\left( \begin{array}{cc} \phi^1_1 & \phi^1_2 \\ \phi^2_1 & \phi^2_2 \end{array} \right) 
		= \left( \begin{array}{cc} u^1_1 & u^1_2 \\ u^2_1 & u^2_2 \end{array} \right)^{-1}, \quad 
	\left( \begin{array}{c} \phi^1_j \\ \phi^2_j \end{array} \right) 
		= -\left( \begin{array}{cc} u^1_1 & u^1_2 \\ u^2_1 & u^2_2 \end{array} \right)^{-1} 
		\left( \begin{array}{c} u^1_j \\ u^2_j \end{array} \right) , 
\end{equation*}
so by (\ref{harm_asymptotics3}) and (\ref{harm_tildex_asym2}), 
\begin{align} \label{phi_asym2}
	\left( \begin{array}{cc} \phi^1_1 & \phi^1_2 \\ \phi^2_1 & \phi^2_2 \end{array} \right) 
	&= \frac{1}{2} \left( \begin{array}{cc} q^1_1 & q^1_2 \\ q^2_1 & q^2_2 \end{array} \right)^{-1} 
		\left( \begin{array}{cc} \op{Re}(c^1 \zeta^{-1}) & -\op{Im}(c^1 \zeta^{-1}) \\ \op{Re}(c^2 \zeta^{-1}) & -\op{Im}(c^2 \zeta^{-1}) \end{array} 
		\right)^{-1} + E \nonumber \\
	&= \frac{1}{2 \op{Im}(c^1 \overline{c^2})} \left( \begin{array}{cc} q^1_1 & q^1_2 \\ q^2_1 & q^2_2 \end{array} \right)^{-1} 
		\left( \begin{array}{cc} \op{Im}(\overline{c^2} \zeta) & -\op{Im}(\overline{c^1} \zeta) \\ 
		-\op{Re}(\overline{c^2} \zeta) & \op{Re}(\overline{c^1} \zeta) \end{array} \right) + E 
\end{align}
where $|D_X^{\alpha} E| \leq C(n,m,|\alpha|) \varepsilon |X-Y|^{1/2+\tau-|\alpha|}$ and 
\begin{equation} \label{phi_asym3}
	\left( \begin{array}{c} \phi^1_j \\ \phi^2_j \end{array} \right) 
	= -\left( \begin{array}{cc} q^1_1 & q^1_2 \\ q^2_1 & q^2_2 \end{array} \right)^{-1} 
		\left( \begin{array}{c} q^1_j \\ q^2_j \end{array} \right) + E
\end{equation}
where $|D_X^{\alpha} E| \leq C(n,m,|\alpha|) \varepsilon |X-Y|^{\tau-|\alpha|}$.  By a similar argument, (\ref{phi_asym3}) holds true in the case that $\mathcal{N} \geq 3/2$.  By (\ref{harm_tildex_asym5}), (\ref{harm_tildex_asym6}), (\ref{transder1}), (\ref{phi_asym1}), (\ref{phi_asym2}), and (\ref{phi_asym3}), $E$ in (\ref{harm_tildex_asym4}) and $E$ in (\ref{phi_asym1}) both satisfy  
\begin{equation} \label{phi_asym4}
	|D_{\tilde x}^{\alpha} D_{\tilde y}^{\beta} E| \leq C(n,m,|\alpha|,|\beta|) \varepsilon |\tilde x|^{2+2\tau-|\alpha|-2|\beta|}. 
\end{equation}
If instead $|X-Y| \geq 2 \op{dist}(X,\mathcal{B}_u)$, by replacing $Y$ with $Z \in \mathcal{B}_u \cap B^{n-2}_{1/4}(0)$ such that $|X-Z| = \op{dist}(X,\mathcal{B}_u)$ and applying (\ref{phi_asym1}) and (\ref{phi_asym4}) with $Z$ in place of $Y$ and using (\ref{branchsetgraph}), (\ref{varphicont}), and the fact that 
\begin{equation*}
	|\op{pr}_{\mathbb{R}^2 \times \{0\}} (q_Z (Y-Z))| \leq C(n,m,\tau) \varepsilon |Y-Z|^{1+\tau} 
\end{equation*}
due to $\mathcal{B}_u$ being a $C^{1,\tau}$ submanifold, where $\op{pr}_{\mathbb{R}^2 \times \{0\}}$ denotes orthogonal projection onto $\mathbb{R}^2 \times \{0\}$, it follows that $E$ in (\ref{harm_tildex_asym4}) and $E$ in (\ref{phi_asym1}) both satisfy 
\begin{gather}
	|E| \leq C(n,m) \varepsilon \left( |\tilde x|^2 + \sum_{j=3}^n |\tilde y_{j-2} - \eta_j| \right)^{1+\tau}, \nonumber \\
	|\tilde x|^{-1} |D_{\tilde x} E| + |D_{\tilde x \tilde x} E| + |D_{\tilde y} E| 
		\leq C(n,m) \varepsilon \left( |\tilde x|^2 + \sum_{j=3}^n |\tilde y_{j-2} - \eta_j| \right)^{\tau}, \label{phi_asym5} 
\end{gather}
and $E$ satisfying (\ref{phi_asym4}) if $|\alpha| + 2|\beta| > 2$.  Finally, note that by (\ref{harm_asymptotics3}), (\ref{harm_tildex_asym4}), (\ref{phi_asym4}), and (\ref{phi_asym5}) (see (\ref{harm_tildex_asym2})) 
\begin{equation} \label{harm_u_asym1}
	D_X^{\gamma} u = \op{Re} \left( \mathcal{N} (\mathcal{N}-1) \cdots (\mathcal{N}-|\gamma|+1) c_Y 
		\left( \frac{i \overline{c^2} \tilde x_1 - i \overline{c^1} \tilde x_2}{\op{Im}(\overline{c^1} c^2)} \right)^{2\mathcal{N} - 2|\gamma|} \right) + E
\end{equation}
whenever $|X-Y| \geq 2 \op{dist}(X,\mathcal{B}_u)$ and $|\gamma| \leq \mathcal{N}-1/2$, where 
\begin{equation} \label{harm_u_asym2}
	|D_{\tilde x}^{\alpha} D_{\tilde y}^{\beta} E| \leq C(n,m,\mathcal{N},|\gamma|,|\alpha|,|\beta|) |\tilde x|^{2\mathcal{N}-2|\gamma|-|\alpha|-2|\beta|+2\tau}
\end{equation}
for all $\alpha$ and $\beta$.  By (\ref{branchsetgraph}) and (\ref{varphicont}), (\ref{harm_u_asym1}) holds true even when $|X-Y| > 2 \op{dist}(X,\mathcal{B}_u)$ with 
\begin{equation*}
	|D_{\tilde x}^{\alpha} E| \leq C(n,m,\mathcal{N},|\gamma|,|\alpha|) |\tilde x|^{2\mathcal{N}-2|\gamma|-|\alpha|} (|\tilde x|^2 + |\tilde y|)^{\tau}
\end{equation*}
whenever $|\alpha| \leq 2\mathcal{N}-2|\gamma|$ and (\ref{harm_u_asym2}) whenever $|\alpha| > 2\mathcal{N}-2|\gamma|$ or $|\beta| \geq 1$.

\section{Partial Legendre transformation for area-stationary graphs} \label{sec:mss_plt_sec}

By Lemma \ref{constfreqthm2_mss}, to prove Theorem \ref{analyticity_thm_mss} it suffices to consider the following setup.  Let $\mathcal{N} = 1/2+k$ for some integer $k \geq 1$.  Let $\varepsilon > 0$ to be determined depending only on $n$, $m$, and $\mathcal{N}$.  Suppose that $u \in C^{1,1/2}(B_2(0);\mathcal{A}_2(\mathbb{R}^m))$ is a two-valued function such that $\mathcal{M} = \op{graph} u$ is area-stationary, $0 \in \mathcal{B}_u$, $\mathcal{N}_{\mathcal{M}}(Y,u(Y)) = \mathcal{N}$ for all $Y \in \mathcal{B}_u$, and 
\begin{equation} \label{mss_closetoplane}
	u(0) = \{0,0\}, \quad Du(0) = \{0,0\}, \quad [Du]_{1/2;B_2(0)} \leq \varepsilon. 
\end{equation}
Suppose that $\mathcal{B}_u$ is the graph $(x_1,x_2) = g(x_3,\ldots,x_n)$ of a $C^{1,\tau}$ function $g : B^{n-2}_1(0) \rightarrow \mathbb{R}^2$ such that (\ref{branchsetgraph}) holds true.  By (\ref{mss_closetoplane}), theory of frequency functions, and Theorem \ref{constfreqthm2_mss}, we may suppose that for every $Y \in B_{1/2}(0) \cap \mathcal{B}_u$, there exists a homogeneous degree $\mathcal{N}$ two-valued function $\widehat{\varphi}_Y : T_p \mathcal{M} \rightarrow \mathcal{A}_2(T_p \mathcal{M}^{\perp})$, where $T_p \mathcal{M}^{\perp}$ denotes the tangent plane to $\mathcal{M}$ at the branch point $p = (Y,u(Y))$, such that after an orthogonal change of coordinates of $T_p \mathcal{M}$ 
\begin{equation} \label{mss_asymptotics1}
	\widehat{\varphi}_Y(X) = \{ \pm \op{Re}(c_Y (x_1+ix_2)^{\mathcal{N}}) \}
\end{equation}
for some $c_Y \in \mathbb{C}^m$ and 
\begin{equation} \label{mss_asymptotics2}
	\rho^{-n} \int_{B_{1/2}(0)} \mathcal{G}\left( \frac{\tilde u_{p,s}}{\|u_s\|_{L^2(B_1(0))}}, \widehat{\varphi}_Y(X) \right)^2 dX 
	\leq \varepsilon \rho^{\mathcal{N}+2\tau} 
\end{equation}
for all $\rho \in (0,1/2)$ for some constants $C \in (0,\infty)$ and $\tau = \tau(n,m) \in (0,1)$, where $\mathcal{M}$ is the graph of $\tilde u_p(X) = \{\tilde u_1(X),\tilde u_2(X)\}$ over $T_p \mathcal{M}$, $\tilde u_{p,s}(X) = \{\pm (\tilde u_{p,2}(X)-\tilde u_{p,1}(X))/2\}$, and, like in Section~\ref{sec:harm_plt_sec}, we can take $C$ to depend only on $n$, $m$, and $\mathcal{N}$ but not on $\varphi_0$ by the compactness of properties of homogeneous degree $\mathcal{N}$, $C^{1,1/2}$, harmonic two-valued functions with $(n-2)$-dimensional spine and the proof of Lemma \ref{constfreqthm2_mss}.

Define 
\begin{equation*}
	v = \frac{u_s}{\|\tilde u_s\|_{L^2(B_1(0))}}
\end{equation*}
We want to establish the asymptotic behavior of $v$ at branch points of $u$.  Let $X_0 \in B_{1/2}(0)$ and $Y \in \mathcal{B}_u \cap B_1(0)$ with $|X_0-Y| \leq 2 \op{dist}(X,\mathcal{B}_u)$.  Let $Q$ denote a rotation of $\mathbb{R}^{n+m}$ taking $\mathbb{R}^n \times \{0\}$ to $T_{(Y,u_a(Y))} \mathcal{M}$ with $|Q - I| \leq C |Du_a(Y)|$ (for example, construct the columns of $Q$ by applying the Gram-Schmidt process to the basis $\{ (e_j, D_j u_a(Y)) \}_{j=1,2,\ldots,n} \cup \{(0,e_j)\}_{j=1,2,\ldots,m}$ for $\mathbb{R}^{n+m}$).  Let $\xi_l(X_0) = \op{pr}_{\mathbb{R}^n \times \{0\}} Q^{-1} (X_0,u_l(X_0))$ for $l = 1,2$ and $\xi_a(X_0) = (\xi_1(X_0)+\xi_2(X_0))/2$, where $\tilde u(X_0) = (\tilde u_1(X_0),\tilde u_2(X_0))$.  Since the image of $(X_0,u_a(X_0)) + Q(X,\tilde u(X))$ for $X \in B_{|X_0-Y|/8}(\xi_a(X_0))$ is contained in the graph of $u |_{B_{|X_0-Y|/4}(X_0)}$ provided $\varepsilon$ is sufficiently small and $\mathcal{B}_u \cap B_{|X_0-Y|/4}(X_0) = \emptyset$, by the local decomposition of $\tilde u$ and $u$ away from branch points and unique continuation we can write $u(X) = \{u_1(X),u_2(X)\}$ on $B_{|X_0-Y|/4}(X_0)$ and $\tilde u_1(X) = \{\tilde u_1(X),\tilde u_2(X)\}$ on $B_{|X_0-Y|/8}(\xi_a(X_0))$ for $C^1$ single-valued solutions $u_1,u_2,\tilde u_1,\tilde u_2$ to the minimal surface system.  Assume $u_1,u_2$ are defined without changing their values at $X_0$.  Letting $\xi_l(X) = \op{pr}_{\mathbb{R}^n \times \{0\}} Q^{-1} (X,u_l(X))$ for $X \in B_{|X_0-Y|/16}(X_0)$ and $l = 1,2$, we can choose $\tilde u_1, \tilde u_2$ so that $(X,u_l(X)) = Q (\xi_l(X),\tilde u_l(\xi_l(X)))$ for $l = 1,2$, i.e. 
\begin{align} 
	\label{rotateu_eqnX} X &= Y + Q_{11} \xi_l(X) + Q_{12} \tilde u_l(\xi_l(X)), \\
	\label{rotateu_eqnU} u_l(X) &= u_a(Y) + Q_{21} \xi_l(X) + Q_{22} \tilde u_l(\xi_l(X)), 
\end{align}
on $B_{|X_0-Y|/16}(X_0)$.  Let $\xi_a(X) = (\xi_1(X)+\xi_2(X))/2$, $\xi_s(X) = (\xi_2(X)-\xi_1(X))/2$, $u_a(X) = (u_1(X)+u_2(X))/2$, $u_s(X) = (u_2(X)-u_1(X))/2$, $\tilde u_a(X) = (\tilde u_1(X)+\tilde u_2(X))/2$, and $\tilde u_s(X) = (\tilde u_2(X)-\tilde u_1(X))/2$.  By taking differences in (\ref{rotateu_eqnX}) and (\ref{rotateu_eqnU}) we get 
\begin{equation} \label{rotateu_eqn1} 
	\xi_s(X) = -Q_{11}^{-1} Q_{12} \frac{\tilde u_2(\xi_2(X)) - \tilde u_1(\xi_1(X))}{2} 
\end{equation}
on $B_{|X_0-Y|/16}(X_0)$ and thus 
\begin{equation} \label{rotateu_eqn2} 
	u_s(X) = (Q_{22} - Q_{21} Q_{11}^{-1} Q_{12}) \frac{\tilde u_2(\xi_2(X)) - \tilde u_1(\xi_1(X))}{2} 
\end{equation}
on $B_{|X_0-Y|/16}(X_0)$.  Observe that 
\begin{align} \label{rotateu_eqn3}
	\op{dist}(\xi_l,\mathcal{B}_{\tilde u}) 
	&\leq \op{dist}((\xi_l,\tilde u_l(\xi_l)), \op{sing} \op{graph} {\tilde u}) 
	= \op{dist}((X,u_l(X)), \op{sing} \op{graph} u) \nonumber \\
	&\leq (1 + C(n,m) \varepsilon) \op{dist}(X,\mathcal{B}_u) 
	= \frac{17}{16} (1 + C(n,m) \varepsilon) |X_0-Y|
\end{align}
and similarly 
\begin{equation*}
	\op{dist}(\xi_l,\mathcal{B}_u) \geq \frac{7}{16} (1 - C(n,m) \varepsilon) |X_0-Y| 
\end{equation*}
and so by (\ref{SimWicThm7}) and (\ref{rotateu_eqn1}) and (\ref{rotateu_eqn2}), 
\begin{gather}
	\sup_{B_{|X_0-Y|/16}(X_0)} |\xi_s| + \sup_{B_{|X_0-Y|/16}(X_0)} |u_s| + \sup_{B_{|X_0-Y|/16}(X_0-Y)} |\tilde u_s| \leq C(n,m) \varepsilon |X_0-Y|^{3/2}, 
		\nonumber \\
	\sup_{B_{|X_0-Y|/16}(X_0)} |Du_s| + \sup_{B_{|X_0-Y|/16}(X_0-Y)} |D\tilde u_s| \leq C(n,m) \varepsilon |X_0-Y|^{1/2}, \nonumber \\
	[Du_s]_{1/2,B_{|X_0-Y|/16}(X_0)} + [D\tilde u_s]_{1/2,B_{|X_0-Y|/16}(0)} 
		+ \sup_{B_{|X_0-Y|/16}(X_0)} |D^2 u_a| + \sup_{B_{|X_0-Y|/16}(X_0-Y)} |D^2 \tilde u_a| \leq C(n,m) \varepsilon, \nonumber \\
	\sup_{B_{|X_0-Y|/16}(X_0)} |D^2 u_s| + \sup_{B_{|X_0-Y|/16}(X_0-Y)} |D^2 \tilde u_s| \leq C(n,m) \varepsilon |X_0-Y|^{-1/2}. \label{rotateu_eqn4}
\end{gather}
Recall that (\ref{mss1}) holds true with $\tilde u_{p,a} = \tilde u_a$, $\tilde u_{p,s} = \tilde u_s$, and $B = B_{|X_0-Y|/16}(X_0)$, so by the Schauder estimates for (\ref{mss1}) together with $\sup_{B_{|X_0-Y|/16}(X_0)} |\tilde u_s| \leq C(n,m) \|\tilde u_s\|_{L^2(B_{1/2}(0))} |X_0-Y|^{\mathcal{N}}$ by the theory of frequency functions and (\ref{rotateu_eqn3}), 
\begin{equation} \label{rotateu_eqn5}
	\sup_{B_{3|X_0-Y|/64}(X_0)} |\tilde u_s| + |X_0-Y| \sup_{B_{3|X_0-Y|/64}(X_0)} |D\tilde u_s| 
	\leq C(n,m) \|\tilde u_s\|_{L^2(B_{1/2}(0))} |X_0-Y|^{\mathcal{N}}. 
\end{equation}
Since $\tilde u_a$ and $\tilde u_s$ satisfy (\ref{mss2}) and $\widehat{\varphi}_Y$ is harmonic, 
\begin{equation*}
	\Delta \left( \frac{\tilde u_s^{\kappa}}{\|u_s\|_{L^2(B_1(0))}} - \widehat{\varphi}_Y \right) 
	= - D_i \left( (A^{ij}_{\kappa \lambda}(D\tilde u_a,D\tilde u_s) - \delta_{ij} \delta_{\kappa \lambda}) 
	\frac{D_j u_s^{\lambda}}{\|u_s\|_{L^2(B_1(0))}} \right) 
\end{equation*}
on $B_{|X_0-Y|/16}(X_0)$, to which we can apply the Schauder estimates together with an estimate of~\cite{KrumWic2}, (\ref{harm_asymptotics2}), (\ref{rotateu_eqn3}), and (\ref{rotateu_eqn5}) to get 
\begin{equation} \label{rotateu_eqn6}
	\sup_{B_{|X_0-Y|/32}(X_0)} \left| \frac{D^{\alpha} \tilde u_s^{\kappa}}{\|u_s\|_{L^2(B_1(0))}} - D^{\alpha} \widehat{\varphi}_Y \right| 
	\leq C(n,m,|\alpha|) \varepsilon |X_0-Y|^{\mathcal{N}+\tau-|\alpha|}  
\end{equation}
for all $\alpha$.  By (\ref{rotateu_eqnX}), (\ref{rotateu_eqn4}) and (\ref{rotateu_eqn6}), 
\begin{equation} \label{rotateu_eqn7}
	|\xi_a(X) - Q_{11}^{-1} (X-Y)| \leq C(n,m) \varepsilon |X_0-Y|^2, \quad |\xi_s(X)| \leq C(n,m) |X_0-Y|^{\mathcal{N}}. 
\end{equation}
By (\ref{rotateu_eqn2}), (\ref{rotateu_eqn4}), and (\ref{rotateu_eqn7}), 
\begin{equation} \label{rotateu_eqn8}
	\sup_{B_{|X_0-Y|/32}(X_0)} \left| \frac{u_s(X)}{\|u_s\|_{L^2(B_1(0))}} - (Q_{22} - Q_{21} Q_{11}^{-1} Q_{12}) \widehat{\varphi}_Y(Q_{11}^{-1} (X-Y)) \right| 
	\leq C(n,m) \varepsilon |X_0-Y|^{\mathcal{N}+\tau}. 
\end{equation}
Let 
\begin{equation*}
	\varphi_Y(X) = (Q_{22} - Q_{21} Q_{11}^{-1} Q_{12}) \widehat{\varphi}_Y(Q_{11}^{-1} X).  
\end{equation*}
Observe that by choosing of the matrices $Q$ to be a Lipschitz function of $Y$ with Lipschitz constant $\leq C(n,m) \varepsilon$ (for example, by constructing $q$ via Gram-Schmidt as described above) and by the proof of Lemma \ref{constfreqthm2_mss}, 
\begin{equation*}
	\varphi_Y(q^{-1}_Y X) = \op{Re}(c_Y (x_1+ix_2)^{\mathcal{N}}) 
\end{equation*}
for some $n \times n$ matrix $q_Y$, not necessarily a rotation matrix, and $c_Y \in \mathbb{C}^m$ such that $q_0$ equals the identity map on $\mathbb{R}^n$, $\op{Re}(c_0) = (1,0,0,\ldots,0)$, $|\op{Im}(c_0)| \leq 1$, and (\ref{varphicont}). 

Since $u_l$ satisfies the minimal surface system, i.e. (\ref{mss1}) holds true with $\tilde u_{p,1} = u_l$ and $B = B_{|X_0-Y|/4}(X_0)$, by the Schauder estimates and (\ref{rotateu_eqn4}), 
\begin{equation} \label{rotateu_eqn9} 
	\sup_{B_{|X_0-Y|/8}(X_0)} |D^{\alpha} u| \leq C(n,m,|\alpha|) \varepsilon |X_0-Y|^{3/2-|\alpha|} 
\end{equation}
for every multi-index $\alpha$ with $|\alpha| \geq 2$.  By scaling (\ref{mss2}) with $\tilde u_{p,a} = u_a$ and $\tilde u_{p,s} = u_s$, 
\begin{equation} \label{rotateu_eqn10}
	D_i (A^{ij}_{\kappa \lambda}(D\tilde u_a(Y),0) D_j \varphi_Y^{\lambda}) = 0 
\end{equation}
weakly on $B_{|X_0-Y|/8}(X_0-Y)$.  By subtracting (\ref{mss2}) with $\tilde u_{p,a} = u_a$, $\tilde u_{p,s} = u_s$, and $B = B_{|X_0-Y|/8}(X_0-Y)$ and (\ref{rotateu_eqn10}), 
\begin{equation*}
	D_i \left( A^{ij}_{\kappa \lambda}(Du_a(Y),0) D_j (v^{\lambda}(X) - \varphi_Y^{\lambda}(X-Y)) \right)  
	= -D_i \left( (A^{ij}_{\kappa \lambda}(Du_a,Du_s) - A^{ij}_{\kappa \lambda}(Du_a(Y),0)) D_j v_s^{\lambda} \right)
\end{equation*}
on $B_{|X_0-Y|/8}(X_0)$ for $\kappa = 1,2,\ldots,m$ and thus by (\ref{rotateu_eqn8}) and the Schauder estimates 
\begin{equation} \label{mss_asymptotics3}
	\sup_{B_{|X_0-Y|/64}(X_0)} \left| D^{\alpha} v(X) - D^{\alpha} \varphi_Y(X-Y) \right| 
	\leq C(n,m,|\alpha|) \varepsilon |X_0-Y|^{\mathcal{N}+\tau-|\alpha|} 
\end{equation}
for every multi-index $\alpha$.  

We similarly want to establish the asymptotic behavior of $u_a$ at branch points of $u$.  Since by (\ref{mss1}), 
\begin{equation*}
	D_i \left( \sqrt{G(Du_a)} G^{ij}(Du_a) D_j u_a - f^i_{\kappa}(Du_a,Du_s) \right) = 0 
\end{equation*}
in $B_1(0)$ where $f^i_{\kappa} : \mathbb{R}^{mn} \times \mathbb{R}^{mn} \rightarrow \mathbb{R}$ is a real analytic function such that $f^i_{\kappa}(P,0) = 0$ and $Df^i_{\kappa}(P,0) = 0$ for all $P \in \mathbb{R}^{mn}$, by elliptic regularity and the Schauder estimates using (\ref{rotateu_eqn9}) and (\ref{mss_asymptotics3}), 
\begin{align} \label{mss_asymptotics4}
	\sup_{B_{\rho}(Y)} |D^{\alpha} u_a - D^{\alpha} P| \leq C(n,m,|\alpha|) \varepsilon \rho^{2\mathcal{N}-1-|\alpha|} 
\end{align}
for all $\rho \in (0,1/2]$ and $\alpha$ for some polynomial $P$ of degree at most $2\mathcal{N}-2$.  In the case that $\mathcal{N} \geq 5/2$ this estimate is sufficient whereas in the case $\mathcal{N} = 3/2$ we need a stronger estimate.  Note that by (\ref{mss1}) for any harmonic polynomial $\tilde P : \mathbb{R}^n \rightarrow \mathbb{R}^m$, 
\begin{gather} 
	D_i \left( A^{ij}_{\kappa \lambda}(D\tilde u_s,D\tilde u_a) D_j \tilde u_a^{\lambda} \right) = 0, \nonumber \\
	\Delta (\tilde u_a^{\kappa} - \tilde P) = - D_i \left( (A^{ij}_{\kappa \lambda}(D\tilde u_s,D\tilde u_a) - \delta_{ij} \delta_{\kappa \lambda}) 
		D_j u_a^{\lambda} \right) . \label{rotateu_eqn11}
\end{gather}

Using the following lemma, we can establish the asymptotic behavior for $\tilde u_a$ at $0$.

\begin{lemma} \label{avgexcessdecay_lemma}
For every $\vartheta \in (0,1/2]$ there exists $\varepsilon = \varepsilon(n,m,\vartheta) > 0$ such that if $\tilde u \in C^{1,1/2}(B_2(0);\mathcal{A}_2(\mathbb{R}^m))$ is a two-valued function whose graph $\mathcal{M}$ is area-stationary,  
\begin{align}
	\label{avged_eqn1} \tilde u(0) = 0, \quad D\tilde u(0) = 0, \quad [D\tilde u]_{1/2,B_2(0)} \leq \varepsilon 
\end{align}
and there is a polynomial $\tilde P : \mathbb{R}^n \rightarrow \mathbb{R}^m$ of degree at most $2$ such that 
\begin{equation} \label{avged_eqn3} 
	\int_{B_1(0)} |\tilde u_a - \tilde P|^2 < \varepsilon, 
\end{equation} 
then there exists a polynomial $\tilde P' : \mathbb{R}^n \rightarrow \mathbb{R}^m$ of degree at most $2$ such that 
\begin{equation} \label{avged_eqn4} 
	\vartheta^{-n-4} \int_{B_{\vartheta}(0)} |\tilde u_a - \tilde P'|^2 
	\leq C \vartheta^2 \left( \int_{B_1(0)} |\tilde u_a - \tilde P|^2 + \|\tilde u\|_{L^2(B_2(0))}^2 \|\tilde u_a\|_{L^2(B_2(0))}^2 \right)
\end{equation} 
for some $C = C(n,m) \in (0,\infty)$. 
\end{lemma}
\begin{proof}
Consider a sequence of $\tilde u_j \in C^{1,1/2}(B_1(0);\mathcal{A}_2(\mathbb{R}^m))$ and polynomials $\tilde P_j : \mathbb{R}^n \rightarrow \mathbb{R}^m$ of degree at most $2k$ such that the graph $\mathcal{M}_j$ of $u_j$ is area-stationary, (\ref{avged_eqn1}) and (\ref{avged_eqn3}) hold true with $\varepsilon = \varepsilon_j \downarrow 0$, $\tilde u = \tilde u_j$, $\mathcal{M} = \mathcal{M}_j$, and $\tilde P = \tilde P_j$.  Let $\tilde u_{j,a}(X) = (\tilde u_{j,1}(X)+\tilde u_{j,2}(X))/2$ and $\tilde u_{j,s}(X) = \{ \pm (\tilde u_{j,1}(X)-\tilde u_{j,2}(X))/2 \}$ where $\tilde u_j(X) = \{\tilde u_{j,1}(X),\tilde u_{j,2}(X)\}$.  Let 
\begin{equation*}
	E_j = \left( \int_{B_1(0)} |\tilde u_{j,a} - \tilde P_j|^2 + \|\tilde u_j\|_{L^2(B_2(0))}^2 \|\tilde u_{j,a}\|_{L^2(B_2(0))}^2 \right)^{1/2}
\end{equation*}
By the Schauder estimates for (\ref{rotateu_eqn11}) with $\tilde u = \tilde u_j$, (\ref{avged_eqn1}), and (\ref{SimWicThm7}), $(u_{j,a} - \tilde P_j)/E_j$ converges to a harmonic single-valued function $w$ in $C^1$ on compact subsets of the interior of $B_1(0)$.  By the Schauder estimates for harmonic functions, 
\begin{equation*} 
	\vartheta^{-n-4} \int_{B_{\vartheta}(0)} |w - p|^2 \leq C \vartheta^2 \int_{B_{3/4}(0)} |w|^2 \leq C \vartheta^2 
\end{equation*} 
for some harmonic polynomial $p$ of degree at most $2$ and $C = C(n,m) \in (0,\infty)$ and thus it follows that (\ref{avged_eqn4}) holds true with $\tilde u = \tilde u_j$, $\tilde P = \tilde P_j$, and $\tilde P' = E_j p$ for infinitely many $j$. 
\end{proof}

By iteratively applying Lemma \ref{avgexcessdecay_lemma} choosing $\vartheta$ so that $C \vartheta \leq 1/8$ and using (\ref{mss_closetoplane}) and~\cite[Lemma 6.9]{SimWic11}, which implies that 
\begin{equation*}
	\rho^{-n/2} \|\tilde u_a\|_{L^2(B_{\rho}(0))} \leq C(n,m) \rho^2 \|\tilde u_a\|_{L^2(B_{1/2}(0))}, \quad 
	\rho^{-n/2} \|\tilde u_s\|_{L^2(B_{\rho}(0))} \leq C(n,m) \rho^{3/2} \|\tilde u_s\|_{L^2(B_{1/2}(0))},
\end{equation*}
for all $\rho \in (0,1/2]$, we can find a sequence of polynomials $\tilde P_j$ for $j = 0,1,2,3,\ldots$ such that $\tilde P_0 = 0$ and 
\begin{equation*}
	(\vartheta^j/4)^{-n} \int_{B_{\vartheta^j/4}(0)} |\tilde u_a - \tilde P_j|^2 + \varepsilon \vartheta^j
	\leq \frac{1}{4} \left( (\vartheta^{j-1}/4)^{-n} \int_{B_{\vartheta^{j-1}/4}(0)} |\tilde u_a - \tilde P_{j-1}|^2 + \varepsilon \vartheta^{j-1} \right) 
\end{equation*}
for all $j = 1,2,3,\ldots$.  It is then standard that $\tilde P_j$ converge uniformly on $B_1(0)$ to a harmonic polynomial $\tilde P_Y$ of degree at most $2$ such that 
\begin{equation} \label{rotateu_eqn12}
	\rho^{-n} \int_{B_{\rho}(0)} |\tilde u_a - \tilde P_Y|^2 \leq C(n,m) \varepsilon \rho^{2+\tau'} 
\end{equation}
for all $\rho \in (0,1/2]$ for some $\tau' = \tau'(n,m) \in (0,1/2]$.  Moreover, for each $Y,Z \in \mathcal{B}_u \cap B_1(0)$ with corresponding harmonic polynomials $P_Y, P_Z$ respectively, 
\begin{equation} \label{polycont}
	\sum_{|\alpha| = 2} |D^{\alpha} P_Y(0) - D^{\alpha} P_Z(0)| \leq C(n,m) \varepsilon |Y-Z|^{\tau'}. 
\end{equation}
By reducing $\tau$ if necessary, we may take $\tau' = \tau$.  By the Schauder estimates applied to (\ref{rotateu_eqn11}), (\ref{mss_asymptotics3}), and (\ref{rotateu_eqn12}), 
\begin{equation} \label{rotateu_eqn13}
	\sup_{B_{|X_0-Y|/32}(X_0-Y)} |D^{\alpha} \tilde u_a - D^{\alpha} \tilde P_Y| \leq C(n,m,|\alpha|) \varepsilon |X_0-Y|^{2+\tau-|\alpha|} 
\end{equation}
for all $\alpha$.  By (\ref{rotateu_eqnX}) and (\ref{rotateu_eqnU}), 
\begin{align*}
	u_a(X) &= u_a(Y) + Du_a(Y) (X-Y) + (Q_{22} - Q_{21} Q_{11}^{-1} Q_{12}) \frac{\tilde u_1(\xi_1(X)) + \tilde u_2(\xi_2(X))}{2} \\
	&= u_a(Y) + Du_a(Y) (X-Y) + (Q_{22} - Q_{21} Q_{11}^{-1} Q_{12}) \frac{\tilde u_a(\xi_1(X)) + \tilde u_a(\xi_2(X))}{2} \\
		&\hspace{5mm} + \frac{1}{2} (Q_{22} - Q_{21} Q_{11}^{-1} Q_{12}) \int_{-1}^1 D\tilde u_s(\xi_a(X)+t\xi_s(X)) \cdot \xi_s(X) dt 
\end{align*}
on $B_{|X_0-Y|/16}(X_0)$, so by (\ref{rotateu_eqn1}), (\ref{rotateu_eqn6}), (\ref{rotateu_eqn7}), and (\ref{rotateu_eqn13}) 
\begin{align*}
	&\sup_{B_{|X_0-Y|/32}(0)} \left| u_a(X) - P_Y(X) 
	+ \|u_s\|_{L^2(B_1(0))}^2 D\tilde \varphi_Y(Q_{11}^{-1} (X-Y)) Q_{11}^{-1} Q_{12} \tilde \varphi_Y(Q_{11}^{-1} (X-Y)) \right|
	\\&\leq C(n,m) \varepsilon |X_0-Y|^{2+\tau}, 
\end{align*}
provided $\tau \leq 1/4$ and reducing $\tau$ otherwise, where 
\begin{equation*}
	P_Y(X) = u_a(Y) + Du_a(Y)(X-Y) + (Q_{22} - Q_{21} Q_{11}^{-1} Q_{12}) \tilde P_Y(Q_{11}^{-1} (X-Y)). 
\end{equation*}
Hence by the Schauder estimates and (\ref{rotateu_eqn11}), 
\begin{align} \label{mss_asymptotics5}
	&\sup_{B_{|X_0-Y|/64}(X_0)} \left| D^{\alpha} u_a(X) - D^{\alpha} P_Y(X) - \|u_s\|_{L^2(B_1(0))}^2 
		D^{\alpha} \left( D\tilde \varphi_Y(Q_{11}^{-1} (X-Y)) Q_{11}^{-1} Q_{12} \tilde \varphi_Y(Q_{11}^{-1} (X-Y)) \right) \right| \nonumber \\
	&\leq C(n,m,|\alpha|) \varepsilon |X_0-Y|^{2+\tau-|\alpha|} 
\end{align}
for all $\alpha$. 

We define the following variant of the partial Legendre transformation of~\cite{KNS}.  We will consider $D_1^k u_a,D_1^k v$ for $k = 0,1,2,\ldots,\mathcal{N}-1/2$ and $D_2 D_1^{\mathcal{N}-3/2} u_a, D_2 D_1^{\mathcal{N}-3/2} v$ as locally solutions to an elliptic system and let $\tilde X = (\tilde x_1,\tilde x_2,\tilde y_1,\tilde y_2,\ldots,\tilde y_{n-2})$ for 
\begin{equation*}
	\tilde x_1 = D_1^{\mathcal{N}-1/2} v^1(X), \quad \tilde x_2 = D_2 D_1^{\mathcal{N}-3/2} v^1(X), \quad \tilde y_{j-2} = x_j \text{ for } j = 3,4,\ldots,n. 
\end{equation*}
where $v^{\kappa}$ denotes the $\kappa$-th coordinate function of $v$.  In other words under the transformation each point $X \in B_1(0)$ maps to two points, $(+D_1^{\mathcal{N}-1/2} v^1_1(X),+D_2 D_1^{\mathcal{N}-3/2} v^1_1(X),x_3,\ldots,x_n)$ and $(-D_1^{\mathcal{N}-1/2} v^1_1(X),-D_2 D_1^{\mathcal{N}-3/2} v^1_1(X),x_3,\ldots,x_n)$ where $v(X) = \{ \pm v_1(X) \}$.  Note that the transformation maps $\mathcal{B}_u$ into $\{0\} \times \mathbb{R}^{n-2}$.   By an argument similar to the one in Section~\ref{sec:harm_plt_sec} using (\ref{mss_asymptotics3}), this transformation is invertible with inverse tranformation is given by (\ref{harm_transvar}) for some functions $\phi^1, \phi^2: B_{1/8}(0) \rightarrow \mathbb{R}$.  By the definition of $\phi^1$ and $\phi^2$, (\ref{transder1}) and (\ref{transder2}) hold true.  Assuming $\varepsilon$ is sufficiently small, (\ref{harm_branchset}) holds true and thus our goal is to study the regularity of $\phi^1$ and $\phi^2$.

Let $\eta' = (\eta_3,\eta_4,\ldots,\eta_n) \in B^{n-2}_{1/4}(0)$ and $Y = (\eta_1,\eta_2,\eta')$ be the unique point in $\mathcal{B}_u \cap B^2_{1/4}(0) \times \{\eta'\}$.  Let $c = (c^1,c^2) = \mathcal{N} (\mathcal{N}-1) \cdots (5/2) (3/2) c^1_Y (1,i)$ if $\mathcal{N} \geq 3/2$ and let $q_Y = (q^i_j)_{i,j=1,2,\ldots,n}$.  By the argument in Section~\ref{sec:harm_plt_sec} using (\ref{mss_asymptotics3}) in place of (\ref{harm_asymptotics3}), we can show that (\ref{harm_tildex_asym4}) and (\ref{phi_asym1}) hold true with $E$ in both equations satisfying (\ref{phi_asym5}) and (\ref{phi_asym4}) for $|\alpha| + 2|\beta| > 2$. 

By (\ref{mss_asymptotics3}), (\ref{mss_asymptotics4}), (\ref{mss_asymptotics5}), (\ref{harm_tildex_asym4}), (\ref{phi_asym4}), and (\ref{phi_asym5}) (see (\ref{harm_tildex_asym2})), 
\begin{equation} \label{mss_u_asym1}
	D_X^{\gamma} v = \op{Re} \left( \mathcal{N} (\mathcal{N}-1) \cdots (\mathcal{N}-|\gamma|+1) c_Y 
		\left( \frac{i \overline{c^2} \tilde x_1 - i \overline{c^1} \tilde x_2}{\op{Im}(\overline{c^1} c^2)} \right)^{2\mathcal{N} - 2|\gamma|} \right) + E
\end{equation}
whenever $|X-Y| \geq 2 \op{dist}(X,\mathcal{B}_u)$ and $|\gamma| \leq \mathcal{N}-1/2$, where 
\begin{equation} \label{mss_u_asym2}
	|D_{\tilde x}^{\alpha} D_{\tilde y}^{\beta} E| \leq C(n,m,\mathcal{N},|\gamma|,|\alpha|,|\beta|) |\tilde x|^{2\mathcal{N}-2|\gamma|-|\alpha|-2|\beta|+2\tau}
\end{equation}
for all $\alpha$ and $\beta$ and 
\begin{equation} \label{mss_u_asym3}
	D_X^{\gamma} u_a = \sum_{|\alpha|/2+|\beta| \leq 2\mathcal{N}-1-|\gamma|} a_{\gamma,\alpha,\beta} \tilde x^{\alpha} \tilde y^{\beta} + E
\end{equation}
whenever $|X-Y| \geq 2 \op{dist}(X,\mathcal{B}_u)$ and $|\gamma| \leq 2\mathcal{N}-1$, where 
\begin{equation} \label{mss_u_asym4}
	|D_{\tilde x}^{\alpha} D_{\tilde y}^{\beta} E| \leq C(n,m,\mathcal{N},|\gamma|,|\alpha|,|\beta|) |\tilde x|^{4\mathcal{N}-2-2|\gamma|-|\alpha|-2|\beta|+2\tau}
\end{equation}
for all $\alpha$ and $\beta$.  By (\ref{branchsetgraph}) and (\ref{varphicont}), (\ref{mss_u_asym1}) holds true even when $|X-Y| > 2 \op{dist}(X,\mathcal{B}_u)$ with 
\begin{equation*}
	|D_{\tilde x}^{\alpha} E| \leq C(n,m,\mathcal{N},|\gamma|,|\alpha|) |\tilde x|^{2\mathcal{N}-2|\gamma|-|\alpha|} (|\tilde x|^2 + |\tilde y|)^{\tau}
\end{equation*}
whenever $|\alpha| \leq 2\mathcal{N}-2|\gamma|$ and (\ref{mss_u_asym2}) whenever $|\alpha| > 2\mathcal{N}-2|\gamma|$ or $|\beta| \geq 1$.  Also, (\ref{mss_u_asym3}) holds true even when $|X-Y| > 2 \op{dist}(X,\mathcal{B}_u)$ with 
\begin{equation*}
	|D_{\tilde x}^{\alpha} E| \leq C(n,m,\mathcal{N},|\gamma|,|\alpha|) (|\tilde x|^2 + |\tilde y|)^{2\mathcal{N}-1-|\gamma|-|\alpha|/2-|\beta|+\tau}
\end{equation*}
whenever $|\alpha|/2+|\beta| \leq 2\mathcal{N}-1-|\gamma|$ and (\ref{mss_u_asym2}) whenever $|\alpha|/2+|\beta| > 2\mathcal{N}-1-|\gamma|$.

\section{Transformed differential system} \label{sec:transpde_sec}

Suppose that either (a) $u \in W^{1,2}(B_1(0);\mathcal{A}_2(\mathbb{R}^m))$ is a Dirichlet energy minimizing function as in Section~\ref{sec:harm_plt_sec} and $\mathcal{N} = 1/2$, (b) $u \in C^{1,1/2}(B_1(0);\mathcal{A}_2(\mathbb{R}^m))$ is locally harmonic in $B_1(0) \setminus \mathcal{B}_u$ and is as in Section~\ref{sec:harm_plt_sec}, or (c) $u \in C^{1,1/2}(B_1(0);\mathcal{A}_2(\mathbb{R}^m))$ is a two-valued functions whose graph is area-stationary as in Section~\ref{sec:mss_plt_sec}.  In cases (a) and (b), note that $u_a \equiv 0$ and let $v = u$ and $\sigma = 1$.  In case (c), let $v = u_s/\sigma$ where $\sigma = \|u_s\|_{L^2(B_1(0))}$.  Let $\tilde X = (\tilde x,\tilde y)$ be the partial Legendre transformation and $\phi^1$ and $\phi^2$ be the functions defined in Sections~\ref{sec:harm_plt_sec} and~\ref{sec:mss_plt_sec} as defined for each case (a), (b), and (c).  Let 
\begin{gather*}
	M^1_1(R) = \frac{R^2_2}{R^1_1 R^2_2 - R^1_2 R^2_1}, \quad 
	M^2_1(R) = \frac{-R^2_1}{R^1_1 R^2_2 - R^1_2 R^2_1}, \\
	M^1_2(R) = \frac{-R^1_2}{R^1_1 R^2_2 - R^1_2 R^2_1}, \quad 
	M^2_2(R) = \frac{R^1_1}{R^1_1 R^2_2 - R^1_2 R^2_1}, \\ 
	M^i_1(R) = M^i_2 = 0, \quad 
	M^1_i(R) = \frac{-R^2_2 R^1_j + R^1_2 R^2_j}{R^1_1 R^2_2 - R^1_2 R^2_1}, \quad
	M^2_i(R) = \frac{R^2_1 R^1_j - R^1_1 R^2_j}{R^1_1 R^2_2 - R^1_2 R^2_1} \quad \text{for } i = 3,4,\ldots,n, \\
	M^i_j(R) = \delta_{ij} \text{ for } i,j = 3,4,\ldots,n, 
\end{gather*}
for all $R \in \mathbb{R}^{2n}$ and let $M^i_j = M^i_j(D\phi)$ for brevity so that by (\ref{transder1}) and (\ref{transder2}), 
\begin{equation} \label{transder3}
	\frac{\partial}{\partial x_i} = \sum_{j=1}^2 M^j_i \frac{\partial}{\partial \tilde x_j} + \sum_{j=3}^n M^j_i \frac{\partial}{\partial \tilde y_{j-2}}  
\end{equation}
for $i = 1,2,\ldots,n$. 

Observe that by the fact that $u$ is locally harmonic on $B_1(0) \setminus \mathcal{B}_u$ in cases (a) and (b) and (\ref{mss2}) and (\ref{mss4}) hold true with $u_a,u_s$ in place of $\tilde u_{p,a}, \tilde u_{p,s}$ in case (c), we have  
\begin{align} \label{pde1}
	&D_i \left( A^{ij}_{\kappa \lambda}(Du_a,\sigma Dv) D_j v^{\lambda} \right) = 0, \nonumber \\
	&D_i \left( A^{ij}_{\kappa \lambda}(\sigma Dv,Du_a) D_j u_a^{\lambda} \right) = 0,
\end{align}
locally on $B_1(0) \setminus \mathcal{B}_u$ for $\kappa = 1,2,\ldots,m$, where $A^{ij}_{\kappa \lambda} = \delta_{ij} \delta_{\kappa \lambda}$ in cases (a) and (b) and $A^{ij}_{\kappa \lambda}$ is defined by (\ref{mss2_notation}) in case (c).  Note that 
\begin{equation} \label{Asymmetry1}
	A^{ij}_{\kappa \lambda}(P,Q) = A^{ij}_{\kappa \lambda}(-P,Q) = A^{ij}_{\kappa \lambda}(P,-Q), \quad
	A^{ij}_{\kappa \lambda}(0,0) = \delta_{ij} \delta_{\kappa \lambda}, \quad DA^{ij}_{\kappa \lambda}(0,0) = 0,
\end{equation} 
where $P,Q \in \mathbb{R}^{mn}$.  For cases (b) and (c), differentiating (\ref{pde1}) by $D^k_1$ yields 
\begin{align} \label{pde2}
	&D_i \left( a^{ij}_{\kappa \lambda}(Du_a,\sigma Dv) D_j D^k_1 v^{\lambda} 
		+ \sigma^{-1} b^{ij}_{\kappa \lambda}(Du_a,\sigma Dv) D_j D^k_1 u_a^{\lambda} \right) \nonumber \\& 
		= \sigma^{-1} D_i \left( h^i_{k,\kappa}(\{DD^l_1 u_a\}_{l \leq k-1}, \{\sigma DD^l_1 v\}_{l \leq k-1} \right) , \nonumber \\
	&D_i \left( a^{ij}_{\kappa \lambda}(\sigma Dv,Du_a) D_j D^k_1 u_a^{\lambda} 
		+ \sigma b^{ij}_{\kappa \lambda}(\sigma Dv,Du_a) D_j D^k_1 v^{\lambda} \right) \nonumber \\& 
		= D_i \left( h^i_{k,\kappa}(\{\sigma DD^l_1 v\}_{l \leq k-1}, \{DD^l_1 u_a\}_{l \leq k-1}) \right) ,
\end{align}
locally on $B_1(0) \setminus \mathcal{B}_u$ for $\kappa = 1,2,\ldots,m$ and $k = 1,2,\ldots,\mathcal{N}-1/2$, where 
\begin{equation*}
	a^{ij}_{\kappa \lambda}(P,Q) = A^{ij}_{\kappa \lambda}(P,Q) + D_{Q^{\lambda}_j} A^{il}_{\kappa \nu}(P,Q) Q^{\nu}_l, \quad 
	b^{ij}_{\kappa \lambda}(P,Q) = D_{P^{\lambda}_j} A^{il}_{\kappa \nu}(P,Q) Q^{\nu}_l, 
\end{equation*}
for $P,Q \in \mathbb{R}^{mn}$ and $h^i_{k,\kappa} : (\mathbb{R}^{mn})^{2k} \rightarrow \mathbb{R}$ are defined by $h^i_{1,\kappa} = 0$ and 
\begin{align*}
	&h^i_{k+1,\kappa}(\{P_t\}_{t \leq k},\{Q_t\}_{t \leq k}) 
	\\&= \sum_{s=0}^{k-1} (D_{P^{\nu}_{s,l}} h^i_{k,\kappa}(\{P_t\}_{t \leq k-1},\{Q_t\}_{t \leq k-1}) P^{\nu}_{s+1,l} 
		+ D_{Q^{\nu}_{s,l}} h^i_{k,\kappa}(\{P_t\}_{t \leq k-1},\{Q_t\}_{t \leq k-1}) Q^{\nu}_{s+1,l}) 
	\\&- D_{P^{\nu}_l} a^{ij}_{\kappa \lambda}(P_0) Q^{\lambda}_{k,j} P^{\nu}_{1,l} 
		- D_{Q^{\nu}_l} a^{ij}_{\kappa \lambda}(P_0) Q^{\lambda}_{k,j} Q^{\nu}_{1,l} 
		- D_{P^{\nu}_l} b^{ij}_{\kappa \lambda}(P_0) P^{\lambda}_{k,j} P^{\nu}_{1,l} 
		- D_{Q^{\nu}_l} b^{ij}_{\kappa \lambda}(P_0) P^{\lambda}_{k,j} Q^{\nu}_{1,l} 
\end{align*}
for all $P_t = (P^{\lambda}_{t,j}), Q_t = (Q^{\lambda}_{t,j}) \in \mathbb{R}^{mn}$ and $k = 1,2,\ldots,\mathcal{N}-3/2$.  Differentiating (\ref{pde1}) and (\ref{pde2}) by $D_2$ yields 
\begin{align} \label{pde3}
	&D_i \left( a^{ij}_{\kappa \lambda}(Du_a,\sigma Dv) D_j D^k_1 D_2 v^{\lambda} 
		+ \sigma^{-1} b^{ij}_{\kappa \lambda}(Du_a,\sigma Dv) D_j D^k_1 D_2 u_a^{\lambda} \right) \\& 
		= \sigma^{-1} D_i \left( f^i_{2k+1,\kappa}(\{(DD^l_1 u_a, DD^l_1 D_2 u_a)\}_{l \leq k-1},DD^k_1 u_a,
			\{(\sigma DD^l_1 v, \sigma DD^l_1 D_2 v)\}_{l \leq k-1}, \sigma DD^k_1 v) \right) , \nonumber \\
	&D_i \left( a^{ij}_{\kappa \lambda}(\sigma Dv,Du_a) D_j D^k_1 D_2 u_a^{\lambda} 
		+ \sigma b^{ij}_{\kappa \lambda}(\sigma Dv,Du_a) D_j D^k_1 D_2 v^{\lambda} \right) \nonumber \\& 
		= D_i \left( f^i_{2k+1,\kappa}(\{(\sigma DD^l_1 v, \sigma DD^l_1 D_2 v)\}_{l \leq k-1}, \sigma DD^k_1 v, 
		\{(DD^l_1 u_a, DD^l_1 D_2 u_a)\}_{l \leq k-1}, DD^k u_a) \right) , \nonumber
\end{align}
locally on $B_1(0) \setminus \mathcal{B}_u$ for $\kappa = 1,2,\ldots,m$ and $k = 0,1,2,\ldots,\mathcal{N}-3/2$, where $f^i_{2k+1,\kappa} : (\mathbb{R}^{mn})^{4k+2} \rightarrow \mathbb{R}$ are defined by $f^i_{1,\kappa} = 0$ and 
\begin{align*}
	&f^i_{2k+1,\kappa}(\{P_t\}_{t \leq 2k},\{Q_t\}_{t \leq 2k}) 
	\\&= \sum_{s=0}^{k-1} (D_{P^{\nu}_{s,l}} h^i_{k,\kappa}(\{P_{2t}\}_{t \leq k-1},\{Q_{2t}\}_{t \leq k-1}) P^{\nu}_{2s+1,l} 
		+ D_{Q^{\nu}_{s,l}} h^i_{k,\kappa}(\{P_{2t}\}_{t \leq k-1},\{Q_{2t}\}_{t \leq k-1}) Q^{\nu}_{2s+1,l}) 
	\\&- D_{P^{\nu}_l} a^{ij}_{\kappa \lambda}(P_0) Q^{\lambda}_{2k,j} P^{\nu}_{1,l} 
		- D_{Q^{\nu}_l} a^{ij}_{\kappa \lambda}(P_0) Q^{\lambda}_{2k,j} Q^{\nu}_{1,l} 
		- D_{P^{\nu}_l} b^{ij}_{\kappa \lambda}(P_0) P^{\lambda}_{2k,j} P^{\nu}_{1,l} 
		- D_{Q^{\nu}_l} b^{ij}_{\kappa \lambda}(P_0) P^{\lambda}_{2k,j} Q^{\nu}_{1,l} 
\end{align*}
for all $P_t = (P^{\lambda}_{t,j}), Q_t = (Q^{\lambda}_{t,j}) \in \mathbb{R}^{mn}$ and $k = 1,2,\ldots,\mathcal{N}-3/2$.  Define $f^i_{2k,\kappa} : (\mathbb{R}^{mn})^{4k} \rightarrow \mathbb{R}$ by $f^i_{2k,\kappa}(\{P_t\}_{t \leq 2k-1},\{Q_t\}_{t \leq 2k-1}) = h^i_{k,\kappa}(\{P_{2t}\}_{t \leq k-1},\{Q_{2t}\}_{t \leq k-1})$.  Note that by the definition of $f^i_{k,\kappa}$ and (\ref{Asymmetry1}), 
\begin{equation} \label{fsymmetry1}
	f^i_{\kappa \lambda}(P,Q) = f^i_{\kappa \lambda}(-P,Q) = -f^i_{\kappa \lambda}(P,-Q), \quad 
	D^l f^i_{\kappa \lambda}(0,0) = 0 \text{ for } l = 0,1,2,
\end{equation} 
where $P,Q \in (\mathbb{R}^{mn})^k$.

Let $\chi_{2k}(\tilde X) = D^k_1 u_a(X)$, $\chi_{2k+1}(\tilde X) = D^k_1 D_2 u_a(X)$, $\psi_{2k}(\tilde X) = D^k_1 v(X)$, and $\psi_{2k+1}(\tilde X) = D^k_1 D_2 v(X)$ for each $k$.  By (\ref{transder3}) and $dx_1 dx_2 \cdots dx_n = (\phi^1_1 \phi^2_2 - \phi^1_2 \phi^2_1) d\tilde x_1 d\tilde x_2 \cdots d\tilde x_n$, (\ref{pde1}) transforms into
\begin{align} \label{transpde1}
	&D_i \left( (\phi^1_1 \phi^2_2 - \phi^1_2 \phi^2_1) M^i_s M^j_t A^{st}_{\kappa \lambda}(D\chi_0 M,\sigma D\psi_0 M) D_j \psi_0^{\lambda} \right) = 0, 
		\nonumber \\
	&D_i \left( (\phi^1_1 \phi^2_2 - \phi^1_2 \phi^2_1) M^i_s M^j_t A^{ij}_{\kappa \lambda}(\sigma D\psi_0 M,D\chi_0 M) D_j \chi_0^{\lambda} \right) = 0,
\end{align}
on $B_{1/8}(0) \setminus \{0\} \times \mathbb{R}^{n-2}$ for $\kappa = 1,2,\ldots,m$.  For cases (b) and (c), (\ref{pde2}) and (\ref{pde3}) transforms into
\begin{align} \label{transpde2}
	&D_i \left( (\phi^1_1 \phi^2_2 - \phi^1_2 \phi^2_1) M^i_s M^j_t a^{st}_{\kappa \lambda}(D\chi_0 M,\sigma D\psi_0 M) D_j \psi_k^{\lambda} \right) \nonumber \\
	&+ \sigma^{-1} D_i \left( (\phi^1_1 \phi^2_2 - \phi^1_2 \phi^2_1) M^i_s M^j_t b^{st}_{\kappa \lambda}(D\chi_0 M,\sigma D\psi_0 M) 
		D_j \chi_k^{\lambda} \right) \nonumber \\
	&= \sigma^{-1} D_i \left( (\phi^1_1 \phi^2_2 - \phi^1_2 \phi^2_1) M^i_s f^s_{k,\kappa}(\{D\chi_l M\}_{l \leq k-1}, \{\sigma D\psi_l M\}_{l \leq k-1}) \right)
		, \nonumber \\
	&D_i \left( (\phi^1_1 \phi^2_2 - \phi^1_2 \phi^2_1) M^i_s M^j_t a^{st}_{\kappa \lambda}(\sigma D\psi_0 M,D\chi_0 M) D_j \chi_k^{\lambda} \right) \nonumber \\
	&+ \sigma D_i \left( (\phi^1_1 \phi^2_2 - \phi^1_2 \phi^2_1) M^i_s M^j_t b^{st}_{\kappa \lambda}(\sigma D\psi_0 M,D\chi_0 M) 
		D_j \psi_k^{\lambda} \right) \nonumber \\
	&= D_i \left( (\phi^1_1 \phi^2_2 - \phi^1_2 \phi^2_1) M^i_s f^s_{k,\kappa}(\{\sigma D\psi_l M\}_{l \leq k-1}, \{D\chi_l M\}_{l \leq k-1}) \right) ,
\end{align}
on $B_{1/8}(0) \setminus \{0\} \times \mathbb{R}^{n-2}$ for $\kappa = 1,2,\ldots,m$ and $k = 1,2,\ldots,2\mathcal{N}-1$.  Note that in case (a), $\psi_0^{\lambda} = \tilde x_{\lambda}$ for $\lambda = 1,2$ and thus in (\ref{transpde1}) and (\ref{transpde2}) we replace $D_j \psi_{2\mathcal{N}-1}^{\lambda}$ with $\delta_{\lambda j}$ for $\lambda = 1,2$ and $j = 1,2,\ldots,n$.  In cases (b) and (c), $\psi_{2\mathcal{N}-3+\lambda}^1 = \tilde x_{\lambda}$ for $\lambda = 1,2$ and in (\ref{transpde1}) and (\ref{transpde2}) we replace $D_j \psi_{2\mathcal{N}-3+\lambda}^1$ with $\delta_{\lambda j}$ for $\lambda = 1,2$ and $j = 1,2,\ldots,n$. 

\section{Schauder theory for singular elliptic systems} \label{sec:schaudersec}

In what follows, it is convenient to write $\tilde X \in \mathbb{R}^n$ as $\tilde X = (\tilde x,\tilde y)$ for $\tilde x \in \mathbb{R}^2$ and $\tilde y \in \mathbb{R}^{n-2}$.  We will use the notation 
\begin{gather*}
	D_i = \partial/\partial \tilde x_i \text{ for } i = 1,2, \quad D_i = \partial/\partial \tilde y_{i-2} \text{ for } i = 3,4,\ldots,n, \quad 
		D_{ij} = D_i D_j \text{ for all } i,j, \\
	D_{\tilde x}^{\alpha} = D_1^{\alpha_1} D_2^{\alpha_2} \text{ for } \alpha = (\alpha_1,\alpha_2), \quad 
		D_{\tilde y}^{\beta} = D_3^{\beta_1} D_4^{\beta_2} \cdots D_n^{\beta_{n-2}} \text{ for } \beta = (\beta_1,\beta_2,\ldots,\beta_{n-2}), \\
	\Delta_{\tilde x} = \sum_{j=1}^2 D_{jj}, \quad \Delta_{\tilde y} = \sum_{j=3}^n D_{jj}. 
\end{gather*}
Consider the metric on $\mathbb{R}^n$ given by 
\begin{equation*}
	g = 4 (\tilde x_1^2 + \tilde x_2^2) (d\tilde x_1^2 + d\tilde x_2^2) + \sum_{j=1}^{n-2} d\tilde y_j^2. 
\end{equation*}
$g$ is the pullback of the Euclidean metric on $\mathbb{R}^n$ under the map $\tilde X \mapsto (\tilde x_1^2 - \tilde x_2^2, 2 \tilde x_1 \tilde x_2, \tilde y)$.  Let $d_g$ denote geodesic distance on $\mathbb{R}^n$ with respect to the metric $g$.  Given an open set $\Omega$, let $\op{diam}_g(\Omega) = \sup_{\tilde X,\tilde X' \in \Omega} d_g(\tilde X,\tilde X')$ be the diameter of $\Omega$ with respect to the metric $g$.  Let $B^g_{\rho}(\tilde X_0)$ denote the geodesic ball with respect to the metric $g$ with center $\tilde X_0$ and radius $\rho > 0$.  Define 
\begin{equation*}
	[\psi]_{g,\tau,\Omega} = \sup_{\tilde X, \tilde Y \in \Omega, \, \tilde X \neq \tilde Y} 
		\frac{|\psi(\tilde X) - \psi(\tilde Y)|}{d_g(\tilde X, \tilde Y)^{\tau}}
\end{equation*}
for every $\tau \in (0,1)$, open set $\Omega \subseteq \mathbb{R}^n$, and real or vector valued function $\psi$ on $\Omega$.  

\begin{defn}
Given an integers $m \geq 1$ and $k \geq 0$, $\tau \in (0,1/2]$, and an open set $\Omega \subseteq \mathbb{R}^n$, define $C^{g,k,\tau}(\Omega;\mathbb{R}^m)$ to be the set of all functions $\psi : \Omega \rightarrow \mathbb{R}^m$ such that $D_{\tilde x}^{\alpha} D_{\tilde y}^{\beta} \psi$ exists and is continuous on $\Omega$ whenever $|\alpha| + 2|\beta| \leq k$ and $[D_{\tilde x}^{\alpha} D_{\tilde y}^{\beta} \psi]_{g,\tau,\Omega} < \infty$ whenever $|\alpha| + 2|\beta| = k$.  When $m = 1$ or the value of $m$ is obvious from the context, we denote $C^{g,k,\tau}(\Omega;\mathbb{R}^m)$ simply by $C^{g,k,\tau}(\Omega)$.  We equip $C^{g,k,\tau}(\Omega;\mathbb{R}^m)$ with the norm 
\begin{equation*}
	\|\psi\|_{C^{g,k,\tau}(\Omega;\mathbb{R}^m)} \equiv \sum_{|\alpha|+2|\beta| \leq k} \sup_{\Omega} |D_{\tilde x}^{\alpha} D_{\tilde y}^{\beta} \psi| 
		+ \sum_{|\alpha|+2|\beta| = k} [D_{\tilde x}^{\alpha} D_{\tilde y}^{\beta} \psi]_{g,\tau,\Omega}. 
\end{equation*}
\end{defn}

\begin{defn}
Given integers $m \geq 1$ and $k \geq 0$, $s \in \mathbb{R}$, $\tau \in (0,1/2]$, and an open set $\Omega \subseteq \mathbb{R}^n$, define $\mathcal{H}^{s,k,\tau}(\Omega;\mathbb{R}^m)$ to be the set of all functions $f \in C^k(\Omega \setminus \{0\} \times \mathbb{R}^{n-2};\mathbb{R}^m)$ such that for every $(0,\tilde y_0) \in \Omega \cap \{0\} \times \mathbb{R}^{n-2}$ there exists a (unique) homogeneous degree $s$ smooth function $f(\, ;\tilde y_0) : \mathbb{R}^2 \rightarrow \mathbb{R}^m$, which we extend to a function $f(\tilde x,\tilde y;\tilde y_0)$ of $(\tilde x,\tilde y) \in \mathbb{R}^n$ that is independent of $\tilde y$, such that 
\begin{align*}
	&\|f\|_{\mathcal{H}^{s,k,\tau}(\Omega)} \equiv \sum_{|\alpha|+|\beta| \leq k} \sup_{\tilde X \in \Omega \setminus \{0\} \times \mathbb{R}^{n-2}} 
		|\tilde x|^{-s+|\alpha|+2|\beta|} |D_{\tilde x}^{\alpha} D_{\tilde y}^{\beta} f(\tilde X)| 
	\\&+ \sum_{|\alpha|+|\beta| = k} \sup_{\tilde X \in \Omega \setminus \{0\} \times \mathbb{R}^{n-2}} 
		|\tilde x|^{-s+|\alpha|+2|\beta|+2\tau} [D_{\tilde x}^{\alpha} D_{\tilde y}^{\beta} f]_{g,\tau,B^g_{|\tilde x|^2/4}(\tilde X)}
	\\&+ \sum_{|\alpha|+|\beta| \leq k} \sup_{(\tilde X,\tilde y_0) \in S_{\Omega}} \op{diam}_g(\Omega)^{\tau} |\tilde x|^{-s+|\alpha|+2|\beta|} 
		\frac{|D_{\tilde x}^{\alpha} D_{\tilde y}^{\beta} f(\tilde X) - D_{\tilde x}^{\alpha} D_{\tilde y}^{\beta} f(\tilde X;\tilde y_0)|}{
		d_g(\tilde X,(0,\tilde y_0))^{\tau}}  
	\\&+ \sum_{|\alpha|+|\beta| = k} \sup_{(\tilde X,\tilde y_0) \in S_{\Omega}, \, d_g(\tilde X,(0,\tilde y_0)) \leq 4 |\tilde x|^2}
		\op{diam}_g(\Omega)^{\tau} |\tilde x|^{-s+|\alpha|+2|\beta|} [D_{\tilde x}^{\alpha} D_{\tilde y}^{\beta} f 
		- D_{\tilde x}^{\alpha} D_{\tilde y}^{\beta} f(.;\tilde y_0)]_{g,\tau,B^g_{|\tilde x|^2/4}(\tilde X)} 
	< \infty, 
\end{align*}
where $S_{\Omega} = \{ (\tilde X,\tilde y_0) : \tilde X \in \Omega \setminus \{0\} \times \mathbb{R}^{n-2}, \, (0,\tilde y_0) \in \Omega \cap \{0\} \times \mathbb{R}^{n-2} \}$.  (Note that the notation $f(\, ;\tilde y_0)$ is ambiguous since $\mathcal{H}^{s,k,\tau}(\Omega;\mathbb{R}^m) \subset \mathcal{H}^{s',k,\tau}(\Omega;\mathbb{R}^m)$ whenever $s > s'$ and thus $f(\, ;\tilde y_0)$ depends on $s$.  Whenever we use the notation $f(\, ;\tilde y_0)$, we will use the value of $s$ for which $f \in \mathcal{H}^{s,k,\tau}(\Omega;\mathbb{R}^m)$ is explicitly stated.)  If instead $k = -1$, we will define $\mathcal{H}^{s,-1,\tau}(\Omega;\mathbb{R}^m)$ to be the set of all distributions of the form $f = \sum_{i=1}^n D_i f^i$ for $f^i \in \mathcal{H}^{s+1,0,\tau}(\Omega;\mathbb{R}^m)$ for $i = 1,2$ and $f^i \in \mathcal{H}^{s+2,0,\tau}(\Omega;\mathbb{R}^m)$ for $i = 3,4,\ldots,n$ with $f^i( \, ;\tilde y_0) = \sum_{i=1}^n D_i f^i( \, ;\tilde y_0)$ as a distribution and 
\begin{equation*}
	\|f\|_{\mathcal{H}^{s,-1,\tau}(\Omega)} = \sum_{i=1}^3 \|f^i\|_{\mathcal{H}^{s+1,0,\tau}(\Omega)} + \sum_{i=3}^n \|f^i\|_{\mathcal{H}^{s+2,0,\tau}(\Omega)}.
\end{equation*} 
When $m = 1$ or the value of $m$ is obvious from the context, we denote $\mathcal{H}^{s,k,\tau}(\Omega;\mathbb{R}^m)$ simply by $\mathcal{H}^{s,k,\tau}(\Omega)$.  
\end{defn}

We shall first prove two relatively simple Schauder estimate, from which the main Schauder estimates that we need follow:

\begin{lemma} \label{schauderlemma}
Let $k \geq 1$ be an integer and $\tau \in (0,1/2)$.  Let $\psi \in C^{g,k,\tau}(B^g_1(0)) \cap C^{\infty}(B^g_1(0) \setminus \{0\} \times \mathbb{R}^{n-2})$ with 
\begin{equation} \label{schauder_evenodd}
	\psi(-\tilde x,\tilde y) = (-1)^k \psi(\tilde x,\tilde y) 
\end{equation} 
for all $(\tilde x,\tilde y) \in B^g_1(0)$ and $f \in \mathcal{H}^{k-2,k-2,\tau}(B^g_1(0))$.  Suppose that $\psi$ is a solution to 
\begin{equation} \label{schauder1_eqn1} 
	\Delta_{\tilde x} \psi + 4 |\tilde x|^2 \Delta_{\tilde y} \psi = f
\end{equation}
on $B^g_1(0) \setminus \{0\} \times B^{n-2}_1(0)$ and 
\begin{equation} \label{schauder1_eqn2} 
	(\Delta_{\tilde x} + 4 |\tilde x|^2 \Delta_{\tilde y}) \left( \sum_{|\alpha|+2|\beta| \leq k} D^{\alpha}_{\tilde x} D^{\beta}_{\tilde y} \psi(0,\tilde y_0) 
		\tilde x^{\alpha} (\tilde y - \tilde y_0)^{\beta} \right) = f(\tilde X;\tilde y_0)
\end{equation}
on $\mathbb{R}^n \setminus \{0\} \times \mathbb{R}^{n-2}$ for all $\tilde y_0 \in B^{n-2}_1(0)$.  Then 
\begin{equation} \label{schauder1_eqn3}
	\|\psi\|_{C^{g,k,\tau}(B^g_{1/2}(0))} \leq C \left( \sup_{B^g_1(0)} |\psi| + \|f\|_{\mathcal{H}^{k-2,k-2,\tau}(B^g_1(0))} \right)
\end{equation}
for some constant $C = C(n,k,\tau) \in (0,\infty)$. 
\end{lemma}

Before proving Lemma \ref{schauderlemma}, we need to prove Schauders estimate away from the singular set, some interpolation inequalities, and an abstract covering lemma similar to CITE.

\begin{lemma}[Schauder estimate away from the singular set] \label{schauder_afbs}
(i) Let $\lambda \in (0,1)$.  Let $\psi \in C^{1,\lambda}(B^g_1(0) \setminus \{0\} \times \mathbb{R}^{n-2})$ and $f^i \in C^{0,\lambda}(B^g_1(0) \setminus  \{0\} \times \mathbb{R}^{n-2})$ for $i = 1,2,\ldots,n$.  Suppose that $\psi$ is a solution to (\ref{schauder1_eqn1}).  Then 
\begin{align*}
	&\sum_{|\alpha|+|\beta| \leq 1} \frac{\rho^{|\alpha|+|\beta|}}{|\tilde x_0|^{|\alpha|}} 
		\sup_{B^g_{\rho/2}(\tilde X_0)} |D_{\tilde x}^{\alpha} D_{\tilde y}^{\beta} \psi| 
	+ \sum_{|\alpha|+|\beta| = 1} \frac{\rho^{|\alpha|+|\beta|+\lambda}}{|\tilde x_0|^{|\alpha|}} 
		[D_{\tilde x}^{\alpha} D_{\tilde y}^{\beta} \psi]_{g,\lambda,B^g_{\rho/2}(\tilde X_0)} 
	\leq C \sup_{B^g_{\rho}(\tilde X_0)} |\psi| \\& + C \left( \sum_{i=1}^2 \frac{\rho}{|\tilde x_0|} \sup_{B^g_{\rho}(\tilde X_0)} |f^i| 
	+ \sum_{i=3}^n \frac{\rho}{|\tilde x_0|^2} \sup_{B^g_{\rho}(\tilde X_0)} |f^i| 
	+ \sum_{i=1}^2 \frac{\rho^{1+\lambda}}{|\tilde x_0|} [f^i]_{g,\lambda,B^g_{\rho}(\tilde X_0)} 
	+ \sum_{i=3}^n \frac{\rho^{1+\lambda}}{|\tilde x_0|^2} [f^i]_{g,\lambda,B^g_{\rho}(\tilde X_0)} \right)
\end{align*}
for every $\tilde X_0 = (\tilde x_0,\tilde y_0) \in B^g_1(0)$ and $0 < \rho \leq \min\{|\tilde x_0|^2/4,1-d_g(\tilde X_0,0)\}$ for some constant $C = C(n,k,\lambda) \in (0,\infty)$. 

(ii) Let $k \geq 2$ be an integer and $\lambda \in (0,1)$.  Let $\psi \in C^{k,\lambda}(B^g_1(0) \setminus \{0\} \times \mathbb{R}^{n-2})$, and $f \in C^{k-2,\lambda}(B^g_1(0) \setminus  \{0\} \times \mathbb{R}^{n-2})$.  Suppose that $\psi$ is a solution to (\ref{schauder1_eqn1}).  Then 
\begin{align*}
	&\sum_{|\alpha|+|\beta| \leq k} \frac{\rho^{|\alpha|+|\beta|}}{|\tilde x_0|^{|\alpha|}} 
		\sup_{B^g_{\rho/2}(\tilde X_0)} |D_{\tilde x}^{\alpha} D_{\tilde y}^{\beta} \psi| 
	+ \sum_{|\alpha|+|\beta| = k} \frac{\rho^{|\alpha|+|\beta|+\lambda}}{|\tilde x_0|^{|\alpha|}} 
		[D_{\tilde x}^{\alpha} D_{\tilde y}^{\beta} \psi]_{g,\lambda,B^g_{\rho/2}(\tilde X_0)} 
	\leq C \sup_{B^g_{\rho}(\tilde X_0)} |\psi| \\& + C \left( \sum_{|\alpha|+|\beta| \leq k-2} \frac{\rho^{2+|\alpha|+|\beta|}}{|\tilde x_0|^{2+|\alpha|}} 
		\sup_{B^g_{\rho}(\tilde X_0)} |D_{\tilde x}^{\alpha} D_{\tilde y}^{\beta} f| 
	+ \sum_{|\alpha|+|\beta| = k-2} \frac{\rho^{2+|\alpha|+|\beta|+\lambda}}{|\tilde x_0|^{2+|\alpha|}} 
		[D_{\tilde x}^{\alpha} D_{\tilde y}^{\beta} f]_{g,\lambda,B^g_{\rho}(\tilde X_0)} \right)
\end{align*}
for every $\tilde X_0 = (\tilde x_0,\tilde y_0) \in B^g_1(0)$ and $0 < \rho \leq \min\{|\tilde x_0|^2/4,1-d_g(\tilde X_0,0)\}$ for some constant $C = C(n,k,\lambda) \in (0,\infty)$. 
\end{lemma}
\begin{proof}
Apply the ordinary Schauder estimates to the rescaled function $\psi(\tilde x_0 + \rho \tilde x/2|\tilde x_0|,\tilde y_0 + \rho \tilde y)$, using the fact that $(\tilde x_0 + \rho \tilde x/2|\tilde x_0|,\tilde y_0 + \rho \tilde y) \in B^g_{\rho/2}(\tilde X_0) \Rightarrow (\tilde x,\tilde y) \in B^n_{1/2}(0)$ and $(\tilde x,\tilde y) \in B_{4/5}(0) \Rightarrow (\tilde x_0 + \rho \tilde x/2|\tilde x_0|,\tilde y_0 + \rho \tilde y) \in B^g_{\rho}(\tilde X_0)$. 
\end{proof}

\begin{lemma}[Interpolation inequality] \label{interplemma}
For every $\eta > 0$, integer $k \geq 1$, and $\tau \in (0,1/2]$, there exists a constant $C = C(n,m,k,\tau,\eta) \in (0,\infty)$ such that for any $\psi \in C^{g,k,\tau}(B^g_{\rho}(0,\tilde y_0);\mathbb{R}^m)$, where $\tilde y_0 \in \mathbb{R}^{n-2}$ and $\rho > 0$, such that (\ref{schauder_evenodd}) holds true for all $(\tilde x,\tilde y) \in B^g_1(0)$, 
\begin{equation} \label{schauderint}
	\sum_{|\alpha|+2|\beta| \leq k} \rho^{|\alpha|/2+|\beta|} \sup_{B^g_{\rho}(0,\tilde y_0)} |D^{\alpha}_{\tilde x} D^{\beta}_{\tilde y} \psi|
	\leq C \sup_{B^g_{\rho}(0,\tilde y_0)} |\psi| 
	+ \eta \sum_{|\alpha|+2|\beta| = k} \rho^{k/2+\tau} [D^{\alpha}_{\tilde x} D^{\beta}_{\tilde y} \psi]_{g,\tau,B^g_{\rho}(0,\tilde y_0)}. 
\end{equation}
\end{lemma}
\begin{proof} 
WLOG assume $\tilde y_0 = 0$, $\rho = 1$, and $m = 1$.  We shall first consider the cases $k = 1,2$ and then prove the cases $k \geq 3$ by induction.  For the case $k = 1$, we want to show that for every $\psi \in C^{g,1,\tau}(B^g_1(0);\mathbb{R}^m)$, 
\begin{equation} \label{schauderint1}
	\sup_{B^g_1(0)} |D_{\tilde x} \psi| \leq C \sup_{B^g_1(0)} |\psi| + \eta [D_{\tilde x} \psi]_{g,\tau,B^g_1(0)}. 
\end{equation}
To see this, observe that given $\sigma \in (0,1)$ and $j \in \{1,2\}$, we can choose a ball $B^g_{\sigma}(\tilde Y) \subset B^g_1(0)$ such that  
\begin{equation*}
	\sup_{B^g_1(0)} |D_j \psi| = \sup_{B^g_{\sigma}(\tilde Y)} |D_j \psi| 
\end{equation*}
so that by the definition of $[D_j \psi]_{g,\tau,B^g_1(0)}$ and the mean value theorem, 
\begin{align*}
	\sup_{B^g_1(0)} |D_j \psi| &\leq \inf_{B^g_{\sigma}(\tilde Y)} |D_j \psi| + (2\sigma)^{\tau} [D_j \psi]_{g,\tau,B^g_1(0)} \\
	&\leq (2\sigma)^{-1} \sup_{B^g_1(0)} |\psi| + (2\sigma)^{\tau} [D_j \psi]_{g,\tau,B^g_1(0)} 
\end{align*}
and thus (\ref{schauderint1}) follows by choosing $\sigma = \sigma(\tau,\eta)$ sufficiently small.  By this same type of argument, for every $\eta > 0$ there exists $C = C(n,\tau,\eta) \in (0,\infty)$ such that for every $\psi \in C^{g,2,\tau}(B^g_1(0);\mathbb{R}^m)$, 
\begin{equation} \label{schauderint2}
	\sup_{B^g_1(0)} |D_{\tilde y} \psi| \leq C \sup_{B^g_1(0)} |\psi| + \eta [D_{\tilde y} \psi]_{g,\tau,B^g_1(0)}. 
\end{equation}

In light of (\ref{schauderint2}), to complete proof for the case $k = 2$, we need to show that for every $\eta > 0$ there exists $C = C(n,\tau,\eta) \in (0,\infty)$ such that for every $\psi \in C^{g,2,\tau}(B^g_1(0);\mathbb{R}^m)$ such that (\ref{schauder_evenodd}) holds true for all $(\tilde x,\tilde y) \in B^g_1(0)$, 
\begin{equation} \label{schauderint3}
	\sup_{B^g_1(0)} |D_{\tilde x} \psi| + \sup_{B^g_1(0)} |D_{\tilde x \tilde x} \psi| 
	\leq C \sup_{B^g_1(0)} |\psi| + \eta [D_{\tilde x \tilde x} \psi]_{g,\tau,B^g_1(0)} + \eta [D_{\tilde y} \psi]_{g,\tau,B^g_1(0)}. 
\end{equation}
First observe that by (\ref{schauder_evenodd}) and (\ref{schauderint1}) with $D_{\tilde x} \psi$ in place of $\psi$, 
\begin{equation} \label{schauderint4}
	\sup_{B^g_1(0)} |D_{\tilde x} \psi| + \sup_{B^g_1(0)} |D_{\tilde x \tilde x} \psi| 
	\leq 2 \sup_{B^g_1(0)} |D_{\tilde x \tilde x} \psi| 
	\leq C(n,\tau,\eta) \sup_{B^g_1(0)} |D_{\tilde x} \psi| + \frac{\eta}{2} [D_{\tilde x \tilde x} \psi]_{g,\tau,B^g_1(0)}. 
\end{equation}
Let $\sigma > 0$.  If $\tilde y_0 \in B^{n-2}_1(0)$ with $B^2_{\sigma}(0) \times \{\tilde y_0\} \subset B^g_1(0)$, then 
\begin{equation} \label{schauderint5}
	\sup_{B^2_{\sigma}(0) \times \{\tilde y_0\}} |D_{\tilde x} \psi| \leq \sigma \sup_{B^g_1(0)} |D_{\tilde x \tilde x} \psi|. 
\end{equation}
If instead $\tilde y_0 \in B^{n-2}_1(0)$ with $B^2_{\sigma}(0) \times \{\tilde y_0\} \not\subset B^g_1(0)$, then by standard interpolation inequalities for every $\varepsilon > 0$,  
\begin{equation} \label{schauderint6}
	\sup_{B^2_{\sigma}(0) \times \{\tilde y_0\}} |D_{\tilde x} \psi| 
	\leq C(n,\tau,\eta,\sigma,\varepsilon) \sup_{B^g_1(0)} |\psi| + \varepsilon \sup_{B^g_1(0)} |D_{\tilde x \tilde x} \psi|. 
\end{equation}
By combining (\ref{schauderint4}), (\ref{schauderint5}), and (\ref{schauderint6}) and choosing $\sigma = \sigma(n,\tau,\eta)$ and $\varepsilon = \sigma(n,\tau,\eta)$ sufficiently small, we obtain (\ref{schauderint3}). 

Now suppose that (\ref{schauderint}) holds true whenever $k < K$ for some integer $K \geq 3$.  Let $\eta > 0$ and $\psi \in C^{g,K,\tau}(B^g_1(0);\mathbb{R}^m)$.  If $K$ is odd, then by (\ref{schauder_evenodd}), the induction hypothesis applied to $D_{\tilde x} \psi$, and (\ref{schauderint1}), 
\begin{align*}
	&\sum_{|\alpha|+2|\beta| \leq K} \sup_{B^g_1(0)} |D^{\alpha}_{\tilde x} D^{\beta}_{\tilde y} \psi| 
	\\&\leq C(n,K,\tau,\eta) \sup_{B^g_1(0)} |D_{\tilde x} \psi| 
		+ \frac{\eta}{2} \sum_{|\alpha|+2|\beta| = K} [D^{\alpha}_{\tilde x} D^{\beta}_{\tilde y} \psi]_{g,\tau,B^g_1(0)}
	\\&\leq C(n,K,\tau,\eta) \sup_{B^g_1(0)} |\psi| + \frac{1}{\sqrt{32}} [D_{\tilde x} \psi]_{g,1/2,B^g_1(0)} 
		+ \frac{\eta}{2} \sum_{|\alpha|+2|\beta| = K} [D^{\alpha}_{\tilde x} D^{\beta}_{\tilde y} \psi]_{g,\tau,B^g_1(0)}
	\\&\leq C(n,K,\tau,\eta) \sup_{B^g_1(0)} |\psi| 
		+ \frac{1}{2} \left( \sup_{B^g_1(0)} |D_{\tilde x \tilde x \tilde x} \psi| + \sup_{B^g_1(0)} |D_{\tilde x \tilde y} \psi| \right) 
		+ \frac{\eta}{2} \sum_{|\alpha|+2|\beta| = K} [D^{\alpha}_{\tilde x} D^{\beta}_{\tilde y} \psi]_{g,\tau,B^g_1(0)}, 
\end{align*}
which implies (\ref{schauderint}) with $k = K$.  If $K$ is even, then by (\ref{schauder_evenodd}), the induction hypothesis applied to $D_{\tilde x \tilde x} \psi$ and $D_{\tilde y} \psi$, (\ref{schauderint2}), and (\ref{schauderint3}), 
\begin{align*}
	&\sum_{|\alpha|+2|\beta| \leq K} \sup_{B^g_1(0)} |D^{\alpha}_{\tilde x} D^{\beta}_{\tilde y} \psi| 
	\\&\leq C(n,K,\tau,\eta) \left( \sup_{B^g_1(0)} |\psi| + \sup_{B^g_1(0)} |D_{\tilde x \tilde x} \psi| + \sup_{B^g_1(0)} |D_{\tilde y} \psi| \right)
		+ \frac{\eta}{2} \sum_{|\alpha|+2|\beta| = K} [D^{\alpha}_{\tilde x} D^{\beta}_{\tilde y} \psi]_{g,\tau,B^g_1(0)}
	\\&\leq C(n,K,\tau,\eta) \sup_{B^g_1(0)} |\psi| 
		+ \frac{1}{\sqrt{32}} \left( [D_{\tilde x \tilde x} \psi]_{g,1/2,B^g_1(0)} + [D_{\tilde y} \psi]_{g,1/2,B^g_1(0)} \right) 
		+ \frac{\eta}{2} \sum_{|\alpha|+2|\beta| = K} [D^{\alpha}_{\tilde x} D^{\beta}_{\tilde y} \psi]_{g,\tau,B^g_1(0)}
	\\&\leq C(n,K,\tau,\eta) \sup_{B^g_1(0)} |\psi| 
		+ \frac{1}{2} \sum_{|\alpha|+2|\beta| = 4} \sup_{B^g_1(0)} |D^{\alpha}_{\tilde x} D^{\beta}_{\tilde y} \psi|
		+ \frac{\eta}{2} \sum_{|\alpha|+2|\beta| = K} [D^{\alpha}_{\tilde x} D^{\beta}_{\tilde y} \psi]_{g,\tau,B^g_1(0)}, 
\end{align*}
which implies (\ref{schauderint}) with $k = K$.  
\end{proof}

\begin{lemma}[Abstract covering lemma] \label{SALemma}
Given $l \geq 0$, there exists $\delta = \delta(n,l) > 0$ such that if $\mathcal{C}$ is a collection of subsets of $B^g_1(0)$ containing all geodesic $g$-balls in $B^g_1(0)$ and $S : \mathcal{C} \rightarrow [0,\infty)$ such that $S$ is subadditive in the sense that 
\begin{equation*}
	S(A) \leq \sum_{j=1}^N S(A_j) \text{ whenever } A \subseteq \bigcup_{j=1}^N A_j
\end{equation*}
for $A,A_1,A_2,\ldots,A_N \in \mathcal{C}$ with and for some $\nu \geq 0$, 
\begin{equation} \label{SAL_eqn1}
	\rho^l S(B^g_{\rho/2}(\tilde x_0,\tilde y_0)) \leq \nu 
\end{equation}
for all $B^g_{\rho}(\tilde x_0,\tilde y_0) \subseteq B^g_1(0)$ with $|\tilde x_0|^2/32 \leq \rho \leq |\tilde x_0|^2/4$ and 
\begin{equation} \label{SAL_eqn2}
	\rho^l S(B^g_{\rho/2}(0,\tilde y_0)) \leq \delta \rho^l S(B^g_{\rho}(0,\tilde y_0)) + \nu 
\end{equation}
for all $\tilde y_0 \in B^{n-2}_1(0)$, then 
\begin{equation} \label{SAL_eqn3}
	\rho^l S(B^g_{\rho/2}(\tilde x_0,\tilde y_0)) \leq C \nu 
\end{equation}
for some constant $C = C(n,l) \in (0,\infty)$.  
\end{lemma}
\begin{proof}
Define 
\begin{equation*}
	Q = \sup_{\tilde y \in B^{n-2}_1(0), \, \rho \in (0,1-|\tilde y|]} \rho^l S(B^g_{\rho/2}(0,\tilde y_0)) .
\end{equation*}
Select $\tilde y \in B^{n-2}_1(0)$ and $\rho \in (0,1-|\tilde y|]$.  Cover $B^2_{\rho^{1/2}/4}(0) \times B^{n-2}_{\rho/2}(\tilde y)$ by balls $B^g_{\rho/8}(0,\tilde z_k)$, $k = 1,2,\ldots,N$, with $\tilde z_k \in B^{n-2}_{\rho/2}(0)$ and $N \leq C(n)$.  Cover $B^g_{\rho/2}(0,\tilde y) \setminus (B^2_{\rho^{1/2}/4}(0) \times B^{n-2}_{\rho/2}(\tilde y))$ by balls $B^g_{\rho/64}(\tilde x_k,\tilde y_k)$, $k = 1,2,\ldots,N'$, with $\tilde x_k \in \mathbb{R}^2$ and $\tilde y_k \in \mathbb{R}^{n-2}$ such that $(\tilde x_k,\tilde y_k) \in B^g_{\rho/2}(0,y)$ and $|\tilde x_k| \geq \rho^{1/2}/4$ and $N' \leq C(n)$.  By (\ref{SAL_eqn1}) and (\ref{SAL_eqn2}), the subadditivity of $S$, and the definition of $Q$, 
\begin{equation} \label{SAL_eqn4}
	Q \leq 4^l N \delta \, Q + (4^l N + 32^l N') \nu.
\end{equation}
Recalling that $N, N' \leq C(n)$ and then choosing $\delta = \delta(n) > 0$ so that $4^l N \delta < 1/2$, (\ref{SAL_eqn4}) yields 
\begin{equation*}
	Q \leq C \nu 
\end{equation*}
for some constant $C = C(n,l) \in (0,\infty)$, which by the definition of $Q$ implies (\ref{SAL_eqn3}).
\end{proof}

\begin{proof}[Proof of Lemma \ref{schauderlemma}]  We claim that in order to prove Lemma \ref{schauderlemma} it suffices to show that for every $\delta > 0$ there exists $\varepsilon = \varepsilon(n,\tau,\nu,\delta) > 0$ and $C = C(n,\tau,\nu,\delta) \in (0,\infty)$ such that under the hypotheses of Lemma \ref{schauderlemma}, 
\begin{align} \label{schauder1_step1_eqn1}
	&\sum_{|\alpha|+2|\beta| = k} [D_{\tilde x}^{\alpha} D_{\tilde y}^{\beta} \psi]_{g,\tau,B^g_{1/2}(0)} 
	\leq \delta \sum_{|\alpha|+2|\beta| = k} [D_{\tilde x}^{\alpha} D_{\tilde y}^{\beta} \psi]_{g,\tau,B^g_1(0)} \nonumber \\
	&+ C \left( \sum_{|\alpha|+2|\beta| \leq k} \sup_{B^g_1(0)} |D_{\tilde x}^{\alpha} D_{\tilde y}^{\beta} \psi| + \|f\|_{\mathcal{H}^{k-2,k-2,\tau}(B^g_1(0))} \right) 
\end{align}
Suppose that $\psi$ satisfies the hypotheses of Lemma \ref{schauderlemma} and suppose that (\ref{schauder1_step1_eqn1}) holds true whenever the hypotheses of Lemma \ref{schauderlemma} hold true.  By applying the translation and scaling $(\tilde x,\tilde y) \mapsto (\rho^{1/2} \tilde x, \tilde y_0 + \rho \tilde y)$ to (\ref{schauder1_step1_eqn1}), we conclude that under the hypotheses of Lemma \ref{schauderlemma}, 
\begin{align} \label{schauder1_step1_eqn2}
	&\sum_{|\alpha|+2|\beta| = k} \rho^{k/2+\tau} [D_{\tilde x}^{\alpha} D_{\tilde y}^{\beta} \psi]_{g,\tau,B^g_{\rho/2}(0,\tilde y_0)} 
	\leq \delta \sum_{|\alpha|+2|\beta| = k} \rho^{k/2+\tau} [D_{\tilde x}^{\alpha} D_{\tilde y}^{\beta} \psi]_{g,\tau,B^g_{\rho}(0,\tilde y_0)} \nonumber \\
	&+ C \left( \sum_{|\alpha|+2|\beta| \leq k} \rho^{|\alpha|/2+|\beta|} \sup_{B^g_{\rho}(0,\tilde y_0)} |D_{\tilde x}^{\alpha} D_{\tilde y}^{\beta} \psi| 
		+ \rho^{k/2} \|f\|_{\mathcal{H}^{k-2,k-2,\tau}(B^g_1(0))} \right) 
\end{align}
for every $\tilde y_0 \in B^{n-2}_1(0)$ and $\rho \in (0,1-|\tilde y_0|]$ for some $C = C(n,\tau,\nu,\delta) \in (0,\infty)$.  By the interpolation inequalities of Lemma \ref{interplemma}, (\ref{schauder1_step1_eqn2}) implies that 
\begin{align} \label{schauder1_step1_eqn3}
	&\sum_{|\alpha|+2|\beta| = k} \rho^{k/2+\tau} [D_{\tilde x}^{\alpha} D_{\tilde y}^{\beta} \psi]_{g,\tau,B^g_{\rho/2}(0,\tilde y_0)} 
	\leq 2\delta \sum_{|\alpha|+2|\beta| = k} \rho^{k/2+\tau} [D_{\tilde x}^{\alpha} D_{\tilde y}^{\beta} \psi]_{g,\tau,B^g_{\rho}(0,\tilde y_0)} \nonumber \\
	&+ C \left( \sup_{B^g_{\rho}(0,\tilde y_0)} |\psi| + \|f\|_{\mathcal{H}^{k-2,k-2,\tau}(B^g_1(0))} \right) 
\end{align}
for every $\tilde y_0 \in B^{n-2}_1(0)$ and $\rho \in (0,1-|\tilde y_0|]$ for some $C = C(n,\tau,\nu,\delta) \in (0,\infty)$.  By Lemma \ref{SALemma} using the Schauder estimate of Lemma \ref{schauder_afbs} and (\ref{schauder1_step1_eqn3}) and then the interpolation inequalities of Lemma \ref{interplemma}, (\ref{schauder1_step1_eqn1}) holds true. 

Now suppose that some $\delta \in (0,1)$ and for $j = 1,2,3,\ldots$ there are a sequence of $\psi_j$ and $f_j$ such that the hypotheses of Lemma \ref{schauderlemma} hold true with $\varepsilon = 1/j$, $\psi = \psi_j$, and $f = f_j$ but 
\begin{align} \label{schauder1_step2_eqn1}
	&\sum_{|\alpha|+2|\beta| = k} [D_{\tilde x}^{\alpha} D_{\tilde y}^{\beta} \psi_j]_{g,\tau,B^g_{1/2}(0)} 
	> \delta \sum_{|\alpha|+2|\beta| = k} [D_{\tilde x}^{\alpha} D_{\tilde y}^{\beta} \psi_j]_{g,\tau,B^g_1(0)} \nonumber \\
	&+ j \left( \sum_{|\alpha|+2|\beta| \leq k} \sup_{B^g_1(0)} |D_{\tilde x}^{\alpha} D_{\tilde y}^{\beta} \psi_j| 
	+ \|f_j\|_{\mathcal{H}^{k-2,k-2,\tau}(B^g_1(0))} \right) 
\end{align}
By scaling we may assume that 
\begin{equation} \label{schauder1_step2_eqn2}
	\sum_{|\alpha|+2|\beta| = k} [D_{\tilde x}^{\alpha} D_{\tilde y}^{\beta} \psi_j]_{g,\tau,B^g_{1/2}(0)} = \delta.
\end{equation}
For each $j$, select distinct $\tilde X_j = (\tilde x_j,\tilde y_j), \tilde X'_j = (\tilde x'_j,\tilde y'_j) \in B^g_{1/2}(0) \setminus \{0\} \times \{0\} \times B^{n-2}_{1/2}(0)$ such that 
\begin{equation} \label{schauder1_step2_eqn3}
	\sum_{|\alpha|+2|\beta| = k} |D_{\tilde x}^{\alpha} D_{\tilde y}^{\beta} \psi_j(\tilde X_j) - D_{\tilde x}^{\alpha} D_{\tilde y}^{\beta} \psi_j(\tilde X'_j)| 
	> \frac{\delta}{2} d_g(\tilde X_j, \tilde X'_j)^{\tau} 
\end{equation}
and let $\rho_j = d_g(\tilde X_j, \tilde X'_j)$.  By (\ref{schauder1_step2_eqn1}) and (\ref{schauder1_step2_eqn3}), $\rho_j \rightarrow 0$ as $j \rightarrow \infty$.  Moreover, $|\tilde x_j| \rightarrow 0$ since otherwise we can pass to a subsequence such that $\inf_j |\tilde x_j| > 0$ and then by (\ref{schauder1_eqn1}) with $\psi = \psi_j$ and $f = f_j$ and the Schauder estimates of Lemma \ref{schauder_afbs}, 
\begin{equation*}
	\lim_{j \rightarrow \infty} \sum_{|\alpha|+2|\beta| = k} [D_{\tilde x}^{\alpha} D_{\tilde y}^{\beta} \psi_j]_{g,\tau,B^g_{|\tilde x_j|^2/8}(\tilde X_j)} = 0, 
\end{equation*}
contradicting (\ref{schauder1_step2_eqn3}).

Suppose that, after passing to a subsequence, $\rho_j/|\tilde x_j|^2 \rightarrow 0$ as $j \rightarrow \infty$.  By (\ref{schauder1_eqn1}) and (\ref{schauder1_eqn2}) with $\psi = \psi_j$ and $f = f_j$, 
\begin{equation*}
	\psi(\tilde X) = \psi_j(\tilde X) - \sum_{|\alpha|+2|\beta| \leq k} D^{\alpha}_{\tilde x} D^{\beta}_{\tilde y} \psi_j(0,\tilde y_j) 
		\tilde x^{\alpha} (\tilde y - \tilde y_j)^{\beta}    
\end{equation*}
satisfies 
\begin{equation*}
	\Delta_{\tilde x} \psi + 4|\tilde x|^2 \Delta_{\tilde y} \psi = f_j(\tilde X) - f_j(\tilde X; \tilde y_j) 
\end{equation*}
in $B^g_1(0)$, so by (\ref{schauder1_step2_eqn1}), (\ref{schauder1_step2_eqn2}), and the Schauder estimates of Lemma \ref{schauder_afbs} with $\lambda = \tau$, 
\begin{align} \label{schauder1_step2_eqn4}
	&\sum_{|\alpha|+2|\beta| > k, \, |\alpha|+|\beta| \leq k} |\tilde x_j|^{|\alpha|+2|\beta|} \sup_{B^g_{|\tilde x_j|^2/8}(\tilde X_j)} 
		|D_{\tilde x}^{\alpha} D_{\tilde y}^{\beta} \psi_j| \nonumber \\
	&+ \sum_{|\alpha|+|\beta| = k} |\tilde x_j|^{|\alpha|+2|\beta|+2\tau} 
		[D_{\tilde x}^{\alpha} D_{\tilde y}^{\beta} \psi_j]_{g,\tau,B^g_{|\tilde x_j|^2/8}(\tilde X_j)} 
	\leq C |\tilde x|^{k+2\tau}
\end{align}
for some $C = C(n,\tau) \in (0,\infty)$.  Define 
\begin{equation*}
	\chi_j(\tilde X) = \frac{|\tilde x_j|^k}{\rho_j^{k+\tau}} 
		\left( \psi_j \left( \tilde x_j + \frac{\rho_j}{|\tilde x_j|} \tilde x, \tilde y_j + \rho_j \tilde y \right) 
		- \sum_{|\alpha|+|\beta| \leq k} \frac{1}{\alpha! \beta!} \frac{\rho_j^{|\alpha|+|\beta|}}{|\tilde x_j|^{|\alpha|}} 
		D_{\tilde x}^{\alpha} D_{\tilde y}^{\beta} \psi_j(\tilde X_j) \tilde x^{\alpha} \tilde y^{\beta} \right) .
\end{equation*}
By (\ref{schauder1_step2_eqn4}) and Taylor's theorem, for every $R > 0$ and $j$ sufficiently large depending on $R$, we have the bound on a standard Euclidean H\"{o}lder coefficient of $[D_{\tilde X}^k \chi_j]_{\tau,B_R(0)} \leq C$ for some $C = C(n,\tau) \in (0,\infty)$ independent of $j$, thus after passing to a subsequence $\chi_j$ converges to some function $\chi$ in $C^k$ on compact subsets of $\mathbb{R}^n$ such that $D_{\tilde X}^l \chi(0) = 0$ for $l = 1,2,\ldots,k$ and $[D_{\tilde X}^k \chi]_{\tau,\mathbb{R}^n} < \infty$.  Moreover by (\ref{schauder1_eqn1}), $\chi_j$ satisfies 
\begin{align*}
	&\Delta_{\tilde x} \chi_j + \frac{|\tilde x_j + \rho_j \tilde x/|\tilde x_j||^2}{|\tilde x_j|^2} \Delta_{\tilde y} \chi_j 
	\\&= \frac{|\tilde x_j|^{k-2}}{\rho_j^{k-2+\tau}} \left( f_j\left( \tilde x_j + \frac{\rho_j}{|\tilde x_j|} \tilde x, \tilde y_j + \rho_j \tilde y \right) 
		\right. \\& \hspace{5mm} \left.  
		- \sum_{|\alpha|+|\beta| \leq k-2} \frac{1}{\alpha! \beta!} \frac{\rho_j^{|\alpha|+|\beta|}}{|\tilde x_j|^{|\alpha|}} 
		(D_{\tilde X}^{\alpha} \Delta_{\tilde x} \psi_j(\tilde X_j) + 4|\tilde x_j + \rho_j \tilde x/|\tilde x_j||^2 D_{\tilde x}^{\alpha} D_{\tilde y}^{\beta} 
		\Delta_{\tilde y} \psi_j(\tilde X_j)) \tilde x^{\alpha} \tilde y^{\beta} \right)
	\\&= \frac{|\tilde x_j|^{k-2}}{\rho_j^{k-2+\tau}} \left( f_j \left( \tilde x_j + \frac{\rho_j}{|\tilde x_j|} \tilde x, \tilde y_j + \rho_j \tilde y \right) 
		- \sum_{|\alpha|+|\beta| \leq k-2} \frac{1}{\alpha! \beta!} \frac{\rho_j^{|\alpha|+|\beta|}}{|\tilde x_j|^{|\alpha|}} 
		D_{\tilde x}^{\alpha} D_{\tilde y}^{\beta} f_j(\tilde X_j) \tilde x^{\alpha} \tilde y^{\beta} \right. \\& \left. \hspace{5mm} 
		- \sum_{|\alpha|+|\beta| = k-2} \frac{8}{\alpha! \beta!} \frac{\rho_j^{|\alpha|+|\beta|+1}}{|\tilde x_j|^{|\alpha|+1}} 
		D_{\tilde x}^{\alpha} D_{\tilde y}^{\beta} \Delta_{\tilde y} \psi_j(\tilde X_j) (\tilde x_j \cdot \tilde x) \tilde x^{\alpha} \tilde y^{\beta} 
		\right. \\& \left. \hspace{5mm} 
		- \sum_{|\alpha|+|\beta| = k-3,k-2} \frac{4}{\alpha! \beta!} \frac{\rho_j^{|\alpha|+|\beta|+2}}{|\tilde x_j|^{|\alpha|+2}} |\tilde x|^2 
		D_{\tilde x}^{\alpha} D_{\tilde y}^{\beta} \Delta_{\tilde y} \psi_j(\tilde X_j) \tilde x^{\alpha} \tilde y^{\beta} \right) ,
\end{align*}
so by (\ref{schauder1_step2_eqn2}) and (\ref{schauder1_step2_eqn4}) $\Delta_{\tilde x} \chi + \Delta_{\tilde y} \chi = 0$ on $\mathbb{R}^n$.  But by Liousville's Theorem no such function $\chi$ can exist.  Therefore 
\begin{equation} \label{schauder1_step2_eqn5}
	\sup_j |\tilde x_j|^2/\rho_j < \infty. 
\end{equation}

Now let 
\begin{align*}
	\chi_j(\tilde x, \tilde y) = {}& \rho_j^{-k/2-\tau} \left( \psi_j(\rho_j^{1/2} \tilde x, y_j + \rho_j \tilde y) 
		- \sum_{|\alpha|+2|\beta| \leq k} \frac{1}{\alpha! \beta!} \rho_j^{|\alpha|/2+|\beta|} D_{\tilde x}^{\alpha} D_{\tilde y}^{\beta} \psi_j(0,y_j) 
		\tilde x^{\alpha} \tilde y^{\beta} \right) 
\end{align*}
for $(\tilde x,\tilde y) \in B^g_{1/2\rho_j}(0)$.  By (\ref{schauder1_eqn1}) and (\ref{schauder1_eqn2}) with $\psi = \psi_j$ and $f = f_j$, 
\begin{equation*}
	\Delta_{\tilde x} \chi_j + 4|\tilde x|^2 \Delta_{\tilde y} \chi_j 
	= \rho_j^{-\tau} (f_j(\rho_j^{1/2} \tilde x, y_j + \rho_j \tilde y) - f_j(\rho_j^{1/2} \tilde x; \tilde y_j)) 
\end{equation*}
on $B^g_1(0) \setminus \{0\} \times B^{n-2}_1(0)$.  Thus by the Schauder estimates, (\ref{schauder1_step2_eqn1}), and (\ref{schauder1_step2_eqn2}), after passing to a subsequence $\chi_j$ converges to a function $\chi$ in $C^k$ on compact subsets of $\mathbb{R}^n \setminus \{0\} \times \mathbb{R}^{n-2}$ as $j \rightarrow \infty$ and $\chi$ satisfies 
\begin{equation} \label{schauder1_step2_eqn6}
	\Delta_{\tilde x} \chi + 4|\tilde x|^2 \Delta_{\tilde y} \chi = 0 
\end{equation}
on $\mathbb{R}^n \setminus \{0\} \times \mathbb{R}^{n-2}$.  By (\ref{schauder1_step2_eqn2}) and Arzela-Ascoli, after passing to a subsequence, $D_{\tilde x}^{\alpha} D_{\tilde y}^{\beta} \chi_j \rightarrow D_{\tilde x}^{\alpha} D_{\tilde y}^{\beta} \chi$ uniformly on compact subsets of $\mathbb{R}^n$ whenever $|\alpha|+2|\beta| \leq k$.  If $k$ is odd, then whenever $|\alpha|+2|\beta| = k-1$, 
\begin{equation*}
	D_{\tilde x}^{\alpha} D_{\tilde y}^{\beta} \chi_j(\tilde x,\tilde y) 
	= \int_0^1 D_{\tilde x} D_{\tilde x}^{\alpha} D_{\tilde y}^{\beta} \chi_j(t \tilde x,\tilde y) \cdot \tilde x dt
\end{equation*}
so $D_{\tilde x}^{\alpha} D_{\tilde y}^{\beta} \chi_j \rightarrow D_{\tilde x}^{\alpha} D_{\tilde y}^{\beta} \chi$ in $C^1$ on compact subsets of $\mathbb{R}^n$.  Since when $k \geq 2$ by Taylor's theorem 
\begin{align} \label{schauder1_step2_eqn7}
	\chi_j(\tilde x,\tilde y) 
	&= \sum_{|\beta| \leq N-1} \frac{1}{\beta!} D_{\tilde y}^{\beta} \chi_j(\tilde x,0) \, \tilde y^{\beta} 
		+ \sum_{|\beta| = N} \frac{N}{\beta!} \int_0^1 (1-t)^{N-1} D_{\tilde y}^{\beta} \chi_j(\tilde x,t\tilde y) dt \, \tilde y^{\beta} \nonumber \\
	&= \sum_{|\beta| \leq N-1} \sum_{|\alpha| = k-2|\beta|} \frac{k-2|\beta|}{\alpha! \beta!} 
		\int_0^1 (1-t)^{k-2|\beta|-1} D_{\tilde x}^{\alpha} D_{\tilde y}^{\beta} \chi_j(t\tilde x,0) dt \, \tilde x^{\alpha} \tilde y^{\beta} \nonumber \\& 
		\hspace{5mm} + \sum_{|\beta| = N} \frac{N}{\beta!} \int_0^1 (1-t)^{N-1} D_{\tilde y}^{\beta} \chi_j(\tilde x,t\tilde y) dt \, \tilde y^{\beta},  
\end{align}
where $N = k/2$ if $k$ is even and $N = (k-1)/2$ if $k$ is odd, $D_{\tilde x}^{\alpha} D_{\tilde y}^{\beta} \chi_j \rightarrow D_{\tilde x}^{\alpha} D_{\tilde y}^{\beta} \chi$ uniformly on compact subsets of $\mathbb{R}^n$ whenever $|\alpha|+2|\beta| \leq k$.  In particular, $\chi \in C^{g,k,\tau}(\mathbb{R}^n)$ with $\chi(-\tilde x,\tilde y) = (-1)^k \chi(\tilde x,\tilde y)$ on $\mathbb{R}^n$, $D_{\tilde x}^{\alpha} D_{\tilde y}^{\beta} \chi_j(0) = 0$ whenever $|\alpha|+2|\beta| \leq k$, and by (\ref{schauder1_step2_eqn1}) and (\ref{schauder1_step2_eqn2}) 
\begin{equation} \label{schauder1_step2_eqn8}
	\sum_{|\alpha|+2|\beta| \leq k} [D_{\tilde x}^{\alpha} D_{\tilde y}^{\beta} \chi]_{g,\tau,\mathbb{R}^n} \leq 1.
\end{equation}
By (\ref{schauder1_step2_eqn5}), after passing to a subsequence, $(\rho_j^{-1/2} \tilde x_j,0)$ and $(\rho_j^{-1/2} \tilde x'_j, \rho_j^{-1} (\tilde y'_j - \tilde y_j))$ converge to some $\xi,\xi' \in \mathbb{R}^n$ such that by (\ref{schauder1_step2_eqn3}), 
\begin{equation} \label{schauder1_step2_eqn9}
	\sum_{|\alpha|+2|\beta| = k} |D_{\tilde x}^{\alpha} D_{\tilde y}^{\beta} \chi(\xi) - D_{\tilde x}^{\alpha} D_{\tilde y}^{\beta} \chi(\xi')| 
	\geq \frac{\delta}{2}  
\end{equation}
for some $\xi,\xi' \in \mathbb{R}^n$. 

Suppose that $k$ is even.  Using the fact that $\chi(-\tilde x,\tilde y) = \chi(\tilde x,\tilde y)$, we define the single-valued function $\varphi : \mathbb{R}^n \rightarrow \mathbb{R}$ by 
\begin{equation} \label{schauder1_step3_eqn1}
	\varphi(X) = \chi(\op{Re}(x_1+ix_2)^{1/2},\op{Im}(x_1+ix_2)^{1/2}, x_3,\ldots,x_n)
\end{equation} 
for $X = (x_1,x_2,\ldots,x_n) \in \mathbb{R}^n$.  Since $\varphi \in C^0(\mathbb{R}^n)$ and $\varphi$ is harmonic on $\mathbb{R}^n \setminus \{0\} \times \mathbb{R}^{n-2}$, $\varphi \in C^{\infty}(\mathbb{R}^n)$ and $\varphi$ is harmonic on $\mathbb{R}^n$.  By applying Taylor's theorem to $\chi$ like we did for (\ref{schauder1_step2_eqn7}) and by (\ref{schauder1_step2_eqn8}), 
\begin{equation} \label{schauder1_step3_eqn2}
	|\varphi(X)| \leq C(n,k,\tau) |X|^{k/2+\tau}
\end{equation}
for all $X \in \mathbb{R}^n$, so by Liousville's theorem $\varphi$ must identically zero.  Thus $\chi$ is identically zero, contradicting (\ref{schauder1_step2_eqn9}).  

Suppose that $k$ is odd.  Using the fact that $\chi(-\tilde x,\tilde y) = -\chi(\tilde x,\tilde y)$, we define the symmetric two-valued function $\varphi : \mathbb{R}^n \rightarrow \mathbb{R}$ by (\ref{schauder1_step3_eqn1}).  By using Taylor's theorem like we did for (\ref{schauder1_step2_eqn7}), the formulas $\cos(\theta) = 2\cos(\theta/2)-1 = 1-2\sin(\theta/2)$ and $\sin(\theta) = 2\cos(\theta/2) \sin(\theta/2)$ for all $\theta \in \mathbb{R}$, and (\ref{schauder1_step2_eqn8}), 
\begin{equation} \label{schauder1_step3_eqn3}
	\left| \varphi(x,y) - \op{Re} \left( \sum_{p+q+|\beta| \leq (k-1)/2} c_{p,q,\beta}(y_0) |x|^p (x_1+i_2)^{1/2+q} (y-y_0)^{\beta} \right) \right| 
	\leq C(n,k,\tau) |X - (0,y_0)|^{k/2+\tau}
\end{equation}
(where $p,q$ are nonnegative integers) for all $x \in \mathbb{R}^2$ and $y,y_0 \in \mathbb{R}^{n-2}$ for some functions $c_{p,q,\beta} \in C^{0,\tau}(\mathbb{R}^{n-2};\mathbb{C})$ with $c_{p,q,\beta}(0) = 0$.  In particular, (\ref{schauder1_step3_eqn2}) holds true. 

Since $\varphi$ is harmonic on $\mathbb{R}^n \setminus \{0\} \times \mathbb{R}^{n-2}$, 
\begin{gather*}
	\int_{B_{\rho}(0,y_0)} |D\varphi|^2 \zeta 
		= \int_{\partial B_{\rho}(0,y_0)} \varphi D_R \varphi \zeta - \int_{B_{\rho}(0,y_0)} \varphi D\varphi \cdot D\zeta, \\
	\int_{B_{\rho}(0,y_0)} \sum_{i,j=1}^n (|D\varphi|^2 \delta_{ij} - 2 D_i \varphi D_j \varphi) D_i (X^j \zeta) 
		= \rho \int_{\partial B_{\rho}(0,y_0)} (|D\varphi|^2 - 2 |D_R \varphi|^2) \zeta, 
\end{gather*}
for all $y_0 \in \mathbb{R}^{n-2}$, $\rho > 0$, and $\zeta \in C^{\infty}(\mathbb{R}^n)$ such that $\zeta \equiv 0$ near $\{0\} \times \mathbb{R}^{n-2}$.  By letting $\zeta$ be smooth functions approximating $\zeta_{\delta}$ given by $\zeta_{\delta}(x,y) = 0$ if $|x| \leq \delta/2$, $\zeta_{\delta}(x,y) = 2|x|/\delta - 1$ if $\delta/2 < |x| < \delta$, and $\zeta_{\delta}(x,y) = 1$ if $|x| \leq \delta$ and then letting $\delta \downarrow 0$ using (\ref{schauder1_step3_eqn3}), we obtain 
\begin{gather}
	\int_{B_{\rho}(0,y_0)} |D\varphi|^2 = \int_{\partial B_{\rho}(0,y_0)} \varphi D_R \varphi, \nonumber \\
	(n-2) \int_{B_{\rho}(0,y_0)} |D\varphi|^2 = \rho \int_{\partial B_{\rho}(0,y_0)} (|D\varphi|^2 - 2 |D_R \varphi|^2), \label{schauder1_step3_eqn4}
\end{gather}
for all $y_0 \in \mathbb{R}^{n-2}$ and $\rho > 0$.  Now we can consider the frequency function $N_{\varphi,(0,y_0)}(\rho)$ by (\ref{defn_freqfn}) with $u = \varphi$ and $Y = (0,y_0)$.  By (\ref{schauder1_step3_eqn4}), 
\begin{equation*}
	\frac{d}{d\rho} \left( \rho^{2-n} \int_{B_{\rho}(0,y_0)} |D\varphi|^2 \right) = 2 \rho^{2-n} \int_{\partial B_{\rho}(0,y_0)} |D_R \varphi|^2
\end{equation*}
and thus 
\begin{equation} \label{schauder1_step3_eqn5}
	N'_{\varphi,(0,y_0)}(\rho) = \frac{2\rho^{3-2n}}{H_{\varphi,(0,y_0)}(\rho)^2} 
		\left( \left( \int_{\partial B_{\rho}(0,y_0)} |\varphi|^2 \right) 
		\left( \int_{\partial B_{\rho}(0,y_0)} |D_R \varphi|^2 \right) - 
		\left( \int_{\partial B_{\rho}(0,y_0)} \varphi D_R \varphi \right)^2 \right) \geq 0
\end{equation}
for all $y_0 \in \mathbb{R}^{n-2}$ and $\rho > 0$ with equality if and only if $\varphi(x,y_0+y)$ is homogeneous degree $N_{\varphi,(0,y_0)}(\rho)$ as a function of $(x,y)$, where 
\begin{equation*}
	H_{\varphi,(0,y_0)}(\rho) = \rho^{1-n} \int_{\partial B_{\rho}(0,y_0)} |\varphi|^2. 
\end{equation*}
Moreover, by (\ref{schauder1_step3_eqn4}), 
\begin{equation} \label{schauder1_step3_eqn6}
	N_{\varphi,(0,y_0)}(\rho) = \frac{2 H'_{\varphi,(0,y_0)}(\rho)}{\rho H_{\varphi,(0,y_0)}(\rho)} 
\end{equation}
for all $y_0 \in \mathbb{R}^{n-2}$ and $\rho > 0$.  By (\ref{schauder1_step3_eqn2}), (\ref{schauder1_step3_eqn5}), and (\ref{schauder1_step3_eqn6}) with $y_0 = 0$, $\log( \rho^{-2k-2\tau} H_{\varphi,0}(\rho))$ is a bounded, convex function of $\log(\rho)$ and thus $\log( \rho^{-k-2\tau} H_{\varphi,0}(\rho))$ is constant with derivative $2N_{\varphi,0}(\rho) - k - 2\tau = 0$.  In other words, $N_{\varphi,0}(\rho) \equiv k/2+\tau$ and thus $\varphi$ is homogeneous degree $k+\tau$.  

Consider the Fourier coefficient of $\varphi$ given by 
\begin{equation*}
	\varphi_{l,\pm}(r,y) = \frac{1}{4\pi} \int_0^{4\pi} e^{\pm (1/2+l) \theta} \chi(r^2 e^{2i\theta},y) d\theta 
\end{equation*}
for $\pm \in \{-,+\}$ and $l \geq 0$ an integer and observe that since $\varphi$ is harmonic on $\mathbb{R}^n \setminus \{0\} \times \mathbb{R}^{n-2}$, 
\begin{equation} \label{schauder1_step3_eqn7}
	\frac{1}{r} \frac{\partial}{\partial r} \left( r \frac{\partial \varphi_{l,\pm}}{\partial r} \right) 
	- \frac{(1/2+l)^2}{r^2} \varphi_{l,\pm} + \Delta_y \varphi_{l,\pm} = 0 
\end{equation}
on $(0,\times) \times \mathbb{R}^{n-2}$.  Since $\varphi$ is homogeneous degree $k/2+\tau$, $\varphi_{l,\pm}(r,y) = r^{k/2+\tau} u(y/r)$ for some function $u \in C^{\infty}(\mathbb{R}^{n-2})$.  By (\ref{schauder1_step3_eqn7}), 
\begin{equation} \label{schauder1_step3_eqn8}
	\Delta_z u + \sum_{i,j=1}^{n-2} z_i z_j D_{z_i z_j} u - (k-1+2\tau) \sum_{i=1}^{n-2} z_i D_{z_i} u + ((k/2+\tau)^2 - (1/2+l)^2) u = 0 
\end{equation}
on $\mathbb{R}^{n-2}$.  By (\ref{schauder1_step3_eqn8}), 
\begin{equation} \label{schauder1_step3_eqn9}
	\Delta_z D^p u + \sum_{i,j=1}^{n-2} z_i z_j D_{z_i z_j} D^p u - (k-1-2p+2\tau) \sum_{i=1}^{n-2} z_i D_{z_i} D^p u + ((k/2-p+\tau)^2 - (1/2+l)^2) D^p u = 0 
\end{equation}
on $\mathbb{R}^{n-2}$ for every integer $p \geq 0$.  Take $v = D^{(k-1)/2-l} u$.  We want to claim that 
\begin{equation} \label{schauder1_step3_eqn010}
	\sup_{B_{\rho}(0) \setminus B_{\rho/2}(0)} |Dv| \leq C(n,k,l,\tau) \rho^{-1} \sup_{B_{2\rho}(0) \setminus B_{\rho/4}(0)} |v|
\end{equation}
for all $\rho \geq 1$.  If $n = 2$, then (\ref{schauder1_step3_eqn010}) follows by (\ref{schauder1_step3_eqn9}) and standard ODE theory.  If $n \geq 3$, then (\ref{schauder1_step3_eqn010}) does not generally hold true.  However, we can write $v(r\omega) = \sum_{k=1}^{\infty} \gamma_k(r) \phi_k(\omega)$, where $r > 0$, $\omega \in S^{n-3}$, and $\{\phi_k\}$ is an $L^2(S^{n-3})$-orthonormal basis of eigenfunctions satisfying $\Delta_{S^{n-3}} \phi_k + \lambda_k \phi_k = 0$ on $S^{n-3}$ for eigenvalues $\lambda_1 \leq \lambda_2 \leq \lambda_3 \leq \cdots$, and replace $v(r\omega)$ with $\gamma_k(r) \phi_k(\omega)$ for each $k$.  By (\ref{schauder1_step3_eqn010}), $\gamma_k$ satisfies the ODE 
\begin{equation*}
	(1+r^2) \frac{\partial^2 \gamma_k}{\partial r^2} - (r^{-1} + (-2l+2\tau) r) \frac{\partial \gamma_k}{\partial r} 
	- (\lambda_k r^{-2} + (-l+\tau)^2 - (1/2+l)^2) \gamma_k = 0 
\end{equation*}
and so by standard ODE theory (\ref{schauder1_step3_eqn010}) holds true.  Let $w = Dv$.  By (\ref{schauder1_step3_eqn3}) and (\ref{schauder1_step3_eqn010}), 
\begin{equation} \label{schauder1_step3_eqn011}
	|w(z)| \leq C(n,k,l,\tau) |z|^{-1}
\end{equation}
for all $z \in \mathbb{R}^{n-2}$.  Rewrite (\ref{schauder1_step3_eqn9}) with $p = (k+1)/2-l$ in polar coordinates as 
\begin{equation} \label{schauder1_step3_eqn10}
	\frac{\partial}{\partial r} \left( h \frac{\partial w}{\partial r} \right) + \frac{h}{r^2} \Delta_{S^{n-3}} w - (1-\tau)(2l+\tau) h w = 0
\end{equation}
weakly on $\mathbb{R}^{n-2}$ where $h = r (1+r^2)^{1/2-l-\tau}$.  Since $r^{-1} \nabla_{S^{n-3}} w$ is bounded as $r \downarrow 0$, we can rewrite (\ref{schauder1_step3_eqn10}) in the weak form 
\begin{equation} \label{schauder1_step3_eqn11}
	\int_{S^{n-3}} \int_0^{\infty} h \left( \frac{\partial w}{\partial r} \frac{\partial \zeta}{\partial r} 
		+ \frac{1}{r^2} \nabla_{S^{n-3}} w \nabla_{S^{n-3}} \zeta + (1-\tau)(2l+\tau) w \zeta \right) = 0
\end{equation}
for all $\zeta \in C^1_c(\mathbb{R}^{n-2};\mathbb{R}^{((k+1)/2-l)(n-2)})$.  By replacing $\zeta$ with $w \zeta(r)^2$ in (\ref{schauder1_step3_eqn11}), where $\zeta \in C^1([0,\infty))$ such that $\zeta'(r) = 0$ near $0$ and $\zeta(r)$ vanishes for $r$ sufficiently large and using Cauchy's inequality, 
\begin{equation*}
	\int_{S^{n-3}} \int_0^{\infty} h \zeta(r)^2 \left( \left| \frac{\partial w}{\partial r} \right|^2 + \frac{1}{r^2} |\nabla_{S^{n-3}} w|^2 
	+ (1-\tau)(2l+\tau) |w|^2 \right) \leq 4 \int_{S^{n-3}} \int_0^{\infty} h |w|^2 \zeta'(r)^2. 
\end{equation*}
For $\rho > 1$, let $\zeta(r) = 1$ for $r \in [0,\rho]$, $\zeta(r) = -\log(r/\rho)/\log(\rho)$ for $r \in (\rho,\rho^2)$, and $\zeta(r) = 0$ for $r \geq \rho^2$ to obtain using (\ref{schauder1_step3_eqn011}) that 
\begin{align*}
	&\int_{S^{n-3}} \int_0^{\rho} h \left( \left| \frac{\partial w}{\partial r} \right|^2 + \frac{1}{r^2} |\nabla_{S^{n-3}} w|^2 
		+ (1-\tau)(2l+\tau) |w|^2 \right) 
	\\& \leq \frac{C}{\log(\rho)^2} \int_{S^{n-3}} \int_{\rho}^{\rho^2} r^{-2l-2\tau} |w|^2
	\leq \frac{C}{\log(\rho)^2} \int_{\rho}^{\rho^2} r^{-2-2l-2\tau}
	\leq \frac{C}{\log(\rho)^2} \rho^{-1-2l-2\tau} 
\end{align*}
for $C = C(n,k,l,\tau) \in (0,\infty)$, so by letting $\rho \rightarrow \infty$ we obtain $w \equiv 0$ on $\mathbb{R}^{n-2}$.  Consequently $u$ is a degree $((k-1)/2-l)$ polynomial in $z$.  By (\ref{schauder1_step3_eqn9}), $u \equiv 0$.  Consequently $\chi \equiv 0$, contradicting (\ref{schauder1_step2_eqn9}). 
\end{proof}

\begin{lemma}[Schauder estimate away from the singular set] \label{schaudersys_afbs}
Given an integer $m \geq 2$, $\tau \in (0,1/2)$, and $\nu \geq 0$, there exists $\varepsilon = \varepsilon(n,m,\tau,\nu) > 0$ such that the following holds true.  For $\kappa = 1,2,\ldots,m$, let $s_{\kappa} \geq 1$ be integers with $s_1 = s_2 = 2$, $s \geq \max s_{\kappa}$, and $\psi^{\kappa} \in C^{s,\tau}(B^g_1(0) \setminus \{0\} \times \mathbb{R}^{n-2})$.  Suppose $a^{ij}_{\kappa \lambda} \in \mathcal{H}^{1-s_{\lambda}+t_{ij},s-1,\tau}(B^g_1(0))$ if $\kappa \leq 2$ and $a^{ij}_{\kappa \lambda} \in \mathcal{H}^{s_{\kappa}-s_{\lambda}+t_{ij},s-1,\tau}(B^g_1(0))$ if $\kappa \geq 3$, where $t_{ij} = 0$ if $i,j \leq 2$, $t_{ij} = 1$ if $i \leq 2 < j$ or $j \leq 2 < i$, and $t_{ij} = 2$ if $i,j \geq 3$, such that 
\begin{gather} 
	\left\| a^{ij}_{\kappa \lambda} + \op{Re} \left( \frac{\sqrt{-1}^{j-i-\kappa-\lambda}}{2(\tilde x_1 + \sqrt{-1} \tilde x_2)} \right)
		\right\|_{\mathcal{H}^{-1,s-1,\tau}(B^g_1(0))} \leq \varepsilon 
		\text{ if } i,j,\kappa, \lambda \leq 2, \nonumber \\
	\|a^{ij}_{\kappa \lambda}\|_{\mathcal{H}^{0,s-1,\tau}(B^g_1(0))} \leq \varepsilon 
		\text{ if } \kappa, \lambda \leq 2 \text{ and } i \leq 2 < j \text{ or } j \leq 2 < i, \nonumber \\
	\left\| a^{ij}_{\kappa \lambda} + \op{Re} \left( 2 \sqrt{-1}^{\kappa+\lambda} (\tilde x_1 + \sqrt{-1} \tilde x_2) \right) 
		\right\|_{\mathcal{H}^{1,s-1,\tau}(B^g_1(0))} \leq \varepsilon \text{ if } \kappa, \lambda \leq 2 < i,j, \nonumber \\
	\|a^{ij}_{\kappa \lambda}\|_{\mathcal{H}^{1-s_{\lambda}+t_{ij},s-1,\tau}(B^g_1(0))} \leq \nu \text{ and } 
		\|a^{ij}_{\kappa \lambda} - a^{ij}_{\kappa \lambda}( \, ;0)\|_{\mathcal{H}^{1-s_{\lambda}+t_{ij},s-1,\tau}(B^g_1(0))} \leq \nu
		\text{ if } \lambda \leq 2 < \kappa, \nonumber \\
	\|a^{ij}_{\kappa \lambda}\|_{\mathcal{H}^{s_{\kappa}-2+t_{ij},-1,\tau}(B^g_1(0))} \leq \varepsilon 
		\text{ if } \kappa \leq 2 < \lambda, \nonumber \\
	\|a^{ij}_{\kappa \lambda} - 1\|_{\mathcal{H}^{0,s-1,\tau}(B^g_1(0))} \leq \varepsilon 
		\text{ if } i = j \leq 2 < \kappa = \lambda \nonumber \\
	\|a^{ij}_{\kappa \lambda} - 4|\tilde x|^2\|_{\mathcal{H}^{2,s-1,\tau}(B^g_1(0))} \leq \varepsilon 
		\text{ if } i = j \geq 3, \, \kappa = \lambda \geq 3, \nonumber \\
	\|a^{ij}_{\kappa \lambda}\|_{\mathcal{H}^{t_{ij},s-1,\tau}(B^g_1(0))} \leq \varepsilon 
		\text{ if } i \neq j, \, \kappa = \lambda \geq 3, \nonumber \\
	\|a^{ij}_{\kappa \lambda}\|_{\mathcal{H}^{s_{\kappa}-s_{\lambda}+t_{ij},s-1,\tau}(B^g_1(0))} \leq \varepsilon 
		\text{ if } \kappa, \lambda \geq 3, \, \kappa \neq \lambda. \label{schaudersyshyp1} 
\end{gather}
Suppose that $\psi$ is a solution to 
\begin{equation} \label{schaudersyshyp2} 
	\sum_{i,j=1}^n \sum_{\lambda=1}^m D_i(a^{ij}_{\kappa \lambda} D_j \psi^{\lambda}) =  f_{\kappa} 
\end{equation}
weakly on $B^g_1(0) \setminus \{0\} \times B^{n-2}_1(0)$ for $\kappa = 1,2,\ldots,m$ for some $f_{\kappa} \in C^{s-2,\tau}(B^g_1(0))$ for $\kappa = 1,2,\ldots,m$.  Then 
\begin{align*}
	&\sum_{\kappa=1}^m \sum_{|\alpha|+|\beta| \leq s} \frac{\rho^{|\alpha|+|\beta|}}{|\tilde x_0|^{s_{\kappa}+|\alpha|}} 
		\sup_{B^g_{\rho/2}(\tilde X_0)} |D_{\tilde x}^{\alpha} D_{\tilde y}^{\beta} \psi^{\kappa}| 
	+ \sum_{\kappa=1}^m \sum_{|\alpha|+|\beta| = s} \frac{\rho^{|\alpha|+|\beta|+\tau}}{|\tilde x_0|^{s_{\kappa}+|\alpha|}} 
		[D_{\tilde x}^{\alpha} D_{\tilde y}^{\beta} \psi^{\kappa}]_{g,\tau,B^g_{\rho/2}(\tilde X_0)} 
	\\&\leq C \left( \sum_{\kappa=1}^m \frac{1}{|\tilde x_0|^{s_{\kappa}}} \sup_{B^g_{\rho}(\tilde X_0)} |\psi^{\kappa}| 
	+ \sum_{\kappa=1}^2 \sum_{|\alpha|+|\beta| \leq s-2} \frac{\rho^{2+|\alpha|+|\beta|}}{|\tilde x_0|^{3+|\alpha|}} 
		\sup_{B^g_{\rho}(\tilde X_0)} |D_{\tilde x}^{\alpha} D_{\tilde y}^{\beta} f_{\kappa}| \right. \\& \left. 
	+ \sum_{\kappa=1}^2 \sum_{|\alpha|+|\beta| = s-2} \frac{\rho^{2+|\alpha|+|\beta|+\tau}}{|\tilde x_0|^{3+|\alpha|}} 
		[D_{\tilde x}^{\alpha} D_{\tilde y}^{\beta} f_{\kappa}]_{g,\tau,B^g_{\rho}(\tilde X_0)} 
	+ \sum_{\kappa=3}^m \sum_{|\alpha|+|\beta| \leq s-2} \frac{\rho^{2+|\alpha|+|\beta|}}{|\tilde x_0|^{2+s_{\kappa}+|\alpha|}} 
		\sup_{B^g_{\rho}(\tilde X_0)} |D_{\tilde x}^{\alpha} D_{\tilde y}^{\beta} f_{\kappa}| \right. \\& \left. 
	+ \sum_{\kappa=3}^m \sum_{|\alpha|+|\beta| = s-2} \frac{\rho^{2+|\alpha|+|\beta|+\tau}}{|\tilde x_0|^{2+s_{\kappa}+|\alpha|}} 
		[D_{\tilde x}^{\alpha} D_{\tilde y}^{\beta} f_{\kappa}]_{g,\tau,B^g_{\rho}(\tilde X_0)} \right) 
\end{align*}
for every $\tilde X_0 = (\tilde x_0,\tilde y_0) \in B^g_1(0)$ and $0 < \rho \leq \min\{|\tilde x_0|^2/4,1-d_g(\tilde X_0,0)\}$ for some constant $C = C(n,m,s,\nu,\tau) \in (0,\infty)$.
\end{lemma}
\begin{proof}
Apply the Schauder estimate~\cite[Theorem 6.8.2]{Morrey} for elliptic systems to $|\tilde x_0|^{-s_{\kappa}} \psi^{\kappa}(\tilde x_0+\rho \tilde x/|\tilde x_0|,\tilde y_0+\rho \tilde y)$.
\end{proof}

\begin{lemma} \label{schaudersyslemma}
Given an integer $m \geq 2$, $\tau \in (0,1/2)$, and $\nu \geq 0$, there exists $\varepsilon = \varepsilon(n,m,\tau,\nu) > 0$ such that the following holds true.  For $\kappa = 1,2,\ldots,m$, let $s_{\kappa} \geq 1$ be integers with $s_1 = s_2 = 2$, $s = \max s_{\kappa}$, and $\psi^{\kappa} \in C^{g,s_{\kappa},\tau}(B^g_1(0)) \cap C^{\infty}(B^g_1(0) \setminus \{0\} \times \mathbb{R}^{n-2})$ with 
\begin{gather*}
	\psi^{\kappa}(-\tilde x,\tilde y) = (-1)^{s_{\kappa}} \psi^{\kappa}(\tilde x,\tilde y) \text{ for all } (\tilde x,\tilde y) \in B^g_1(0), \\
	D_{\tilde x}^{\alpha} \psi^{\kappa}(0,y) = 0 \text{ for } y \in B^{n-2}_1(0), \, 1 \leq |\alpha| < s_{\kappa}. 
\end{gather*}
Suppose $a^{ij}_{\kappa \lambda} \in \mathcal{H}^{1-s_{\lambda}+t_{ij},s-1,\tau}(B^g_1(0))$ if $\kappa \leq 2$ and $a^{ij}_{\kappa \lambda} \in \mathcal{H}^{s_{\kappa}-s_{\lambda}+t_{ij},s-1,\tau}(B^g_1(0))$ if $\kappa \geq 3$, where $t_{ij} = 0$ if $i,j \leq 2$, $t_{ij} = 1$ if $i \leq 2 < j$ or $j \leq 2 < i$, and $t_{ij} = 2$ if $i,j \geq 3$, such that (\ref{schaudersyshyp1}) holds true.  Suppose that $\psi$ is a solution to (\ref{schaudersyshyp2}) and 
\begin{equation} \label{schaudersyshyp3} 
	\sum_{i,j=1}^n \sum_{\lambda=1}^m \sum_{|\alpha| = s_{\lambda}-1} \frac{1}{\alpha!} D_i(a^{ij}_{\kappa \lambda}(\tilde x;\tilde y_0) 
	D_j D_{\tilde x}^{\alpha} \psi^{\lambda}(0,\tilde y_0) \tilde x^{\alpha}) =  f_{\kappa}(\tilde x;\tilde y_0) 
\end{equation}
weakly on $\mathbb{R}^2 \setminus \{0\}$ for $\kappa = 1,2,\ldots,m$ and $\tilde y_0 \in B^{n-2}_1(0)$ for some $f_{\kappa} \in \mathcal{H}^{-1,s-2,\tau}(B^g_1(0))$ for $\kappa = 1,2$ and $f_{\kappa} \in \mathcal{H}^{s_{\kappa}-2,s-2,\tau}(B^g_1(0))$ for $\kappa = 3,4,\ldots,m$.  Then 
\begin{equation} \label{schaudersyseqn1}
	\sum_{\kappa=1}^m \|\psi^{\kappa}\|_{C^{g,s_{\kappa},\tau}(B^g_{1/2}(0))} 
	\leq C \left( \sum_{\kappa=1}^m \sup_{B^g_1(0)} |\psi^{\kappa}| 
		+ \sum_{\kappa=1}^2 \|f_{\kappa}\|_{\mathcal{H}^{-1,s-2,\tau}(B^g_1(0))} 
		+ \sum_{\kappa=3}^m \|f_{\kappa}\|_{\mathcal{H}^{s_{\kappa}-2,s-2,\tau}(B^g_1(0))}\right)
\end{equation}
for some constant $C = C(n,m,s,\tau,\nu) \in (0,\infty)$. 
\end{lemma}
\begin{proof}
First consider the special case where $m = 2$ and 
\begin{align} \label{schaudersyseqn2}
	&D_1 \left( \frac{\tilde x_1}{2 |\tilde x|^2} (D_1 \psi^1 + D_2 \psi^2) + \frac{\tilde x_2}{2 |\tilde x|^2} (D_2 \psi^1 - D_1 \psi^2) \right) \\
		&+ D_2 \left( \frac{-\tilde x_2}{2 |\tilde x|^2} (D_1 \psi^1 + D_2 \psi^2) + \frac{\tilde x_1}{2 |\tilde x|^2} (D_2 \psi^1 - D_1 \psi^2) \right) 
		+ \sum_{j=3}^n D_j (2\tilde x_1 D_j \psi^1 - 2\tilde x_2 D_j \psi^2) = f_1, \nonumber \\	
	&D_1 \left( \frac{-\tilde x_2}{2 |\tilde x|^2} (D_1 \psi^1 + D_2 \psi^2) + \frac{\tilde x_1}{2 |\tilde x|^2} (D_2 \psi^1 - D_1 \psi^2) \right) \nonumber \\
		&+ D_2 \left( \frac{-\tilde x_1}{2 |\tilde x|^2} (D_1 \psi^1 + D_2 \psi^2) - \frac{\tilde x_2}{2 |\tilde x|^2} (D_2 \psi^1 - D_1 \psi^2) \right) 
		+ \sum_{j=3}^n D_j (-2\tilde x_2 D_j \psi^1 - 2\tilde x_1 D_j \psi^2) = f_2. \nonumber 
\end{align}
on $\mathbb{R}^n \setminus \{0\} \times \mathbb{R}^{n-2}$.  By expanding (\ref{schaudersyseqn2}), 
\begin{align} 
	\label{schaudersyseqn3} &\frac{\tilde x_1}{2|\tilde x|^2} \Delta_{\tilde x} \psi^1 + 2\tilde x_1 \Delta_{\tilde y} \psi^1 
		- \frac{\tilde x_2}{2|\tilde x|^2} \Delta_{\tilde x} \psi^2 - 2\tilde x_2 \Delta_{\tilde y} \psi^1 = f_1, \\
	\label{schaudersyseqn4} &\frac{-\tilde x_2}{2|\tilde x|^2} \Delta_{\tilde x} \psi^1 + 2\tilde x_2 \Delta_{\tilde y} \psi^1 
		- \frac{\tilde x_1}{2|\tilde x|^2} \Delta_{\tilde x} \psi^2 - 2\tilde x_1 \Delta_{\tilde y} \psi^1 = f_2, 
\end{align}
on $\mathbb{R}^n \setminus \{0\} \times \mathbb{R}^{n-2}$.  The equations $2\tilde x_1 \cdot $(\ref{schaudersyseqn3})$-2\tilde x_2 \cdot$(\ref{schaudersyseqn4}) and $-2\tilde x_2 \cdot $(\ref{schaudersyseqn3})$-2\tilde x_1 \cdot$(\ref{schaudersyseqn4}) yield 
\begin{align} \label{schaudersyseqn5}
	\Delta_{\tilde x} \psi^1 + 4 |\tilde x|^2\Delta_{\tilde y} \psi^1 &= -\tilde x_1 f_1 + \tilde x_2 f_2, \nonumber \\
	\Delta_{\tilde x} \psi^2 + 4 |\tilde x|^2\Delta_{\tilde y} \psi^2 &= -\tilde x_2 f_1 - \tilde x_1 f_2, 
\end{align}
on $\mathbb{R}^n \setminus \{0\} \times \mathbb{R}^{n-2}$.  Then by Lemma \ref{schauderlemma} we obtain the Schauder estimate for solutions $\psi^1,\psi^2$ to (\ref{schaudersyseqn2}) of 
\begin{equation} \label{schaudersyseqn6}
	\sum_{\kappa=1}^2 \|\psi^{\kappa}\|_{C^{g,2,\tau}(B^g_{1/2}(0))} 
	\leq C(n,\tau) \left( \sum_{\kappa=1}^2 \sup_{B^g_1(0)} |\psi^{\kappa}| + \sum_{\kappa=1}^2 \|f_{\kappa}\|_{\mathcal{H}^{-1,0,\tau}(B^g_1(0))} \right)
\end{equation}

Now consider the case of solutions $\psi^{\kappa}$ to the general systems (\ref{schaudersyshyp2}) and (\ref{schaudersyshyp3}) subject to (\ref{schaudersyshyp1}).  By subtracting (\ref{schaudersyshyp2}) and (\ref{schaudersyshyp3}) and then applying the Schauder estimate away from $\{0\} \times \mathbb{R}^{n-2}$ of Lemma \ref{schaudersys_afbs} with $\rho = |\tilde x_0|^2/4$ and using (\ref{schaudersyshyp1}), 
\begin{align} \label{schaudersyseqn7}
	&\sum_{\kappa=1}^m \sum_{|\alpha|+2|\beta| > s_{\kappa}, \, |\alpha|+|\beta| \leq s} |\tilde x|^{-s_{\kappa}+|\alpha|+2|\beta|-2\tau} 
		|D_{\tilde x}^{\alpha} D_{\tilde y}^{\beta} \psi^{\kappa}(\tilde X)| 
	+ \sum_{\kappa=1}^m \sum_{|\alpha|+|\beta| = s} |\tilde x|^{-s_{\kappa}+|\alpha|+2|\beta|} 
		[D_{\tilde x}^{\alpha} D_{\tilde y}^{\beta} \psi^{\kappa}]_{g,\tau,B^g_{|\tilde x|^2/8}(\tilde X)} \nonumber \\
	&\leq C \left( \varepsilon \sum_{\kappa=1}^m \|\psi^{\kappa}\|_{C^{g,s_{\kappa},\tau}(B^g_1(0))} 
		+ \sum_{\kappa=1}^2 \|f_{\kappa}\|_{\mathcal{H}^{-1,s-2,\tau}(B^g_1(0))} 
		+ \sum_{\kappa=3}^m \|f_{\kappa}\|_{\mathcal{H}^{s_{\kappa}-2,s-2,\tau}(B^g_1(0))}\right)
\end{align}
on $B^g_{3/4}(0)$ for some constant $C = C(n,m,s,\tau,\nu) \in (0,\infty)$.  By (\ref{schaudersyseqn6}) and Lemma \ref{schauderlemma} applied to the systems (\ref{schaudersyshyp2}) and (\ref{schaudersyshyp3}) using (\ref{schaudersyshyp1}) and (\ref{schaudersyseqn7}), 
\begin{align*}
	&\sum_{\kappa=1}^2 \|\psi^{\kappa}\|_{C^{g,2,\tau}(B^g_{5/8}(0))} 
	\leq C \left( \sum_{\kappa=1}^2 \sup_{B^g_1(0)} |\psi^{\kappa}| 
		+ \varepsilon \sum_{\lambda=1}^m \|\psi^{\lambda}\|_{C^{g,s_{\lambda},\tau}(B^g_1(0))} \nonumber \right. \\& \left. 
		+ \sum_{\lambda=1}^2 \|f_{\lambda}\|_{\mathcal{H}^{-1,s-2,\tau}(B^g_1(0))} 
		+ \sum_{\lambda=3}^m \|f_{\lambda}\|_{\mathcal{H}^{s_{\lambda}-2,s-2,\tau}(B^g_1(0))} \right)
\end{align*}
and 
\begin{align*}
	&\|\psi^{\kappa}\|_{C^{g,s_{\kappa},\tau}(B^g_{1/2}(0))} 
	\leq C \left( \sup_{B^g_1(0)} |\psi^{\kappa}| 
		+ \nu \sum_{\lambda=1}^2 \|\psi^{\lambda}\|_{C^{g,2,\tau}(B^g_{5/8}(0))} 
		+ \varepsilon \sum_{\lambda=1}^m \|\psi^{\lambda}\|_{C^{g,s_{\lambda},\tau}(B^g_1(0))} \nonumber \right. \\& \left. 
		+ \sum_{\lambda=1}^2 \|f_{\lambda}\|_{\mathcal{H}^{-1,s-2,\tau}(B^g_1(0))} 
		+ \sum_{\lambda=3}^m \|f_{\lambda}\|_{\mathcal{H}^{s_{\lambda}-2,s-2,\tau}(B^g_1(0))} \right)
\end{align*}
for $\kappa = 3,4,\ldots,m$ for some constant $C = C(n,m,s,\tau,\nu) \in (0,\infty)$.  Hence 
\begin{align} \label{schaudersyseqn8}
	&\sum_{\kappa=1}^m \|\psi^{\kappa}\|_{C^{g,s_{\kappa},\tau}(B^g_{1/2}(0))} 
	\leq C \left( \sum_{\kappa=1}^m \sup_{B^g_1(0)} |\psi^{\kappa}| 
		+ \varepsilon \sum_{\kappa=1}^m \|\psi^{\kappa}\|_{C^{g,s_{\kappa},\tau}(B^g_1(0))} \nonumber \right. \\& \left. 
		+ \sum_{\kappa=1}^2 \|f_{\kappa}\|_{\mathcal{H}^{-1,s-2,\tau}(B^g_1(0))} 
		+ \sum_{\kappa=3}^m \|f_{\kappa}\|_{\mathcal{H}^{s_{\kappa}-2,s-2,\tau}(B^g_1(0))} \right)
\end{align}
for some constant $C = C(n,m,s,\tau,\nu) \in (0,\infty)$.  By scaling (\ref{schaudersyseqn7}) replacing $\psi^{\kappa}$ with $\rho^{-s_{\kappa}/2} \psi^{\kappa}(\rho^{1/2} \tilde x,\tilde y_0+\rho \tilde y)$ and then using the Schauder estimate away from $\{0\} \times \mathbb{R}^{n-2}$ of Lemma \ref{schaudersys_afbs}), the interpolation inequalities of Lemma \ref{interplemma}, and Lemma \ref{SALemma} with $l = s+\tau$, we obtain (\ref{schaudersyseqn1}). 
\end{proof}

For the next section, it will be convenient to define the notation 
\begin{align*}
	\vvvert \psi \vvvert_{k,l,\tau,B^g_{\rho}(0,\tilde y_0)} 
	= {}& \sum_{|\alpha|+2|\beta| \leq k} \rho^{-k/2+|\alpha|/2+|\beta|} \sup_{B^g_{\rho}(0,\tilde y_0)} |D_{\tilde x}^{\alpha} D_{\tilde y}^{\beta} \psi| 
	+ \sum_{|\alpha|+2|\beta| = k} \rho^{\tau} [D_{\tilde x}^{\alpha} D_{\tilde y}^{\beta} \psi]_{g,\tau,B^g_{\rho}(0,\tilde y_0)}
	\\&+ \sum_{|\alpha|+2|\beta| > k, \, |\alpha|+|\beta| \leq l} \rho^{\tau} \sup_{B^g_{\rho}(0,\tilde y_0)} |\tilde x|^{-k-2\tau+|\alpha|+2|\beta|} 
		|D_{\tilde x}^{\alpha} D_{\tilde y}^{\beta} \psi| 
	\\&+ \sum_{|\alpha|+|\beta| = l} \sup_{(\tilde x,\tilde y) \in B^g_{\rho}(0,\tilde y_0)} \rho^{\tau} |\tilde x|^{-k+|\alpha|+2|\beta|} 
		[D_{\tilde x}^{\alpha} D_{\tilde y}^{\beta} \psi]_{g,\tau,B^g_{|\tilde x|^2/4}(\tilde x,\tilde y)} 
\end{align*}
for nonnegative integers $k \leq l$, $\tau \in (0,1/2]$, and $\psi \in C^{g,k,\tau}(B^g_{\rho}(0,\tilde y_0)) \cap C^{l,\tau}(B^g_{\rho}(0,\tilde y_0) \setminus \{0\} \times \mathbb{R}^{n-2})$.  

\begin{rmk} \label{schauder_rmk}
Suppose that $\psi \in C^{g,k,\tau}(B^g_{\rho}(0,\tilde y_0)) \cap C^{l,\tau}(B^g_{\rho}(0,\tilde y_0) \setminus \{0\} \times \mathbb{R}^{n-2})$ with $D_{\tilde x}^{\alpha} \psi = 0$ on $\{0\} \times B^{n-2}_1(0)$ whenever $|\alpha| < k$ and $\vvvert D_{\tilde x}^k \psi \vvvert_{k,l,\tau,B^g_{\rho}(0,\tilde y_0)} < \infty$.  By using the definition of $\vvvert D_{\tilde x}^k \psi \vvvert_{k,l,\tau,B^g_{\rho}(0,\tilde y_0)}$ and the fact that by Taylor's theorem 
\begin{equation*}
	D_{\tilde x}^{\alpha} D_{\tilde y}^{\beta} \psi(\tilde x,\tilde y) 
	= \sum_{|\gamma| = k-|\alpha|-2|\beta|} \frac{|\gamma|}{\gamma!} \int_0^1 (1-t)^{|\gamma|-1} 
	D_{\tilde x}^{\alpha+\gamma} D_{\tilde y}^{\beta} \psi(t\tilde x,\tilde y) dt \, \tilde x^{\alpha} 
\end{equation*}
on $B^g_{\rho}(0,\tilde y_0)$ whenever $|\alpha|+2|\beta| \leq k$, it is easy to check that $\psi \in \mathcal{H}^{k,l,\tau}(B^g_1(0);\mathbb{R}^m)$ with 
\begin{equation*}
	\psi(\tilde X;\tilde z) = \sum_{|\alpha| = k} \frac{1}{\alpha!} D_{\tilde x}^{\alpha} \psi(0,\tilde z) \tilde x^{\alpha}, \quad
	\|\psi\|_{\mathcal{H}^{k,l,\tau}(B^g_{\rho}(0,\tilde y_0))} \leq C(n,m,k,\tau) \vvvert \psi \vvvert_{k,l,\tau,B^g_{\rho}(0,\tilde y_0)}. 
\end{equation*}
\end{rmk}

Notice that by applying Lemma \ref{schaudersys_afbs} to $\rho^{-s_{\kappa}/2} \psi^{\kappa}(\rho^{1/2} \tilde x,\tilde y_0+\rho \tilde y)$ and also subtracting (\ref{schaudersyshyp2}) and (\ref{schaudersyshyp3}) and then applying Lemma \ref{schaudersys_afbs} in the case $\rho = |\tilde x_0|^2/4$ to the resulting system, we obtain: 

\begin{cor} \label{schaudersyscor}
Under the hypotheses of Lemma \ref{schaudersyslemma}, 
\begin{align*}
	\sum_{\kappa=1}^m \vvvert \psi^{\kappa} \vvvert_{s_{\kappa},s,\tau,B^g_{\vartheta \rho}(0,\tilde y_0)} 
	\equiv {}& C \left( \sum_{\kappa=1}^m \rho^{-s_{\kappa}/2} \sup_{B^g_{\vartheta \rho}(0,\tilde y_0)} |\psi^{\kappa}| 
		+ \sum_{\kappa=1}^2 \|f_{\kappa}\|_{\mathcal{H}^{-1,s-2,\tau}(B^g_{\vartheta \rho}(0,\tilde y_0))} \right. \\& \left. 
		+ \sum_{\kappa=3}^m \rho^{1-s_{\kappa}/2} \|f_{\kappa}\|_{\mathcal{H}^{s_{\kappa}-2,s-2,\tau}(B^g_{\vartheta \rho}(0,\tilde y_0))} \right)
\end{align*}
for all $\vartheta \in (0,1)$ and $B^g_{\rho}(0,\tilde y_0) \subset B^g_1(0,0)$ for some constant $C = C(n,m,s,\tau,\nu,\vartheta) \in (0,\infty)$. 
\end{cor}

\section{Higher regularity of the branch set} \label{sec:analyticitysec}

Let $\tilde X = (\tilde x,\tilde y)$, $\phi^1$, $\phi^2$, $\chi^{\kappa}_k$, and $\psi^{\kappa}_k$ be as in Section~\ref{sec:transpde_sec}.  By (\ref{harm_tildex_asym4}), (\ref{phi_asym1}), (\ref{harm_u_asym1}), (\ref{mss_u_asym1}), (\ref{mss_u_asym2}) and the corresponding error estimates, $\phi^1,\phi^2,\chi^{\kappa}_k \in C^{g,2,\tau}(B^g_{1/8}(0))$ and $\psi^{\kappa}_k \in C^{g,2\mathcal{N}-2\lceil k/2 \\rceil,\tau}(B^g_{1/8}(0))$ with $D_{\tilde x}^{\alpha} \psi^{\kappa}_k = 0$ on $\{0\} \times B^{n-2}_{1/8}(0)$ whenenver $|\alpha| < 2\mathcal{N}-2\lceil k/2 \\rceil$.

Note that if we differentiate (\ref{transpde1}) and (\ref{transpde2}) by $D_{\tilde y}^{\gamma}$ (with $|\gamma| \geq 1$) we obtain 
\begin{align} \label{transpde_diff1}
	&D_i \left( (\phi^1_1 \phi^2_2 - \phi^1_2 \phi^2_1) M^i_s M^j_t a^{st}_{\kappa \lambda}(D\chi_0 M,\sigma D\psi_0 M) 
		D_j D_{\tilde y}^{\gamma} \psi_k^{\lambda} \right) \nonumber \\
	&+ \sigma^{-1} D_i \left( (\phi^1_1 \phi^2_2 - \phi^1_2 \phi^2_1) M^i_s M^j_t b^{st}_{\kappa \lambda}(D\chi_0 M,\sigma D\psi_0 M) 
		D_j D_{\tilde y}^{\gamma} \chi_k^{\lambda} \right) \nonumber \\
	&+ \sigma^{-1} D_i \left( c^{ij}_{k, \kappa \lambda}(D\phi,\{D\chi_l\}_{l \leq k}, \{\sigma D\psi_l\}_{l \leq k}) 
		D_j D_{\tilde y}^{\gamma} \phi^{\lambda} \right) \nonumber \\
	&- D_i \left( (\phi^1_1 \phi^2_2 - \phi^1_2 \phi^2_1) M^i_s D_{Q^{\lambda}_{l,t}} f^s_{k,\kappa}(\{D\chi_l M\}_{l \leq k-1}, 
		\{\sigma D\psi_l M\}_{l \leq k-1}) D_j D_{\tilde y}^{\gamma} \psi_l^{\lambda} \right), \nonumber \\
	&- \sigma^{-1} D_i \left( (\phi^1_1 \phi^2_2 - \phi^1_2 \phi^2_1) M^i_s D_{P^{\lambda}_{l,t}} f^s_{k,\kappa}(\{D\chi_l M\}_{l \leq k-1}, 
		\{\sigma D\psi_l M\}_{l \leq k-1}) D_j D_{\tilde y}^{\gamma} \chi_l^{\lambda} \right) \nonumber \\
	&= \sigma^{-1} D_i \left( h^i_{\gamma,k,\kappa}(\{DD^{\alpha} \phi\}_{\alpha < \gamma}, \{DD^{\alpha} \chi_l\}_{l \leq k, \, \alpha < \gamma}, 
		\{\sigma DD^{\alpha} \psi_l\}_{l \leq k, \, \alpha < \gamma}) \right)
\end{align}
and 
\begin{align} \label{transpde_diff2}
	&D_i \left( (\phi^1_1 \phi^2_2 - \phi^1_2 \phi^2_1) M^i_s M^j_t a^{st}_{\kappa \lambda}(\sigma D\psi_0 M,D\chi_0 M) 
		D_j D_{\tilde y}^{\gamma} \chi_k^{\lambda} \right) \nonumber \\
	&+ \sigma D_i \left( (\phi^1_1 \phi^2_2 - \phi^1_2 \phi^2_1) M^i_s M^j_t b^{st}_{\kappa \lambda}(\sigma D\psi_0 M,D\chi_0 M) 
		D_j D_{\tilde y}^{\gamma} \psi_k^{\lambda} \right) \nonumber \\
	&+ D_i \left( c^{ij}_{k, \kappa \lambda}(D\phi,\{\sigma D\psi_l\}_{l \leq k},\{D\chi_l\}_{l \leq k}) 
		D_j D_{\tilde y}^{\gamma} \phi^{\lambda} \right) \nonumber \\
	&- D_i \left( (\phi^1_1 \phi^2_2 - \phi^1_2 \phi^2_1) M^i_s D_{Q^{\lambda}_{l,t}} f^s_{k,\kappa}(\{\sigma D\psi_l M\}_{l \leq k-1},
		\{D\chi_l M\}_{l \leq k-1}) D_j D_{\tilde y}^{\gamma} \chi_l^{\lambda} \right), \nonumber \\
	&- \sigma D_i \left( (\phi^1_1 \phi^2_2 - \phi^1_2 \phi^2_1) M^i_s D_{P^{\lambda}_{l,t}} f^s_{k,\kappa}(\{\sigma D\psi_l M\}_{l \leq k-1},
		\{D\chi_l M\}_{l \leq k-1}) D_j D_{\tilde y}^{\gamma} \psi_l^{\lambda} \right) \nonumber \\
	&= D_i \left( h^i_{\gamma,k,\kappa}(\{DD^{\alpha} \phi\}_{\alpha < \gamma}, \{\sigma DD^{\alpha} \psi_l\}_{l \leq k, \, \alpha < \gamma}, 
		\{DD^{\alpha} \chi_l\}_{l \leq k, \, \alpha < \gamma}) \right)
\end{align}
on $B_{1/8}(0) \setminus \{0\} \times \mathbb{R}^{n-2}$ for $\kappa = 1,2,\ldots,m$ and $k = 0,1,2,\ldots,2\mathcal{N}-1$, where we let $f^s_{0 \lambda} = 0$, 
\begin{align*}
	c^{ij}_{0, \kappa \lambda}(R,P,Q) 
		&= D_{R^{\lambda}_j} \left( (R^1_1 R^2_2 - R^1_2 R^2_1) M^i_s(R) M^l_t(R)) A^{st}_{\kappa \nu}(P \cdot M(R), Q \cdot M(R)) Q^{\nu}_l \right) 
		\text{ when } k = 0, \\
	c^{ij}_{k, \kappa \lambda}(R,P,Q) 
		&= D_{R^{\lambda}_j} \left( (R^1_1 R^2_2 - R^1_2 R^2_1) M^i_s(R) M^l_t(R)) a^{st}_{\kappa \nu}(P_0 \cdot M(R), Q_0 \cdot M(R)) Q^{\nu}_{0,l} \right) \\
		&\hspace{5mm} + D_{R^{\lambda}_j} \left( (R^1_1 R^2_2 - R^1_2 R^2_1) M^i_s(R) M^l_t(R)) b^{st}_{\kappa \nu}(P_0 \cdot M(R), Q_0 \cdot M(R)) 
			P^{\nu}_{0,l} \right) \\
		&\hspace{5mm} + D_{R^{\lambda}_j} \left( (R^1_1 R^2_2 - R^1_2 R^2_1) M^i_s(R) f^s_{k,\kappa}(P \cdot M(R), Q \cdot M(R)) \right) \text{ when } k \geq 1,
\end{align*}
for all $R \in \mathbb{R}^{2n}$, $P \in (\mathbb{R}^{mn})^k$, and $Q \in (\mathbb{R}^{mn})^k$ using the notation $P_0 = (P^{\kappa}_{0,j}), Q_0 = (Q^{\kappa}_{0,j}) \in \mathbb{R}^{mn}$, $h^i_{\gamma,k,\kappa} : (\mathbb{R}^n)^{(2+2mk)(\gamma_1 \gamma_2 \cdots \gamma_n - 1)} \rightarrow \mathbb{R}$.  Recall that in case (a), we replace $\psi_0^{\lambda}$ with $\tilde x_{\lambda}$ for $\lambda = 1,2$ in (\ref{transpde_diff1}) and (\ref{transpde_diff2}).  In cases (b) and (c), we replace $\psi_{2\mathcal{N}-3+\lambda}^1$ with $\tilde x_{\lambda}$ for $\lambda = 1,2$ in (\ref{transpde_diff1}) and (\ref{transpde_diff2}). 

First we will show that for every $\gamma$, $D_{\tilde y}^{\gamma} \phi^{\kappa} \in C^{g,2,\tau}(B^g_{\rho}(0))$, $D_{\tilde y}^{\gamma} \chi^{\kappa}_k \in C^{g,2,\tau}(B^g_{\rho}(0))$, and $D_{\tilde y}^{\gamma} \psi^{\kappa}_k \in C^{g,2\mathcal{N}-2\lceil k/2 \rceil,\tau}(B^g_{\rho}(0))$ for all $\rho \in (0,1/8)$ with $D_{\tilde x}^{\alpha} D_{\tilde y}^{\gamma} \psi^{\kappa}_k = 0$ on $\{0\} \times B^{n-2}_{1/8}(0)$ whenever $|\alpha| < 2\mathcal{N}-2\lceil k/2 \rceil$ and 
\begin{align} \label{smoothness_eqn1}
	&\sum_{\kappa=1}^2 \vvvert D_{\tilde y}^{\gamma} \phi^{\kappa} \vvvert_{2,2\mathcal{N},\tau,B^g_{\rho}(0)} 
	+ \sum_{\kappa=1}^m \sum_{k=0}^{2\mathcal{N}-1} \vvvert D_{\tilde y}^{\gamma} \chi^{\kappa}_k \vvvert_{2,2\mathcal{N},\tau,B^g_{\rho}(0)} \nonumber \\
	&+ \sum_{\kappa=1}^m \sum_{k=0}^{2\mathcal{N}-1} \vvvert D_{\tilde y}^{\gamma} \psi^{\kappa}_k 
		\vvvert_{2\mathcal{N}-2\lceil k/2 \rceil,2\mathcal{N},\tau,B^g_{\rho}(0)} 
	\leq C(n,m,\mathcal{N},|\gamma|,\rho) 
\end{align}
for all $\rho \in (0,1/8)$.  (Note that $D_{\tilde y}^{\gamma} \psi_0^{\lambda} = 0$ for $\lambda = 1,2$ in case (a) and  $D_{\tilde y}^{\gamma} \psi_{2\mathcal{N}-3+\lambda}^1 = 0$ for $\lambda = 1,2$ in cases (b) and (c).)  We will proceed by induction, assuming that this is all true if $|\gamma| < s$ for some integer $s \geq 1$, and proving this is the case when $|\gamma| = s$ using a standard difference quotient argument involving the Schauder estimates, in particular Corollary \ref{schaudersyscor}. 

For $l = 3,4,\ldots,n$ and $h \in \mathbb{R}$ with $h \neq 0$, define the difference quotient operator $\delta_{l,h}$ by 
\begin{equation} \label{diffquot}
	\delta_{l,h} \psi(\tilde X) = \frac{\psi(\tilde X + he_l) - \psi(\tilde X)}{h}
\end{equation}
for every function $\psi : B^g_{\rho}(0) \rightarrow \mathbb{R}^N$, where $\rho > |h| > 0$ and $\tilde X \in B^g_{\rho-|h|}(0)$ and $e_1,e_2,\ldots,e_n$ denotes the standard basis for $\mathbb{R}^n$.  Note that if $\psi \in C^{g,k,\tau}(B_{\rho}(0);\mathbb{R}^m)$, then $\delta_{h,l} \psi \in C^{g,k,\tau}(B_{\rho-|h|}(0);\mathbb{R}^m)$ whenever $l = 3,4,\ldots,n$; however, we cannot similarly take difference quotients in the $e_1$ and $e_2$ directions.  Moreover, if $\psi, D_{\tilde y} \psi \in C^{g,k,\tau}(B_{\rho}(0);\mathbb{R}^m)$, then $\|\delta_{h,l} \psi\|_{C^{g,k,\tau}(B_{\rho-|h|}(0);\mathbb{R}^m)} \leq \|D_l \psi\|_{C^{g,k,\tau}(B_{\rho}(0);\mathbb{R}^m)}$. 

Now take any $\gamma$ with $|\gamma| = s-1$ and any $l \in \{3,4,\ldots,n\}$ and $h \neq 0$.  By applying $\delta_{l,h}$ to (\ref{transpde1}) and (\ref{transpde2}) if $s = 1$ and to (\ref{transpde_diff1}) and (\ref{transpde_diff2}) if $s \geq 2$ and applying Corollary \ref{schaudersyscor} using the induction hypothesis and Remark \ref{schauder_rmk}, we obtain 
\begin{align} \label{smoothness_eqn2}
	&\sum_{\kappa=1}^2 \vvvert \delta_{h,l} D_{\tilde y}^{\gamma} \phi^{\kappa} \vvvert_{2,2\mathcal{N},\tau,B^g_{\rho}(0)} 
	+ \sum_{\kappa=1}^m \sum_{k=0}^{2\mathcal{N}-1} \vvvert \delta_{h,l} D_{\tilde y}^{\gamma} \chi^{\kappa}_k \vvvert_{2,2\mathcal{N},\tau,B^g_{\rho}(0)} 
		\nonumber \\
	&+ \sum_{\kappa=1}^m \sum_{k=0}^{2\mathcal{N}-1} \vvvert \delta_{h,l} D_{\tilde y}^{\gamma} \psi^{\kappa}_k 
		\vvvert_{2\mathcal{N}-2\lceil k/2 \rceil,2\mathcal{N},\tau,B^g_{\rho}(0)} 
	\leq C
\end{align}
whenever $\rho \in (0,1/8)$ and $0 < |h| < (1-\rho)/2$ for some constant $C = C(n,m,\mathcal{N},|\gamma|,\rho) \in (0,\infty)$ independent of $h$.  Note that the hypotheses of Lemma \ref{schaudersyslemma} are easily checked using the definitions of $M^i_j$, $A^{ij}_{\kappa \lambda}$, $a^{ij}_{\kappa \lambda}$, $b^{ij}_{\kappa \lambda}$, $f^i_{k,\kappa}$, and $c^{ij}_{\kappa \lambda}$, (\ref{Asymmetry1}), (\ref{fsymmetry1}), (\ref{phi_asym1}), (\ref{harm_u_asym1}), (\ref{mss_u_asym1}), (\ref{mss_u_asym2}), and the induction hypothesis.  Now by letting $h \rightarrow 0$ using Arzela-Ascoli and the smoothness of $D_{\tilde y}^{\gamma} \phi^{\kappa}$, $D_{\tilde y}^{\gamma} \chi^{\kappa}_k$, and $D_{\tilde y}^{\gamma} \psi^{\kappa}_k$ on $B^g_{1/8}(0) \setminus \{0\} \times B^{n-2}_{1/8}(0)$, $\delta_{h,l} D_{\tilde y}^{\gamma} \phi^{\kappa} \rightarrow D_l D_{\tilde y}^{\gamma} \phi^{\kappa}$, $\delta_{h,l} D_{\tilde y}^{\gamma} \chi_k^{\kappa} \rightarrow D_l D_{\tilde y}^{\gamma} \chi_k^{\kappa}$, and $\delta_{h,l} D_{\tilde y}^{\gamma} \phi_k^{\kappa} \rightarrow D_l D_{\tilde y}^{\gamma} \phi_k^{\kappa}$ uniformly in $B^g_{1/8-\sigma}(0) \setminus B^2_{\sigma} \times \mathbb{R}^{n-2}$ for all $\sigma \in (0,1/16)$.  By letting $h \rightarrow 0$ using (\ref{smoothness_eqn2}), Arzela-Ascoli, and using series expansions like we did in the proof of Lemma \ref{schaudersyslemma} to establish the $C^{g,k,\tau}$ convergence of $\chi_j \rightarrow \chi$, we obtain that $\delta_{h,l} D_{\tilde y}^{\gamma} \phi^{\kappa} \rightarrow D_l D_{\tilde y}^{\gamma} \phi^{\kappa}$ and $\delta_{h,l} D_{\tilde y}^{\gamma} \chi_k^{\kappa} \rightarrow D_l D_{\tilde y}^{\gamma} \chi_k^{\kappa}$ in $C^{g,2,\tau}(B^g_{\rho}(0))$ and $\delta_{h,l} D_{\tilde y}^{\gamma} \psi_k^{\kappa} \rightarrow D_l D_{\tilde y}^{\gamma} \psi_k^{\kappa}$ in $C^{g,2\mathcal{N}-2\lceil k/2 \rceil,\tau}(B^g_{\rho}(0))$ for all $\rho \in (0,1/8)$.  In particular, $D_l D_{\tilde y}^{\gamma} \phi^{\kappa} \in C^{g,2,\tau}(B^g_{\rho}(0))$ and $D_l D_{\tilde y}^{\gamma} \chi_k^{\kappa} \in C^{g,2,\tau}(B^g_{\rho}(0))$ and $D_l D_{\tilde y}^{\gamma} \phi_k^{\kappa} \in C^{g,2\mathcal{N}-2\lceil k/2 \rceil,\tau}(B^g_{\rho}(0))$ for all $\rho \in (0,1/8)$ and $l = 1,2,\ldots,n-2$.  $D_{\tilde x}^{\alpha} D_l D_{\tilde y}^{\gamma} \psi^{\kappa}_k = 0$ on $\{0\} \times B^{n-2}_{1/8}(0)$ whenever $|\alpha| < 2\mathcal{N}-2\lceil k/2 \rceil$.  By letting $h \rightarrow 0$ in (\ref{smoothness_eqn2}), (\ref{smoothness_eqn1}) holds true whenever $|\gamma| = s$. 

Now we will show that $\mathcal{B}_u$ is a real analytic $(n-2)$-dimensional submanifold by showing that $\phi^1(0,\tilde y)$ and $\phi^2(0,\tilde y)$ are real analytic functions of $\tilde y$.  This will follow if we can show that 
\begin{align} \label{analyticity_eqn1}
	&\sum_{\kappa=1}^2 \vvvert D_{\tilde y}^{\gamma} \phi^{\kappa} \vvvert_{2,2\mathcal{N},\tau,B^g_{\rho/|\gamma|}(0)} 
	+ \sum_{\kappa=1}^m \sum_{k=0}^{2\mathcal{N}-1} \vvvert D_{\tilde y}^{\gamma} \chi^{\kappa}_k \vvvert_{2,2\mathcal{N},\tau,B^g_{\rho/|\gamma|}(0)} \\
	&+ \sum_{\kappa=1}^m \sum_{k=0}^{2\mathcal{N}-1} \vvvert D_{\tilde y}^{\gamma} \psi^{\kappa}_k 
		\vvvert_{2\mathcal{N}-2\lceil k/2 \rceil,2\mathcal{N},\tau,B^g_{\rho/|\gamma|}(0)} 
	\leq \left\{ \begin{array}{ll} 
		H_0 \rho^{-|\gamma|} & \text{if } |\gamma| \leq 2, \\
		(|\gamma|-2)! H_0 H^{|\gamma|-2} \rho^{-|\gamma|} & \text{if } |\gamma| > 2. 
	\end{array} \right. \nonumber 
\end{align}
for all $B^g_{\rho}(0,\tilde y_0) \subseteq B^g_{1/8}(0)$ and $\gamma$ with $|\gamma| \geq 1$ for some constants $H_0,H \geq 1$ depending only on $n$, $m$, and $\mathcal{N}$ (independent of $\tilde y_0$ and $\gamma$).  We will prove (\ref{analyticity_eqn1}) by inductively applying the Schauder estimates, in particular Corollary \ref{schaudersyscor}, to (\ref{transpde_diff1}) and (\ref{transpde_diff2}).  Note that by (\ref{smoothness_eqn1}), (\ref{analyticity_eqn1}) holds true whenever $|\gamma| \leq 2\mathcal{N}+4$ by choosing $H_0$ large enough.  Suppose that for some integer $s \geq 2\mathcal{N}+4$, (\ref{analyticity_eqn1}) holds true whenever $|\gamma| < s$.  We now want to prove (\ref{analyticity_eqn1}) holds true for an arbitrary $\gamma$ with $|\gamma| = s$.  By Corollary \ref{schaudersyscor} applied to (\ref{transpde_diff1}) and (\ref{transpde_diff2}), 
\begin{align} \label{analyticity_eqn2}
	&\sum_{\kappa=1}^2 \vvvert D_{\tilde y}^{\gamma} \phi^{\kappa} \vvvert_{2,2\mathcal{N},\tau,B^g_{\rho/2s}(0)} 
	+ \sum_{\kappa=1}^m \sum_{k=0}^{2\mathcal{N}-1} \vvvert D_{\tilde y}^{\gamma} \chi^{\kappa}_k \vvvert_{2,2\mathcal{N},\tau,B^g_{\rho/2s}(0)} \nonumber \\
	&+ \sum_{\kappa=1}^m \sum_{k=0}^{2\mathcal{N}-1} \vvvert D_{\tilde y}^{\gamma} \psi^{\kappa}_k 
		\vvvert_{2\mathcal{N}-2\lceil k/2 \rceil,2\mathcal{N},\tau,B^g_{\rho/2s}(0)} 
	\leq C \left( \sum_{\kappa=1}^2 (\rho/p)^{-1} \sup_{B^g_{\rho/s}(0,\tilde y_0)} |D_{\tilde y}^{\gamma} \phi^{\kappa}| \right. \nonumber \\
	&+ \sum_{\kappa=1}^m \sum_{k=0}^{2\mathcal{N}-1} (\rho/p)^{-1} \sup_{B^g_{\rho/s}(0,\tilde y_0)} |D_{\tilde y}^{\gamma} \chi^{\kappa}_k| 
	+ \sum_{\kappa=1}^m \sum_{k=0}^{2\mathcal{N}} (\rho/p)^{-\mathcal{N}-\lceil k/2 \rceil} 
		\sup_{B^g_{\rho/s}(0,\tilde y_0)} |D_{\tilde y}^{\gamma} \psi^{\kappa}_k| \nonumber \\
	&+ \sigma^{-1} \sum_{\kappa=1}^m \sum_{k=0}^{2\mathcal{N}-1} \|D_i h^i_{\gamma,k,\kappa}(\{DD^{\alpha} \phi\}, \{DD^{\alpha} \chi_l\}, 
		\{\sigma DD^{\alpha} \psi_l\})\|_{\mathcal{H}^{2\mathcal{N}-2\lceil k/2 \rceil-2,2\mathcal{N}-2,\tau}(B^g_{\rho/s}(0,\tilde y_0))} \nonumber \\
	&\left. + \sum_{\kappa=1}^m \sum_{k=0}^{2\mathcal{N}-1} \|D_i h^i_{\gamma,k,\kappa}(\{DD^{\alpha} \phi\}, \{\sigma DD^{\alpha} \psi_l\}, 
		\{DD^{\alpha} \chi_l\})\|_{\mathcal{H}^{0,2\mathcal{N}-2,\tau}(B^g_{\rho/s}(0,\tilde y_0))} \right)
\end{align}
for all $B^g_{\rho}(0,\tilde y_0) \subseteq B^g_{1/8}(0)$ for $C = C(n,m,\mathcal{N}) \in (0,\infty)$.  We now need to bound the $h^i_{\gamma,k,\kappa}$ terms in (\ref{analyticity_eqn2}).  To do so, we need to handle the fact that (\ref{transpde_diff1}) and (\ref{transpde_diff2}) are not quite elliptic differential systems with real analytic date but rather are singular wherever $\phi^1_1 \phi^2_2 - \phi^1_2 \phi^2_1 = 0$ and we need a technique for managing the computation for obtaining the bounds on the $h^i_{\gamma,k,\kappa}$ terms required for the real analyticity result.  The latter will come from a technique that is due to Friedman in~\cite{Friedman} and used with modification by the author in~\cite{KrumThesis}. 

Consider any real analytic function $f : \mathbb{R}^{n^2} \times \mathbb{R}^{mn} \rightarrow \mathbb{R}$ and assume that for $K_0 > 0$, $K \geq a \geq 2$, and nonnegative integers $p_i$, $q_i$, $p = \sum_{i=1}^n p_i$, and $q = \sum_{i=1}^n q_i$, 
\begin{gather} 
	|D_S^{\alpha} D_P^{\beta} f(S,P)| \leq (|\alpha|+|\beta|+p+q)! K_0 K^{2|\alpha|+|\beta|+2p+q} \prod_{i=1}^n |S^i|^{p_i+q_i-|\alpha_i|} 
		\text{ if } |\alpha| \leq p+q, \nonumber \\
	D_S^{\alpha} f(S,P) = 0 \text{ if } |\alpha| > p+q, \label{singanalyticfn_eqn1}
\end{gather}
for all $S \in \mathbb{R}^{n^2}$, $P \in \mathbb{R}^{mn}$ with $|P| \leq a$, and $\alpha$, where for the variable $S$ we write $S = (S^i_j)_{i,j=1,\ldots,n}$ and $S^i = (S^i_j)_{j=1,\ldots,n}$ and similarly for $P$ and other such variables and for the multi-index $\alpha$ we write $\alpha = (\alpha^i_j)_{i,j=1,\ldots,n}$, $\alpha_j = (\alpha^i_j)_{j=1,\ldots,n}$, and $D_S^{\alpha} = \prod D_{S^i_j}^{\alpha^i_j}$ and similarly for $\beta$ and other such multi-indices.  Then, using the fact that 
\begin{equation*}
	M(R) = \left( \begin{matrix}
		R^1_1 & R^1_2 & R^1_3 & R^1_4 & \cdots & R^1_n \\
		R^2_1 & R^2_2 & R^2_3 & R^2_4 & \cdots & R^2_n \\
		0 & 0 & 1 & 0 & \cdots & 0 \\
		0 & 0 & 0 & 1 & \cdots & 0 \\
		\vdots & \vdots & \vdots & \vdots & \ddots & \vdots \\
		0 & 0 & 0 & 0 & \cdots & 1 
	\end{matrix} \right)^{-1}, 
\end{equation*}
we have 
\begin{equation} \label{singanalyticfn_eqn2}
	D_R^{\gamma} D_P^{\delta} \left( (R^1_1 R^2_2 - R^1_2 R^2_1) f(M(R), PM(R)) \right) = (R^1_1 R^2_2 - R^1_2 R^2_1) f_{\gamma,\delta}(M(R), PM(R)) 
\end{equation}
whenever $|\gamma|+|\delta| = 1$ for the real analytic function $f_{\gamma,\delta} : \mathbb{R}^{n^2} \times \mathbb{R}^{mn} \rightarrow \mathbb{R}$ given by 
\begin{align} \label{singanalyticfn_eqn3}
	&f_{\gamma,\delta}(S,P) = f(S,P) S^i_{\kappa} - D_{S^l_j} f(S,P) S^i_j S^l_{\kappa} 
		- D_{P_j^{\lambda}} f(S,P) S^i_j P^{\lambda}_{\kappa} \text{ if } \gamma = e_i^{\kappa}, \, |\delta|=0, \nonumber \\
	&f_{\gamma,\delta}(S,P) = D_{P_j^{\kappa}} f(S,P) S^i_j \text{ if } |\gamma|=0, \, \delta = e_i^{\kappa}, 
\end{align}
for $S \in \mathbb{R}^{2n}$ and $P \in \mathbb{R}^{mn}$.  By differentiating (\ref{singanalyticfn_eqn3}) and using (\ref{singanalyticfn_eqn1}), if $|\gamma|+|\delta| = 1$ then 
\begin{equation} \label{singanalyticfn_eqn5}
	|D_S^{\alpha} D_P^{\beta} f_{\gamma,\delta}(S,P)| \leq (|\alpha|+|\beta|+|\gamma|+|\delta|+p+q)! 2^{2|\gamma|+|\delta|} 
		K_0 K^{2|\alpha|+|\beta|+2|\gamma|+|\delta|+2p+q} \prod_{i=1}^n |S^i|^{|\gamma_i|+|\delta_i|+p_i+q_i-|\alpha^i|} 
\end{equation} 
if $|\alpha| \leq 1+p+q$ and $D_S^{\alpha} f_{\gamma,\delta}(S,P) = 0$ if $|\alpha| > 1+p+q$ for all $S \in \mathbb{R}^{n^2}$, $P \in \mathbb{R}^{mn}$ with $|P| \leq a$, $\alpha$, and $\beta$.  By induction, we can show that for every $\gamma$ and $\delta$ there exists a real analytic function $f_{\gamma,\delta} : \mathbb{R}^{2n} \times \mathbb{R}^{mn} \rightarrow \mathbb{R}$ such that (\ref{singanalyticfn_eqn2}) holds true, (\ref{singanalyticfn_eqn5}) holds true if $|\alpha| \leq |\gamma|+|\delta|+p+q$ and $D_S^{\alpha} f_{\gamma,\delta}(S,P) = 0$ if $|\alpha| > |\gamma|+|\delta|+p+q$ for all $S \in \mathbb{R}^{n^2}$, $P \in \mathbb{R}^{mn}$ with $|P| \leq a$, $\alpha$, and $\beta$.  By (\ref{singanalyticfn_eqn2}) and (\ref{singanalyticfn_eqn5}), 
\begin{align} \label{singanalyticfn_eqn6}
	&\left| D_R^{\alpha} D_P^{\beta} \left( (R^1_1 R^2_2 - R^1_2 R^2_1) f(M(R), PM(R)) \right) \right| \nonumber \\
	&\leq (|\alpha|+|\beta|+p+q)! K_0 2^{2|\alpha|+|\beta|} K^{2|\alpha|+|\beta|+2p+q} 
	|R^1_1 R^2_2 - R^1_2 R^2_1| \left( \frac{|R_x| (1+|R_y|)}{|R^1_1 R^2_2 - R^1_2 R^2_1|} \right)^{\sum_{i=1}^2 (|\alpha_i|+|\beta_i|+p_i+q_i)} 
\end{align}
wherever $R^1_1 R^2_2 - R^1_2 R^2_1 \neq 0$, where $R_x = (R^{\kappa}_i)_{i \leq 2}$ and $R_y = (R^{\kappa}_i)_{i \geq 3}$. 

Now we can rewrite (\ref{transpde_diff1}) in the form 
\begin{equation} \label{analyticity_eqn3}
	\sum_{i=1}^n D_i \left( F^i_{k,\kappa}(D\phi,D\chi,D\psi) \right) = 0 
\end{equation}
on $B_{1/8}(0) \setminus \{0\} \times \mathbb{R}^{n-2}$ for $k = 0,1,2,\ldots,2\mathcal{N}-1$ and $\kappa = 1,2,\ldots,m$, where $\phi = (\phi^{\lambda})$, $\chi = (\chi_l^{\lambda})_{l \leq k}$, and $\psi$ consists of $\psi_l^{\lambda}$ for $l \leq k$ excluding $\psi_0^1, \psi_0^2$ in case (a) and excluding $\psi_{2\mathcal{N}-2}^1, \psi_{2\mathcal{N}-1}^1$ in cases (b) and (c) and where $F^i_{k,\kappa} : \mathbb{R}^{2n} \times \mathbb{R}^{kmn} \times \mathbb{R}^{kmn-2n} \rightarrow \mathbb{R}$ are functions that are smooth wherever $R^1_1 R^2_2 - R^1_2 R^2_1 \neq 0$.  By the real analyticity of $A^i_{\kappa}$ and (\ref{singanalyticfn_eqn6}) we may assume that 
\begin{align} \label{analyticity_eqn4a}
	&|D_{(R,P,Q)}^{\beta} F_{k,\kappa}^i(R,P,Q)| 
	\\&\leq \left\{ \begin{array}{ll} 
		K_0 K^{2|\beta_R|+|\beta_P|+|\beta_Q|} |R^1_1 R^2_2 - R^1_2 R^2_1| 
			\left( \frac{|R_x| (1+|R_y|)}{|R^1_1 R^2_2 - R^1_2 R^2_1|} \right)^{|\beta_x|+p_i} |Q| |M(R)| 
			& \text{if } |\beta| \leq 2\mathcal{N}+2, \\
		(|\beta|-2\mathcal{N}-2)! K_0 K^{2|\beta_R|+|\beta_P|+|\beta_Q|} |R^1_1 R^2_2 - R^1_2 R^2_1| 
			\left( \frac{|R_x| (1+|R_y|)}{|R^1_1 R^2_2 - R^1_2 R^2_1|} \right)^{|\beta_x|+p_i} 
			& \text{if } |\beta| \geq 2\mathcal{N}+3, 
	\end{array} \right. \nonumber 
\end{align}
if $|\beta_{Q_0}| = |\beta_Q|$, 
\begin{align} \label{analyticity_eqn4b}
	&|D_{(R,P,Q)}^{\beta} F_{k,\kappa}^i(R,P,Q)| 
	\\&\leq \left\{ \begin{array}{ll} 
		K_0 K^{2|\beta_R|+|\beta_P|+|\beta_Q|} |R^1_1 R^2_2 - R^1_2 R^2_1| \left( \frac{|R_x| (1+|R_y|)}{|R^1_1 R^2_2 - R^1_2 R^2_1|} 
			\right)^{\sum_{i=1}^2 (|\alpha_i|+|\beta_i|)+p_i} & \text{if } |\beta| \leq 2\mathcal{N}+2, \\
		(|\beta|-2\mathcal{N}-2)! K_0 K^{2|\beta_R|+|\beta_P|+|\beta_Q|} |R^1_1 R^2_2 - R^1_2 R^2_1| 
			\left( \frac{|R_x| (1+|R_y|)}{|R^1_1 R^2_2 - R^1_2 R^2_1|} \right)^{|\beta_x|+p_i} 
			& \text{if } |\beta| \geq 2\mathcal{N}+3, 
	\end{array} \right. \nonumber 
\end{align}
if either $|\beta_{Q_0}| < |\beta_Q|$ and $|\beta_{Q_k}| = 0$ or $|\beta_{Q_0}| + 1 = |\beta_Q|$ and $|\beta_{Q_k}| = 1$, and 
\begin{equation} \label{analyticity_eqn4c}
	D_{(R,P,Q)}^{\beta} F_{k,\kappa}^i(R,P,Q) = 0 
\end{equation}
if $|\beta_{Q_0}| + 1 < |\beta_Q|$ and $|\beta_{Q_k}| \geq 1$ for all $R \in \mathbb{R}^{2n}$, $P \in \mathbb{R}^{kmn}$ with $|P| \leq \sqrt{kmn}$, $Q \in \mathbb{R}^{kmn}$ with $|Q_l| \leq \sqrt{mn}$ for $l < k$, and $\beta$ for some constants $K_0,K \geq 1$ depending only on $n$, $m$, $\mathcal{N}$ (independent of $\beta$), where $p_i = 1$ if $i = 1,2$ and $p_i = 0$ if $i = 3,4,\ldots,n$ and we use the notation 
\begin{gather*}
	\beta = (\beta_R,\beta_P,\beta_Q), \quad \beta_R = (\beta_{R^{\lambda}_j}), \quad \beta_P = (\beta_{P^{\lambda}_{l,j}}), \quad 
		\beta_Q = (\beta_{Q^{\lambda}_{l,j}}), \\ 
	\beta_{Q_l} = (\beta_{Q^{\lambda}_{l,j}}) \text{ for each } l = 0,1,\ldots,2\mathcal{N}-1, \quad 
	\beta_x = (\beta_{R^{\lambda}_j},\beta_{P^{\lambda}_{l,j}},\beta_{Q^{\lambda}_{l,j}})_{j \leq 2}, \\ 
	D_{(R,P,Q)}^{\beta} = \prod D_{R^{\lambda}_j}^{\beta_{R^{\lambda}_j}} \cdot \prod D_{P^{\lambda}_{l,j}}^{\beta_{P^{\lambda}_{l,j}}} 
		\cdot \prod D_{Q^{\lambda}_{l,j}}^{\beta_{Q^{\lambda}_{l,j}}}.  
\end{gather*}
By the sum, product, and chain rule for derivatives,  
\begin{align} \label{analyticity_eqn5}
	&D_{\tilde X}^{\alpha} h^i_{\gamma,k,\kappa}(D\phi,D\chi,D\psi) 
	= \sum c_{\beta,\{k_l\},\{\kappa_l\},\{j_l\},\{\alpha_l\},\{\gamma_l\}} D_{(R,P,Q)}^{\beta} F_{k,\kappa}^i(D\phi,D\chi,D\psi) \nonumber \\& \cdot 
	\prod_{\kappa_l \leq 2} D_{j_l} D_{\tilde X}^{\alpha_l} D_{\tilde y}^{\gamma_l} \phi^{\kappa_l} 
	\prod_{2 < \kappa_l \leq m+2} D_{j_l} D_{\tilde X}^{\alpha_l} D_{\tilde y}^{\gamma_l} \chi_{k_l}^{\kappa_l-2} 
	\prod_{\kappa_l > m+2} D_{j_l} D_{\tilde X}^{\alpha_l} D_{\tilde y}^{\gamma_l} \psi_{k_l}^{\kappa_l-m-2} 
\end{align}
where $c_{\beta,\{k_l\},\{\kappa_l\},\{j_l\},\{\gamma_l\}}$ are nonnegative integers and the sum is over $2 \leq |\beta| \leq |\alpha|+|\gamma|$ and $k_l$, $\kappa_l$, $j_l$, and $\gamma_l$ for $l = 1,2,\ldots,|\beta|$ such that 
\begin{equation*}
	\sum_{\kappa_l \leq 2} e^{\kappa_l} = \beta_R, \quad 
	\sum_{2 < \kappa_l \leq m+2} e_{j_l}^{\kappa_l-2} = \beta_P, \quad 
	\sum_{\kappa_l > m+2} e_{j_l}^{\kappa_l-m-2} = \beta_Q, \quad 
	\sum_{l=1}^{|\beta|} \alpha_l = \alpha, \quad \sum_{l=1}^{|\beta|} \gamma_l = \gamma. 
\end{equation*}
Note that we are using the convention that sums over empty sets to equal zero and products over empty sets to equal one.  We shall take $k_l = 0$ if $\kappa_l \leq 2$ and assume that in (\ref{analyticity_eqn5}) that if $c_{\beta,\{k_l\},\{\kappa_l\},\{j_l\},\{\alpha_l\},\{\gamma_l\}} \neq 0$ then $c_{\beta,\{k'_l\},\{\kappa'_l\},\{j'_l\},\{\alpha'_l\},\{\gamma'_l\}}$ for every permutation $\{(\kappa'_l,j'_l,\alpha'_l,\gamma'_l)\}$ of $\{(\kappa_l,j_l,\alpha_l,\gamma_l)\}$.  

Suppose that $\Phi : \mathbb{R} \rightarrow \mathbb{R}$ and $v : \mathbb{R} \rightarrow \mathbb{R}$ are functions, called \textit{majorants}, such that $\Phi(0) = 0$, $D\Phi(0) = 0$, 
\begin{equation} \label{analyticity_eqn6}
	|D_{(R,P,Q)}^{\beta} F_{k,\kappa}^i(R,P,Q)| 
	\leq \left\{ \begin{array}{l} 
		D^{|\beta|} \Phi(0) |R^1_1 R^2_2 - R^1_2 R^2_1| \left( \frac{|R_x| (1+|R_y|)}{|R^1_1 R^2_2 - R^1_2 R^2_1|} \right)^{|\beta_x|+p_i} |Q M(R)| 
			\text{ if } |\beta_{Q_0}| = |\beta_Q|, \\
		0 \text{ if } |\beta_{Q_0}| + 1 < |\beta_Q|, \, |\beta_{Q_k}| \geq 1, \\ 
		D^{|\beta|} \Phi(0) |R^1_1 R^2_2 - R^1_2 R^2_1| \left( \frac{|R_x| (1+|R_y|)}{|R^1_1 R^2_2 - R^1_2 R^2_1|} \right)^{|\beta_x|+p_i} 
			\text{ otherwise, }
	\end{array} \right. \nonumber 
\end{equation}
for all $R \in \mathbb{R}^{2n}$, $P \in \mathbb{R}^{kmn}$ with $|P| \leq \sqrt{kmn}$, $Q \in \mathbb{R}^{kmn}$ with $|Q_l| \leq \sqrt{mn}$ for $l < k$, and $\beta$ and 
\begin{align} \label{analyticity_eqn7}
	&\sum_{\kappa=1}^2 \vvvert D_{\tilde y}^{\gamma'} \phi^{\kappa} \vvvert_{2,2\mathcal{N},\tau,B^g_{\rho/s}(0)} 
	+ \sum_{\kappa=1}^m \sum_{k=0}^{2\mathcal{N}-1} \vvvert D_{\tilde y}^{\gamma'} \chi^{\kappa}_k \vvvert_{2,2\mathcal{N},\tau,B^g_{\rho/s}(0)} \nonumber \\
	&+ \sum_{\kappa=1}^m \sum_{k=0}^{2\mathcal{N}-1} \vvvert D_{\tilde y}^{\gamma'} \psi^{\kappa}_k 
		\vvvert_{2\mathcal{N}-2\lceil k/2 \rceil,2\mathcal{N},\tau,B^g_{\rho/s}(0)} 
	\leq D^{|\gamma'|} v(0) \rho^{-|\gamma'|}
\end{align}
whenever $|\gamma'| < s$.  By taking the derivative of the composition functions formed by $\Phi(\eta_1+\eta_2+\cdots+\eta_{2\mathcal{N}mn})$, $\eta_j = e^{\xi_1+\xi_2+\cdots+\xi_n} v(\zeta_1+\zeta_2+\cdots+\zeta_{n-2}) - v(0)$ for all $j$, $\xi_j = \xi$ for all $j$, and $\zeta_j = \zeta$ for all $j$, 
\begin{equation} \label{analyticity_eqn8}
	\left. D_{\xi}^{|\alpha|} D_{\zeta}^{|\gamma|} \left( \Phi(2\mathcal{N}mn e^{n\xi} v((n-2)\zeta) - v(0)) \right) \right|_{\xi=\zeta=0} 
	= \sum c_{\beta,\{k_l\},\{\kappa_l\},\{j_l\},\{\alpha_l\},\{\gamma_l\}} D^{|\beta|} \Phi(0) \prod_{l=1}^{|\beta|} D^{|\gamma_l|} v(0) 
\end{equation}
where the sum is over $\beta$, $k_l$, $\kappa_l$, $j_l$, and $\gamma_l$ as in (\ref{analyticity_eqn5}) and $c_{\beta,\{k_l\},\{\kappa_l\},\{j_l\},\{\gamma_l\}}$ are as in (\ref{analyticity_eqn5}).  By comparing (\ref{analyticity_eqn5}) and (\ref{analyticity_eqn8}) using (\ref{analyticity_eqn6}) and (\ref{analyticity_eqn7}) and simplifying, 
\begin{align} \label{analyticity_eqn9}
	&|\tilde x|^{-2\mathcal{N}+2\lceil k/2 \rceil+2+|\alpha_x|+2|\alpha_y|} |D_{\tilde X}^{\alpha} D_i h^i_{\gamma,k,\kappa}(D\phi,D\chi,D\psi)| \nonumber \\
	&\leq \rho^{-|\gamma|} \left. D_{\xi}^{|\alpha|} D_{\zeta}^{|\gamma|} \left( \Phi(4\mathcal{N}mn (e^{\xi} v(\zeta) - v(0))) \right) \right|_{\xi=\zeta=0},
\end{align}
where $\alpha_x = (\alpha_1,\alpha_2)$ and $\alpha_y = (\alpha_3,\alpha_4,\ldots,\alpha_n)$.  By the properties of H\"{o}lder coefficients and by taking the differences between (\ref{analyticity_eqn5}) and (\ref{analyticity_eqn5}) with $\phi( \, ; \tilde y_0)$, $\chi^{\kappa}_l( \, ; \tilde y_0)$, and $\psi^{\kappa}_l( \, ; \tilde y_0)$ in place of $\phi$, $\chi^{\kappa}_l$, and $\psi^{\kappa}_l$ respectively, one can similarly argue using (\ref{analyticity_eqn6}) and (\ref{analyticity_eqn7}) that 
\begin{align} \label{analyticity_eqn10}
	&|\tilde x|^{-2\mathcal{N}+2\lceil k/2 \rceil+2+|\alpha_x|+2|\alpha_y|+2\tau} 
		[D_{\tilde X}^{\alpha} D_i h^i_{\gamma,k,\kappa}(D\phi,D\chi,D\psi)]_{g,\tau,B^g_{|\tilde x|^2/4}(\tilde X)} \nonumber \\
		&\leq \rho^{-|\gamma|} \left. D_{\xi}^{|\alpha|+1} D_{\zeta}^{|\gamma|} \left( \Phi(4\mathcal{N}mn (e^{\xi} v(\zeta) - v(0))) \right) 
			\right|_{\xi=\zeta=0}, \nonumber \\
	&|D_{\tilde X}^{\alpha} D_i h^i_{\gamma,k,\kappa}(D\phi,D\chi,D\psi) 
			- D_{\tilde X}^{\alpha} D_i h^i_{\gamma,k,\kappa}(D\phi( \, ;\tilde z),D\chi( \, ;\tilde z),D\psi( \, ;\tilde z))| \nonumber \\
		&\leq \left. (\rho/s)^{-\tau} |\tilde x|^{2\mathcal{N}-2\lceil k/2 \rceil-2-|\alpha_x|-2|\alpha_y|} d(\tilde X,\tilde z)^{\tau} \cdot 
			\rho^{-|\gamma|} D_{\xi}^{|\alpha|+1} D_{\zeta}^{|\gamma|} \left( \Phi(4\mathcal{N}mn (e^{\xi} v(\zeta) - v(0))) \right) \right|_{\xi=\zeta=0}, 
			\nonumber \\
	&[D_{\tilde X}^{\alpha} D_i h^i_{\gamma,k,\kappa}(D\phi,D\chi,D\psi) - D_{\tilde X}^{\alpha} D_i h^i_{\gamma,k,\kappa}(D\phi( \, ;\tilde z),
			D\chi( \, ;\tilde z),D\psi( \, ;\tilde z))]_{g,\tau,B^g_{|\tilde x|^2/4}(\tilde X)} \nonumber \\
		&\leq \left. (\rho/s)^{-\tau} |\tilde x|^{2\mathcal{N}-2\lceil k/2 \rceil-2-|\alpha_x|-2|\alpha_y|} \cdot \rho^{-|\gamma|} 
			D_{\xi}^{|\alpha|+2} D_{\zeta}^{|\gamma|} \left( \Phi(4\mathcal{N}mn (e^{\xi} v(\zeta) - v(0))) \right) \right|_{\xi=\zeta=0}, 
\end{align}
where $\tilde X \in B^g_{\rho/s}(0,\tilde y_0)$, $\tilde z \in B^{n-2}_{\rho/s}(0,\tilde y_0)$, $|\alpha| \leq 2\mathcal{N}-2$, and in the last equation we require that $d_g(\tilde X,(0,\tilde z)) \leq 4|\tilde x|^2$.  By the same argument except we write (\ref{transpde_diff2}) as (\ref{analyticity_eqn3}) and in place of (\ref{analyticity_eqn6}) we assume 
\begin{align*}
	&|D_{(R,P,Q)}^{\beta} F_{k,\kappa}^i(R,P,Q)| 
	\\&\leq \left\{ \begin{array}{l} 
		D^{|\beta|} \Phi(0) |R^1_1 R^2_2 - R^1_2 R^2_1| \left( \frac{|R_x| (1+|R_y|)}{|R^1_1 R^2_2 - R^1_2 R^2_1|} \right)^{|\beta_x|} |Q_0 M(R)| |Q_k M(R)| 
			\text{ if } |\beta_Q| = 0, \\
		D^{|\beta|} \Phi(0) |R^1_1 R^2_2 - R^1_2 R^2_1| \left( \frac{|R_x| (1+|R_y|)}{|R^1_1 R^2_2 - R^1_2 R^2_1|} \right)^{|\beta_x|} |Q_k M(R)| 
			\text{ if } |\beta_Q| = |\beta_{Q_0}| \geq 1, \\
		D^{|\beta|} \Phi(0) |R^1_1 R^2_2 - R^1_2 R^2_1| \left( \frac{|R_x| (1+|R_y|)}{|R^1_1 R^2_2 - R^1_2 R^2_1|} \right)^{|\beta_x|} |Q_0 M(R)| 
			\text{ if } |\beta_Q| = |\beta_{Q_k}| = 1, \\
		D^{|\beta|} \Phi(0) |R^1_1 R^2_2 - R^1_2 R^2_1| \left( \frac{|R_x| (1+|R_y|)}{|R^1_1 R^2_2 - R^1_2 R^2_1|} \right)^{|\beta_x|} 
			\text{ if } |\beta_Q| = |\beta_{Q_0}| + 1 \geq 2, \, |\beta_{Q_k}| = 1, \\
		D^{|\beta|} \Phi(0) |R^1_1 R^2_2 - R^1_2 R^2_1| \left( \frac{|R_x| (1+|R_y|)}{|R^1_1 R^2_2 - R^1_2 R^2_1|} \right)^{|\beta_x|} 
			\text{ if } |\beta_{Q_0}| < |\beta_Q|, \, |\beta_{Q_k}| = 0, \\
		0 \text{ if } |\beta_{Q_k}| \geq 2,
	\end{array} \right. 
\end{align*}
for all $R \in \mathbb{R}^{2n}$, $P \in \mathbb{R}^{kmn}$ with $|P| \leq \sqrt{kmn}$, $Q \in \mathbb{R}^{kmn}$ with $|Q_l| \leq \sqrt{mn}$ for $l < k$, and $\beta$, (\ref{analyticity_eqn9}) and (\ref{analyticity_eqn10}) continue to hold true if we interchange $D\chi$ and $D\psi$ and replace each $-2\mathcal{N}+2\lceil k/2 \rceil+2$ with $0$.  By (\ref{analyticity_eqn1}), (\ref{analyticity_eqn4a}), (\ref{analyticity_eqn4b}), and (\ref{analyticity_eqn4c}), we can take 
\begin{equation*}
	\Phi(t) = \sum_{k=2}^{2\mathcal{N}+2} K_0 K^{2k} t^k + \sum_{k=2\mathcal{N}+3}^{\infty} \frac{K_0 K^{2k}}{(k-2\mathcal{N}-2)^{2\mathcal{N}+2}} t^k, \quad
	v(\zeta) = H_0 \zeta + \sum_{k=2}^{|\gamma|} \frac{H_0 H^{k-2}}{(k-1)^2} \zeta^k.  
\end{equation*}

Given any functions $f : \mathbb{R} \rightarrow \mathbb{R}$ and $g : \mathbb{R}^2 \rightarrow \mathbb{R}$ with $f(0) = 0$ and $g(0) = 0$, 
\begin{equation*}
	D^{\alpha} (f \circ g)(0) = \sum_{k=1}^{|\alpha|} D^k f(0) P_k(\{D^{\beta} g(0)\}_{\beta \leq \alpha})
\end{equation*}
for some polynomials $P_k : \mathbb{R}^{\prod \alpha_j} \rightarrow \mathbb{R}$ that are independent of $f$ and $g$.  Since $f$ is arbitrary, we may take $f(t) = t^j$ for each integer $j$ and thereby deduce that 
\begin{equation*}
	D^{\alpha} (f \circ g)(0) = \sum_{k=1}^{|\alpha|} \frac{1}{k!} D^k f(0) D^{\alpha}(g^k)(0). 
\end{equation*}
Letting $f(t) = \Phi(4\mathcal{N}mn t)$ and $g(\xi,\zeta) = e^{n\xi} v((n-2)\zeta) - v(0)$ yields 
\begin{equation} \label{analyticity_eqn11}
	\left. D_{\xi}^j D_{\zeta}^{|\gamma|} \Phi(4\mathcal{N}mn (e^{\xi} v(\zeta) - v(0))) \right|_{\xi=\zeta=0} 
	= \sum_{k=2}^{|\gamma|} \frac{(4\mathcal{N}mn)^k k^j}{k!} D^k \Phi(0) D_{\zeta}^{|\gamma|} (v^k)(0) 
\end{equation}
for $j = 0,1,2,\ldots,2\mathcal{N}$. 

Observe that 
\begin{equation} \label{analyticity_eqn12}
	\frac{1}{k!} D_{\zeta}^k (v^k)(0) = H_0^k 
\end{equation}
for $k = 1,2,\ldots,|\gamma|$.  We claim that 
\begin{equation} \label{analyticity_eqn13}
	\frac{1}{j!} D_{\zeta}^j (v^k)(0) \leq \frac{6^{k-1} H_0^k H^{j-k-1}}{(j-k)^2} 
\end{equation}
for $1 \leq k < j \leq |\gamma|$.  We can prove (\ref{analyticity_eqn13}) by induction on $k$, observing that (\ref{analyticity_eqn13}) holds true when $k = 1$ and using the induction step 
\begin{equation*}
	\frac{1}{(k+1)!} D_{\zeta}^{k+1} (v^k)(0) 
	= \frac{1}{(k+1)!} \sum_{l=1}^2 \frac{(k+1)!}{l! (k+1-l)!} D_{\zeta}^l v(0) D_{\zeta}^{k+1-l} (v^{k-1})(0) 
	\leq 2 \cdot 6^{k-2} H_0^k 
\end{equation*}
for $2 \leq k < |\gamma|$ and 
\begin{align*}
	\frac{1}{j!} D_{\zeta}^j (v^k)(0) 
	&= \frac{1}{j!} \sum_{l=1}^{j-k+1} \frac{j!}{l! (j-l)!} D_{\zeta}^l v(0) D_{\zeta}^{j-l} (v^{k-1})(0) 
	\\&\leq 6^{k-2} H_0^k \left( \frac{2}{(j-k)^2} H^{j-k-1} + \sum_{l=2}^{j-k} \frac{1}{(l-1)^2 (j-l-k+1)^2} H^{j-k-2} \right) 
	\\&\leq \frac{6^{k-1} H_0^k H^{j-k-1}}{(j-k)^2} 
\end{align*}
for $2 \leq k$ and $k+2 \leq j \leq |\gamma|$ using the fact that 
\begin{equation} \label{analyticity_eqn14}
	\sum_{l=1}^{N-1} \frac{1}{l^2 (N-l)^2} 
	= \sum_{l=1}^{N-1} \frac{1}{N^2} \left( \frac{1}{l^2 (N-l)^2} + \frac{1}{l^2 (N-l)^2} \right)^2 
	\leq \frac{4}{N^2} \sum_{l=1}^{N-1} \frac{1}{l^2} 
	\leq \frac{2\pi^2}{3N^2} 
	\leq \frac{6}{N^2} 
\end{equation}
for every integer $N \geq 2$.  By (\ref{analyticity_eqn11}), (\ref{analyticity_eqn12}), (\ref{analyticity_eqn13}), and (\ref{analyticity_eqn14}), 
\begin{align} \label{analyticity_eqn15}
	&\left. D_{\xi}^j D_{\zeta}^{|\gamma|} \Phi(4\mathcal{N}mn (e^{\xi} v(\zeta) - v(0))) \right|_{\xi=\zeta=0} \nonumber \\
	&\leq C(\mathcal{N}) |\gamma|! K_0 H^{|\gamma|-1} \left( \sum_{k=2}^{2\mathcal{N}+2} \frac{1}{(|\gamma|-k)^2} 
		\left( \frac{24\mathcal{N}mn K^2 H_0}{H} \right)^k \right. \nonumber \\& \hspace{5mm} \left. 
		+ \sum_{k=2\mathcal{N}+3}^{|\gamma|-1} \frac{1}{(k-2\mathcal{N}-2)^2 (|\gamma|-k)^2} \left( \frac{24\mathcal{N}mn K^2 H_0}{H} \right)^k 
		+ \frac{1}{(|\gamma|-2\mathcal{N}-2)^2} \left( \frac{24\mathcal{N}mn K^2 H_0}{H} \right)^{|\gamma|} \right) \nonumber \\ 
	&\leq C(n,m,\mathcal{N}) (|\gamma|-2)! K_0 K^4 H_0^2 H^{|\gamma|-3} 
\end{align}
for $j = 0,1,2,\ldots,2\mathcal{N}$ provided $H > 24\mathcal{N}mn K^2 H_0$. 

By (\ref{analyticity_eqn1}), (\ref{analyticity_eqn2}), (\ref{analyticity_eqn9}), (\ref{analyticity_eqn10}), and (\ref{analyticity_eqn15}), 
\begin{align} \label{analyticity_eqn16}
	&\sum_{\kappa=1}^2 \vvvert D_{\tilde y}^{\gamma} \phi^{\kappa} \vvvert_{2,2\mathcal{N},\tau,B^g_{\rho/2s}(0)} 
	+ \sum_{\kappa=1}^m \sum_{k=0}^{2\mathcal{N}-1} \vvvert D_{\tilde y}^{\gamma} \chi^{\kappa}_k \vvvert_{2,2\mathcal{N},\tau,B^g_{\rho/2s}(0)} 
		+ \sum_{\kappa=1}^m \sum_{k=0}^{2\mathcal{N}-1} \vvvert D_{\tilde y}^{\gamma} \psi^{\kappa}_k 
		\vvvert_{2\mathcal{N}-2\lceil k/2 \rceil,2\mathcal{N},\tau,B^g_{\rho/2s}(0)} \nonumber \\
	&\leq C \left( (s-2)! H_0 H^{s-3} \rho^{-s-1} + (s-2)! K_0 K^4 H_0^2 H^{s-3} \rho^{-s} \right) 
\end{align}
for all $B^g_{\rho}(0,\tilde y_0) \subseteq B^g_{1/8}(0)$ and for some $C = C(n,m,\mathcal{N}) \in (0,\infty)$. 

To complete the proof, let $B^g_{\rho}(0,\tilde y_0) \subseteq B^g_{1/8}(0)$.  If $\tilde X = (\tilde x,\tilde y) \in B^g_{\rho/s}(0,\tilde y_0)$ then $B^g_{(1-2/s)\rho}(0,\tilde y) \subseteq B^g_{1/8}(0)$, so by (\ref{analyticity_eqn16}) and the fact that $(1+1/j)^j \leq e$ for $j = s-2$, 
\begin{align} \label{analyticity_eqn17}
	(\rho/s)^{-\mathcal{N}-\lceil k/2 \rceil-|\alpha|/2-|\beta|} |D_{\tilde x}^{\alpha} D_{\tilde y}^{\beta+\gamma} \psi_k^{\kappa}(\tilde X)| 
	&\leq C (1 + K_0 K^4 H_0) (s-2)! H_0 H^{s-3} \left( (1-2/s) \rho \right)^{-s} \nonumber \\
	&\leq C (1 + K_0 K^4 H_0) (s-2)! H_0 H^{s-3} \rho^{-s-1} 
\end{align}
whenever $|\alpha|+2|\beta| \leq 2\mathcal{N}-2\lceil k/2 \rceil$.  Similarly, 
\begin{align} \label{analyticity_eqn18}
	(\rho/s)^{\tau} |\tilde x|^{-2\mathcal{N}-2\lceil k/2 \rceil-2\tau} |D_{\tilde x}^{\alpha} D_{\tilde y}^{\beta+\gamma} \psi_k^{\kappa}(\tilde X)| 
		&\leq C (1 + K_0 K^4 H_0) (s-2)! H_0 H^{s-3} \rho^{-s-1}, \nonumber \\
	(\rho/s)^{\tau} |\tilde x|^{-2\mathcal{N}-2\lceil k/2 \rceil} [D_{\tilde x}^{\alpha} D_{\tilde y}^{\beta+\gamma} \psi_k^{\kappa}
		]_{g,\tau,B^g_{|\tilde x|^2/4}(\tilde X)} &\leq C (1 + K_0 K^4 H_0) (s-2)! H_0 H^{s-3} \rho^{-s-1}, 
\end{align}
whenever $|\alpha|+2|\beta| > 2\mathcal{N}-2\lceil k/2 \rceil$ and $|\alpha|+|\beta| \leq 2\mathcal{N}$.  Next suppose that $\tilde X,\tilde X' \in B^g_{\rho/s}(0,\tilde y_0)$.  If $d_g(X,X') \leq (s-3) \rho/2s^2$, then by (\ref{analyticity_eqn16}) and the fact that $(1+1/j)^j \leq e$ for $j = s-2$, 
\begin{align} \label{analyticity_eqn19}
	(\rho/s)^{\tau} \frac{|D_{\tilde x}^{\alpha} D_{\tilde y}^{\beta+\gamma} \psi_k^{\kappa}(\tilde X) 
		- D_{\tilde x}^{\alpha} D_{\tilde y}^{\beta+\gamma} \psi_k^{\kappa}(\tilde X')|}{d_g(X,X')^{\tau}}
	&\leq (\rho/s)^{\tau} [D_{\tilde x}^{\alpha} D_{\tilde y}^{\beta+\gamma} \psi_k^{\kappa}]_{g,\tau,B^g_{(s-3) \rho/2s^2}(0,\tilde y)} \nonumber \\
	&\leq C (1 + K_0 K^4 H_0) (s-2)! H_0 H^{s-3} (1-2/s)^{-s-\tau} \rho^{-s} \nonumber \\
	&\leq C (1 + K_0 K^4 H_0) (s-2)! H_0 H^{s-3} \rho^{-s-1} 
\end{align}
whenever $|\alpha|+2|\beta| = 2\mathcal{N}-2\lceil k/2 \rceil$.  If instead $d_g(X,X') > (s-3) \rho/2s^2$ then by (\ref{analyticity_eqn17}),
\begin{align} \label{analyticity_eqn20}
	(\rho/s)^{\tau} \frac{|D_{\tilde x}^{\alpha} D_{\tilde y}^{\beta+\gamma} \psi_k^{\kappa}(\tilde X) 
		- D_{\tilde x}^{\alpha} D_{\tilde y}^{\beta+\gamma} \psi_k^{\kappa}(\tilde X')|}{d_g(X,X')^{\tau}} 
	&\leq 2^{\tau} (1-3/s)^{\tau} \sup_{B^g_{\rho}(0,\tilde y_0)} |D_{\tilde x}^{\alpha} D_{\tilde y}^{\beta+\gamma} \psi_k^{\kappa}| \nonumber \\
	&\leq C (1 + K_0 K^4 H_0) (s-2)! H_0 H^{s-3} \rho^{-s-1}. 
\end{align}
By (\ref{analyticity_eqn17}), (\ref{analyticity_eqn17}), (\ref{analyticity_eqn17}), and (\ref{analyticity_eqn17}) and similar computations for $\phi^{\kappa}$ and $\chi_k^{\kappa}$, we obtain (\ref{analyticity_eqn1}) for $|\gamma| = s$.  Therefore (\ref{analyticity_eqn1}) holds true for all $\gamma$, completing the proof that $\mathcal{B}_u$ is a real analytic $(n-2)$-dimensional submanifold near the origin.  

We claim that by modifying the same argument one can show that, letting 
\begin{equation*}
	\|\psi\|'_{C^{k,\tau}(B^g_{\rho}(\tilde X_0))} = \sum_{|\alpha|+|\beta| \leq k} \frac{\rho^{|\alpha|+|\beta|}}{|\tilde x_0|^{|\alpha|}} 
		\sup_{B^g_{\rho/2}(\tilde X_0)} |D_{\tilde x}^{\alpha} D_{\tilde y}^{\beta} \psi| 
	+ \sum_{|\alpha|+|\beta| = k} \frac{\rho^{|\alpha|+|\beta|+\tau}}{|\tilde x_0|^{|\alpha|}} 
		[D_{\tilde x}^{\alpha} D_{\tilde y}^{\beta} \psi]_{g,\tau,B^g_{\rho/2}(\tilde X_0)} 
\end{equation*}
for $\psi \in C^{k,\tau}(B^g_{\rho}(\tilde X_0))$ where $\tilde X_0 = (\tilde x_0,\tilde y_0)$ and $\rho \leq |\tilde x_0|^2/4$, 
\begin{align} \label{analyticity_eqn21}
	&\sum_{\kappa=1}^2 \frac{1}{|\tilde x_0|^2} \|D_{\tilde y}^{\gamma} \phi^{\kappa}\|'_{C^{2\mathcal{N},\tau}(B^g_{\rho/|\gamma|}(\tilde X_0))}
	+ \sum_{\kappa=1}^m \sum_{k=0}^{2\mathcal{N}-1} \frac{1}{|\tilde x_0|^2} \|D_{\tilde y}^{\gamma} \chi_k^{\kappa}\|'_{C^{2\mathcal{N},\tau}(B^g_{\rho/|\gamma|}(\tilde X_0))} 
		\nonumber \\
	&+ \sum_{\kappa=1}^m \sum_{k=0}^{2\mathcal{N}-1} \frac{1}{|\tilde x_0|^{2\mathcal{N}-2\lceil k/2 \rceil}} 
		\|D_{\tilde y}^{\gamma} \psi_k^{\kappa}\|'_{C^{2\mathcal{N},\tau}(B^g_{\rho/|\gamma|}(\tilde X_0))}
	\leq \left\{ \begin{array}{ll} 
		H_0 \rho^{-|\gamma|} & \text{if } |\gamma| \leq 2, \\
		(|\gamma|-2)! H_0 H^{|\gamma|-2} \rho^{-|\gamma|} & \text{if } |\gamma| > 2. 
	\end{array} \right. 
\end{align}
for all $B^g_{\rho}(\tilde X_0) \subseteq B^g_{1/80}(0)$ with $\tilde X_0 = (\tilde x_0,\tilde y_0)$ and $\rho/|\gamma| \leq |\tilde x_0|^2/8$ and $\gamma$ with $|\gamma| \geq 1$ for some constants $H_0,H \geq 1$ depending only on $n$, $m$, and $\mathcal{N}$ (independent of $\tilde y_0$ and $\gamma$).  In particular, we again argue by induction on $|\gamma|$, assuming that for some $s \geq \max\{8,2\mathcal{N}+4\}$, (\ref{analyticity_eqn21}) holds true when $|\gamma| < s$ and then proving (\ref{analyticity_eqn21}) for $|\gamma| = s$.  In place of using Lemma \ref{schaudersys_afbs} to obtain (\ref{analyticity_eqn2}) we use the Schauder estimate away from $\{0\} \times \mathbb{R}^{n-2}$ of Lemma \ref{schaudersys_afbs} instead of Corollary \ref{schaudersyscor} to obtain an estimate bounding $\phi^{\kappa}$, $\chi_k^{\kappa}$, and $\psi_k^{\kappa}$ on $B^g_{\rho/s}(\tilde X_0)$.  We claim that by (\ref{analyticity_eqn1}) and (\ref{analyticity_eqn21}), 
\begin{align} \label{analyticity_eqn22}
	&\sum_{\kappa=1}^2 \frac{1}{|\tilde x_0|^2} \|D_{\tilde y}^{\gamma'} \phi^{\kappa}\|'_{C^{2\mathcal{N},\tau}(B^g_{\rho/s}(\tilde X_0))}
	+ \sum_{\kappa=1}^m \sum_{k=0}^{2\mathcal{N}-1} \frac{1}{|\tilde x_0|^2} \|D_{\tilde y}^{\gamma'} \chi_k^{\kappa}\|'_{C^{2\mathcal{N},\tau}(B^g_{\rho/s}(\tilde X_0))} 
		\nonumber \\
	&+ \sum_{\kappa=1}^m \sum_{k=0}^{2\mathcal{N}-1} \frac{1}{|\tilde x_0|^{2\mathcal{N}-2\lceil k/2 \rceil}} 
		\|D_{\tilde y}^{\gamma'} \psi_k^{\kappa}\|'_{C^{2\mathcal{N},\tau}(B^g_{\rho/s}(\tilde X_0))}
	\leq \left\{ \begin{array}{ll} 
		C H_0 \rho^{-|\gamma'|} & \text{if } |\gamma'| \leq 2, \\
		(|\gamma'|-2)! C H_0 H^{|\gamma'|-2} \rho^{-|\gamma'|} & \text{if } |\gamma'| > 2. 
	\end{array} \right. 
\end{align}
for every $\gamma'$ with $1 \leq |\gamma'| < s$ for some constant $C = C(n,m,\mathcal{N}) \in (0,\infty)$.  If $\rho/|\gamma'| \leq |\tilde x_0|^2/8$, then (\ref{analyticity_eqn22}) follows from (\ref{analyticity_eqn21}) with $\gamma'$ in place of $\gamma$.  If instead $\rho/|\gamma'| > |\tilde x_0|^2/8$, then  $B^g_{\rho/|\gamma'|}(\tilde X_0) \subset B^g_{9\rho/|\gamma'|}(0,\tilde y_0)$ and $B^g_{\rho}(\tilde X_0) \subseteq B^g_{1/80}(0)$ implies that $B^g_{9\rho}(0,\tilde y_0) \subseteq B^g_{1/8}(0)$, so (\ref{analyticity_eqn22}) follows from (\ref{analyticity_eqn1}) with $\gamma'$ and $B^g_{9\rho}(0,\tilde y_0)$ in place of $\gamma$ and $B^g_{\rho}(0,\tilde y_0)$ respectively and Remark \ref{schauder_rmk}.  Now by the argument involving the majorants $\Phi$ and $v$ above using (\ref{analyticity_eqn22}) and with obvious modifications, one can bound the necessary terms in the Schauder estimate to prove (\ref{analyticity_eqn21}) for $|\gamma| = s$.  Therefore (\ref{analyticity_eqn21}) holds true for all $\gamma$.  

By combining (\ref{analyticity_eqn1}) and (\ref{analyticity_eqn21}), 
\begin{align} \label{analyticity_eqn23}
	&\sum_{\kappa=1}^2 \sum_{|\alpha| \leq 2} \sup_{B^g_{1/160}(0)} |\tilde x|^{-2+|\alpha|} 
		|D_{\tilde x}^{\alpha} D_{\tilde y}^{\gamma} \phi^{\kappa}(\tilde X)| 
	+ \sum_{\kappa=1}^2 \sum_{|\alpha| = 2} (160 |\gamma|)^{\tau} 
		[D_{\tilde x}^{\alpha} D_{\tilde y}^{\gamma} \phi^{\kappa}]_{g,\tau,B^g_{1/160}(0)} \nonumber 
	\\&+ \sum_{\kappa=1}^m \sum_{k=0}^{2\mathcal{N}-1} \sum_{|\alpha| \leq 2} \sup_{B^g_{1/160}(0)} |\tilde x|^{-2+|\alpha|} 
		|D_{\tilde x}^{\alpha} D_{\tilde y}^{\gamma} \chi_k^{\kappa}(\tilde X)| 
	+ \sum_{\kappa=1}^m \sum_{k=0}^{2\mathcal{N}-1} \sum_{|\alpha| = 2} (160 |\gamma|)^{\tau} 
		[D_{\tilde x}^{\alpha} D_{\tilde y}^{\gamma} \chi_k^{\kappa}]_{g,\tau,B^g_{1/160}(0)} \nonumber 
	\\&+ \sum_{\kappa=1}^m \sum_{k=0}^{2\mathcal{N}-1} \sum_{|\alpha| \leq 2\mathcal{N}-2\lceil k/2 \rceil} \sup_{B^g_{1/160}(0)} 
		|\tilde x|^{-2\mathcal{N}+2\lceil k/2 \rceil+|\alpha|} |D_{\tilde x}^{\alpha} D_{\tilde y}^{\gamma} \psi_k^{\kappa}(\tilde X)| \nonumber 
	\\&+ \sum_{\kappa=1}^m \sum_{k=0}^{2\mathcal{N}-1} \sum_{|\alpha| = 2\mathcal{N}-2\lceil k/2 \rceil} (160 |\gamma|)^{\tau} 
		[D_{\tilde x}^{\alpha} D_{\tilde y}^{\gamma} \psi_k^{\kappa}]_{g,\tau,B^g_{1/160}(0)} 
	\leq \left\{ \begin{array}{ll} 
		H_0 160^{|\gamma|} & \text{if } |\gamma| \leq 2, \\
		(|\gamma|-2)! H_0 H^{|\gamma|-2} 160^{|\gamma|} & \text{if } |\gamma| > 2, 
	\end{array} \right. 
\end{align}
for all $\gamma$ for some constants $H_0,H \geq 1$ depending only on $n$, $m$, and $\mathcal{N}$ (independent of $\gamma$).  We can interpret (\ref{analyticity_eqn23}) as implying that $\phi^{\kappa}(\tilde x,\tilde y)$, $\chi^{\kappa}(\tilde x,\tilde y)$, and $\psi^{\kappa}(\tilde x,\tilde y)$ are locally real analytic with respect to $\tilde y$ with bounds on $\|D_{\tilde y}^{\gamma} \phi^{\kappa}\|_{C^{g,2,\tau}(B^g_{1/160}(0))}$, $\|D_{\tilde y}^{\gamma} \chi_k^{\kappa}\|_{C^{g,2,\tau}(B^g_{1/160}(0))}$, and $\|D_{\tilde y}^{\gamma} \psi_k^{\kappa}\|_{C^{g,2\mathcal{N}-2\lceil k/2 \rceil,\tau}(B^g_{1/160}(0))}$ given by (\ref{analyticity_eqn23}).  

\end{document}